\documentclass[12pt,twoside]{amsbook}
\usepackage{amsmath}
\usepackage{amssymb}
\usepackage{amsthm}
\usepackage{mathptmx}
\usepackage{txfonts}
\usepackage{yhmath}
\usepackage{multirow}
\usepackage{bigdelim}
\usepackage{url}
\usepackage[arrow, matrix, curve]{xy}

\usepackage[latin1,ansinew]{inputenc}

\usepackage[T1]{fontenc}

\usepackage{textcomp}

\usepackage{array}

\usepackage[pdftex]{graphicx}

\graphicspath{{Bilder/}}

\usepackage{courier}               

\usepackage{rotating}

\usepackage{color}

\definecolor{LinkColor}{rgb}{0,0,0}
\usepackage[%
    pdftitle={Thesis_Brinkmann},%
  ]{hyperref}
\hypersetup{colorlinks=true,%
    linkcolor=LinkColor,%
    citecolor=LinkColor,%
    filecolor=LinkColor,%
    menucolor=LinkColor,%
    urlcolor=LinkColor}

\usepackage[savemem]{listings}
\definecolor{lbcolor}{rgb}{0.85,0.85,0.85}

\usepackage[numbers]{natbib}

\usepackage{verbatim}

\usepackage{tipa}

\usepackage{fancybox}

\newcommand{\mm}{\mathfrak m}
\newcommand{\nn}{\mathfrak n}
\newcommand{\pp}{\mathfrak p}

\newcommand{\A}{\mathbb{A}}
\newcommand{\C}{\mathbb{C}}
\newcommand{\D}{\mathbb{D}}

\newcommand{\I}{\mathbb{I}}
\newcommand{\N}{\mathbb{N}}
\newcommand{\Obb}{\mathbb{Obb}}
\newcommand{\Prim}{\mathbb{P}}
\newcommand{\Q}{\mathbb{Q}}
\newcommand{\R}{\mathbb{R}}
\newcommand{\T}{\mathbb{T}}
\newcommand{\Z}{\mathbb{Z}}
\newcommand{\Ac}{\mathcal{A}}
\newcommand{\Bc}{\mathcal{B}}

\newcommand{\Dc}{\mathcal{D}}
\newcommand{\Ec}{\mathcal{E}}
\newcommand{\Fc}{\mathcal{F}}
\newcommand{\Gc}{\mathcal{G}}
\newcommand{\Ic}{\mathcal{I}}

\newcommand{\Mcc}{\mathcal{M}}
\newcommand{\Nc}{\mathcal{N}}
\newcommand{\Lc}{\mathcal{L}}
\newcommand{\Oc}{\mathcal{O}}

\newcommand{\Rc}{\mathcal{R}}
\newcommand{\Sc}{\mathcal{S}}
\newcommand{\Tc}{\mathcal{T}}

\newcommand{\Te}{\text{e}}
\newcommand{\Th}{\text{h}}
\newcommand{\Ts}{\text{s}}
\newcommand{\TF}{\text{F}}

\newcommand{\TT}{\text{T}}

\DeclareMathOperator{\chara}{char}

\DeclareMathOperator{\dirsum}{\oplus}
\DeclareMathOperator{\id}{id}
\DeclareMathOperator{\Ima}{im}
\DeclareMathOperator{\Ker}{ker}
\DeclareMathOperator{\coKer}{coker}

\DeclareMathOperator{\Ann}{Ann}
\DeclareMathOperator{\Ass}{Ass}
\DeclareMathOperator{\Hom}{Hom}
\DeclareMathOperator{\Ext}{Ext}
\DeclareMathOperator{\Dim}{dim}
\DeclareMathOperator{\Det}{det}
\DeclareMathOperator{\Spec}{Spec}
\DeclareMathOperator{\Depth}{depth}
\DeclareMathOperator{\projdim}{proj.dim}
\DeclareMathOperator{\lK}{H}
\DeclareMathOperator{\ev}{ev}
\DeclareMathOperator{\LC}{LC}
\DeclareMathOperator{\Syz}{Syz}

\DeclareMathOperator{\Proj}{Proj}
\DeclareMathOperator{\Deg}{deg}
\DeclareMathOperator{\Max}{max}
\DeclareMathOperator{\Min}{min}
\DeclareMathOperator{\sg}{syzgap}
\DeclareMathOperator{\Id}{Id}
\DeclareMathOperator{\MF}{MF}
\DeclareMathOperator{\SL}{SL}
\DeclareMathOperator{\GL}{GL}
\DeclareMathOperator{\Rank}{rank}
\DeclareMathOperator{\open}{D}
\DeclareMathOperator{\HKF}{HK}
\DeclareMathOperator{\HKM}{e_{HK}}
\DeclareMathOperator{\LHKM}{e^{\infty}_{HK}}
\DeclareMathOperator{\FS}{F_{sign}}

\newcommand{\qpot}{^{[q]}}
\newcommand{\pb}{^{\ast}}
\newcommand{\fpb}[1]{\text{F}^{#1\ast}}
\newcommand{\Lim}[1]{\lim_{#1}}

\newcommand{\inj}{\hookrightarrow}
\newcommand{\ra}{\rightarrow}
\newcommand{\lra}{\longrightarrow}

\newcommand{\trans}{^{\text{T}}}
\newcommand{\dimglo}[1]{\text{h}^0(#1)}
\newcommand{\dimext}[1]{\text{h}^1(#1)}
\newcommand{\punctured}[2]{\Spec(#1)\setminus\{#2\}}

\DeclareMathOperator{\pnt}{\raise 0.5mm \hbox{\large\bf.}}

\numberwithin{section}{chapter}
\numberwithin{equation}{chapter}
\numberwithin{table}{chapter}

\newtheoremstyle{thm}{}{}%
     {\em}
     {}
     {\bf}
     {.}
     {0.5em}
     {\thmname{#1}\thmnumber{ #2}\normalfont\thmnote{ (#3)}}

\newtheoremstyle{thm*}{}{}%
     {\em}
     {}
     {\bf}
     {.}
     {0.5em}
     {\thmname{#1}\normalfont\thmnote{ (#3)}}

\newtheoremstyle{def}{}{}%
     {\rm}
     {}
     {\bf}
     {.}
     {0.5em}
     {\thmname{#1}\thmnumber{ #2}\thmnote{ #3}}

\theoremstyle{thm}

\newtheorem{thm}{Theorem}[chapter]
\newtheorem{lem}[thm]{Lemma}
\newtheorem{cor}[thm]{Corollary}
\newtheorem{prop}[thm]{Proposition}

\newtheorem{que}[thm]{Question}
\newtheorem{Not}[thm]{Notation}

\newtheorem{func}[thm]{Function}

\theoremstyle{thm*}

\newtheorem{thm*}{Theorem}
\newtheorem{lem*}{Lemma}
\newtheorem{cor*}{Corollary}

\theoremstyle{def}

\newtheorem{defi}[thm]{Definition}
\newtheorem{rem}[thm]{Remark}
\newtheorem{exa}[thm]{Example}

\newtheorem{setup}[thm]{Setup}
\newtheorem{con}[thm]{Construction}

\makeatletter
\def\qed{%
  \ifmmode 
  \tag*{\qedsymbol}
  \else \leavevmode\unskip\penalty9999 \hbox{}\nobreak\hfill
  \quad\hbox{\qedsymbol}\pagebreak[1]\vskip 2\topsep\fi}

\def\cocoa{{\hbox{\rm C\kern-.13em o\kern-.07em C\kern-.13em o\kern-.15em A}}}

\let\epsilon=\varepsilon
\let\phi=\varphi
\begin{document}

\thispagestyle{empty}
\begin{center}
Dissertation

zur Erlangung des Doktorgrades (Dr. rer. nat.)

des Fachbereichs Mathematik/Informatik

der Universit\"at Osnabr\"uck

\vspace{5cm}

\LARGE
\textbf{Hilbert-Kunz functions of surface rings of type ADE}

\vspace{2cm}

\normalsize
vorgelegt

von

\vspace{2cm}

\Large
Daniel Brinkmann

\large
Osnabr\"uck

eingereicht: Mai 2013

ver\"offentlicht: August 2013

\vspace{4.5cm}

\Large
Betreuer: Prof. Dr. Holger Brenner
\end{center}
\newpage

\thispagestyle{empty}
\tableofcontents

\chapter*{Introduction}
\section*{The goal of this thesis and a sketch of the most important ideas}

The starting point for Hilbert-Kunz theory was the paper ``Characterizations of regular local rings of characteristic $p$'' and its sequel 
``On Noetherian rings of characteristic $p$'' by Kunz published 1969 resp. 1976 (cf. \cite{kunz1} and \cite{kunz2}). The central objects are the 
\textit{Hilbert-Kunz function} and the \textit{Hilbert-Kunz multiplicity}, which are defined in the case of affine $k$-algebras $R:=k[X_1,\ldots,X_n]/I$ 
in dependence on the classical Hilbert-Samuel multiplicity and function as
$$\HKF(R,p^e)=\Dim_k\left(R/\left(X_1^{p^e},\ldots,X_n^{p^e}\right)\right)\quad\text{and}\quad \HKM(R)=\lim_{e\ra\infty}\frac{\HKF(R,p^e)}{p^{e\cdot\Dim(R)}},$$
where $k$ denotes a field of characteristic $p>0$ and $\Dim(R)$ denotes the Krull dimension of $R$. For an $(X_1,\ldots,X_n)$-primary ideal $J:=(f_1,\ldots,f_m)$, the \textit{Hilbert-Kunz function of $R$ with respect to $J$} and the \textit{Hilbert-Kunz multiplicity of $R$ with respect to $J$} are defined as
$$\HKF(J,R,p^e)=\Dim_k\left(R/\left(f_1^{p^e},\ldots,f_m^{p^e}\right)\right)\quad\text{resp.}\quad \HKM(R)=\lim_{e\ra\infty}\frac{\HKF(J,R,p^e)}{p^{e\cdot\Dim(R)}}.$$
 Seven years later, Monsky proved the existence of 
Hilbert-Kunz multiplicities in a more general situation (cf. \cite[Theorem 1.8]{hkmexists}). 
In the same paper the term ``Hilbert-Kunz function'' firstly appeared. Hilbert-Kunz theory has applications in Iwasawa theory as well as in tight closure theory 
(see for example \cite{monskyiwa} and \cite{tightapp}). Initially, Kunz expected a relation to the resolution of singularities in positive characteristics 
(cf. the discussion after Example 4.3 of \cite{kunz2}) but up to today there are no substantial results in this direction. 
Maybe the most central question in Hilbert-Kunz theory was whether Hilbert-Kunz multiplicities have to be rational or not. For long time the favourite answer was 
that they have to be rational (cf. \cite{hkmexists}) and there were many results in this direction (e.g. \cite{bredim2}, \cite{eto1}, \cite{triflag}, \cite{monskyhan}, \cite{seibert2}, \cite{tri1} or \cite{watanabe}). But 2008 Monsky came up with a conjecture whose correctness would imply the existence of a four-dimensional ring of characteristic two with irrational Hilbert-Kunz multiplicity (cf. \cite{hkmrat}). Extending the geometric approach to Hilbert-Kunz theory (cf. \cite{bredim2} or \cite{tri1}), Brenner recently proved the existence of irrational Hilbert-Kunz multiplicities in all characteristics (cf. \cite{holgerirre}).
\vspace{5mm}

The goal of this thesis is to compute the Hilbert-Kunz functions of the surface rings
\begin{itemize}
\item $A_n:=k[X,Y,Z]/(X^{n+1}-YZ)$ with $n\geq 0$,
\item $D_n:=k[X,Y,Z]/(X^2+Y^{n-1}+YZ^2)$ with $n\geq 4$
\item $E_6:=k[X,Y,Z]/(X^2+Y^3+Z^4)$,
\item $E_7:=k[X,Y,Z]/(X^2+Y^3+YZ^3)$ and
\item $E_8:=k[X,Y,Z]/(X^2+Y^3+Z^5)$,
\end{itemize}
where $k$ denotes an algebraically closed field. These are called \textit{surface rings of type ADE}.

In characteristic zero surface rings of type ADE show up as quotient singularities of $\C^2$ and are unique up to isomorphism as it was shown by Klein 
(cf. \cite{klein}) in 1886. This is also how they have share in string theory (cf. \cite{string}). 
Surface rings of type ADE were studied for example by du Val (cf. \cite{duval1}, \cite{duval2}, \cite{duval3}) and Brieskorn (cf. \cite{brieskorn}) and 
are sometimes called \textit{Kleinian} or \textit{du Val singularities}. Due to Artin (cf. \cite{artin2}) they are exactly the \textit{rational double points}. In positive characteristics Artin 
classified the rational double points of surfaces. In characteristics at least seven, these are exactly the surface rings of type ADE. In characteristics two, 
three and five all rational double points are given by surface rings of type ADE or certain deformations of them (cf. \cite{artin1}). Note that the rational double 
points are exactly the \textit{simple surface singularities} (cf. \cite{artin1}).

Due to Watanabe and Yoshida (cf. \cite{watyo1}) the Hilbert-Kunz multiplicity of the surface rings of type ADE is known to be $2-\tfrac{1}{|G|}$, where $G\subsetneq\SL_2(\C)$ is 
the group corresponding to the singularity in characteristic zero and under the additional assumption that the corresponding ring is $F$-rational (cf. \cite{watyo1}). Note that this assumption 
is satisfied in almost all characteristics. 
The Hilbert-Kunz function of the surface rings of type $A_n$ is known due to Kunz (cf. \cite{kunz2}) and the Hilbert-Kunz functions of surface rings of 
type $E_6$ or $E_8$ can be computed at least for explicitly given characteristic using the algorithm of Han and Monsky (cf. \cite{monskyhan}). In general 
the Hilbert-Kunz function of surface rings $R$ of type ADE has due to Huneke, McDermott and Monsky the shape
$$\HKF(R,p^e)=\HKM(R)\cdot p^{2e}+\gamma(p^e),$$
where $\gamma(p^e)$ is a bounded function (cf. \cite{hkfnormal}).

By a geometric approach of Brenner and Trivedi for two-dimensional standard-graded rings $S$ one needs to control the Frobenius pull-backs of the cotangent bundle 
on $C:=\Proj(S)$, e.g. the sheaves
$$\left(\Syz_S\left(S_+^{[p^e]}\right)\right)^{\sim},$$ 
to compute the Hilbert-Kunz function of $S$ (cf. \cite{bredim2} or \cite{tri1}). Unfortunately, surface rings of type ADE are not standard-graded. But we can associate to them the 
standard-graded rings $S:=k[U,V,W]/(F)$, where 
$(F)$ is the image of the principal ideal defining a surface ring of type ADE under the map 
$$k[X,Y,Z]\ra k[U,V,W], \text{ }X\mapsto U^{\Deg(X)}, \text{ }Y\mapsto V^{\Deg(Y)}, \text{ }Z\mapsto W^{\Deg(Z)}.$$
Since the map $R\ra S$ is a flat extension of Noetherian domains, one obtains
$$\HKF(R,p^e)=\frac{\HKF((U^{\Deg(X)},V^{\Deg(Y)},W^{\Deg(Z)}),S,p^e)}{[Q(S):Q(R)]},$$
where $[Q(S):Q(R)]$ denotes the degree of the extension of the fields of fractions of $R$ and $S$. The benefit of this approach is that we can use tools from 
algebraic geometry to compute 
$$\HKF\left(\left(U^{\Deg(X)},V^{\Deg(Y)},W^{\Deg(Z)}\right),S,p^e\right).$$ 
The disadvantage of this approach is that $\Proj(S)$ is much more complicated 
then $\Proj(R)$, which is just $\Prim_k^1$. This is because the quotient map $\A_k^2\ra\Spec(R)$ induces a finite map $\Prim_k^1\ra\Proj(R)$ of smooth projective curves. 
By Hurwitz' formula the curve $\Proj(R)$ has genus zero, hence $\Proj(R)\cong\Prim^1_k$. To compute 
$$\HKF\left(\left(U^{\Deg(X)},V^{\Deg(Y)},W^{\Deg(Z)}\right),S,p^e\right),$$ 
we need to control the $\Oc_{\Proj(S)}$-modules corresponding to the $S$-modules 
$$N_e:=\Syz_S\left(U^{\Deg(X)\cdot p^e},V^{\Deg(Y)\cdot p^e},W^{\Deg(Z)\cdot p^e}\right).$$
To do so, we compute the $R$-modules
$$M_e:=\Syz_R\left(X^{p^e},Y^{p^e},Z^{p^e}\right)$$
and use our map $\eta:R\ra S$ as follows.

\begin{enumerate}
 \item The map $\eta$ induces a map $\phi:\punctured{S}{S_+}\ra\punctured{R}{R_+}$.
 \item The pull-back of $\widetilde{M_e}^{\text{pt}}$ along $\phi$ is just $\widetilde{N_e}^{\text{pt}}$, where $\widetilde{M_e}^{\text{pt}}$ denotes the 
 restriction to $\punctured{R}{R_+}$ of the sheaf on $\Spec(R)$ associated to $M_e$ and analogue for $\widetilde{N_e}^{\text{pt}}$.
 \item The isomorphism $\phi\pb(\widetilde{M_e}^{\text{pt}})\cong \widetilde{N_e}^{\text{pt}}$ extends to a global isomorphism $\eta(M_e)\cong N_e$.
\end{enumerate}

But how to compute the $M_e$?

Since over two-dimensional rings first syzygy modules of homogeneous ideals are graded maximal Cohen-Macaulay modules, we will make use of the fact that the surface rings of 
type ADE are of \textit{finite graded Cohen-Macaulay type}, meaning that there are up to degree shifts only 
finitely many isomorphism classes of indecomposable, graded maximal Cohen-Macaulay modules (cf. \cite{knoerrer} and \cite{buchmcm}). Representatives for the 
isomorphism classes are described in \cite{matfacade} using Eisenbud's theory of matrix factorizations (cf. \cite{eismat}). In our situation, a matrix factorization 
is a pair $(\phi,\psi)$ of $n\times n$ matrices with coefficients in $k[X,Y,Z]$ such that 
$$\phi\cdot\psi=\psi\cdot\phi=f\cdot\Id,$$
where $f$ is the generator of the principal ideal defining a surface ring $R$ of type ADE. Then $\coKer(\phi)$ is a graded maximal Cohen-Macaulay $R$-module. 
Following the common strategy for classifications, we need an invariant that detects the isomorphism class 
of a given (indecomposable) graded maximal Cohen-Macaulay module and then compute the invariant for all $R$-modules $M_e$. The first step is to find 
``good representatives'' for the finitely many isomorphism classes of indecomposable, graded maximal Cohen-Macaulay modules. It is not surprising that the phrase 
``good representatives'' means first syzygy modules of homogeneous ideals in our case. We will use sheaf-theoretic tools to find isomorphisms of the form 
$$\coKer(\phi)\cong\Syz_R(I_{\phi}),$$
where the $I_{\phi}$ are homogeneous ideals of $R$ and $(\phi,\psi)$ runs through a complete list of matrix factorizations representing the isomorphism classes 
of indecomposable, graded maximal Cohen-Macaulay $R$-modules. To be more precise, we will show (in our situation) that for every matrix factorization $(\phi,\psi)$, 
where $\phi$ is a $n\times n$-matrix and $\coKer(\phi)$ has rank $r$, there are $r+1$ columns of $\psi$ (!) such that on $U:=\punctured{R}{R_+}$ those $r+1$ columns generate 
$\Ima(\psi)$ and on $R$ they generate $\Syz_R(f_1,\ldots,f_{r+1})$ for suitable homogeneous $f_i\in R$. Therefore, we have the short exact sequence
$$\xymatrix{
0\ar[r] & (f_1,\ldots,f_n)^{\sim}|_U\ar[r] & \Oc_U^{r+1}\ar[r] & \widetilde{\Ima(\psi)}|_U\ar[r] & 0,
}$$
showing that the dual module $\Ima(\psi)^{\vee}$ and $\Syz_R(f_1,\ldots,f_{r+1})$ correspond on $U$. Since all rings in question are Gorenstein in codimension one, this isomorphism extends to a global isomorphism. It remains to observe $\Ima(\psi)\cong\coKer(\phi)$.

Since we have only a finite list of isomorphism classes, one might hope that they can be distinguished by their Hilbert-series. This is not true in general. 
But we will see that the Hilbert-series of a given (indecomposable) graded maximal Cohen-Macaulay $R$-module detects the isomorphism class of $M$ 
\textit{up to dualizing}, by which we mean that there is a graded maximal Cohen-Macaulay module $N$, having the same Hilbert-series as $M$, such that either 
$M\cong N$ or $M\cong N^{\vee}$. We will prove a theorem which allows us to compute the Hilbert-series of 
the modules $M_e$ by reducing the computation to the computation of the Hilbert-series of first syzygy modules of homogeneous ideals in 
a polynomial ring in only two variables. The idea to this theorem was inspired by work of Brenner (cf. \cite{miyaoka}).

We will see that the Hilbert-series of $M_e$ is enough to detect its isomorphism class (up to dualizing) if 
$R$ is of type $A_n$ or $E_8$. In the other cases we need to show that the modules $M_e$ are indecomposable. This is done by using invariant theory. To be more precise, we will 
use the facts that the surface rings of type ADE are rings of invariants of $k[x,y]$ under finite subgroups of $\SL_2(k)$ and that normal subgroups $H\subseteq G$ induce maps 
$\iota: \Spec(k[x,y]^H)\ra \Spec(k[x,y]^G)$ resp. on the punctured spectra. The claimed result will follow by comparing (on the punctured spectra) the pull-backs of $\widetilde{M_e}$ along $\iota$ with the Frobenius pull-backs of $\iota\pb(\widetilde{M_0})$.

From the indecomposability of the modules $M_e$ we will obtain a periodicity up to degree shift of the sequence 
$$\left(\Syz_R\left(X^{p^e},Y^{p^e},Z^{p^e}\right)\right)_{e\in\N}$$
by which we mean that there exists an $e_0$ such that
$$\Syz_R\left(X^{p^{e+e_0}},Y^{p^{e+e_0}},Z^{p^{e+e_0}}\right)\cong\Syz_R\left(X^{p^e},Y^{p^e},Z^{p^e}\right)(m_e)$$
holds for all $e\in\N$ and suitable $m_e\in\Z$. This leads to a periodicity up to degree shift of the sequence
$$\left(\Syz_S\left(U^{\Deg(X)\cdot p^e},V^{\Deg(Y)\cdot p^e},W^{\Deg(Z)\cdot p^e}\right)\right)_{e\in\N}.$$
Putting everything together, we will be able to compute the Hilbert-Kunz functions of surface rings of type ADE.

\begin{thm*}[cf. Theorem \ref{hkfdn}]
 The Hilbert-Kunz function of $D_{n+2}$ is given by (with $n\geq 2$, $q=p^e$ and $q\equiv r$ mod $2n$, where $r\in\{0,\ldots,2n-1\}$)
$$e\longmapsto\left\{\begin{aligned}
          & 2\cdot q^2 && \text{ if } e\geq 1,p=2,\\
 & \left(2-\frac{1}{4n}\right)q^2-\frac{r+1}{2}+\frac{r^2}{4n} && \text{ otherwise.}
          \end{aligned}\right.$$
\end{thm*}

\begin{thm*}[cf. Theorem \ref{hkfe6}]
The Hilbert-Kunz function of $E_6$ is given by
$$e\longmapsto\left\{\begin{aligned}
 & 2\cdot q^2 && \text{ if } e\geq 1\text{ and }p\in\{2,3\},\\
 & \left(2-\frac{1}{24}\right)q^2-\frac{23}{24} && \text{ otherwise.}
          \end{aligned}\right.$$
\end{thm*}

\begin{thm*}[cf. Theorem \ref{hkfe7}]
 The Hilbert-Kunz function of $E_7$ is given by
$$e\longmapsto\left\{\begin{aligned}
          & 2\cdot q^2 && \text{ if } e\geq 1\text{ and }p\in\{2,3\},\\
 & \left(2-\frac{1}{48}\right)q^2-\frac{47}{48} && \text{ if }q\text{ mod }24\in\{\pm1,\pm7\},\\
 & \left(2-\frac{1}{48}\right)q^2-\frac{71}{48} && \text{ otherwise.}
          \end{aligned}\right.$$
\end{thm*}

\begin{thm*}[cf. Theorem \ref{hkfe8}]
 The Hilbert-Kunz function of $E_8$ is given by
$$e\longmapsto\left\{\begin{aligned}
          & 2\cdot q^2 && \text{ if } e\geq 1\text{ and }p\in\{2,3,5\},\\
 & \left(2-\frac{1}{120}\right)q^2-\frac{119}{120} && \text{ if }q\text{ mod }30\in\{\pm1,\pm11\},\\
 & \left(2-\frac{1}{120}\right)q^2-\frac{191}{120} && \text{ otherwise.}
          \end{aligned}\right.$$
\end{thm*}

Note that Seibert gave in \cite{seibert2} an algorithm to compute Hilbert-Kunz functions of maximal Cohen-Macaulay rings and modules. Given a ring 
$R:=k[f_1,\ldots,f_n]\subseteq k[x,y]$, where $k$ is a field of characteristic $p>0$, the first step of this algorithm is to decompose $R$ into 
indecomposable $R^p$-modules, which gets very hard if the $f_i$ are not monomial. This is why we did not use this algorithm. For similar reasons it 
seems not practicable to count the monomials in $k[X,Y,Z]/(F,X^q,Y^q,Z^q)$ directly if $F$ is not binomial.

Applying the developed methods in the case of Fermat rings $R:=k[X,Y,Z]/(X^n+Y^n+Z^n)$ and using work of Kaid, Kustin et al. and Trivedi 
(cf. \cite{almar}, \cite{vraciu} resp. \cite{tri2}), we will obtain results on the Hilbert-Kunz function of $R$ as well as on the behaviour of the 
Frobenius pull-backs of the cotangent bundle on $\Proj(R)$. We say that a vector bundle $\Sc$ admits a \textit{$(s,t)$-Frobenius periodicity} if there are $s<t\in\N$ such that 
$\fpb{t}(\Sc)\cong\fpb{s}(\Sc)$ holds up to degree shift.

\begin{thm*}[cf. Theorem \ref{frobpernontri}]
Let $C$ be the projective Fermat curve of degree $n$ over an algebraically closed field of characteristic $p>0$ with $\gcd(p,n)=1$. 
The bundle $\Syz_C(X,Y,Z)$ admits a Frobenius periodicity if and only if 
$$\delta\left(\frac{1}{n},\frac{1}{n},\frac{1}{n}\right)=0.$$
Moreover, the length of this periodicity is bounded above by the order of $p$ modulo $2n$.  
\end{thm*}

\begin{thm*}[cf. Theorem \ref{hkfsss}]
Let $R:=k[X,Y,Z]/(X^n+Y^n+Z^n)$ with an algebraically closed field $k$ of characteristic $p>0$ with $\gcd(p,n)=1$. Assume that 
$\Syz_{\Proj(R)}(X,Y,Z)$ is strongly semistable and is not trivialized by the Frobenius. Let $q=p^e=n\theta+r$ with $\theta$, $r\in\N$ and $0\leq r<n$. 
Then 
$$\HKF(R,q)= \left\{\begin{aligned} & \frac{3n}{4}\cdot q^2-\frac{3n}{4}r^2+r^3 && \text{if }\theta\text{ is even,}\\ & \frac{3n}{4}\cdot q^2-\frac{3n}{4}(n-r)^2+(n-r)^3 && \text{if }\theta\text{ is odd.}\end{aligned}\right.$$
\end{thm*}

These theorems generalize the result of Brenner and Kaid who obtained them under the condition $p\equiv -1$ modulo $2n$ (cf. \cite{holgeralmar}).

Moreover, the last theorem recovers a result of Han and Monsky who proved that the Hilbert-Kunz function of a Fermat ring $R$ has the shape $\HKM(R)\cdot p^{2e}+\Delta_e$, where the function $e\mapsto\Delta_e$ is eventually periodic with non-positive values (cf. \cite{monskyhan}).

We will give examples showing that in every odd characteristic for every $1\leq l$ there exists a $n\in\N$ such that the bundle $\Syz(X,Y,Z)$ on the 
projective Fermat curve of degree $n$ admits a $(0,l)$-Frobenius periodicity. In characteristic two only $(1,2)$-Frobenius periodicities appear. 
Moreover, this happens only in the case $n=3$.

\section*{Summary of the thesis}

In the first chapter we will give a survey on Hilbert-Kunz theory. We divide this survey into three sections. 
In the first one we give the basic definitions and state various results as the existence of Hilbert-Kunz multiplicities, bounds for Hilbert-Kunz 
multiplicities, some classification results and the connection to tight closure theory. In the second section we explain the geometric approach to Hilbert-Kunz theory in graded 
dimension two by Brenner and Trivedi (cf. \cite{bredim2} and \cite{tri1}). This section will start with short discussions of vector bundles, Frobenius 
maps, syzygy bundles and various stability conditions for vector bundles. It will end with the classification of Hilbert-Kunz functions of the homogeneous coordinate rings of projective 
curves of degree three defined over algebraically closed fields of positive characteristic. In the third section we define the invariants 
limit Hilbert-Kunz multiplicity and 
$F$-signature which are strongly connected to Hilbert-Kunz multiplicities. At the end of the first chapter we give a list of further references 
related to Hilbert-Kunz theory.

In the first section of the second chapter we compute some explicit examples of Hilbert-Kunz functions with the help of the geometric approach. For 
example we will compute (cf. Example \ref{hkf2lm}) the Hilbert-Kunz functions of the diagonal hypersurfaces 
$$R:=k[X,Y,Z]/(X^2+Y^l+Z^m),$$ 
where $l$, $m$ are positive integers such 
that the triple $\left(\tfrac{1}{2},\tfrac{1}{l},\tfrac{1}{m}\right)$ does not satisfy the strict triangle inequality.

The goal of the second section is to generalize a theorem of Monsky which connects the Hilbert-Kunz multiplicities of standard-graded quotients 
$R:=k[X,Y,Z]/(F)$, where $F$ is a so-called \textit{regular}, irreducible trinomial of degree $d$, with Han's $\delta$ function. In this case regular means that the trinomial
$F=M_X+M_Y+M_Z$ satisfies $\Deg_X(M_X)>\tfrac{d}{2}$, $\Deg_Y(M_Y)>\tfrac{d}{2}$ and $\Deg_Z(M_Z)>\tfrac{d}{2}$. Explicitly, Monsky showed
$$\HKM\left((X^s,Y^s,Z^s),R\right)=\frac{3d}{4}s^2+\frac{\lambda^2}{4d}\delta\left(\frac{s\alpha}{\lambda},\frac{s\beta}{\lambda},\frac{s\gamma}{\lambda}\right)^2,$$
where $s$ is a positive integer and $\alpha$, $\beta$, $\gamma$, $\lambda$ are natural numbers depending only on the exponents appearing in $F$ 
(cf. \cite[Theorem 2.3]{irred} or Corollary \ref{irredmon}).
Leaning on Monsky's proof, we will show that a positively-graded analogue of Han's $\delta$ function as well as 
$$\mu(l_1,\ldots,l_n):=\HKM\left(\left(x_1^{l_1},\ldots,x_n^{l_n}\right),R\right)$$
for standard-graded domains $R$ and homogeneous $R_+$-primary ideals $\left(x_1^{l_1},\ldots,x_n^{l_n}\right)$ extend to continuous functions on $[0,\ldots,\infty)^3$ resp. $[0,\ldots,\infty)^n$.

\begin{thm*}[cf. Theorem \ref{myhkmthm} resp. Corollary \ref{myhkmcor}]
With the previous notations we have
\begin{align*}
\begin{split}
& \HKM\left(\left(X^a,Y^b,Z^c\right),k[X,Y,Z]/(F)\right) \\
= & \frac{d}{4}\cdot (2ab+2ac+2bc-a^2-b^2-c^2)+\frac{\lambda^2}{4d}\left(\delta\left(\frac{\alpha(a,b,c)}{\lambda},\frac{\beta(a,b,c)}{\lambda},\frac{\gamma(a,b,c)}{\lambda}\right)\right)^2,
\end{split}
\end{align*}
where $\lambda$ is a positive integer depending only on the exponents appearing in $F$ and the abbreviations $\alpha(a,b,c)$, $\beta(a,b,c)$, $\gamma(a,b,c)$ denote natural numbers 
depending on $a$, $b$, $c$ and the exponents in $F$.
\end{thm*}

An explicit implementation of this Theorem for CoCoA can be found in the Appendix. Using a result of Watanabe and Yoshida (cf. \cite[Theorem 2.7]{watyo1} or Theorem \ref{watyorank}) on the behaviour of Hilbert-Kunz multiplicities under 
module-finite extensions of Noetherian local domains, the above result enables us to compute Hilbert-Kunz multiplicities of positively-graded quotients 
$k[U,V,W]/(G)$ such that the image of $G$ under the map 
$$k[U,V,W]\lra k[X,Y,Z], U\mapsto X^{\Deg(U)},V\mapsto Y^{\Deg(V)},W\mapsto Z^{\Deg(W)}$$
is a regular, irreducible trinomial (cf. Remark \ref{hkmstgnonstg}). In particular, we recover Han's result on the Hilbert-Kunz multiplicities of 
two-dimensional diagonal hypersurfaces 
$$k[U,V,W]/\left(U^{d_1}+V^{d_2}+W^{d_3}\right),$$
with $2\leq d_i\in\N$ (cf. \cite[Theorem 2.30]{handiss} or Corollary 
\ref{hkmdiag}) and the result of Watanabe and Yoshida on the Hilbert-Kunz multiplicities of the surface rings of type ADE (cf. \cite[Theorem 5.4]{watyo1}, 
Theorem \ref{watyogroup} or Example \ref{lastexahan}).

In the third section of the second chapter we use Theorem \ref{myhkmthm} resp. Remark \ref{hkmstgnonstg} to study the behaviour of the Hilbert-Kunz 
multiplicity in certain families of positively-graded quotients $k[U,V,W]/(G)$ of the form described above. Writing the trinomial $G$ as $G=M_U+M_V+M_W$, the families are definied by the condition that one of the exponents $\Deg_U(M_U)$, $\Deg_V(M_V)$, $\Deg_W(M_W)$ is free and all other exponents in $G$ are fixed. For a fixed family we will study the Hilbert-Kunz multiplicity for large values of the free exponent. We will see that in all cases the sequence of the Hilbert-Kunz multiplicities in such a family converges to the Hilbert-Kunz multiplicity of a binomial quotient ring, where the binomial is the deformation of the trinomial $G$ obtained by leaving out the monomial with the free exponent.

In the first section of Chapter 3 we give a brief overview of properties of maximal Cohen-Macaulay modules including a proof of the fact that the concepts maximal Cohen-Macau\-lay and reflexive coincide in normal domains of dimension two. We will demonstrate how one might compute 
Hilbert-Kunz functions of two-dimensional rings of finite Cohen-Macaulay type with the help of the geometric approach. The example will be the $n$-th Veronese subring $R$ 
of $k[x,y]$, where $k$ denotes an algebraically closed field of characteristic $p>0$. We obtain (cf. Example \ref{veronese})
$$\HKF(R,p^e)=\frac{n+1}{2}\cdot \left(p^{2e}-r^2\right)+n\cdot\frac{r(r+1)}{2}+r-n,$$
where $r$ is the smallest non-negative representative of the class of $p^e$ in $\Z/(n)$. During this computation we will give a complete list of first syzygy modules of ideals representing the isomorphism classes of maximal Cohen-Macaulay $R$-modules, show how $\Syz_R(R_+)$ decomposes in this choice of representatives and explain the action of the Frobenius morphism on the Picard group.

In Section 3.2 we will explain the notion of matrix factorizations developed by Eisenbud in \cite{eismat} and their connection to maximal 
Cohen-Macaulay modules over rings of the form $S/(f)$, where $(S,\mm)$ is a local regular Noetherian ring and $f\in\mm^2\setminus\{0\}$.

In Section 3.3 we will develop a sheaf-theoretic approach to compute first syzygy modules of ideals that are isomorphic to the maximal Cohen-Macaulay module 
defined by a fixed matrix factorization. For this approach we have to restrict to two-dimensional local rings which are Gorenstein in codimension one.

In Sections 3.4-3.8 we use the approach from Section 3.3 to obtain a system of first syzygy modules of ideals representing the isomorphism classes of the 
indecomposable, maximal Cohen-Macaulay modules over the completions of the surface rings of type ADE. Each section ends with a theorem gathering the following informations.

\begin{itemize}
 \item A complete list of non-isomorphic first syzygy modules of ideals representing the isomorphism classes of non-free, indecomposable, maximal 
	Cohen-Macaulay modules.
 \item Representatives of the isomorphism classes of the dual modules.
 \item The isomorphism classes of the corresponding determinant bundles on the punctured spectrum.
 \item Ideals representing the isomorphism classes of the non-free, indecomposable, maximal Cohen-Macaulay modules of rank one.
\end{itemize}

For example if $R$ is the $(X,Y,Z)$-adic completion of $E_6$ and $U$ its punctured spectrum, the corresponding theorem 
is the following.

\begin{thm*}[cf. Theorem \ref{syze6}]
The pairwise non-isomorphic modules 
\begin{align*}
M_1 &= \Syz_R(X,Y,Z),\\
M_2 &= \Syz_R(X,Y^2,YZ,Z^2),\\
M_3 &= \Syz_R(iX+Z^2,Y^2,YZ),\\
M_4 &= \Syz_R(-iX+Z^2,Y^2,YZ),\\
M_5 &= \Syz_R(-iX+Z^2,Y),\\
M_6 &= \Syz_R(iX+Z^2,Y)
\end{align*}
give a complete list of representatives of the isomorphism classes of indecomposable, non-free, maximal Cohen-Macaulay modules. Moreover, $M_1$ and $M_2$ are 
selfdual, $M_3^{\vee}\cong M_4$ and $M_5^{\vee}\cong M_6$. The Determinants of $\widetilde{M_3}|_U$ resp. $\widetilde{M_4}|_U$ are $\widetilde{M_5}|_U$ resp. $\widetilde{M_6}|_U$. 
For the rank one modules we have the isomorphisms $M_5\cong (iX+Z^2,Y)$ and $M_6\cong (-iX+Z^2,Y)$.
\end{thm*}

In Section 3.9 we discuss an idea how to avoid sheaves in the approach of Section 3.3 and give an example showing the disadvantages of this idea.

In Section 3.10 we comment on the correpondence between the local and the graded situation. Especially, we will explain a theorem due to Auslander and Reiten saying that for a positively-graded ring $R$, where $R_0$ is a perfect field, the isomorphism classes of indecomposable, graded maximal Cohen-Macaulay $R$-modules up to degree shift and the isomorphism classes of indecomposable, maximal Cohen-Macaulay $\hat{R}$-modules coincide (cf. \cite{ausreit}).

In Chapter 4 we want to find an easy way to compute the Hilbert-series of syzygy modules of the form 
$$M_a:=\Syz_R(X^a\cdot V_1,\ldots,X^a\cdot V_m,V_{m+1},\ldots,V_{m+l}),$$
where $R$ is a positively-graded ring of the form $k[X,Y_1,\ldots,Y_n]/(X^d-F(Y_1,\ldots,Y_n))$ and the $V_i$ are monomials in the variables $Y_j$. 
Inspired by Brenner's work in \cite{miyaoka}, we will prove the following theorem.

\begin{thm*}[cf. Theorem \ref{ses}]
Let $a=d\cdot q+r$ with $q$, $r\in\N$ and $0\leq r\leq d-1$. We have a short exact sequence
\begin{align*}
0\lra M_{a+d-2r}(s-\alpha\cdot r)\lra M_{dq}(s-\alpha\cdot r)\oplus M_{dq+d}(s)\lra M_a(s)\lra 0
\end{align*}
for all $s\in\Z$, where $\alpha=\Deg(X)$.
\end{thm*}

Note that the modules $M_{dq}$ and $M_{dq+d}$ in the middle spot are already defined over the polynomial ring $k[Y_1,\ldots,Y_n]$. From this theorem we can compute the Hilbert-series of the $R$-modules $M_a$.

\begin{thm*}[cf. Theorem \ref{HilbSer}]
The Hilbert-series of $M_a$ is given by
\begin{equation*}
\lK_{M_a}(t) = \frac{(t^{\alpha\cdot r}-t^{\alpha\cdot d})\cdot \lK_{M_{dq}}(t)+(1-t^{\alpha\cdot r})\lK_{M_{dq+d}}(t)}{1-t^{\alpha\cdot d}}.
\end{equation*}
\end{thm*}

In the second section we compute for each isomorphism class of indecomposable, maximal Cohen-Macaulay modules over surface rings of type ADE the Hilbert-series 
of the representing first syzygy modules of ideals found in Chapter 3. Moreover, we discuss which splitting behaviours of a given maximal Cohen-Macaulay 
module $M$ can be excluded by computing its Hilbert-series.

In Chapter 5 we use the fact that the surface rings of type ADE defined over $\C$ are the rings of invariants of $\C[x,y]$ under actions defined by the 
finite subgroups of $\SL_2(\C)$. If $R$ is of type ADE over an algebraically closed field of characteristic $p>0$ and $G\subsetneq\SL_2(\C)$ is the 
corresponding group in characteristic zero with order invertible in $k$, it remains true that $R$ is the ring of invariants of $k[x,y]$ under the action 
of a finite subgroup $G'\subsetneq\SL_2(k)$ with $|G'|=|G|$. If $H$, $G$ are finite subgroups of $\SL_2(k)$ such that $H$ is a normal subgroup of $G$, then 
$k[x,y]^G$ is a subring of $k[x,y]^H$, inducing a morphism $\iota:V\ra U$, where $V$ and $U$ are the punctured spectra of $\Spec\left(k[x,y]^H\right)$ resp. 
$\Spec\left(k[x,y]^G\right)$. Fixing a surface ring of type ADE with corresponding 
group $G$ and non-isomorphic, indecomposable, maximal Cohen-Macaulay modules $M_1,\ldots,M_n$ representing all isomorphism classes, we will 
compute the pull-backs of all the corresponding $\Oc_U$-modules under all maps $\iota$ as above. These explicit computations are done in Sections 5.1 to 5.5. 
In Section 5.6 we will gather the results of the previous sections and obtain the following lemma.

\begin{lem*}[cf. Lemma \ref{fpbindec}]
Let $U$ be the punctured spectrum of a surface ring $R$ of type $D$ or $E$. Assume that the 
order of the group corresponding to the singularity is invertible in the ground field. Then $\Syz_U\left(X^{p^e},Y^{p^e},Z^{p^e}\right)$ is indecomposable for all $e\in\N$.
\end{lem*}

In Chapter 6 we compute the Hilbert-Kunz functions of the surface rings of type DE (and A). During these computations the whole strength of Theorem 
\ref{HilbSer} becomes obvious.

In the last chapter we discuss possible extensions of the approach using matrix factorizations. More explicitly, we compute the Hilbert-Kunz functions 
of $E_8$ with respect to the ideals $(X,Y,Z^2)$ and $(X,Y^2,YZ,Z^2)$ (cf. Examples \ref{exa1chap7} and \ref{exa2chap7}) as well as the Hilbert-Kunz 
functions of $A_{\infty}:=k[X,Y,Z]/(XY)$ and $D_{\infty}:=k[X,Y,Z]/(X^2Y-Z^2)$ (cf. Sections 7.2 and 7.3). We also compute for every isomorphism class 
of non-free, indecomposable, maximal Cohen-Macaulay modules over $A_{\infty}$ and $D_{\infty}$ a representing first syzygy module of an ideal and the 
corresponding Hilbert-series.

In the fourth section we focuss on projective Fermat curves $C$ of degree $n$ with $C:=\Proj(R)$ and $R:=k[X,Y,Z]/(X^n+Y^n+Z^n)$, where we assume $k$ to 
be an algebraically closed field of characteristic $p>0$ coprime to $n$. We start the section by a brief discussion of Kaid's result on the strong 
Harder-Narasimhan filtration of $\Syz_C\left(X^{p^e},Y^{p^e},Z^{p^e}\right)$ for $e\gg 0$ and the Hilbert-Kunz function of $R$ in the case, where $\Syz_C(X,Y,Z)$ is not 
strongly semistable. Since the vector bundle $\Syz_C(X,Y,Z)$ is strongly semistable and is not trivialized by the Frobenius if and only if all quotients $R/(X^q,Y^q,Z^q)$ have infinite 
projective dimension (cf. Corollary \ref{deltafinprodim}), we can use a theorem of Kustin, Rahmati and Vraciu which gives the free resolution of 
the quotient $Q:=R/\left(X^N,Y^N,Z^N\right)$ if the projective dimension of $Q$ is infinite (cf. \cite{vraciu} or Theorem \ref{inftyfree}) to obtain

\begin{cor*}[cf. Corollary \ref{inftysyz}]
Let $N=\theta\cdot n+r$ with $\theta$, $r\in\N$ and $0\leq r<n$. Assume that $R/\left(X^N,Y^N,Z^N\right)$ has infinite projective dimension. Then
$$\Syz_R\left(X^N,Y^N,Z^N\right)(m)\cong \left\{\begin{aligned} & \Syz_R(X^r,Y^r,Z^r) && \text{if }\theta\text{ is even,}\\ & \Syz_R(X^{n-r},Y^{n-r},Z^{n-r}) && \text{if }\theta\text{ is odd.}\end{aligned}\right.$$
holds for some $m\in\Z$.
\end{cor*}

This corollary is used to prove that $$\delta\left(\frac{1}{n},\frac{1}{n},\frac{1}{n}\right)=0$$ is sufficient for having a Frobenius periodicity of $\Syz_{\Proj(R)}(X,Y,Z)$. 
Moreover, the corollary enables us to compute the Hilbert-Kunz function of $R$, provided $\Syz_C(X,Y,Z)$ is strongly semistable and is not trivialized by the Frobenius. 

At the end we discuss the question which Frobenius periodicities the bundles $\Syz_C(X,Y,Z)$ might admit.

Finally, the last section of Chapter 7 is reserved for a discussion of open questions.

We should mention that we used CoCoA for a numerical treatment of all explicit results on Hilbert-Kunz functions and multiplicities (of course only for small values of the various parameters).

\section*{Acknowledgements}
First of all, my deepest thanks go to my supervisor Holger Brenner. Without our countless 
discussions it would not have been possible to write this thesis. I thank him for various usefull comments 
and his extraordinary patience during the last years.

I would like to thank Tim R\"omer and Heinz Spindler for their great lectures on commutative algebra and algebraic geometry. 
Especially, because many of these courses were held in addition to their teaching duties.

I thank Almar Kaid for several discussions on Hilbert-Kunz theory. These made the beginning of my 
Ph.D. studies easier.

I am thankful to Axel St\"abler for very helpfull discussions on vector bundles and his usefull comments to my thesis. 

Special thanks go to Igor Burban who brought the theory of matrix factorizations to my attention and to Kevin Tucker
who is responsible for the corollaries about the $F$-signature functions of the surface rings of type ADE.

I thank Alessio Caminata for his corrections to my thesis as well as the whole institute of mathematics in Osnabr\"uck for the great time I had in the last years. There was always a warm and familar working atmosphere.

Finally my deep gratitude goes to my family and my friends for their support and patience.

\section*{Notations and conventions}
During this thesis we will make use of the following notations and conventions.

\begin{itemize}
\item $\N:=\{0,1,2,\ldots\}$ and $\N_{\geq a}:=\{a,a+1,a+2,\ldots\}$.
\item All rings are assumed to be non-zero, commutative and Noetherian. Moreover, they have a unit element.
\item A standard-graded ring $R$ is generated over $R_0$ by finitely many elements of degree one.
\item A positively-graded ring $R$ is generated over $R_0$ by finitely many elements of positive degree.
\item A hypersurface denotes a quotient of a polynomial ring by a principal ideal.
\item Prime and maximal ideals are always proper ideals.
\item We denote by $\hat{R}$ the $\mm$-adic completion of a local ring $(R,\mm)$ or the $R_+$-adic completion of a graded ring.
\item If $R\ra S$ is a homomorphism of rings and $I\subseteq R$ an ideal, we denote by $IS$ the ideal of $S$ which is generated by the image of $I$.
\item Krull dimensions of rings and modules are denoted by $\Dim(\_)$.
\item If $V$ is a $k$-vector space, its dimension is denoted by $\Dim_k(V)$.
\end{itemize}

\chapter{A survey on Hilbert-Kunz theory}
This chapter is devoted to give an overview of Hilbert-Kunz theory. In the first section we will give the basic definitions and state some results. 
In the second section we will give a brief introduction to vector bundles and explain how they appear in Hilbert-Kunz theory. Finally, in section three 
we will define the related invariants limit Hilbert-Kunz multiplicity and $F$-signature and list up some further references.

\section{The algebraic viewpoint of Hilbert-Kunz theory}
We start by gathering some properties of the length function $\lambda_R$.

\begin{prop}\label{proplength}
 Let $(R,\mm)$, $(S,\nn)$ be local rings, let $M$ be a finitely generated $R$-module and $N$ a finitely generated $S$-module. Then the following hold.

\begin{enumerate}
 \item If $R$ is a $k$-algebra with $\Dim_k(R/\mm)<\infty$, then $\lambda_R(M)\cdot\Dim_k(R/\mm)=\Dim_k(M).$
 \item If $R\rightarrow S$ is local and $\lambda_R(M)$, $\lambda_S(S/\mm S)$ are finite, then 
	$$\lambda_R(M)\cdot\lambda_S(S/\mm S)\geq \lambda_S(S\otimes_RM).$$
	Moreover, equality holds if the morphism is flat.
 \item Let $R\rightarrow S$ be local with $\lambda_S(N)<\infty$ and $[S/\nn:R/\mm]<\infty$. Then we have 
	$$\lambda_R(N)=\lambda_S(N)\cdot [S/\nn:R/\mm].$$
\end{enumerate}
\end{prop}

\begin{proof}
See Exercises 1.6 (d), 1.8 (b) and Lemma 1.36 (a) in Chapter 7 of \cite{liu}.
\end{proof}

\begin{defi}
Let $(R,\mm)$ be a local ring of characteristic $p>0$. Let $I=(f_1,\ldots,f_m)$ be an $\mm$-primary ideal and 
$M$ a finitely generated $R$-module. The function $\HKF(I,M,p^e):\N\ra\N$, 
$$e\longmapsto\lambda_R\left(M/\left(f_1^{p^e},\ldots,f_m^{p^e}\right)M\right)$$
is called the \textit{Hilbert-Kunz function of $M$ with respect to $I$}.
\end{defi}

\begin{rem}
\begin{enumerate}
 \item Denote by $\TF:R\ra R$ the Frobenius morphism $r\mapsto r^p$. Since this is a homomorphism and $\left(f_1^{p^e},\ldots,f_m^{p^e}\right)=\TF^e(I)$, it is 
      immediately clear that the Hilbert-Kunz function only depends on $I$ and not on the generators. This is not true for the so-called \textit{generalized 
      Hilbert-Kunz function} $\HKF(I,M,n):\N\ra\N,$
      $$n\longmapsto\lambda_R\left(M/(f_1^{n},\ldots,f_m^{n})M\right).$$
      To show this, we compute the values for $n=2$ with respect to two different representations of $\mm$ in a ring of characteristic three. Let
      $$R:=\Z/(3)\llbracket X,Y,Z\rrbracket/(X^2+Y^3+Z^4),$$ 
      $I=\mm$ and $M=R$. Choosing $I=(X,Y,Z)$ we have 
      $$R/(X^2,Y^2,Z^2)=\Z/(3)\llbracket X,Y,Z\rrbracket/(X^2,Y^2,Z^2),$$
      which has length eight. But taking $I=(X+YZ,Y,Z)$, we have 
      $$R/((X+YZ)^2,Y^2,Z^2)=\Z/(3)\llbracket X,Y,Z\rrbracket/(X^2,Y^2,Z^2,XYZ),$$
      which has length seven.
 \item Using Proposition \ref{proplength}, one can show that the Hilbert-Kunz function is independent under completion of $R$ resp. taking the algebraic 
      closure of $R/\mm$ (compare \cite[Lemma 3.1]{kunz1}). Note that this proof also works for the generalized Hilbert-Kunz function.
 \item In view of the previous remark, the Hilbert-Kunz function of a positively-graded affine $k$-algebra, where $k$ is a field and may assumed to be 
      algebraically closed, is always the Hilbert-Kunz function of the $R_+$-adic completion of $R$.
 \item Assume that $(R,\mm)$ is an integral domain of characteristic $p>0$ with perfect residue class field. Denote by $R^{1/q}$ the ring of $q$-th roots of 
      elements in $R$ in an algebraic closure of the field of fractions of $R$. If $R^{1/q}$ is a finite $R$-module for all $q=p^e$ then 
      we have by \cite[Remark 1.3]{watanabe} for any $\mm$-primary ideal $I$ the isomorphism 
      $$R/\TF^e(I)\cong R^{1/q}/IR^{1/q}.$$
\end{enumerate}
\end{rem}

\begin{Not}
\begin{enumerate}
\item If $I=\mm$, we will omit the ideal in our notation, hence $$\HKF(M,q):=\HKF(\mm,M,q).$$ Moreover, we will call this function the Hilbert-Kunz function of $M$.
\item We will also use the notation $I\qpot:=\TF^e(I)$ with $q=p^e$.
\end{enumerate}
\end{Not}

\begin{exa}\label{genhkfan}
We will compute the generalized Hilbert-Kunz function of the surface rings of type $A_n$, given by $R:=k[X,Y,Z]/(X^{n+1}-YZ)$, with respect to the graded 
maximal ideal $\mm=(X,Y,Z)$. We need to compute the length of $R/(X^m,Y^m,Z^m)$ for all $m\in\N$, which is the same as the dimension of the quotient as 
a $k$-vector space. For any $m\in\N$, let $m=l(n+1)+r$ with $l\in\N$ and $0\leq r\leq n$. 
Counting the monomials in 
$$k[X,Y,Z]/(X^m,Y^m,Z^m,X^{n+1}-YZ),$$ 
we can always reduce the power of $X$ to a number in $\{0,\ldots,n\}$ using the defining equation $X^{n+1}=YZ$. Since all monomials divisible by $X^m$ vanish, we obtain the equations
$$\begin{array}{cccccc}
X^m & = & X^rY^lZ^l &=& 0 & \rdelim\}{3}{1cm}[ $n+1-r$ equations]\\
 & \vdots & & & \\
X^{m+n-r} & = & X^nY^lZ^l &=& 0 & \\
X^{m+n-r+1} & = & Y^{l+1}Z^{l+1} &=& 0 & \rdelim\}{3}{1cm}[ $r$ equations]\\
 & \vdots & & & \\
X^{m+n} & = & X^{r-1}Y^{l+1}Z^{l+1} &=& 0 & \\
\end{array}$$
Since the powers of $X$ are linearly independent, we see by the above equations that a monomial $X^hY^iZ^j\in R/(X^m,Y^m,Z^m)$ is non-zero if and only if the exponents satisfy the following relations

\begin{enumerate}
\item If $0\leq h\leq r-1$, then either $0\leq i\leq l$ and $0\leq j\leq m-1$ or $l+1\leq i\leq m-1$ and $0\leq j\leq l$.
\item If $r\leq h\leq n$, then either $0\leq i\leq l-1$ and $0\leq j\leq m-1$ or $l\leq i\leq m-1$ and $0\leq j\leq l-1$.
\end{enumerate}

Counting the number of triples $(h,i,j)$ satisfying the above constraints, we get
\begin{align*}
\HKF(R,m) &= \Dim_k(k[X,Y,Z]/(X^m,Y^m,Z^m,X^{n+1}-YZ))\\
&= (n+1-r)\left[\sum_{i=0}^{l-1}\sum_{j=0}^{m-1}1+\sum_{i=l}^{m-1}\sum_{j=0}^{l-1}1\right]
 +r\left[\sum_{i=0}^{l}\sum_{j=0}^{m-1}1+\sum_{i=l+1}^{m-1}\sum_{j=0}^{l}1\right]\\
&= (n+1-r)(ml+(m-l)l)+r((l+1)m+(m-(l+1))(l+1))\\
&= (n+1-r)(2ml-l^2)+r(2(l+1)m-(l+1)^2)\\
&= (n+1-r)\left(2m\frac{m-r}{n+1}-\left(\frac{m-r}{n+1}\right)^2\right)+r\left(2m\left(\frac{m-r}{n+1}+1\right)-\left(\frac{m-r}{n+1}+1\right)^2\right)\\
\begin{split} &= (n+1-r)\left(\frac{2m^2}{n+1}-\frac{2rm}{n+1}-\frac{m^2}{(n+1)^2}+\frac{2rm}{(n+1)^2}-\frac{r^2}{(n+1)^2}\right)\\
  &\quad +r\left(\frac{2m^2}{n+1}+\frac{2n+2-2r}{n+1}m-\frac{m^2-2rm+r^2}{(n+1)^2}-\frac{2m-2r}{n+1}-1\right)\end{split}\\
\begin{split} &= (n+1-r)\left(\frac{2n+1}{(n+1)^2}m^2-\frac{2rn}{(n+1)^2}m-\frac{r^2}{(n+1)^2}\right)\\
  &\quad +r\left(\frac{2n+1}{(n+1)^2}m^2+\left(2-\frac{2r}{n+1}+\frac{2r}{(n+1)^2}-\frac{2}{n+1}\right)m-\frac{r^2}{(n+1)^2}+\frac{2r}{n+1}-1\right)\end{split}\\
\begin{split} &= \frac{2n+1}{n+1}m^2+\left(-\frac{2rn}{n+1}+\frac{2r^2n}{(n+1)^2}+2r-\frac{2r^2}{n+1}+\frac{2r^2}{(n+1)^2}-\frac{2r}{n+1}\right)m\\
  &\quad -\frac{r^2}{n+1}+\frac{r^3}{(n+1)^2}-\frac{r^3}{(n+1)^2}+\frac{2r^2}{n+1}-r\end{split}\\
&= \left(2-\frac{1}{n+1}\right)m^2+\frac{r^2}{n+1}-r.
\end{align*}
\end{exa}

\begin{rem}\label{kunzbsp}
Similarly to the last example Kunz computed in \cite[Example 4.3]{kunz2} the Hilbert-Kunz functions of binomial hypersurfaces of the form 
$$R:=k[X,Y_1,\ldots,Y_n]/(X^d-Y_1^{a_1}\cdot\ldots\cdot Y_n^{a_n}),$$
where $k$ is a field of characteristic $p>0$ coprime to $d\in\N_{\geq 2}$ and $a_1,\ldots,a_n$ are natural numbers with $\sum a_i>1$. He obtains 
(compare with the introduction of \cite{hkmexists})
$$\HKM(R)=\left\{\begin{aligned}
 & d\cdot\left(1-\prod\left(1-\frac{a_i}{d}\right)\right) && \text{if }a_i<d\text{ for all }i\in\{1,\ldots,n\},\\
 & d && \text{otherwise.}
\end{aligned}\right.$$
\end{rem}

The first deep theorem in Hilbert-Kunz theory was the following.

\begin{thm}[Kunz]\label{kunzthm}
Let $(R,\mm)$ be a local ring of dimension $d$ and positive characteristic $p$. The following statements are equivalent.

\begin{enumerate}
\item $R$ is regular.
\item $R$ is reduced and a flat $R^p$-module.
\item $\HKF(R,p)=p^d$.
\item $R$ is reduced and a flat $R^{p^e}$-module for all $e\in\N$.
\item $\HKF(R,p^e)=p^{ed}$ for all $e\in\N$.
\end{enumerate}
\end{thm}

\begin{proof} See \cite[Theorem 2.1, Theorem 3.3]{kunz1}. \end{proof}

This theorem leads to the question how the quotient $\HKF(I,M,p^e)/p^{ed}$ behaves for large $e$. This question was answered by Monsky. 

Recall that for a set $X$ and a function $f:X\ra\R$, the set $O(f)$ consists of all functions $g:X\ra\R$ with the property that $|g/f|$ is bounded.

\begin{thm}[Monsky]\label{thmhkmexists}
Let $(R,\mm)$ be a local ring of positive characteristic $p$ and dimension $d$. Let $I$ be an $\mm$-primary ideal and $M$ a finitely generated $R$-module.
\begin{enumerate}
\item There is a positive real constant $c(M)$ such that 
	$$\lambda_R\left(M/I\qpot M\right)=c(M)q^d+O\left(q^{d-1}\right).$$
\item If $0\ra L\ra M\ra N\ra 0$ is a short exact sequence of finitely generated $R$-modules, then 
	$$\lambda_R\left(M/I\qpot M\right)=\lambda_R\left(L/I\qpot L\right)+\lambda_R\left(N/I\qpot N\right)+O\left(q^{d-1}\right).$$
\end{enumerate}
\end{thm}

\begin{proof}
See \cite[Theorem 1.8]{hkmexists}.
\end{proof}

Note that part (i) of this theorem ensures that the limit in the following definition exists.

\begin{defi} Let $(R,\mm)$ be a local ring of positive characteristic $p$ and dimension $d$. Let $I$ be an $\mm$-primary ideal and $M$ a finitely generated $R$-module. We call the limit
$$\lim_{e\ra\infty}\frac{\HKF(I,M,p^e)}{p^{ed}}$$
the \textit{Hilbert-Kunz multiplicity of $M$ with respect to $I$}. We denote this limit by $\HKM(I,M)$ and set $\HKM(M):=\HKM(\mm,M)$ for short. We will 
also call $\HKM(M)$ the Hilbert-Kunz multiplicity of $M$.
\end{defi}

\begin{rem}
We give a few remarks on results one obtains on the way to prove Theorem \ref{thmhkmexists}. The references are \cite{hkmexists} and \cite{tightapp}.
\begin{enumerate}
\item We have the inequalities ($\Te(I)$ denotes the ordinary multiplicity of $I$)
      $$\frac{\Te(I)}{d!}\leq\HKM(I,R)\leq\Te(I).$$
      Hence, the Hilbert-Kunz multiplicity of $R$ with respect to $I$ equals the multiplicity of $I$ if $\Dim(R)\leq 1$ and generalizes it in higher dimensions.
\item If $I$ is generated by a system of parameters, one always has $\HKM(I,R)=\Te(I)$.
\item The \textit{associativity formula} holds for Hilbert-Kunz multiplicities, that is 
      $$\HKM(I,M)=\sum_{\substack{\pp\in\Ass(M),\\ \Dim(R/\pp)=\Dim(R)}}\HKM(I,R/\pp)\cdot\lambda_{R_{\pp}}(M_{\pp}).$$
\item We have $\HKM(I,M)=0$ if and only if $\Dim(M)<\Dim(R)$.
\item Two finitely generated $R$-modules have the same Hilbert-Kunz multiplicity with respect to the ideal $I$ if their localizations at all minimal primes of 
      maximal dimension are isomorphic. Hence, the Hilbert-Kunz multiplicity does not detect embedded primes.
\end{enumerate}
\end{rem}

We give an example to the last remark.

\begin{exa}\label{exahkm1}
Let $p$ be a prime number, $R:=\Z/(p)\llbracket X,Y,Z\rrbracket/(XZ,YZ)$ and $M:=R/(Z)$. There is only one minimal prime of maximal dimension of $R$, 
namely $(Z)$. The localizations of $R$ and $M$ at $(Z)$ are both isomorphic to the field of fractions of the ring $R/(Z)$. Their Hilbert-Kunz functions 
are
\begin{align*}
\HKF(M,p^e) &= p^{2e}\text{ and} \\
\HKF(R,p^e) &= p^{2e}+p^e-1
\end{align*}
as one sees by counting monomials as in Example \ref{genhkfan}. This gives $\HKM(R)=\HKM(M)=1$. Treating $R$ and $M$ as $\Z/(p)\llbracket X,Y,Z\rrbracket$-modules, 
their Hilbert-Kunz functions are the same but their Hilbert-Kunz multiplicities are zero.
\end{exa}

The next example shows that the limit 
$$\lim_{n\ra\infty}\frac{\lambda_R\left(M/\left(f_1^n,\ldots,f_m^n\right) M\right)}{n^{\Dim(R)}},$$ 
sometimes called \textit{generalized Hilbert-Kunz multiplicity}, does not exist even in easy cases.

\begin{exa}
Let $R:=k[X,Y,Z]/(X+Y+Z)$, where $k$ is any field of characteristic $p>0$. Consider the sequence 
$$\left(\frac{\Dim_k\left(R/(X^n,Y^n,Z^n)\right)}{n^2}\right)_{n\in\N}.$$
The subsequence to indices of the form $p^e$ converges to 1, but the subsequence to indices of the form $2p^e$ converges to $\tfrac{3}{4}$ as one sees by 
counting monomials as in Example \ref{genhkfan}.
\end{exa}

By Theorem \ref{kunzthm} one has that every regular local ring of positive characteristic has Hilbert-Kunz multiplicity one. In fact, one can 
show in this situation that
$$\HKF(I,R,p^e)=\lambda_R(R/I)\cdot p^{e\cdot\Dim(R)}$$
holds for any $\mm$-primary ideal (see \cite[1.4]{watyo1} and \cite[Exercise 8.2.10]{brunsherzog}). 
By Example \ref{exahkm1} the converse does not hold in general. One can only expect it to be true if there are no embedded primes. In fact this is - as for the Hilbert-Samuel multiplicity (see \cite{nagata}) - already enough.

\begin{thm}[Watanabe, Yoshida]\label{hkm1meansreg}
Let $(R,\mm)$ be an unmixed local ring of positive characteristic. If $R$ has Hilbert-Kunz multiplicity one, then $R$ is regular.
\end{thm}

\begin{proof}
See \cite[Theorem 1.5]{watyo1} or \cite[Theorem 3.1]{hkm1reg} for an alternative proof.
\end{proof}

If the ring has dimension zero, we always have $I\qpot M=0$ for $q$ large enough. Therefore, all Hilbert-Kunz functions in these cases are eventually 
constant. Due to Monsky the structure of Hilbert-Kunz functions over one-dimensional rings is known completely.

\begin{thm}[Monsky]
Let $(R,\mm)$ be a local ring of positive characteristic $p$ and dimension one. Let $I$ be an $\mm$-primary ideal and $M$ a finitely generated $R$-module. There is a 
periodic function $\phi:\N\ra\N$ such that
$$\HKF(I,M,p^e)=\HKM(I,M)\cdot p^e+\phi(e)$$
for all $e\gg 0$.
\end{thm}

\begin{proof}
See \cite[Theorem 3.11]{hkmexists}.
\end{proof}

\begin{rem}
In \cite{hkf1dim} the author gives an algorithm in the case of an one-dimensional, standard-graded, affine $k$-algebra $R$, where $k$ is a 
finite field, that computes $\phi$ and a bound $e_0\in\N$ such that the above theorem holds for all $e\geq e_0$.
\end{rem}

For the next theorems dealing with the structure of Hilbert-Kunz functions over normal domains, we need the \textit{divisor class group} $C(R)$ of $R$, which is 
the quotient of the group of Weil divisors by the normal subgroup of principal divisors. For a finitely generated $R$-module $M$, we choose a primary decomposition with 
quotients $R/P_i$. We attach to $M$ the divisor $-\sum P_j$, where $P_j$ runs through all $P_i$ of height one (with repetitions) in the primary decomposition. The image of 
$-\sum P_j$ in $C(R)$ is independent of the chosen decomposition and is denoted by $c(M)$.

\begin{thm}[Huneke, McDermott, Monsky]
Let $(R,\mm)$ be a local normal domain of positive characteristic $p$ and dimension $d$. Let $M$ and $N$ be finitely generated, torsion-free $R$-modules of the same rank 
$r$. Then the following statements hold.

\begin{enumerate}
\item If $c(M)=0$, then $\HKF(M,q)=r\cdot \HKF(R,q)+O\left(q^{d-2}\right)$.
\item If $c(M)=c(N)$, then $\HKF(M,q)=\HKF(N,q)+O\left(q^{d-2}\right)$.
\end{enumerate}
\end{thm}

\begin{proof}See \cite[Theorem 1.4, Lemma 1.6]{hkfnormal}.\end{proof}

\begin{thm}[Huneke, McDermott, Monsky]
Let $(R,\mm,k)$ be a local, excellent, normal domain of positive characteristic $p$, dimension $d$ and with $k$ perfect. Let $I$ be an $\mm$-primary 
ideal and $M$ a finitely generated $R$-module. Then there exists a constant $\beta(I,M)\in\R$ with 
$$\HKF(I,M,q)=\HKM(I,M)\cdot q^d+\beta(I,M)\cdot q^{d-1}+O\left(q^{d-2}\right).$$ 
\end{thm}

\begin{proof}See \cite[Theorem 1.12]{hkfnormal}.\end{proof}

\begin{rem}
\begin{enumerate}
 \item The numbers $\beta(I,M)$ in the previous theorem can be computed using values of a homomorphism $\tau:C(R)\ra(\R,+)$. If $\tau$ is constant zero 
 (for example if $C(R)$ is torsion), all $\beta(I,M)$ are zero. An important example, where $C(R)$ is torsion, is the case, where $(R,\mm,k)$ is a 
 complete, local, two-dimensional, normal domain and $k$ is the algebraic closure of a finite field (cf. \cite[Lemma 2.1]{hkfnormal}).
 \item In the introduction of \cite{hkfnormal} the authors discuss briefly (non-)possible extensions of the previous theorem to the $O\left(q^{d-2}\right)$-term. They point out that by an example in \cite{monskyhan} the Hilbert-Kunz function of the ring 
 $$R=\Z/(5)[X,Y,Z,W]/(X^4+Y^4+Z^4+W^4)$$
 is given by the formula
 $$\HKF(R,5^e)=\frac{168}{61}5^{3e}-\frac{107}{61}3^e,$$
 which shows that one cannot hope to have a constant factor in front of the $q^{d-2}$-term, even for rings that are regular in codimension two.
 \item In \cite[Corollary 2.5]{secondcoeff} the authors weaken the assumptions in the above theorem to an excellent, equidimensional, reduced ring such that the ideal of the singular locus has height at most two.
 \item In \cite[Theorem 3.2]{jean2} the authors weaken the assumptions in the previous theorem to an excellent, local ring with perfect residue class field such that all localizations at primes $\pp$ of dimension $d$ resp. $d-1$ are fields resp. DVRs. 
	Moreover, they argue with the Chow group instead of the divisor class group.
 \item In \cite{contessa} Contessa uses Koszul homology and the fact 
      $$\lambda_R\left(M/I\qpot M\right)=\lambda_R\left(\lK_0\left(I\qpot;M\right)\right)$$ 
      to compute Hilbert-Kunz functions of finitely generated modules over regular rings with algebraically closed residue class field with respect to ideals generated by 
      a system of parameters. She shows that for a finitely generated 
      torsion-free module $M$ one has 
      $$\HKF(I,M,q)=\Rank(M)\cdot q^d+\text{const}\cdot q^{d-2}+O\left(q^{d-3}\right),$$
      where the constant is a non-negative real number, depending only on $(M^{\vee})^{\vee}/M$ (see \cite[Theorem 3.4]{contessa}). In particular, the constant is zero if $M$ is reflexive.
\end{enumerate}
\end{rem}

The next two statements give bounds for Hilbert-Kunz functions (of rings).

\begin{lem}
Let $(R,\mm,k)$ be a complete local ring of positive characteristic $p$ and assume $k$ to be algebraically closed. Then the inequalities
$$\Min_{\substack{y_1,\ldots,y_d\\ \text{system of parameters}}}\lambda_R(R/(y_1,\ldots,y_d))\cdot q^d\geq \HKF(R,q)\geq q^d$$
hold. Moreover, if $R$ is Cohen-Macaulay, we have $\Te(R)\cdot q^d\geq \HKF(R,q)\geq q^d$.
\end{lem}

\begin{proof}
See \cite[Proposition 3.2]{kunz1}.
\end{proof}

\begin{thm}[Watanabe, Yoshida / Huneke, Yao]\label{watyorelhkm}
Let $(R,\mm)$ be a local ring of characteristic $p>0$. Let $I\subseteq J$ be ideals, where $I$ is $\mm$-primary (and $J=R$ is allowed). Then
$$\HKF(I,R,q)\leq\HKF(J,R,q)+\lambda_R(J/I)\cdot\HKF(R,q).$$
In particular, for $J=R$ one has $\HKF(I,R,q)\leq\lambda_R(R/I)\cdot\HKF(R,q)$.
\end{thm}

\begin{proof} See \cite[Lemma 4.2 (2)]{watyo1} for the analogue statement about the Hilbert-Kunz multiplicities and \cite[Lemma 2.1, Corollary 2.2 (1)]{hkm1reg} for the general case. \end{proof}

The next theorem shows how Hilbert-Kunz multiplicities behave under module-finite extensions of local domains.

\begin{thm}[Watanabe, Yoshida]\label{watyorank}
Let $(R,\mm)\subseteq (S,\nn)$ be an extension of local domains of positive characteristic, where $S$ is a finite $R$-module. Let $I$ be an $\mm$-primary ideal of $R$. Then we have 
$$\HKM(I,R)=\frac{\HKM(IS,S)\cdot\left[S/\nn:R/\mm\right]}{\left[Q(S):Q(R)\right]},$$
where $Q(\_)$ denotes the field of fractions.
\end{thm}

\begin{proof}
See \cite[Theorem 2.7]{watyo1} resp. \cite[Theorem 1.6]{watyo3dim}.
\end{proof}

\begin{rem}\label{hkfflat}
Using Proposition \ref{proplength}, it follows easily that Theorem \ref{watyorank} holds even for the Hilbert-Kunz functions under the additional 
assumptions that $R$ is an $R/\mm$-algebra and that the extension $R\subseteq S$ is flat. Explicitly, we have (the numbers refer to Proposition 
\ref{proplength})

\begin{align*}
\lambda_R\left(R/I\qpot\right) &\stackrel{(ii)}{=} \frac{\lambda_S\left(S/(IS)\qpot\right)}{\lambda_S(S/\mm S)}\\
					&\stackrel{(iii)}{=} \frac{\lambda_S\left(S/(IS)\qpot\right)\cdot \left[S/\nn:R/\mm\right]}{\lambda_R(S/\mm S)}\\
					&\stackrel{(i)}{=} \frac{\lambda_S\left(S/(IS)\qpot\right)\cdot \left[S/\nn:R/\mm\right]}{\Dim_k(S/\mm S)}.
\end{align*}

The denominator is the same as in the theorem, since $$\Dim_k(S/\mm S)=\Dim_{Q(R)}(Q(R)\otimes_kk\otimes_RS)=\Dim_{Q(R)}Q(S).$$
\end{rem}

Using Theorem \ref{watyorank} Brenner reproved a theorem in invariant theory of Smith.

\begin{thm}[Smith / Brenner]
Let $B:=k[X_1,\ldots,X_n]$ with an algebraically closed field $k$ of positive characteristic $p$. Let $G\subsetneq\GL(n,k)$ be a finite subgroup and consider its 
natural linear action on $B$. Let $A:=B^G$ and $\mm:=A_+$. We set $h:=\Dim_k(B/\mm B)$, where $\mm B$ is the Hilbert ideal. Let $R$ be one of the rings 
$A_{\mm}$ and $\widehat{A_{\mm}}$. Then the following statements hold.

\begin{enumerate}
 \item $h=|G|\cdot\HKM(R)\geq |G|.$
 \item Equality holds if and only if $A$ is a polynomial ring.
\end{enumerate}
\end{thm}

\begin{proof}
The original proof was given in \cite{larry} and the alternative proof can be found in \cite{invariant}.
\end{proof}

A second application of Theorem \ref{watyorank} is the computation of the Hilbert-Kunz multiplicities of surface rings with isolated singularities of type ADE.

\newpage

\begin{defi}
Let $R:=k[X,Y,Z]/(F)$, where $k$ is a field. We say that $R$ and $\widehat{R}$ are \textit{surface rings of type} 

\begin{itemize}
 \item[$A_n$] if $F=X^{n+1}+YZ$ and $n\geq 0$,
 \item[$D_n$] if $F=X^2+Y^{n-1}+YZ^2$ and $n\geq 4$,
 \item[$E_6$] if $F=X^2+Y^3+Z^4$,
 \item[$E_7$] if $F=X^2+Y^3+YZ^3$,
 \item[$E_8$] if $F=X^2+Y^3+Z^5$.
\end{itemize}

In any case a surface ring of type ADE has in almost all characteristics an isolated singularity at the origin and we refer to these singularities as 
\textit{singularities of type ADE}.
\end{defi}

\begin{rem}
Recall that the surface rings $\C[X,Y,Z]/(F)$ of type ADE appear as rings of invariants of $\C[x,y]$ by the actions of the finite subgroups of $\SL_2(\C)$. 
The groups corresponding to the singularities of type $A_n$, $D_n$, $E_6$, $E_7$ resp. $E_8$ are the cylic group with $n+1$ elements, the binary 
dihedral group of order $4n-8$, the binary tetrahedral group of order 24, the binary octahedral group of order 48 resp. the binary icosahedral group 
of order 120. If $k$ is algebraically closed of characteristic $p>0$ the groups above can be viewed as finite subgroups of $\SL_2(k)$, provided their order 
is invertible in $k$. In these cases $k[X,Y,Z]/(F)$ is again the ring of invariants of $k[x,y]$ under the action of the corresponding group 
(cf. \cite[Chapter 6, \S 2]{leuwie}).
\end{rem}

Now, we show how one might use Theorem \ref{watyorank} to compute the Hilbert-Kunz multiplicities of surface rings of type ADE. See \cite[Theorem 5.4]{watyo1} 
for a stronger formulation of the next theorem.

\begin{thm}[Watanabe, Yoshida]\label{watyogroup}
Let $R$ be a surface ring of type ADE. Assume that $k$ is algebraically closed and that the order of the corresponding group $G$ is invertible modulo 
$\chara(k)>0$. Then 
$$\HKM(R)=2-\frac{1}{|G|}.$$ 
\end{thm}

\begin{proof}
We prove the graded version of the statement. The proof carries over to the complete situation. Let $R$ be of type $A_n$. Then Example \ref{genhkfan} 
shows $\HKM(R)=2-\tfrac{1}{n+1}$. Let $R$ be of type $D_n$ or $E$. Denote the corresponding group by $G$ and let $k[x,y]$ be the ring on which $G$ acts. 
Then the extension $Q(k[x,y]):Q(R)$ has degree $|G|$ (cf. \cite[Proposition 5.4]{leuwie}). The ring $R=k[x,y]^G$ is generated as $k$-algebra by three 
invariant homogeneous polynomials $X$, $Y$, $Z$ such that $X^2$ is a polynomial in $Y$ and $Z$ and such that $\Deg(Y)\cdot\Deg(Z)=2\cdot|G|$ holds 
(cf. Chapter \ref{pbmaxcm}). Since $k[x,y]$ is regular, we have $\HKM(I,k[x,y])=\lambda(k[x,y]/I)$ for every $(x,y)$-primary ideal $I$. This shows
$$\HKM((X,Y,Z),k[x,y])=\lambda(k[x,y]/(X,Y,Z))=\lambda(k[x,y]/(Y,Z))-1=2\cdot|G|-1.$$
By Theorem \ref{watyorank} we have
$$\HKM(R)=\frac{\HKM((X,Y,Z),k[x,y])}{|G|}=2-\frac{1}{|G|}.$$
\end{proof}

Keeping the assumptions that $R$ is a Cohen-Macaulay local ring of dimension two, one can prove the inequality
$$\HKM(R)\geq\frac{r+2}{2r+2}\cdot\Te(R),$$
where $r$ denotes the number of generators of $\mm/J$ for a minimal reduction $J$ of $\mm$ (cf. \cite[Lemma 5.5]{watyo1}). 
In \cite{watyo2} the authors study the numbers $\HKM(I^n,R)$ for $\mm$-primary ideals $I$. In particular, they show that for a local, two-dimensional, 
Cohen-Macaulay ring with algebraically closed residue class field one has 
$$\HKM(\mm^n,R)\geq \frac{\Te(R)}{2}n^2+\frac{n}{2}$$
and that this minimal bound is achieved for all $n$ if and only if $\HKM(R)=\frac{\Te(R)+1}{2}$ (cf. \cite[Theorem 2.5]{watyo2}). Moreover, they show 
that a local, two-dimensional, Cohen-Macaulay ring $R$ with algebraically closed residue class field has minimal Hilbert-Kunz multiplicity 
($n=1$ in the above inequality) if and only if the associated graded ring of $R$ is isomorphic to the $\Te(R)$-th Veronese subring of $k[X,Y]$ (cf. \cite[Theorem 3.1]{watyo2}).

One big class of rings in which Hilbert-Kunz functions and multiplicities are studied is the class of rings that arise from combinatorial problems. 
For example monoid and toric rings are studied in \cite{eto1} and \cite{watanabe}. An explicit treatment of the two-dimensional toric case is done 
in \cite{choi}. In \cite{watyo1} and \cite{eto2} the authors focus on Rees algebras. In \cite{eto3} (see also \cite{eto2} and \cite{bcp}), Eto computes 
the Hilbert-Kunz multiplicity of Segre products of polynomial rings.
The (generalized) Hilbert-Kunz functions of $2\times 2$ determinantal rings (in $m\cdot n$ many variables) are studied in \cite{lance}. They obtain in the case $m=2$ the following theorem (cf. \cite[Theorem 4.4, Corollary 4.5]{lance}).

\begin{thm}[Miller, Swanson]
Let $k$ be a field and 
$$R:=k[X_{i,j}|i=1,2;j=1,\ldots,n]/I_2,$$ 
where $I_2$ is the ideal generated by the two-minors of the matrix $(X_{ij})_{ij}$. For any natural number $m\in\N$ we obtain

\begin{enumerate}
\item $$\HKF(R,m)=\frac{nm^{n+1}-(n-2)m^n}{2}+n\binom{n+m-1}{n+1}$$
\item $$\HKM(R)=\frac{n}{2}+\frac{n}{(n+1)!}$$
\end{enumerate}
\end{thm}

\begin{rem}
The above result was extended in \cite{swanson} to the case 
$$R:=k[X_{i,j}|i=1,\ldots,m;j=1,\ldots,n]/I_2.$$
\end{rem}

The structure of (generalized) Hilbert-Kunz functions for monomial rings and binomial hypersurfaces is studied in \cite{conca}.

\begin{thm}[Conca]
Let $I$ be a monomial ideal in $S=k[X_1,\ldots,X_n]$, where $k$ is a field and let $R=S/I$. Then there is a polynomial $P_R(Y)\in\Z[Y]$ of degree $\Dim(R)$ with leading coefficient $\Te(R)$ such that $\HKF(R,m)=P_R(m)$ for all $m\in\N$ that are at least the highest exponent appearing in the generators of $I$. 
\end{thm}

\begin{proof}
See \cite[Theorem 2.1]{conca}.
\end{proof}

\begin{thm}[Conca]
Let $R=k[X_1,\ldots,X_{n+m}]/(F)$, where $k$ is a field and 
$$F=X_1^{a_1}\cdot\ldots\cdot X_m^{a_m}-X_{m+1}^{b_{m+1}}\cdot\ldots\cdot X_{m+n}^{b_{m+n}}$$
is a homogeneous binomial with maximal appearing exponent $u$. Then there exists a rational polynomial $P_R(Y,Z)\in\Q[Y,Z]$ of degree $n+m-1$ with leading 
coefficient $C$ and an integer $\alpha\geq 0$ with
$$\HKF(R,l)=P_R(l,\epsilon)$$
for all $l\geq\alpha$ and $\epsilon\equiv l\text{ }(u)$ with $\epsilon\in\{0,\ldots,u-1\}$. Moreover, we have
$$C=\sum_{i=1}^m\sum_{j=1}^n(-1)^{i+j}s_{im}(a)s_{jn}(b)\frac{ij}{(i+j-1)u^{i+j-1}},$$
with $a=(a_1,\ldots,a_m)$, $b=(b_{m+1},\ldots,b_{m+n})$ and $s_{\alpha\beta}$ denotes the elementary symmetric polynomial of degree $\alpha$ in $\beta$ indeterminates.
\end{thm}

\begin{proof}
See \cite[Theorem 3.1]{conca}.
\end{proof}

We now turn to the hypersurface case. Studying properties of the function $D_F:\N^n\ra \N$,
$$(a_1,\dots,a_n)\mapsto\Dim_k(k[X_1,\ldots,X_n]/(X_1^{a_1},\ldots,X^{a_n},X_1+\ldots+X_n)),$$
Han and Monsky developed in \cite{monskyhan} an
algorithm to compute Hilbert-Kunz functions of \textit{diagonal hypersurfaces} 
$$k[X_1,\ldots,X_n]/\left(X^{d_1}+\ldots+X^{d_n}\right)\text{ with }d_1,\ldots,d_n\in\N_{\geq 2}.$$
In particular, they obtain the following result on the structure of Hilbert-Kunz functions in this case.

\begin{thm}[Han, Monsky]
Let $R$ be a $d$-dimensional diagonal hypersurface over a field $k$ of positive characteristic $p$. Then the following holds.

\begin{enumerate}
 \item If $\chara(k)=2$ or $d=2$, we have 
      $$\HKF(R,p^e)=\HKM(R)\cdot p^{de}+\Delta_e,$$
      where $e\mapsto\Delta_e$ is eventually periodic.
 \item If $\chara(k)\geq 3$ and $d\geq 3$ there are integers $\mu\geq 1$ and $0\leq l\leq p^{(d-2)\mu}$ such that
      $$\HKF(R,p^e)=\HKM(R)\cdot p^{de}+\Delta_e,$$
      where $\Delta_{e+\mu}=l\cdot\Delta_e$ for $e\gg 0$.
\end{enumerate}

Moreover, $\Delta_e\leq 0$ in any case.
\end{thm}

\begin{proof}
See \cite[Theorem 5.7]{monskyhan} and \cite[Corollary 5.9]{monskyhan}. See \cite[Theorem 1.1]{chiang1} for a different proof using the same tools.
\end{proof}

\begin{rem}
\begin{enumerate}
\item The ideas of \cite{monskyhan} are generalized in \cite{chiang2} to hypersurfaces $R=k[X_{i,j}|i=1,\ldots,m,j=1,\ldots,t_i]/(f)$ with
$$f=\sum_{i=1}^m\prod_{j=1}^{t_i}X_{i,j}^{d_{i,j}}.$$
\item In her thesis \cite{chang}, Chang used the algorithm of Han and Monsky to compute the Hilbert-Kunz function of the Fermat-quartic 
$$R=\Z/(p)[W,X,Y,Z]/(W^4+X^4+Y^4+Z^4)$$ 
for odd primes $p$. Explicitly, she obtains
\begin{equation*}\HKF(R,p^e)=\left\{\begin{aligned}
 & \frac{2}{3}\left(\frac{8p^2+8p+12}{2p^2+2p+1}\right)p^{3e}-\frac{1}{3}\left(\frac{10p^2+10p+21}{2p^2+2p+1}\right)\left(\frac{p+1}{2}\right)^e && \text{if }p\equiv 1\text{ } (4),\\
 & \frac{2}{3}\left(\frac{8p^3+4p+12}{2p^2-p+1}\right)p^{3e}-\frac{1}{3}\left(\frac{10p^3+11p+21}{2p^3-p+1}\right)\left(\frac{p-1}{2}\right)^e && \text{if }p\equiv 3\text{ } (4).
\end{aligned}\right.\end{equation*}
\item In his thesis \cite{chen}, Chen remarks that there cannot be a numerical function of polynomial type having the Hilbert-Kunz multiplicity of the Fermat-quartic as leading coefficient.
\item In the special case $n=2$, the function $D_F$ is also known as Han's $\delta$-function, which we will study in more detail in Chapter \ref{chaphan}.
\item For a hypersurface $R$, the sequence $$\left(\frac{\HKF(R,p^e)}{p^{e\cdot\Dim(R)}}\right)_{e\in\N}$$ is non-decreasing by \cite[Theorem 5.8]{monskyhan}.
\end{enumerate}
\end{rem}

The special case of the Hilbert-Kunz multiplicities of $d$-dimensional Fermat-quadrics reappear as lower bounds for non-regular complete intersections.

\begin{thm}[Enescu, Shimomoto]
Let $(R,\mm,k)$ be a local, non-regular, complete intersection of dimension $d\geq 2$ with $\chara(k)\neq 2$. Then 
$$\HKM(R)\geq\HKM\left(k\llbracket X_0,\ldots,X_d\rrbracket/\left(\sum_{i=0}^dX_i^2\right)\right).$$
\end{thm}

\begin{proof} This bound was conjectured in \cite{watyo3dim} and the $d$-dimensional case was proven in \cite[Theorem 4.6]{hkmupper}.\end{proof}

\begin{rem}
In \cite[Theorem 3.8]{limitmonsky} the authors showed that 
$$\left(\HKM\left(k\llbracket X_0,\ldots,X_d\rrbracket/\left(\sum_{i=0}^dX_i^2\right)\right)\right)_{\chara(k)=p}$$ 
tends to 1+ the coefficient of $t^{d-1}$ in the power series of $\sec(t)+\tan(t)$ around zero as $p$ goes to infinity.
\end{rem}

Recall that for regular rings one has $\HKF(I,R,p^e)=\lambda_R(R/I)\cdot p^{e\cdot\Dim(R)}$ (discussion before Theorem \ref{hkm1meansreg}). This is generalized 
to complete intersections by the following theorem.

\begin{thm}[Dutta, Miller]\label{thmclaudia}
Let $(R,\mm,k)$ be a local ring of dimension $d$ which is a complete intersection and let $M$ be a finitely generated $R$-module of finite 
length. Then the following statements hold.
\begin{enumerate}
 \item We have $\lambda_R(\TF^e(M))\geq \lambda_R(M)\cdot p^{ed} \text{ for all }e\in\N$. 
 \item The sequence $(\lambda_R(\TF^e(M))\cdot p^{-ed})_{e\in\N}$ is non-decreasing.
 \item The following statements are equivalent.
 
  \begin{enumerate}
    \item The module $M$ has finite projective dimension.
    \item We have $\lambda_R(\TF^e(M))=\lambda_R(M)\cdot p^{ed}$ for all $e\in\N$.
    \item The limit $\lim_{e\ra\infty}\lambda_R(\TF^e(M))\cdot p^{-ed}$ exists and equals $\lambda_R(M)$.
  \end{enumerate}
\noindent In particular, if $I$ is an $\mm$-primary ideal such that the quotient $R/I$ has finite projective dimension, one has 
$$\HKF(I,R,p^e)=\lambda_R(R/I)\cdot p^{ed}.$$
\end{enumerate}
\end{thm}

\begin{proof}
Part (i) and the implication (a) $\Rightarrow$ (b) of (iii) can be found in \cite[Theorem 1.9]{dutta}. For the other statements see \cite[Theorem 2.1]{claudia}.
\end{proof}

An open question in Hilbert-Kunz theory is whether Hilbert-Kunz multiplicities are rational or not. Many experts thought for a long time Hilbert-Kunz 
multiplicities would always be rational (cf. \cite{hkmexists}), since this was true in all explicit examples and in all cases, where the Hilbert-Kunz multiplicity was formulated 
in terms of a different theory as combinatorial algebra or vector bundles. This opinion has changed due to a conjecture of Monsky.

\begin{thm}[Monsky]
Let $S:=\Z/(2)[X,Y,Z]$ and $h:=X^3+Y^3+XYZ\in S$. Suppose that Monskys conjecture \cite[Conjecture 1.5]{hkmrat} on the value of 
$$\HKF\left(S/\left(h^{2l+1}\right),2^{e+1}\right)-\frac{1}{2}\cdot\left(\HKF\left(S/\left(h^{2l},2^{e+1}\right)\right)+\HKF\left(S/\left(h^{2l+2},2^{e+1}\right)\right)\right)$$
holds for all $0\leq l<2^e$. Then the following statements hold.

\begin{enumerate}
 \item Let $2^e\geq j$. Then there are recursively defined integers $v_j$ and $u_j$ such that
      $$\HKF(S/(h^j),2^e)=\left\{\begin{aligned} & \frac{7j}{3}\cdot 4^e-\frac{j^2}{3}\cdot 2^e+u_j && \text{if }j\equiv 1\text{ }(3), \\ & \frac{7j}{3}\cdot 4^e-\frac{j^2}{3}\cdot 2^e+v_j && \text{if }j\equiv 2\text{ }(3).\end{aligned}\right.$$
 \item $$\HKM(\Z/(2)[X,Y,Z,U,V]/(UV+h))=\frac{4}{3}+\frac{5}{14\cdot\sqrt{7}}.$$
 \item The number $$\sum_{n\geq 0}\binom{2n}{n}^2\cdot\left(\frac{1}{2^{16}}\right)^n$$ is transcendental and a $\Q$-linear combination of Hilbert-Kunz 
	multiplicities of hypersurfaces in characteristic two.
\end{enumerate}
\end{thm}

\begin{proof}
For (i) and (ii) see \cite[Theorem 1.9, Corollary 2.7]{hkmrat}. The proof that the number in (iii) is transcendental (and exists) can be found in \cite{schneider} and the rest of the statement in \cite[Theorem 2.2]{hkmtrans}.
\end{proof}

Very recently, Brenner came up in \cite{holgerirre} with a quite explicit example that Hilbert-Kunz multiplicities might be irrational.

\begin{thm}[Brenner]\label{hkmirratholger}
There exists a local domain whose Hilbert-Kunz multiplicity is irrational and there exists a three-dimensional hypersurface ring 
$$R=k[X,Y,Z,W]/(F),$$ 
where $F$ is homogeneous of degree 4 and $k$ is an algebraically closed field of 
large characteristic, and an Artinian $R$-module $M$ such that $\HKM(M)$ is irrational.
\end{thm}

At the end of this section we state some results on the interaction of Hilbert-Kunz multiplicities with the theory of tight closure.

\begin{defi}
Let $R$ be a ring of characteristic $p>0$ and let $I$ be an ideal of $R$. We say that $x\in R$ belongs to the \textit{tight closure} of $I$ 
if $cx^q\in I\qpot$ for a $c$ in the complement of the minimal primes of $R$ not depending on $q$ and all large $q$. The tight closure of $I$ is an ideal, 
which we denote by $I\pb$ and we say that $I$ is \textit{tightly closed} in the case $I=I\pb$.
\end{defi}

The following theorem shows that the Hilbert-Kunz multiplicity measures the membership to the tight closure.

\begin{thm}[Hochster, Huneke]
Let $(R,\mm)$ be a local ring of characteristic $p>0$ and let $I\subseteq J$ be $\mm$-primary ideals. Then 
$$I\pb=J\pb \text{ }\Longrightarrow \text{ } \HKM(I,R)=\HKM(J,R).$$
The converse holds, if $R$ is quasi-unmixed and analytically unramified.
\end{thm}

\begin{proof}
See \cite[Theorem 5.4]{tightapp}.
\end{proof}

The previous theorem tells us that Hilbert-Kunz multiplicities behave to tight closure as Hilbert-Samuel multiplicities behave to integral closure.

\begin{defi}
We call a ring $R$ of characteristic $p>0$ an \textit{F-rational} ring if every ideal generated by parameters is tightly closed.
\end{defi}

The following lemma gives a criterion for Cohen-Macaulay rings of multiplicity two to be $F$-rational.

\begin{lem}
Let $(R,\mm)$ be Cohen-Macaulay of positive characteristic with multiplicity two. Then $R$ is $F$-rational if and only if $\HKM(R)<2$. 
\end{lem}

\begin{proof} See \cite[5.3]{watyo1}.\end{proof}

In view of Theorem \ref{watyorank} we obtain the following corollary.

\begin{cor}
Let $R$ be a surface ring of type ADE such that the order of the associated group is invertible in the ground field. Then $R$ is $F$-rational.
\end{cor}

\section{The geometric viewpoint of Hilbert-Kunz theory}
In this section we introduce some basic notations concerning vector bundles, discuss briefly the different Frobenius morphisms appearing in algebraic geometry, 
focus a little more on syzygy bundles and different stability properties of vector bundles. We will often restrict ourselves to projective curves, since the definitions get much simpler 
in this situation. Thereafter we will explain the connection of the theory of vector bundles and Hilbert-Kunz theory. We will finish this section by a discussion 
of the Hilbert-Kunz functions of the homogeneous coordinate rings of elliptic curves and irreducible curves of degree three.

\subsection{Generalities on vector bundles}
\begin{defi}
Let $X$ be a scheme. A (geometric) \textit{vector bundle} of rank $r$ over $X$ is a scheme $f:E\ra X$ together with an open covering $\{U_i\}_{i\in I}$ of $X$ 
and isomorphisms $\Psi_i:f^{-1}(U_i)\ra\A^r_{U_i}$ such that for any $i,j$ and for any open affine set $V:=\Spec(A)\subseteq U_i\cap U_j$ the automorphism
$$\Psi:=\Psi_j|_{f^{-1}(U_j)}\circ\Psi^{-1}_i|_{\A_V^r}$$
is given by an $A$-linear automorphism of $A[X_1,\ldots,X_r]$.
\end{defi}

\begin{rem}
\begin{enumerate}
 \item There is an equivalence of the categories of vector bundles over $X$ and the category of locally free sheaves of constant (finite) rank over $X$. This equivalence attaches 
    to a vector bundle of rank $r$ its sheaf of sections, which carries a natural $\Oc_X$-module structure and is locally free of rank $r$. For the converse, attach to a locally free 
    sheaf of rank $r$ the relative spectrum of its symmetric algebra. This gives a vector bundle of rank $r$, whose sheaf of sections is the dual of the locally free sheaf we started with
    (cf. \cite[Exercise II.5.18]{hartshorne}).
 \item Note that the quotient of a vector bundle by a subbundle is again a vector bundle (cf. \cite[Proposition 1.7.1]{lepotier}). Hence an inclusion of locally free sheaves $\Sc\subseteq\Tc$
    yields by the above equivalence a closed immersion of vector bundles if and only if the quotient $\Tc/\Sc$ is itself a locally free sheaf (cf. \cite[Proposition 1.7.11]{ega2}).
\end{enumerate}
\end{rem}

\begin{defi}
Let $\Ec$ be a locally free sheaf of rank $r$ on $X$. Then 
$$\Det(\Ec):=\bigwedge^r\Ec$$ 
is called the \textit{determinant-bundle} or \textit{determinant} of $\Ec$.
\end{defi}

\begin{exa}
If $\Lc$ is an invertible sheaf, we have $\Det(\Lc)=\Lc$.
\end{exa}

By \cite[Exercise II.5.16]{hartshorne} determinants satisfy the following properties.

\newpage

\begin{prop}\label{propdet} The situation is the same as in the previous definition.
\begin{enumerate}
 \item Determinants are invertible.
 \item If $r=2$, there is an isomorphism $$\Ec\cong\Ec^{\vee}\otimes\Det(\Ec).$$
 \item If $f:Y\ra X$ is a morphism of schemes and $\Ec$ a locally free $\Oc_X$-module of rank $r$, we have an isomorphism
    $$f\pb(\Det(\Ec))\cong\Det(f\pb(\Ec)).$$
 \item From a short exact sequence 
    $$0\ra\Fc\ra\Ec\ra\Gc\ra 0$$
    of locally free sheaves, we obtain the isomorphism
    $$\Det(\Fc)\otimes\Det(\Gc) \cong \Det(\Ec).$$
\end{enumerate}
\end{prop}

\begin{defi}
Let $\Ec$ be a locally free sheaf of rank $r$ on a smooth projective curve $C$. Then $\Deg(\Ec):=\Deg(\Det(\Ec))$ is called the \textit{degree} of $\Ec$, where the degree 
of the line bundle $\Det(\Ec)$ is the degree of the corresponding Weil divisor. If $C$ is a projective curve with an ample invertible sheaf $\Oc_C(1)$ 
the \textit{degree} of $C$ is defined as $\Deg(C):=\Deg(\Oc_C(1))$.
\end{defi}

\begin{exa}
We have $\Deg(\Oc_C(m))=m\cdot\Deg(\Oc_C(1))=m\cdot\Deg(C)$.
\end{exa}

\begin{rem}
\begin{enumerate}
 \item Note that one could also define the degree of $\Ec$ as $$\chi(\Ec)-r\cdot\chi(\Oc_X)=\chi(\Ec)-r(1-g),$$ where $\chi$ denotes the Euler characteristic. 
      This definition is equivalent to our by \cite[Theorem 2.6.9]{lepotier}.
 \item From the alternative definition it becomes clear that the degree is additive on short exact sequences (see \cite[Lemma 1.5 (2)]{maruyama}). 
 \item By \cite[Lemma 1.5 (4)]{maruyama} we have the formula
      $$\Deg(\Ec\otimes\Fc)=\Rank(\Ec)\Deg(\Fc)+\Rank(\Fc)\Deg(\Ec).$$
\end{enumerate}
\end{rem}

We close this subsection by stating two prominent theorems for locally free sheaves.

\begin{thm}[Serre duality]
Let $X$ be a smooth projective variety of positive dimension $d$ defined over an algebraically closed field and let $\omega_X:=\bigwedge^d\Omega_X$ 
be the canonical bundle on $X$. For every locally free sheaf $\Ec$ on $X$ there are natural isomorphisms
$$\lK^i(X,\Ec)\cong\lK^{d-i}(X,\Ec^{\vee}\otimes\omega_X)^{\ast}$$
for all $i=0,\ldots,d$, where $^{\ast}$ denotes the dual vector space.
\end{thm}

\begin{proof}
See \cite[Corollary III.7.7]{hartshorne}.
\end{proof}

\begin{thm}[Riemann-Roch]\label{thmrr}
Let $C$ be a smooth projective curve of genus $g$ defined over an algebraically closed field $k$ and let $\Ec$ be a locally free sheaf on $C$. Then we have 
$$\text{h}^0(C,\Ec)-\text{h}^1(C,\Ec)=\Deg(\Ec)+(1-g)\cdot\Rank(\Ec),$$
where $\Th^i(\_):=\Dim_k(\lK^i(\_))$.
\end{thm}

\begin{proof}
See \cite[Theorem 2.6.9]{lepotier}.
\end{proof}

\subsection{Frobenius}

In this section we define the (absolute) Frobenius morphism.

\begin{defi}
Let $X$ be a scheme over a field $k$ of positive characteristic $p$. Then the map $\TF_X:X\ra X$ given by the identity on the topological spaces and the Frobenius homomorphism on the section rings, is called the \textit{(absolute) Frobenius}. Note that we will write $\TF$ instead of $\TF_X$ if we deal only with one scheme.
\end{defi}

\begin{lem}\label{frobprop}
\begin{enumerate}
\item Let $g:X\ra Y$ be a morphism of schemes over $\Z/(p)$. Then the equality $\TF_Y\circ g=g\circ \TF_X$ holds.
\item For $x\in X$, we have $\TF_X(x)=x$.
\item If all stalks of $X$ are regular rings, then $\TF_X$ is flat.
\item If $X$ is integral and geometrically reduced, we have $\Deg(\TF_X)=p^{\Dim(X)}$.
\end{enumerate}
\end{lem}

\begin{proof}
See \cite[Lemma 3.2.22]{liu} for parts (i) and (ii). Part (iii) follows from Kunz' Theorem \ref{kunzthm}. Part (iv) is Corollary 3.2.27 of \cite{liu}.
\end{proof}

Since the absolute Frobenius acts by $x\mapsto x^p$ on $k$, the morphism $\TF_X:X\ra X$ is $k$-linear if and only if $k=\Z/(p)$. In the general case, one obtains a $k$-linear morphism by the following construction. First, form the pull-back
$$X^{(p)}:=X\times_{\Spec(k)}\Spec(k)\stackrel{\text{pr}_2}{\ra}\Spec(k)$$
with respect to the structure map $f:X\ra\Spec(k)$ and the absolute Frobenius $\TF_{\Spec(k)}$ (see the diagram below). 

Secondly, one obtains by the universal property of pull-backs (and part (i) of the previous lemma) a morphism $\TF_{\text{rel}}:X\ra X^{(p)}$ of varieties over $k$, which we call \textit{relative Frobenius} or \textit{$k$- linear Frobenius}.

The following diagram shows all involved maps and schemes.
$$\xymatrix{
X\ar@/_/[ddrr]_f\ar@/^/[drrrr]^{\TF_X}\ar@{-->}[drr]^{\TF_{\text{rel}}} & & & & \\
& & X^{(p)}\ar[rr]^{\text{pr}_1}\ar[d]^{\text{pr}_2} & & X\ar[d]^f \\
& & \Spec(k)\ar[rr]^{\TF_{\Spec(k)}} & & \Spec(k)
}$$
If $X$ is an affine variety $\Spec(k[T_1,\ldots,T_n]/I)$ over $k$, then the absolute Frobenius acts on $\Oc_X$ by sending $g\in k[T_1,\ldots,T_n]/I$ to $g^p$, while the relative Frobenius acts on $\Oc_X$ by raising only the (images of the) variables $T_i$ to their $p$-th power (compare \cite[Section 2.1]{liedtke}). See \cite[Section 3.2.4]{liu} for a more general definition of the absolute and relative Frobenius morphisms.

Note that if $k$ is perfect, then $X$ and $X^{(p)}$ are isomorphic as schemes over $k$, since their structural morphisms differ only by an isomorphism.

\begin{defi}
Let $\TF:X\ra X$ be the absolute Frobenius morphism and $\Gc$ an $\Oc_X$-module. We call the pull-back $\fpb{}(\Gc)$ of $\Gc$ defined by the following diagram
$$\xymatrix{
\fpb{}(\Gc)\ar[r]\ar[d] & \Gc\ar[d]\\
X\ar[r]^{\TF} & X
}$$
the \textit{Frobenius pull-back} of $\Gc$.
\end{defi}

For line bundles the Frobenius pull-back is easy to compute.

\begin{lem}\label{fpbline}
If $X$ is separable and $k$ is algebraically closed, we have $$\fpb{}(\Lc)\cong\Lc^p$$ for every invertible sheaf $\Lc$ on $X$.
\end{lem}

\begin{proof}
See \cite[Lemma 1.2.6]{frobsplit}.
\end{proof}

We end up this section by defining the Hasse invariant of elliptic curves (compare with \cite[Chapter IV.4]{hartshorne}). Let $X$ be an elliptic curve over a perfect field $k$ of positive characteristic $p$. Then the absolute Frobenius $\TF:X\ra X$ induces a map 
$$\lK^1(\TF):\lK^1(X,\Oc_X)\ra\lK^1(X,\Oc_X)$$ 
on cohomology. This map fulfills 
$$\lK^1(\TF)(\lambda\cdot a)=\lambda^p\cdot\lK^1(\TF)(a)$$
for all $\lambda\in k$ and $a\in \lK^1(X,\Oc_X)$. Since $X$ is elliptic, we have $\Dim_k(\lK^1(X,\Oc_X))=1$. Thus, the map $\lK^1(\TF)$ has to be either 0 or bijective, since $k$ is perfect.

\begin{defi}
The \textit{Hasse invariant} of $X$ is zero if $\lK^1(\TF)$ is the zero map and one if $\lK^1(\TF)$ is bijective.
\end{defi}

\subsection{Syzygy bundles}

In this subsection we will turn our attention to syzygy bundles, since they give a nice class of examples of vector bundles and because exactly these are 
the vector bundles that will appear in the application of vector bundles to Hilbert-Kunz theory.

\begin{defi}
Let $k$ be a field and $R$ an affine $k$-algebra. Let $f_1,\ldots,f_n\in R$ and $X:=\Spec(R)$. 
\begin{enumerate}
\item We call the kernel of the map 
	$$\Oc_X^n\stackrel{f_1,\ldots,f_n}{\lra}\Oc_X$$
	the \textit{sheaf of syzygies} for $f_1,\ldots,f_n$ and denote it by $\Syz_X(f_1,\ldots,f_n)$. 
\item Let $\Ic_Z$ be the ideal sheaf of the closed subscheme $Z:=V(f_1,\ldots,f_n)\subseteq X$. We call the short exact sequence
	$$0\lra\Syz_X(f_1,\ldots,f_n)\lra\Oc_X^n\stackrel{f_1,\ldots,f_n}{\lra}\Ic_Z\lra 0$$
	the \textit{presenting sequence} of the sheaf of syzygies.
\end{enumerate}
\end{defi}

\begin{defi}
Let $k$ be a field and $R$ a standard-graded, affine $k$-algebra. Let $f_1,\ldots,f_n\in R$ be homogeneous elements with $\Deg(f_i)=d_i$. Let $Y:=\Proj(R)$. 
\begin{enumerate}
\item We call the kernel of the map 
	$$\bigoplus_{i=1}^n\Oc_Y(-d_i)\stackrel{f_1,\ldots,f_n}{\lra}\Oc_Y$$
	the \textit{sheaf of syzygies} for $f_1,\ldots,f_n$ and denote it by $\Syz_Y(f_1,\ldots,f_n)$. 
\item Let $\Ic_Z$ be the ideal sheaf of the closed subscheme $Z:=V_+(f_1,\ldots,f_n)\subseteq Y$. We call the short exact sequence
	$$0\lra\Syz_Y(f_1,\ldots,f_n)\lra\bigoplus_{i=1}^n\Oc_Y(-d_i)\stackrel{f_1,\ldots,f_n}{\lra}\Ic_Z\lra 0$$
	the \textit{presenting sequence} of the sheaf of syzygies.
\end{enumerate}
\end{defi}

Note that in both situations the sheaf of syzygies $\Syz_X(f_1,\ldots,f_n)$ resp. $\Syz_Y(f_1,\ldots,f_n)$ is nothing but the sheafification of the (graded) 
$R$-module $\Syz_R(f_1,\ldots,f_n)$. Therefore we may define in the graded case for any $m\in\Z$ the \textit{sheaf of syzygies of total degree $m$} of 
$f_1,\ldots,f_n$, denoted by $\Syz_Y(f_1,\ldots,f_n)(m)$, as the sheafification of the $R$-module $\Syz_R(f_1,\ldots,f_n)(m)$. Note that the equality
$$\Syz_Y(f_1,\ldots,f_n)(m)=\Syz_Y(f_1,\ldots,f_n)\otimes\Oc_Y(m)$$ 
holds, since $R$ is standard-graded (for counter-examples in the non standard-graded case see for example \cite[Section 1.5]{weighted}). Moreover, if 
$R$ has dimension at least two and satisfies $S_2$ the sheaf 
$\Syz_Y(f_1,\ldots,f_n)(m)$ encodes the homogeneous component of degree $m$ of the $R$-module $\Syz_R(f_1,\ldots,f_n)$. 
We get this homogeneous component back by taking global sections
$$\Gamma\left(Y,\Syz_Y(f_1,\ldots,f_n)(m)\right)\cong \Syz_R(f_1,\ldots,f_n)_m$$
(cf. \cite[Exercise III.3.5]{hartshorne}).

\begin{exa}
Let $R:=k[X_0,\ldots,X_n]$ and $Y:=\Proj(R)=\Prim^n_k$. Then there is a short exact sequence (cf. \cite[Theorem II.8.13]{hartshorne})
$$0\lra\Omega_{Y/k}\lra\Oc_Y^{n+1}(-1)\stackrel{X_0,\ldots,X_n}{\lra}\Oc_Y\lra 0,$$
which arises from the Euler sequence. 
This sequence is a presenting sequence of $\Omega_{Y/k}$ and of $\Syz_Y(X_0,\ldots,X_n)$, hence they are isomorphic and $\Syz_Y(R_+)$ is nothing but 
the cotangent bundle on $Y$. If $Y'\subseteq Y$ is a closed subvariety, the sheaf of syzygies $\Syz_{Y'}(X_0,\ldots,X_n)$ on $Y'$ is just the restriction of 
the cotangent bundle $\Omega_{Y/k}$ to $Y'$.
\end{exa}

Recall some basic properties of syzygy modules.

\begin{lem}\label{manipulationlemma} Let $R$ be a standard-graded, affine $k$-domain and let $f,f_1,\ldots,f_n\in R$ be homogeneous and $f$ non-zero. The following statements hold.
\begin{enumerate}
 \item Let $g_1,\ldots,g_m\in(f_1,\ldots,f_n)$ be homogeneous elements with $g_j=\sum_{i=1}^nh_{ij}\cdot f_i$. This defines a graded $R$-module homomorphism
	$$\Syz_R(g_1,\ldots,g_m)\stackrel{(h_{ij})_{ij}}{\lra}\Syz_R(f_1,\ldots,f_n).$$
 \item The following isomorphism holds
	$$\Syz_R(ff_1,\ldots,ff_n)\cong\Syz_R(f_1,\ldots,f_n)(-\Deg(f)).$$
 \item If $f,f_n$ is an $R$-regular sequence, we have 
    $$\Syz_R(ff_1,\ldots,ff_{n-1},f_n)\cong\Syz_R(f_1,\ldots,f_n)(-\Deg(f)).$$
 \item If $\Deg(ff_i)=\Deg(f_n)$ for some $i\in\{1,\ldots,n-1\}$, we have $$\Syz_R(f_1,\ldots,f_n)\cong\Syz_R(f_1,\ldots,f_{n-1},f_n+ff_i).$$ 
 \end{enumerate}
\end{lem}

\begin{proof}
Since the parts (i), (ii) and (iv) are obvious, we only prove part (iii). Let 
$$(s_1,\ldots,s_n)\in\Syz_R(ff_1,\ldots,ff_{n-1},f_n)$$ 
be a syzygy of total degree $m$. We obtain
$$0=\sum_{i=1}^{n-1}s_iff_i+s_nf_n\equiv s_nf_n\text{ modulo }f.$$
This forces $s_n\in (f)$, since $f_n$ is a non-zero divisor in $R/(f)$. But then $(s_1,\ldots,s_{n-1},s_n/f)$ is a syzygy for $f_1,\ldots,f_n$ of 
total degree $m-\Deg(f)$. On the other hand, if $(t_1,\ldots,t_n)$ is a syzygy for $f_1,\ldots,f_n$ of total degree $m'$, then $(t_1,\ldots,t_{n-1},ft_n)$ 
is a syzygy for the elements $ff_1,\ldots,ff_{n-1},f_n$ of total degree $m'+\Deg(f)$. These operations lead to two homomorphisms that are inverse to each other.
\end{proof}

\begin{prop}\label{syzprop}
Let $k$ be an algebraically closed field and $R$ a standard-graded $k$-algebra of dimension at least two and $Y:=\Proj(R)$. Let 
$f_1,\ldots,f_n\in R$ be homogeneous elements with $\Deg(f_i)=d_i$. Assume that at least one of the $f_i$ is a non-zero 
divisor. Then the following statements hold.

\begin{enumerate}
\item The rank of $\Syz_Y(f_1,\ldots,f_n)$ is $n-1$.
\item If the ideal generated by the $f_i$ is $R_+$-primary, then the morphism 
	$$\bigoplus_{i=1}^n\Oc_Y(-d_i)\stackrel{f_1,\ldots,f_n}{\lra}\Oc_Y$$
	is surjective and $\Syz_Y(f_1,\ldots,f_n)$ is locally free.
\item The sheaf $\Syz_Y(f_1,\ldots,f_n)$ is locally free on $U:=D_+(f_1,\ldots,f_n)\subseteq Y$.
\item If $Y$ is an irreducible curve and the ideal $(f_1,\ldots,f_n)$ is $R_+$-primary, then
	$$\Det(\Syz_Y(f_1,\ldots,f_n)(m))\cong\Oc_Y\left((n-1)m-\sum_{i=1}^nd_i\right).$$
	In particular, we have $\Deg(\Syz_Y(f_1,\ldots,f_n)(m))=\left((n-1)m-\sum_{i=1}^nd_i\right)\cdot \Deg(Y)$.
\item Let $\phi:R\ra S$ be a morphism of degree $d$ of normal, standard-graded $k$-domains of dimension two. Let $g:Y':=\Proj(S)\ra Y$ be the 
	corresponding morphism of smooth projective curves. Then the following holds 
	$$g\pb(\Syz_Y(f_1,\ldots,f_n)(m))\cong\Syz_{Y'}(\phi(f_1),\ldots,\phi(f_n))(dm).$$
\item Assume that $\chara(k)=p>0$ and that $Y$ is an irreducible curve. For the iterated pull-backs under the Frobenius, we have
	$$\fpb{e}(\Syz_Y(f_1,\ldots,f_n)(m))\cong\Syz_Y(f_1^q,\ldots,f_n^q)(mq),$$
	where $q=p^e$ and $e\in\N$.
\end{enumerate}
\end{prop}

\begin{proof}
\begin{enumerate}
\item Since $f_1,\ldots,f_n$ contains a non-zero divisor, the rank of $\Ic_Z$ (on every irreducible component of $Y$) is one. Since the rank is additive on short exact sequences, 
    the statement follows from the presenting sequence of $\Syz_Y(f_1,\ldots,f_n)$.
\item Since $(f_1,\ldots,f_n)$ is $R_+$-primary, the ideal sheaf $\Ic_Z$ is just $\Oc_Y$. The supplement follows, since kernels of surjective maps between 
    locally free sheaves are locally free.
\item On the affine patches $D_+(f_i)$, the map
\begin{align*}
 && \bigoplus_{j\in\{1,\ldots,n\},j\neq i}\Oc_Y(-d_j)|_{D_+(f_i)} &\longrightarrow \Syz_Y(f_1,\ldots,f_n)|_{D_+(f_i)},\\
 && (g_1,\ldots,g_{i-1},g_{i+1},\ldots,g_n) &\longmapsto \left(g_1,\ldots,g_{i-1},-\frac{\sum_{j\in\{1,\ldots,n\},j\neq i}f_jg_j}{f_i},g_{i+1},\ldots,g_n\right)
\end{align*}
      is an isomorphism.
\item By part (ii) the presenting sequence of $\Syz_Y(f_1,\ldots,f_n)$ ends in $\Oc_Y$. We obtain
    \begin{align*}
    \Det(\Syz_Y(f_1,\ldots,f_n)(m))\otimes\Oc_Y(m) &\cong \Det\left(\bigoplus_{i=1}^n\Oc_Y(m-d_i)\right)\\
    &\cong \Oc_Y\left(nm-\sum_{i=1}^nd_i\right).
    \end{align*}
    Tensoring with $\Oc_Y(-m)$ gives the result. The supplement is clear.
\item The statement follows by taking the pull-back of the presenting sequence of $\Syz_Y(f_1,\ldots,f_n)$ and by \cite[Proposition II.6.9]{hartshorne}.
\item Similarly to (v), the statement follows by taking the Frobenius pull-back of the presenting sequence of $\Syz_Y(f_1,\ldots,f_n)$ using Lemma \ref{fpbline}.
\end{enumerate}
\end{proof}

\begin{rem} In the affine situation the analogues of parts (i), (iii), (v) and (vi) of the above proposition are still true.
\end{rem}

\subsection{Stability}

An important property of vector bundles is the semistablility. It is essential to construct meaningfull moduli spaces (see \cite{huyb}). 

\begin{defi} Let $C$ be a smooth projective curve defined over an algebraically closed field and let $\Ec$ be a locally free $\Oc_C$-module. We call the quotient 
$\mu(\Ec):=\frac{\Deg(\Ec)}{\Rank(\Ec)}$ the \textit{slope} of the sheaf $\Ec$. 
\end{defi}

\begin{defi} Let $C$ be a smooth projective curve defined over an algebraically closed field and let $\Ec$ be a locally free $\Oc_C$-module.
We call $\Ec$ \textit{semistable} if for every locally free subsheaf $0\neq \Fc\subsetneq\Ec$ the inequality 
$\mu(\Fc)\leq\mu(\Ec)$ holds.
We say that $\Ec$ is \textit{stable} if the above inequality is always strict.
\end{defi}

Note that our definition of stability is due to Mumford (another not equivalent definition of stability using Euler 
characteristics is due to Gieseker).

\begin{exa}
Let $C\subsetneq \Prim^2_k$ be a smooth plane curve of degree at least two defined over an algebraically closed field $k$. Then $\Omega_{\Prim^2_k/k}|_C$ 
is semistable as shown in \cite[Proposition 6.2]{tightdim2} and \cite[Corollary 3.5]{tri1}.
\end{exa}

\begin{prop}
Let $C$ be a smooth projective curve defined over an algebraically closed field and let $\Ec$, $\Fc$ be locally free $\Oc_X$-modules. The following statements hold.

\begin{enumerate}
\item Line bundles are stable.
\item If $\Ec$ and $\Fc$ are semistable then $\Ec\oplus\Fc$ is semistable if and only if $\mu(\Fc)=\mu(\Ec)$.
\item The $\Oc_X$-module $\Ec$ is (semi-)stable if and only if for every line bundle $\Lc$ the sheaf $\Ec\otimes\Lc$ is (semi-)stable.
\end{enumerate}
\end{prop}

\begin{proof}
To prove (i) assume that the line bundle $\Lc$ is not semistable. Then there exists a line bundle $\Ec\subsetneq\Lc$ with $\Deg(\Ec)>\Deg(\Lc)$. The inclusion 
$$0\ra\Ec\ra\Lc$$ leads to the inclusion $$0\ra\Oc\ra\Lc\otimes\Ec^{\vee}.$$
Because $\Oc$ injects into $\Lc\otimes\Ec^{\vee}$ its degree $\Deg(\Lc\otimes\Ec^{\vee})=\Deg(\Lc)-\Deg(\Ec)$ has to be non-negative contradicting $\Deg(\Ec)>\Deg(\Lc)$.

To prove part (ii) assume that $\Ec\oplus\Fc$ is semistable. Then $\Ec\oplus 0$ is a locally free subsheaf of $\Ec\oplus\Fc$. The inequality 
$$\mu(\Ec\oplus 0)=\mu(\Ec)\leq \mu(\Ec\oplus\Fc)$$
is equivalent to $\mu(\Ec)\leq\mu(\Fc)$. Replacing $\Ec$ by $\Fc$ gives $\mu(\Fc)\leq\mu(\Ec)$, hence $\mu(\Ec)=\mu(\Fc)$.
Now let $\mu(\Ec)=\mu(\Fc)$. It is easy to see that $\mu(\Ec\oplus\Fc)=\mu(\Ec)=\mu(\Fc)$ holds. Assume that $\Ec\oplus\Fc$ is not semistable. Then there exists a locally free subsheaf $\Gc\subsetneq\Ec\oplus\Fc$ with 
$\mu(\Gc)>\mu(\Ec\oplus\Fc)$. From the short exact sequence
$$0\ra\Ec\ra\Ec\oplus\Fc\ra\Fc\ra 0$$
we obtain locally free subsheaves $\Gc':=\Ker(\Gc\ra\Fc)$ and $\Gc'':=\Ima(\Gc\ra\Fc)$ of $\Ec$ resp. $\Fc$ and the short exact sequence
\begin{equation}\label{gcses}0\ra\Gc'\ra\Gc\ra\Gc''\ra 0.\end{equation}
If $\Gc'=0$, we obtain the contradiction $$\mu(\Fc)\geq\mu(\Gc'')=\mu(\Gc)>\mu(\Ec\oplus\Fc)=\mu(\Fc).$$
Similarly, if $\Gc''=0$, one gets $\mu(\Ec)>\mu(\Ec).$
If $\Gc'$ and $\Gc''$ are both non-zero, one obtains from (\ref{gcses}) the contradiction
$$\Max(\mu(\Gc'),\mu(\Gc''))\geq\mu(\Gc)>\mu(\Ec\oplus\Fc)=\mu(\Ec)=\mu(\Fc)$$
since $\Ec$ and $\Fc$ are semistable.

To prove part (iii) fix a line bundle $\Lc$. If $0\neq\Fc\subsetneq\Ec$ contradicts (semi-)stability, then $0\neq\Fc\otimes\Lc\subsetneq\Ec\otimes\Lc$ 
contradicts (semi-)stability. Similarly, if $0\neq\Fc'\subsetneq\Ec\otimes\Lc$ contradicts (semi-)stability, then $0\neq\Fc'\otimes\Lc^{\vee}\subsetneq\Ec$ 
contradicts (semi-)stability.
\end{proof}

\begin{thm}[Grothendiek]
Let $\Ec$ be a vector bundle of rank $r$ on $\Prim_k^1$, where $k$ is an algebraically closed field. Then there are uniquely determined 
integers $a_1\geq\ldots\geq a_r$ with
$$\Ec\cong\bigoplus_{i=1}^r\Oc_{\Prim_k^1}(a_i).$$
\end{thm}

\begin{proof}
See \cite[Theorem 1.3.1]{huyb}.
\end{proof}

Using this theorem and the last proposition, one gets immediately a classification of all (semi-)stable vector bundles on the projective line:

\begin{cor}
Let $\Ec$ be a vector bundle on $\Prim_k^1$, where $k$ is an algebraically closed field. Then $\Ec$ is semistable if and only if all $a_i$ 
in Grothendiecks Theorem are equal. Moreover, a vector bundle is stable if and only if it is a line bundle.
\end{cor}

The next theorem shows that one can find for a given vector bundle a filtration with semistable quotients.

\begin{thm}[Harder, Narasimhan]
Let $C$ be a smooth projective curve. Then every locally free sheaf $\Ec$ over $C$ has a unique filtration
$$0=\Ec_0\subsetneq\Ec_1\subsetneq\ldots\subsetneq\Ec_t=\Ec$$
satisfying the following properties 
\begin{enumerate}
\item all quotients $\Ec_i/\Ec_{i-1}$ are semistable and
\item the slopes $\mu_i:=\mu(\Ec_i/\Ec_{i-1})$ form a strictly decreasing chain $\mu_1>\ldots>\mu_t$.
\end{enumerate}
\end{thm}

\begin{proof}
See \cite[Lemma 1.3.7]{hnfilt}.
\end{proof}

\begin{defi}
The filtration from the previous theorem is called \textit{Harder-Narasimhan filtration}. The sheaf $\Ec_1$ in the filtration is called the 
\textit{maximal destabilizing subsheaf} of $\Ec$ and the quotient $\Ec/\Ec_{t-1}$ is called the \textit{minimal destabilizing quotient} 
of $\Ec$. The corresponding slopes $\mu_{\text{max}}(\Ec):=\mu_1$ and $\mu_{\text{min}}(\Ec):=\mu_t$ are called \textit{maximal slope} resp. 
\textit{minimal slope} of $\Ec$.
\end{defi}

Working in positive characteristic, it is not true that Frobenius pull-backs of semistable vector bundles are again semistable. A counter example to this 
can be found on the curve
$$C:=\Proj(k[X,Y,Z]/(X^4+Y^4+Z^4)),$$ 
where $k$ is an algebraically closed field with $\chara(k)=3$. Then $\Syz_C(X,Y,Z)$ is semistable, but 
$\fpb{}(\Syz_C(X,Y,Z))$ is not as shown in \cite[Example 3.1]{holgeralmar}. The following definition is due to Miyaoka (cf. \cite[Section 5]{defsss}).

\begin{defi} Let $C$ be a smooth projective curve over an algebraically closed field $k$ of positive characteristic. A locally free 
sheaf $\Ec$ over $C$ is called \textit{strongly semistable}, if all its Frobenius pull-backs $\fpb{e}(\Ec)$, $e\geq 0$, are semistable.
\end{defi}

\begin{defi}
We call a Harder-Narasimhan filtration a \textit{strong Harder-Narasimhan filtration} if all quotients are strongly semistable.
\end{defi}

The next theorem shows that for every locally free sheaf there is a Frobenius pull-back with a strong Harder-Narasimhan filtration. In fact the Frobenius 
pull-backs stabilize in a certain way (see the next remark for the precise statement).

\begin{thm}[Langer]\label{langer}
Let $C$ be a smooth projective curve over an algebraically closed field $k$ with $\chara(k)=p>0$. For every locally free sheaf $\Ec$ there 
exists a (minimal) natural number $e_0$ such that the Harder-Narasimhan filtration of $\fpb{e_0}(\Ec)$ is strong.
\end{thm}

\begin{proof}
See \cite[Theorem 2.7]{langer}.
\end{proof}

\begin{rem}\label{pbofstronghnfilt}
Let $0\subsetneq\Ec_1\subsetneq\ldots\subsetneq\Ec_{t-1}\subsetneq\fpb{e_0}(\Ec)$ be the strong Harder-Narasimhan filtration from Langers Theorem. 
Since Frobenius pull-backs respect inclusions and commute with quotients, the filtration
$$0\subsetneq\fpb{e-e_0}(\Ec_1)\subsetneq\ldots\subsetneq\fpb{e-e_0}(\Ec_{t-1})\subsetneq\fpb{e}(\Ec)$$
is a strong Harder-Narasimhan filtration for $\fpb{e}(\Ec)$ for all $e\geq e_0$.
\end{rem}


\subsection{Relation to Hilbert-Kunz theory}
In this subsection we explain how the theory of vector bundles arises in Hilbert-Kunz theory. The next proposition will be crucial.

\begin{prop}\label{hkgeomapp}
Let $k$ be an algebraically closed field of positive characteristic $p$, let $R$ be a standard-graded $k$-domain of dimension at least two satisfying $S_2$ and 
let $Y=\Proj(R)$. Let $I=(f_1,\ldots,f_n)$ be an $R_+$-primary ideal with $\Deg(f_i)=d_i$. Then for every $q=p^e\geq 1$ and 
$m\in\Z$ the following formula holds
$$\Dim_k\left(R/I\qpot\right)_m=\Th^0(Y,\Oc_Y(m))-\sum_{i=1}^n\Th^0(Y,\Oc_Y(m-q\cdot d_i))+\Th^0(Y,\Syz_Y(f_1^q,\ldots,f_n^q)(m)).$$
\end{prop}

\begin{proof}
Starting with the exact sequence
$$0\lra\Syz_Y(f_1^q,\ldots,f_n^q)(m)\lra\sum_{i=1}^n\Oc_Y(m-q\cdot d_i)\stackrel{f_1^q,\ldots,f_n^q}{\lra}\Oc_Y(m)\lra 0,$$
we obtain by taking global sections the exact sequence
\begin{equation}\label{seqgeom}0\lra\Gamma(Y,\Syz_Y(f_1^q,\ldots,f_n^q)(m))\lra\sum_{i=1}^nR_{m-q\cdot d_i}\stackrel{f_1^q,\ldots,f_n^q}{\lra}R_m,\end{equation}
where $\Gamma(Y,\Oc_Y(l))=R_l$ for any $l\in\Z$, since $R$ satisfies $S_2$ (cf. \cite[Exercise III.3.5]{hartshorne}). Obviously, the image of the last map is $\left(I\qpot\right)_m$, hence we may add 
$\ra\left(R/I\qpot\right)_m\ra 0$ at the right hand side of (\ref{seqgeom}). Computing the alternating sum of the $k$-vector space dimensions the sum 
$\sum_m\left(R/I\qpot\right)_m$ is finite since $I\qpot$ is $R_+$-primary and it is exactly $\HKF(I,R,q)$.
\end{proof}

At this point we should mention that - although Proposition \ref{hkgeomapp} holds for rings of higher dimension than two - there are no examples in dimension at least three in which the above proposition is helpfull. 
The reason for this is that second syzygy modules of quotients $R/I$ do not behave well if $\Dim(R)\geq 3$. In fact, the $\Dim(R)$-th syzygy modules of 
$R/I$ are the correct modules to deal with. This is emphasized by a recent work of Brenner. Using more general tools from algebraic geometry than we 
presented, he was able to prove the following theorem (compare with \cite{holgerirre}).

\begin{thm}[Brenner]
Let $R$ be a standard-graded Cohen-Macaulay domain of dimension $d+1\geq 2$ with an isolated singularity over an algebraically closed field $k$ of positive characteristic $p$. Let $Y:=\Proj(R)$ and let $I\subsetneq R$ be an $R_+$-primary ideal. Let 
\begin{equation}\label{complex}\ldots\ra F_2\stackrel{\delta_2}{\ra} F_1\stackrel{\delta_1}{\ra} F_0\ra R/I\ra 0\end{equation}
be a graded complex with $F_i:=\bigoplus_{j\in J_i}R(-\beta_{ij})$. Let $\Syz_R^d:=\ker(\delta_d)$. If the complex on $\punctured{R}{R_+}$ associated to (\ref{complex}) is exact, we have
$$\HKF(I,R,p^e) = \sum_{m\in\N}\Th^d(\fpb{e}(\Syz_Y^d)(m))+\sum_{i=0}^d(-1)^{d-1-i}\left(\sum_{j\in J_i}\left(\sum_{m\in\N}\Th^d(\Oc_Y(-\beta_{ij}p^e+m))\right)\right).$$
\end{thm}

In the next chapter we will give some examples how one can compute Hilbert-Kunz functions with Proposition \ref{hkgeomapp}.

Starting with Proposition \ref{hkgeomapp}, using Langer's Theorem \ref{langer} and properties of strongly semistable vector bundles, Brenner and Trivedi computed independently 
the Hilbert-Kunz multiplicity of two-dimensional, standard-graded, normal domains containing an algebraically closed field of positive characteristic. To state their result, we need some more notation.

\begin{Not}
In the situation of Proposition \ref{hkgeomapp} assume $\Dim(R)=2$ and that $Y$ is smooth. Then there exists by Langer's Theorem \ref{langer} an $e_0$ such that for every 
$e\geq e_0$ the $e$-th Frobenius pull-back of the sheaf $\Syz_Y(f_1,\ldots,f_n)$ has a strong Harder-Narasimhan filtration 
$$0=\Sc_0\subsetneq \Sc_1\subsetneq\ldots\subsetneq\Sc_{t-1}\subsetneq\Sc_t=\fpb{e}(\Syz_Y(f_1,\ldots,f_n)).$$
For $i=1,\ldots,t$, we define the numbers 
$$\nu_i:=-\frac{\mu(\Sc_i/\Sc_{i-1})}{\Deg(Y)\cdot p^e} \quad \text{and} \quad r_i:=\Rank(\Sc_i/\Sc_{i-1}).$$
Note that these numbers are independent of $e$ by Remark \ref{pbofstronghnfilt}.
\end{Not}

\begin{thm}[Brenner/Trivedi]\label{holger}
In the situation of Proposition \ref{hkgeomapp} assume $\Dim(R)=2$ and that $Y$ is smooth. Then the Hilbert-Kunz multiplicity of $R$ with respect to $I$ is given by
$$\HKM(I,R)=\frac{\Deg(Y)}{2}\cdot\left(\sum_{i=1}^tr_i\nu_i^2-\sum_{j=1}^nd_j^2\right),$$
where the $r_i$ and $\nu_i$ are computed with respect to $\Syz_Y(I)$.
\end{thm}

\begin{proof}
See \cite[Theorem 3.6]{bredim2} for the general statement and \cite[Theorem 4.12]{tri1} for the case $I=R_+$.
\end{proof}




We will state some generalizations of Theorem \ref{holger}.

\begin{thm}[Trivedi]\label{trivedi}
Drop the assumption in Theorem \ref{holger} that $Y$ is smooth. Then the inequality
$$\HKM(I,R)\geq\frac{\Deg(Y)}{2}\left(\frac{\left(\sum_{i=1}^nd_i\right)^2}{n-1}-\sum_{i=1}^nd_i^2\right)$$
holds. Equality holds if and only if the pull-back of $\Syz_Y(I)$ to the normalization of $Y$ is strongly semistable. In particular, if $Y$ is already smooth, 
equality holds if and only if $\Syz_Y(I)$ is strongly semistable.
\end{thm}

\begin{proof}
See \cite[Theorem 2.2]{tri2}.
\end{proof}

\begin{cor}\label{coralmar}
Let $k$ be an algebraically closed field of positive characteristic $p$. Let $R:=k[X,Y,Z]/(F)$ be standard-graded with $F$ irreducible and homogeneous 
of degree $d$ such that $C:=\Proj(R)$ is a smooth projective curve. Let $I=(f_1,f_2,f_3)\subsetneq R$ be a homogeneous and $R_+$-primary ideal with all 
generators having the same degree $a\geq 1$. Then the following statements hold.

\begin{enumerate}
\item If $\Syz_C(I)$ is not semistable, then there exists an integer $0<l\leq a d$ with $l\equiv ad\text{ }(2)$ such that 
	$$\HKM(I,R)=\frac{3d}{4}\cdot a^2+\frac{l^2}{4d}.$$
\item If $\Syz_C(I)$ is semistable but not strongly semistable, then there exists an integer $0<l\leq d(d-3)$ with $l\equiv pad\text{ }(2)$ such that 
	$$\HKM(I,R)=\frac{3d}{4}\cdot a^2+\frac{l^2}{4d\cdot p^{2e}},$$
	where $e\geq 1$ is such that $\fpb{e-1}(\Syz_C(I))$ is semistable, but $\fpb{e}(\Syz_C(I))$ is not.
\item If $\Syz_C(I)$ is strongly semistable, then  
	$$\HKM(I,R)=\frac{3d}{4}\cdot a^2.$$
\end{enumerate}
\end{cor}

\begin{proof}
For the case $a=1$ see \cite[Theorem 5.3]{tri1} and \cite[Corollary 4.6]{bredim2}. For the general case see 
\cite[Corollary 1.4.9 (2), Proposition 1.4.11]{almar}.
\end{proof}

\begin{thm}[Brenner]
In the situation of Theorem \ref{holger}, the Hilbert-Kunz function of $R$ with respect to $I$ has the shape
$$\HKF(I,R,p^e)=\HKM(I,R)\cdot p^{2e}+\gamma(p,e),$$
where $\gamma(p,\_)$ is a bounded function. Moreover, if $k$ is the algebraic closure of a finite field, the function $\gamma(p,\_)$ is eventually periodic.
\end{thm}

\begin{proof}
See \cite[Theorem 6.1]{bredim2func}.
\end{proof}

\subsection{Curves of degree three}
Using the just explained geometric approach, Brenner and Hein computed the Hilbert-Kunz function of the homogeneous coordinate ring $R$ associated 
to a very ample line bundle $\Oc_C(1)$ of an elliptic curve $C$ over an algebraically closed field $k$. See \cite[Corollary 4.6]{brehein} for 
the case where $\Syz_C(f_1,\ldots,f_n)$ is semistable and \cite[Theorem 4.4]{brehein} for the general case. In the same paper they also look more 
carefully at the case, where $C$ is embedded by a complete linear system $|\Oc_C(1)|$ into some projective space $\Prim_k^N$. In 
\cite[Theorem 5.2]{brehein} they give a full discription of the Hilbert-Kunz function of $R$. In particular, the Hilbert-Kunz multiplicity of $R$ is 
$\HKM(R)=\frac{(N+1)^2}{2N}$. We just state the explicit results of the Hilbert-Kunz functions in the cases $N=2,3,4$, given in 
\cite[Corollaries 5.3-5.6]{brehein}. Note that these results had already been proved in \cite[Theorem 3.17]{triflag}.

\begin{prop} Let $C$ be an elliptic curve over an algebraically closed field $k$ of characteristic $p>0$, embedded by a complete linear 
system into $\Prim_k^N$. Then the Hilbert-Kunz function of the homogeneous coordinate ring $R$ of $C$ is given by the following formulas.

\begin{enumerate}
\item Let $N=2$ and let $h$ be the Hasse invariant of $C$. Then
    $$\HKF_R(R,q)=\left\{\begin{aligned} & \frac{9}{4}q^2-\frac{5}{4} && \text{if }q\text{ is odd},\\ & 8-h && \text{if }q=2,\\ & \frac{9}{4}q^2-h && \text{otherwise},\end{aligned}\right.$$
\item $$\HKF_R(R,q)=\left\{\begin{aligned} & \frac{8}{3}q^2-\frac{5}{3} && \text{if }N=3\text{ and }p\neq 3,\\
 & \frac{25}{8}q^2-\frac{17}{8} && \text{if }N=4\text{ and }p\neq 2.\end{aligned}\right.$$
\end{enumerate}
\end{prop}

The Hilbert-Kunz functions of the homogeneous coordinate rings of curves of degree three have been studied by many authors using a whole bunch of different tools. In his thesis (cf. \cite{pardue}), 
Pardue made a conjecture on the Hilbert-Kunz functions of the homogeneous coordinate rings of plane curves of degree three. This conjecture was proven case by case by different authors. We 
summarize the results and references.

If $C$ is an irreducible nodal cubic, Pardue showed in his thesis (cf. \cite{pardue} and also \cite[Theorem 3]{buchweitz}) 
$$\HKF(R,q)=\frac{7}{3}q^2-\frac{1}{3}q-D(q)$$ 
with $D(q)=\frac{5}{3}$, if $q\equiv 2\text{ }(3)$ and $D(q)=1$ otherwise. Using Burbans results from \cite{cycleprojline} on the classification of 
Frobenius pull-backs of line bundles on plane rational nodal curves, Monsky showed in \cite{nodal}, that the Hilbert-Kunz function of $R$ with 
respect to an $\mm$-primary ideal $I$ is of the form $\HKF(I,R,q)=\mu q^2+\alpha q +D(q)$, where $\mu$ and $\alpha$ are constants depending on the 
data of $\Syz_C(I)$ and $D(q)$ depends on these data and on the class of $q$ modulo three.

If $C$ is cuspidal, we have 
$$\HKF(R,3^e)=\frac{7}{3}9^e\quad\text{and}\quad \HKF(R,q)=\frac{7}{3}q^2-\frac{4}{3}$$ 
for $p\neq 3$ (see \cite{pardue}). If $C$ is smooth, the Hilbert-Kunz function of $R$ was computed for odd characteristics in \cite[Theorem 4]{buchweitz} based on 
Geronimus' formula for the determinant of a matrix, whose entries are Legendre polynomials. In characteristic two, one might use the 
algorithm from \cite{monskyhan} if the Hasse invariant is zero. If the Hasse invariant is one, Monsky computed the Hilbert-Kunz function in \cite{char2cubic} using tools 
from invariant theory.

\section{Related invariants and further readings}
In this section we want to discuss some invariants related to Hilbert-Kunz theory. We also give a list 
of references related to Hilbert-Kunz theory that have not been mentioned in this survey.

\subsection{Limit Hilbert-Kunz multiplicity}
The basic question in this subsection is whether one can define a ``Hilbert-Kunz multiplicity'' in characteristic zero. We have already defined the 
generalized Hilbert-Kunz function $n\mapsto \lambda_R(R/(f_1^n,\ldots,f_m^n)R)$. From this definition one might get the idea to study the limit
$$\lim_{n\ra\infty}\frac{\lambda_R(R/(f_1^n,\ldots,f_m^n)R)}{n^{\Dim(R)}},$$
provided it exists. Consider for example $R:=\Z[X,Y,Z]/(X+Y+Z)$. Then by \cite[Theorem 2.14]{monskyhan} we obtain
$$\Dim_{\Q}(\Q\otimes_{\Z}R/(X^n,Y^n,Z^n))=\left\lceil\frac{3n^2}{4}\right\rceil.$$
Dividing this number by $n^2$ and taking the limit $n\ra\infty$, we get 
$$\lim_{n\ra\infty}\frac{\Dim_{\Q}(\Q\otimes_{\Z}R/(X^n,Y^n,Z^n)}{n^2}=\frac{3}{4}.$$
But since $\Q\otimes_{\Z}R\cong\Q[X,Y]$ is regular, we want the ``Hilbert-Kunz multiplicity'' to be one. Hence the above limit cannot be a good candidate.

\begin{defi}
Let $R$ be a positively-graded $\Z$-domain and let $I=(f_1,\ldots,f_m)$ be a homogeneous $R_+$-primary ideal. For a prime $p$ let 
$R_p:=R\otimes_{\Z}\Z/(p)$ and denote by $I_p$ the extended ideal $IR_p$. 
We call the limit (provided it exists)
$$\LHKM(I,R):=\lim_{p\ra\infty}\HKM(I_p,R_p)$$
the \textit{limit Hilbert-Kunz multiplicity} of $R$ with respect to $I$.
\end{defi}

The existence of $\LHKM(\mm,R)$, where $R$ is a diagonal hypersurface over $\Z$ was shown by Gessel and Monsky in \cite{limitmonsky}. To state their 
result, we need one more definition.

\begin{defi}
Let $d_1,\ldots,d_s$ be positive integers. For any integer $\lambda$ we define 
$$C_{\lambda}:=\sum_{\substack{(\epsilon_1,\ldots,\epsilon_s)\in\{\pm 1\}^s,\\ \frac{\epsilon_1}{d_1}+\ldots+\frac{\epsilon_s}{d_s}>2\lambda}}(\epsilon_1\cdot\ldots\cdot\epsilon_{s})\cdot\left(\frac{\epsilon_1}{d_1}+\ldots+\frac{\epsilon_s}{d_s}-2\lambda\right)^{s-1}.$$
\end{defi}

\begin{thm}[Gessel, Monsky] Let $R:=\Z[X_1,\ldots,X_s]/\left(\sum_{i=1}^sX_i^{d_i}\right)$ with positive integers $d_i$ and let $d:=\prod d_i$. 
The limit Hilbert-Kunz multiplicity of $R$ is given by
$$\LHKM(R)=\frac{d(2^{1-s})}{(s-1)!}\cdot\left(C_0+2\sum_{\lambda>0}C_{\lambda}\right).$$
\end{thm}

\begin{proof}
See \cite[Theorem 2.4]{limitmonsky}.
\end{proof}

Another big class of rings for which the limit Hilbert-Kunz multiplicity is known to exist, is given by standard-graded, two-dimensional $k$-domains, 
where $k$ is an algebraically closed field of characteristic zero. To state the theorem, we need to define spreads.

\begin{defi}
Let $R$ be a finitely generated, positively-graded, two-dimensional domain over an algebraically closed field $k$ of characteristic 0 and let $I$ be a 
homogeneous $R_+$-primary ideal. There exist a finitely generated $\Z$-algebra $A\subseteq k$, a finitely generated, positively-graded algebra $R_A$ over $A$ 
and a homogeneous ideal $I_A\subsetneq R_A$ with $R_A\otimes_Ak=R$ and such that the following properties hold for all closed points $s\in\Spec(A)$

\begin{enumerate}
 \item the ring $R_s:=R_A\otimes_A\kappa(s)$ is a finitely generated, positively-graded, two-dimen\-sional domain over $\kappa(s)$,
 \item the rings $R_s$ are normal, if $R$ is and
 \item the ideal $I_s:=\Ima(I_A\otimes_A\kappa(s))\subsetneq R_s$ is homogeneous and $(R_s)_+$-primary.
\end{enumerate}

We call the data $(A,R_A,I_A)$ a \textit{spread} of the pair $(R,I)$. Note that the spread $(A,R_A,I_A)$ can be chosen such that $A$ contains a given 
finitely generated $k$-subalgebra $A_0$.
\end{defi}

\begin{thm}[Trivedi]
Let $R$ be a standard-graded, two-dimensional domain over an algebraically closed field $k$ of characteristic 0. Let $I \subsetneq R$ be a
homogeneous $R_+$-primary ideal and let $(A,R_A,I_A)$ be a spread of the pair $(R,I)$. Denote by $s_0$ the generic point of $\Spec(A)$. Then the limit
$$\LHKM(I,R)=\lim_{s\ra s_0}\HKM(I_s,R_s)$$
over the closed points $s\in\Spec(A)$ exists. Moreover, if $I=(f_1,\ldots,f_n)$ with $\Deg(f_i)=d_i$, this limit equals 
$$\frac{\Deg(\Proj(S))}{2}\left(\sum_{i=1}^tr_i\left(\frac{\mu_i}{\Deg(\Proj(S))}\right)^2-\sum_{i=1}^nd_i^2\right),$$
where $S$ is the normalization of $R$ and the numbers $r_i$, $\mu_i$ are the ranks and slopes of the consecutive quotients of the bundles in the 
Harder-Narasimhan filtration of $\Syz_{\Proj(S)}(IS)$.
\end{thm}

\begin{proof} See \cite[Theorem 2.4]{triredmodp}.\end{proof}

Trivedi remarks (cf. \cite[Remark 2.6]{triredmodp}) that all $\HKM(I_s,R_s)$ are bounded below by $\LHKM(I,R)$ and that equality holds for a closed point 
$s\in\Spec(A)$ if and only if the Harder-Narasimhan filtration 
of $\Syz_{\Proj(R_s)}(I_s)$ is strong.

At least in this situation, the limit Hilbert-Kunz multiplicity behaves to solid closure like the Hilbert-Kunz multiplicity to tight closure (cf. \cite[Theorem 3.3]{solid}).

The result of Trivedi was recovered in \cite{limitholger} for standard-graded, flat $\Z$-domains $R$, where almost all fiber rings $R_p:=R\otimes_{\Z}\Z/(p)$ 
are geometrically normal, two-dimensional domains. Moreover, they show that in this situation, one does not have to compute all $\HKM(I_p,R_p)$ to compute 
$\LHKM(I,R)$. In fact, one can skip the limit in the first number.

\begin{thm}[Brenner, Li, Miller]
In the above situation, one has 
$$\LHKM(I,R)=\lim_{p\ra\infty}\frac{\lambda_{R_p}\left(R_p/I_p^{[p^e]}\right)}{p^{2e}}$$
for every fixed integer $e\geq 1$.
\end{thm}

\subsection{F-Signature}
In this subsection we will define the F-signature (function) and connect them to Hilbert-Kunz theory. For further properties of F-signature and 
explicit examples we refer to the literature.

\begin{defi}
Let $(R,\mm,k)$ be a reduced local ring of dimension $d$ and characteristic $p>0$. Assume that $k$ is perfect and that $R$ is $F$-finite 
(meaning that $\TF:R\ra R$ is a finite map). 
For any $e\in\N$ denote by $a_q$ the maximal rank of a free submodule of $R^{1/q}$. Then the function 
$$e\mapsto\FS(R,q):=a_q$$
is called the \textit{$F$-signature function} of $R$. The limit 
$$\Ts(R):=\lim_{e\ra\infty}\frac{\FS(R,q)}{q^d}$$
(provided it exists) is called the \textit{$F$-signature} of $R$.
\end{defi}

The $F$-signature was firstly studied in \cite{smithvan} and the existence was proven by Tucker in \cite[Theorem 4.9]{fsignexists}. Yao has shown 
in \cite[Proposition 4.1]{yao} that the $F$-signature of $R$ is nothing but the minimal relative Hilbert-Kunz multiplicity of $R$, that is 
$$\inf\{\HKM(I,R)-\HKM(J,R)|I\subsetneq J,\text{ } \sqrt{I}=\mm,\text{ } \lambda_R(J/I)=1\}$$
and was defined in \cite{minrelhkm}.

The following theorem gives a connection between the $F$-signature function and the Hilbert-Kunz function.

\begin{thm}[Huneke, Leuschke]\label{fsighkf}
Let $(R,\mm)$ be Gorenstein, let $I$ be an $\mm$-primary ideal generated by a system of parameters and $J:=(I,y)$, where $y\in R$ is a representative for the socle of $R/I$. We then get the equality
$$\FS(R,q)=\HKF(I,R,q)-\HKF(J,R,q).$$
\end{thm}

\begin{proof}
See \cite[Theorem 11(2)]{2thmsmcm}.
\end{proof}

\subsection{Further readings}

We give some more references on papers related to Hilbert-Kunz theory that are not mentioned so far.

\begin{enumerate}
\item In \cite{fields} Fields investigates the theory of quasi-polynomial functions to length functions of generalized Tor-modules. This theory has applications to Hilbert-Kunz functions.
\item More on Hilbert-Kunz functions of hypersurfaces can be found in \cite{chen}, \cite{upad2}, \cite{upad1}, \cite{upad4} and \cite{upad3}.
\item For more explicit examples of Hilbert-Kunz functions (in characteristic two) see \cite{lines4}, \cite{points4} and \cite{zdp4}.
\item The geometric approach to Hilbert-Kunz theory is used in \cite{triflag} and \cite{irredcurves}.
\item For a deeper analysis of Han's $\delta$-function and similar functions, especially of their fractal structure, see \cite{montei1}, \cite{montei2}, \cite{texdiss} and \cite{texnew}.
\item More bounds for Hilbert-Kunz multiplicities can be found in \cite{boundshkmaber}, \cite{boundshkmaber2}, \cite{manuel}, \cite{boundshkm} and \cite{fsignexists}.
\item For more on Frobenius functors and Frobenius pull-backs in higher dimensions see \cite{dutta}, \cite{dutta2}, \cite{dutta3}, \cite{seibert1} and \cite{fropb}.
\item For some generalizations of Hilbert-Kunz functions and multiplicities for example to ideals that are not $\mm$-primary see \cite{dao}, \cite{neil1} and \cite{neil2}.
\item More properties and explicit examples of $F$-signatures can be found in \cite{fpure}, \cite{aberfsign}, \cite{eric}, \cite{fpairs2}, \cite{fpairs1}, \cite{korff}, \cite{singh} and \cite{sannai}.
\item More about tight and solid closure and their interactions with Hilbert-Kunz theory can be found for example in \cite{tightclo}, \cite{solidclo}, \cite{tightapp}, \cite{holgerhab}, \cite{solid} and \cite{holgerbuch}.
\item Results on the Hilbert-Kunz function of maximal Cohen-Macaulay modules can be found in \cite{jean}, \cite{seibert2} and \cite{watyomcm}.
\end{enumerate}

\chapter{Examples of Hilbert-Kunz functions and Han's $\delta$ function}\label{chaphan}
In the first section we will discuss some examples of Hilbert-Kunz functions. The second section is devoted to a more detailed study of Han's $\delta$ 
function and a generalization of a theorem due to Monsky. We will end section two by computing some examples of Hilbert-Kunz multiplicities with the help 
of this generalization. In the third section we will use this generalization to study the behaviour of the Hilbert-Kunz multiplicity in certain 
families of trinomial surface rings.
\section{Examples of Hilbert-Kunz functions}
We fix some notations for the next two lemmata in which we will compute formulas for values of Hilbert-Kunz functions. These formulas will be used to 
compute Hilbert-Kunz functions of explicit examples (cf. Example \ref{hkf2lm} as well as the computations in the Chapters \ref{chaphkfade} and \ref{chapextfurexa}). 
Let $R:=k[X,Y,Z]/(F)$ be a normal, standard-graded domain with an algebraically closed field $k$ of prime characteristic $p$ and $\Deg(F)=d>0$. Moreover, 
we fix the $R_+$-primary ideal $I:=(X^{\alpha},Y^{\beta},Z^{\gamma})$ with $\alpha$, $\beta$, $\gamma\in\N_{\geq 1}$. Let $Q:\N^3\ra\Z$ be the quadratic form defined as
$$(n_1,n_2,n_3)\longmapsto 2(n_1n_2+n_1n_3+n_2n_3)-n_1^2-n_2^2-n_3^2.$$

\begin{lem}\label{hkffrei} 
Let $e\in\N$ be such that
\begin{align}\label{syzsplit}\Syz_R\left(I\qpot\right)=R(-n)\oplus R(-l)\text{ for some }n,l\in\Z.\end{align}
Then the value at $e$ of the Hilbert-Kunz function of $R$ with respect to $I$ is given by
$$\HKF(I,R,p^e)=\frac{d\cdot Q(\alpha,\beta,\gamma)\cdot p^{2e}+d\cdot|n-l|^2}{4}.$$
Note that the assumption (\ref{syzsplit}) means that the quotient $R/I\qpot$ has finite projective dimension and this carries over to all quotients 
$R/I^{[q']}$ for $q'\geq q$.
\end{lem}

\begin{proof}
Computing the degrees of the $C:=\Proj(R)$-bundles corresponding to the $R$-modules in (\ref{syzsplit}), one finds
$$l+n=q(\alpha+\beta+\gamma).$$
Therefore, we can write $l$ and $n$ as 
\begin{equation}\label{lundn}\frac{q(\alpha+\beta+\gamma)\pm |l-n|}{2}.\end{equation}
Using Proposition \ref{hkgeomapp} and Riemann-Roch \ref{thmrr}, we obtain
\begin{align*}
 &\quad \Dim_k(R/(X^{\alpha q},Y^{\beta q},Z^{\gamma q}))\\
 &= \lim_{x\ra\infty}\left[\sum_{m=0}^x\dimglo{\Oc_C(m)}-\sum_{m=0}^x\dimglo{\Oc_C(m-\alpha q)}-\sum_{m=0}^x\dimglo{\Oc_C(m-\beta q)}-\sum_{m=0}^x\dimglo{\Oc_C(m-\gamma q)}\right.\\
 &\quad \left. +\sum_{m=0}^x\dimglo{\Oc_C(m-l)}+\sum_{m=0}^x\dimglo{\Oc_C(m-n)}\right]\\
 &= \lim_{x\ra\infty}\left[\sum_{m=0}^x\dimglo{\Oc_C(m)}-\sum_{m=0}^{x-\alpha q}\dimglo{\Oc_C(m)}-\sum_{m=0}^{x-\beta q}\dimglo{\Oc_C(m)}-\sum_{m=0}^{x-\gamma q}\dimglo{\Oc_C(m)}\right.\\
 &\quad \left. +\sum_{m=0}^{x-l}\dimglo{\Oc_C(m)}+\sum_{m=0}^{x-n}\dimglo{\Oc_C(m)}\right]\\
 &= \lim_{x\ra\infty}\left[\sum_{m=x-\alpha\cdot q+1}^x\dimglo{\Oc_C(m)}-\sum_{m=x-l+1}^{x-\beta q}\dimglo{\Oc_C(m)} -\sum_{m=x-n+1}^{x-\gamma q}\dimglo{\Oc_C(m)}\right]\\
 &= \lim_{x\ra\infty}\left[\sum_{m=x-\alpha\cdot q+1}^x(dm+1-g)+\sum_{m=x-\alpha\cdot q+1}^x\dimext{\Oc_C(m)}-\sum_{m=x-l+1}^{x-\beta q}(dm+1-g)\right.\\
 &\quad \left. -\sum_{m=x-l+1}^{x-\beta q}\dimext{\Oc_C(m)}-\sum_{m=x-n+1}^{x-\gamma q}(dm+1-g) -\sum_{m=x-n+1}^{x-\gamma q}\dimext{\Oc_C(m)}\right],
\end{align*}
where $g$ in the last equality denotes the genus of $C$. With the help of Equation (\ref{lundn}) one sees that the various sums over $dm+1-g$ sum up 
to the constant
$$\frac{d Q(\alpha,\beta,\gamma)q^2+d|l-n|^2}{4}.$$
Since both bounds of the sums over $\Th^1(\Oc_C(m))$ depend linearly on $x$ and $\Th^1(\Oc_C(m))$ vanishes for large $m$ by a theorem of Serre 
(cf. \cite[Theorem III.5.2]{hartshorne}) these sums have no share in the limit over $x$. All in all, the claim follows.
\end{proof}

\begin{lem}\label{hkfdreivar} 
Let $e\in\N$ be such that
\begin{align}\label{eqsyzsyz}\Syz_R\left(I\qpot\right)\cong\Syz_R(X^a,Y^b,Z^c)(-n)\end{align}
for some $n\in\Z$ and some positive integers $a$, $b$, $c$ such that at least one of the inequalities $a<\alpha q$, $b<\beta q$, $c<\gamma q$ holds. 
Computing the degrees of the $C:=\Proj(R)$-bundles corresponding to the $R$-modules in (\ref{eqsyzsyz}), we find
$$n=\frac{q(\alpha+\beta+\gamma)-a-b-c}{2}.$$
With $D:=\Dim_k(R/(X^a,Y^b,Z^c))$ the value at $e$ of the Hilbert-Kunz function of $R$ with respect to $I$ is given by
$$\HKF(I,R,p^e)=\frac{d\cdot Q(\alpha,\beta,\gamma)\cdot p^{2e}-d\cdot Q(a,b,c)}{4}+D.$$
\end{lem}

\begin{proof}
From the presenting sequence of $\Syz_C(X^a,Y^b,Z^c)(m-n)$ we obtain
\begin{align}
\nonumber &\quad \Th^0(C,\Syz_C(X^a,Y^b,Z^c)(m-n))\\ 
&= \Th^0(C,\Oc_C(m-n-a))+\Th^0(C,\Oc_C(m-n-b))+\Th^0(C,\Oc_C(m-n-c)) \label{subst} \\
\nonumber &\quad -\Th^0(C,\Oc_C(m-n))+\Dim_k(R/(X^a,Y^b,Z^c))_{m-n}.
\end{align}
As in the previous example, using Proposition \ref{hkgeomapp}, Riemann-Roch \ref{thmrr} and substituting 
$\Th^0\left(C,\Syz_C\left(I\qpot\right)(m)\right)$ with the right hand side of (\ref{subst}), we get
\begin{align*}
 &\quad \Dim_k(R/(X^{\alpha q},Y^{\beta q},Z^{\gamma q}))\\
 &= \lim_{x\ra\infty}\left[\sum_{m=0}^x\dimglo{\Oc_C(m)}-\sum_{m=0}^x\dimglo{\Oc_C(m-\alpha q)}-\sum_{m=0}^x\dimglo{\Oc_C(m-\beta q)}\right.\\
 &\quad -\sum_{m=0}^x\dimglo{\Oc_C(m-\gamma q)}+\sum_{m=0}^x\dimglo{\Oc_C(m-n-a)}+\sum_{m=0}^x\dimglo{\Oc_C(m-n-b)}\\
 &\quad \left. +\sum_{m=0}^x\dimglo{\Oc_C(m-n-c)}-\sum_{m=0}^x\dimglo{\Oc_C(m-n)}\right]+D\\
 &= \lim_{x\ra\infty}\left[\sum_{m=0}^x\dimglo{\Oc_C(m)}-\sum_{m=0}^{x-\alpha q}\dimglo{\Oc_C(m)}-\sum_{m=0}^{x-\beta q}\dimglo{\Oc_C(m}-\sum_{m=0}^{x-\gamma q}\dimglo{\Oc_C(m)}\right.\\
 &\quad \left. +\sum_{m=0}^{x-n-a}\dimglo{\Oc_C(m)}+\sum_{m=0}^{x-n-b}\dimglo{\Oc_C(m)}+\sum_{m=0}^{x-n-c}\dimglo{\Oc_C(m)}-\sum_{m=0}^{x-n}\dimglo{\Oc_C(m)}\right]+D\\
 &= \frac{d\cdot Q(\alpha,\beta,\gamma)\cdot q^2-d\cdot Q(a,b,c)}{4}+D.
\end{align*}
\end{proof}

An important application of Lemma \ref{hkfdreivar} is the case, where $\Syz_{\Proj(R)}(X,Y,Z)$ admits a Frobenius periodicity.

\begin{defi}
Let $X$ be a scheme over an algebraically closed field $k$ of characteristic $p>0$. Let $\Sc$ be a vector bundle over $X$. 
Assume there are $s<t\in\N$ such that the Frobenius pull-backs $\fpb{e}(\Sc)$ of $\Sc$ are pairwise non-isomorphic 
for $0\leq e\leq t-1$ and $\fpb{t}(\Sc)\cong\fpb{s}(\Sc)$. We say that $\Sc$ admits a \textit{$(s,t)$-Frobenius periodicity.} 
Finally, the bundle $\Sc$ admits a \textit{Frobenius periodicity} if there are $s<t\in\N$ such that $\Sc$ admits a $(s,t)$-Frobenius periodicity.
\end{defi}

\begin{rem}
\begin{enumerate}
\item By a result of Lange and Stuhler the vector bundles $\Sc$ admitting a $(0,t)$-Frobenius periodicity are exactly those which are \'etale trivializable 
(cf. \cite[Satz 1.4]{langestuhler}), where $\Sc$ is \textit{\'etale trivializable} if there exists a finite \'etale, surjective morphism $\phi:D\ra C$ such that 
$\phi\pb(\Sc)$ is trivial.
\item If $X$ is a projective scheme only bundles of degree zero might admit a Frobenius periodicity in the strong sence above. For bundles of non-zero degree 
we will use the following weaker definition of Frobenius periodicity. We say that $\Sc$ admits a $(s,t)$-Frobenius periodicity if 
$$\fpb{e}(\Sc)\ncong\fpb{e'}(\Sc)(m) \quad\text{and}\quad \fpb{t}(\Sc)\cong\fpb{s}(\Sc)(n)$$
hold for all $0\leq e'<e\leq t-1$, all $m\in\Z$ and some $n\in\Z$.
\end{enumerate}
\end{rem}

\begin{exa}
Let $d\geq 2$ be an integer and let $k$ be a field of characteristic $p$ with $p\equiv -1\text{ }(2d)$. Let 
$$C:=\Proj\left(k[X,Y,Z]/\left(X^d+Y^d-Z^d\right)\right)$$ 
be the Fermat curve of degree $d$ and $\Sc:=\Syz_C(X,Y,Z)$. In this case Brenner and Kaid showed that $\Sc$ is always strongly semistable and admits a 
$(0,1)$-Frobenius periodicity (cf. \cite[Lemma 3.2, Theorem 3.4]{holgeralmar}). To be more precise, one has
$$\fpb{}(\Sc)\cong\Sc\left(-\frac{3p-3}{2}\right).$$ 
From this periodicity one can compute the Hilbert-Kunz function of $R$ as (cf. \cite[Corollary 4.1]{holgeralmar} or use Lemma \ref{hkfdreivar} with 
$\alpha=\beta=\gamma=a=b=c=D=1$)
$$\HKF(R,q)=\frac{3d}{4}\cdot q^2+1-\frac{3d}{4}.$$
In Chapter \ref{chapextfurexa} we recover the periodicity by a different approach as in \cite{holgeralmar}. The isomorphisms $\fpb{}(\Sc)\cong\Sc\left(-\frac{3p-3}{2}\right)$ 
as well as the trivializing \'{e}tale covers are computed explicitly in \cite{axel}.
\end{exa}

\begin{exa}\label{hkf2lm}
We are interested in the Hilbert-Kunz functions of the diagonal hypersurfaces of the form
$$R:=k[X,Y,Z]/(X^2+Y^l+Z^m),$$ 
where $k$ is an algebraically closed field of characteristic $p>0$. With $\Deg(X)=lm$, $\Deg(Y)=2m$ and $\Deg(Z)=2l$ the ring $R$ is homogeneous. 
Moreover, we assume $l=3$ and $m\geq 6$ or $4\leq l\leq m$. Note that $l$ and $m$ are chosen such that 
\begin{equation}\label{cond2lm}\frac{1}{2}+\frac{1}{l}+\frac{1}{m}\leq 1.\end{equation}
The above inequality is equivalent to the condition that the triple $\left(\tfrac{1}{2},\tfrac{1}{l},\tfrac{1}{m}\right)$ does not satisfy the strict 
triangle inequality. Since the map 
$$R\ra S:=k[U,V,W]/\left(U^{2lm}+V^{2lm}+W^{2lm}\right)$$ 
with $X\mapsto U^{lm}$, $Y\mapsto V^{2m}$, $Z\mapsto W^{2l}$ is flat and the extension 
$$k[U,V,W]/\left(U^{lm},V^{2m},W^{2l}\right):k$$
has degree $4l^2m^2$ (cf. \cite[Theorem 22.2]{matsu}), we obtain
$$\HKF(R,q)=\frac{\HKF\left(\left(U^{lm},V^{2m},W^{2l}\right),S,q\right)}{4l^2m^2}$$
by Remark \ref{hkfflat}. Using the Lemmata \ref{hkffrei} and \ref{hkfdreivar}, we obtain
\begin{equation}\label{equ1} \HKF(R,q)=\frac{2lm\cdot Q(lm,2m,2l)\cdot q^2+2lm\cdot\delta^2}{4\cdot 4l^2m^2}\end{equation} 
if $\Syz_S(U^{lmq},V^{2mq},W^{2lq})\cong S(-n)\oplus S(-n-\delta)$ for some $n,\delta\in\Z$ and 
\begin{equation}\label{equ2} \HKF(R,q)=\frac{2lm\cdot Q(lm,2m,2l)\cdot q^2- 2lm\cdot Q(lma,2mb,2lc)}{4\cdot 4l^2m^2}+\frac{\Dim_k\left(S/(U^a,V^b,Z^c)\right)}{4l^2m^2}\end{equation}
if $\Syz_S(U^{lmq},V^{2mq},W^{2lq})\cong\Syz_S(U^a,V^b,W^c)(-n)$ for some $n\in\Z$ and some $a,b,c\in\N$ such that at least one of the inequalities 
$a<lmq$, $b<2mq$, $c<2lq$ holds.

At first we present a geometric computation which will show that $\Syz_S(U^{lmq},V^{2mq},W^{2lq})$ gets free if $q$ is large and the inequality 
(\ref{cond2lm}) is strict. Afterwards, we turn to the algebraic situation to compute the values of the Hilbert-Kunz functions that are not covered by the 
geometric computation. The computation in the algebraic situation will recover the results from the geometric one.

Denote by $C$ the projective spectrum of $S$ and assume that the inequality (\ref{cond2lm}) is strict. This means that $(0,W^{2l},-V^{2m})$ is a syzygy 
for $U^{lm},V^{2m},W^{2l}$ of minimal possible degree $2(l+m)$. Therefore, the vector bundle $\Oc_C(-2(l+m))$ is the maximal destabilizing subsheaf of 
$\Syz_C\left(U^{lm},V^{2m},W^{2l}\right)$ and we obtain the short exact sequence
$$0\ra\Oc_C(-2(l+m))\ra\Syz_C\left(U^{lm},V^{2m},W^{2l}\right)\ra\Oc_C(-lm)\ra 0.$$
The $e$-th Frobenius pull-back of this sequence splits if 
\begin{align*}
\Ext^1(\Oc_C(-lmp^e),\Oc_C(-2(l+m)p^e)) & \cong \lK^1(C,\Oc_C(-(2l+2m-lm)p^e))\\
 & \cong \lK^0(C,\Oc_C((2l+2m-lm)p^e+2lm-3))
\end{align*}
vanishes, where the last isomorphism is due to Serre duality with $\omega_C\cong\Oc_C(2lm-3)$. Since $2m+2l-lm<0$, the degree shift $(2l+2m-lm)p^e+2lm-3$ gets 
negative for large $e$ such that $\Oc_C((2l+2m-lm)p^e+2lm-3)$ has no global sections. This shows (with $q=p^e$)
$$\Syz_C(U^{lmq},V^{2mq},W^{2lq})\cong\Oc_C(-2(l+m)q)\oplus\Oc_C(-lmq)$$
for large $e$. We obtain the splitting 
$$\Syz_S(U^{lmq},V^{2mq},W^{2lq})\cong S(-2(l+m)q)\oplus S(-lmq)$$
of $S$-modules, showing 
\begin{align*}
\HKF\left(\left(U^{lm},V^{2m},W^{2l}\right),S,q\right) & = \frac{2lm\cdot Q(lm,2m,2l)\cdot q^2+2lm\cdot\delta^2}{4}=8l^2m^2q^2 \text{ and}\\
\HKF(R,q) & = 2q^2
\end{align*}
for all $q$ bigger than $\tfrac{2lm-3}{ml-2m-2l}$.

We now show by an algebraic computation that we always have a splitting
$$\Syz_S(U^{lmq},V^{2mq},W^{2lq})\cong S(-2(l+m)q)\oplus S(-lmq)$$
if $2l+2m<lm$ or $2l+2m=lm$ and $p\not\equiv 1\text{ }(l)$. In the case $p=2$, we have (for $q\geq 2$)
$$U^{lmq}=(V^{2lm}+W^{2lm})^{\frac{q}{2}}=V^{lmq}+W^{lmq}.$$
From this equality, we deduce that $\Syz_S(U^{lmq},V^{2mq},W^{2lq})$ is free. The generators are 
$$\left(-1,V^{(l-2)mq},W^{(m-2)lq}\right)\text{ and }(0,W^{2lq},-V^{2mq})$$
of total degrees $mlq$ and $2mq+2lq$.

If $q$ is divisible by $l$, we have 
$$V^{2mq}=(-U^{2ml}-W^{2ml})^{\frac{q}{l}}=-U^{2mq}-W^{2mq}.$$
From this equality we obtain with Lemma \ref{manipulationlemma}
\begin{align*}
\Syz_S(U^{lmq},V^{2mq},W^{2lq}) & \cong \Syz_S(U^{lmq},-U^{2mq}-W^{2mq},W^{2lq}) \\
 & \cong \Syz_S(U^{lmq},U^{2mq},W^{2lq}) \\
 & \cong \Syz_S(U^{(l-2)mq},1,W^{2lq})(-2mq)\\
 & \cong (S(-(l-2)mq)\oplus S(-2lq))(-2mq)\\
 & \cong S(-lmq)\oplus S(-2mq-2lq).
\end{align*}
Now we can assume that $p$ is odd and that we can write $q=ln+r$ with $1\leq r\leq l-1$. Using $U^{2lm}=-V^{2lm}-W^{2lm}$, we get
\begin{equation}\label{xqzer}U^{lmq}=-U^{lm}\sum_{i=0}^{\frac{q-1}{2}}\binom{\frac{q-1}{2}}{i}(-1)^{\frac{q-1}{2}}V^{2lmi}W^{2lm\left(\frac{q-1}{2}-i\right)}.\end{equation}
The summands for $i\geq \left\lceil\frac{q}{l}\right\rceil=n+1$ are divisible by $V^{2mq}$ and the summands with 
$$i\leq \left\lfloor\frac{(m-2)q-m}{2m}\right\rfloor$$
are divisible by $W^{2lq}$. We take a closer look at the bound for divisibility by $W^{2lq}$.
\begin{align}\label{zqdiv}
\frac{(m-2)q-m}{2m} & = \frac{(m-2)(ln+r)-m}{2m} \nonumber \\
 & = n+\frac{(ml-2m-2l)n+m(r-1)-2r}{2m}.
\end{align}
If the numerator in the fraction in the last line is non-negative, all summands for indices $i\leq n$ in Equation (\ref{xqzer}) are divisible by $W^{2lq}$. 
Since all summands with index $i\geq n+1$ are multiples of $V^{2mq}$, we obtain $U^{lmq}\in (V^{2mq},W^{2lq})$. As in the case $p=2$, this implies
$$\Syz_S(U^{lmq},V^{2mq},W^{2lq})\cong S(-lmq)\oplus S(-2mq-2lq).$$
The number $ml-2m-2l$ vanishes if and only if $(l,m)\in\{(3,6),(4,4)\}$, otherwise it is positive. The cases $(l,m)=(3,6)$ and $(l,m)=(4,4)$ will be 
discussed seperatedly, hence we assume $ml-2m-2l\geq 1$. Starting with the case $r=1$, we have to show $(ml-2m-2l)n\geq 2$. From $r=1$, we obtain 
$n\geq 1$. If $l$ is even, $ml-2l-2m$ is even and the result follows. If $l$ is odd, we get that $n$ is even, since $q$ is odd. Now let $r>1$. To 
prove that the numerator in (\ref{zqdiv}) is non-negative, it suffices to show $m(r-1)-2r\geq 0$. This is true, because otherwise we obtain the contradiction
$$m(r-1)<2r\Longleftrightarrow r<1+\frac{2}{m-2}\leq 2,$$
since $m\geq 4$ in all cases. This argument is still valid in the cases $(l,m)\in\{(3,6),(4,4)\}$.

Up to now, we found (except for the cases $q$ odd with $q\equiv 1\text{ }(l)$ and $(l,m)\in\{(3,6),(4,4)\}$)
$$\Syz_S(U^{lmq},V^{2mq},W^{2lq})\cong S(-lmq)\oplus S(-2mq-2lq).$$
The Hilbert-Kunz function of $R$ is given by Equation (\ref{equ1}) as 
\begin{align*}
\HKF(R,q) &= \frac{2lm\cdot Q(lm,2l,2m)\cdot q^2+2lm\cdot (lm-2m-2l)^2\cdot q^2}{16l^2m^2}\\
 &= 2q^2
\end{align*}
We only have to treat the cases $q=nl+1$ for $(l,m)\in\{(3,6),(4,4)\}$ with $q$ odd. If $p\not\equiv 1 \text{ }(l)$ we obtain a splitting 
$$\Syz_S(U^{lmp},V^{2mp},W^{2lp})\cong S(-lmp)\oplus S(-2mp-2lp),$$
which carries over to all higher exponents of $p$. Hence, we may assume $q=p$ with $p=nl+1$. We start with the case $(l,m)=(3,6)$ and work with the 
grading $\Deg(X)=18$, $\Deg(Y)=12$ and $\Deg(Z)=6$. With $p=3n+1$ only the summand for $i=n$ in Equation (\ref{xqzer}) does not belong to the ideal 
$(V^{12p},W^{6p})$. The binomial coefficient of this summand is non-zero since all its factors are $<p$.  This shows
\begin{align*}
 \Syz_S(U^{18p},V^{12p},Z^{6p}) &\cong \Syz_S\left(U^{18}V^{12(p-1)}W^{6(p-1)},V^{12p},W^{6p}\right)\\
 &\cong \Syz_S(U^{18},V^{12},W^{6})(-18(p-1)).
\end{align*}
By Equation (\ref{equ2}), we obtain
\begin{align*}
\HKF(R,q) &= \frac{36\cdot Q(18,12,6)\cdot q^2-36\cdot Q(18,12,6)+4\cdot 18\cdot 12\cdot 6}{4\cdot 4\cdot 9\cdot 36}\\
 &= 2q^2-1.
\end{align*}
In the case $(l,m)=(4,4)$ we work with the grading $\deg(X)=16$, $\Deg(Y)=\Deg(Z)=8$. Let $q=4n+1$. Arguing as in the case $(l,m)=(3,6)$ we may assume 
$q=p$ and obtain
$$\Syz_S(U^{16p},V^{8p},W^{8p}) \cong \Syz_S(U^{16},V^{8},W^{8})(-16(p-1)).$$
Using Equation (\ref{equ2}) in this situation, we find again $\HKF(R,q)=2q^2-1$.

All in all we have
$$\HKF(R,p^e)=\begin{cases}
1 & \text{if }e=0,\\
2q^2-1 & \text{if }e>0, \text{ } p\equiv 1\text{ }(l) \text{ and } (l,m)\in\{(3,6),(4,4)\},\\
2q^2 & \text{else}.
\end{cases}$$
\end{exa}

\begin{rem}
The previous example covers all two-dimensional diagonal hypersurfaces of multiplicity two which are not of type $A_n$, $E_6$ or $E_8$ with one exceptional family. Namely the cases 
$R:=k[X,Y,Z]/(X^2+Y^2+Z^l)$ for $l\geq 2$ and $\chara(k)=2$ are not treated. In this case an easy calculation shows $\HKF(R,2^e)=2\cdot 4^e$.
\end{rem}

\section{Han's $\delta$ function}
In this section we introduce "Han's $\delta$ function", which measures syzygy gaps in a polynomial ring in two variables. It was first studied by Han in her thesis \cite{handiss}.
We also give a result of Monsky, that allows us to compute Hilbert-Kunz multiplicities of certain curves in $\Prim^2$ with the help of this function.
All statements are reproved in a slightly more general context to make them work in non-standard graded situations.

Let $k$ denote an algebraically closed field of $\chara k=p>0$ and let $R=k[U,V]$, where $U$ and $V$ have the degrees $a$ resp. $b\in\N_{\geq 1}$. All polynomials are assumed to be (quasi-)\-homogeneous.

\begin{defi}
Let $F_1$, $F_2$, $F_3\in R$ be homogeneous and coprime. By Hilberts Syzygy Theorem we get a splitting $\Syz_R(F_1,F_2,F_3)\cong R(-\alpha)\oplus R(-\beta)$ for some $\alpha\geq\beta\in\N$.
Let $$\sg(F_1,F_2,F_3):=\alpha-\beta.$$
\end{defi}

\begin{rem}
Let $I$ be the ideal generated by the $F_i$ and $d_i:=\deg(F_i)$. Then by \cite[Proposition 5.8.9]{cocoa2} the Hilbert-series of $R/I$ is $$\lK_{R/I}(t)=\sum_{i=0}^{\infty}\Dim_k((R/I)_i)\cdot t^i=\frac{1-t^{d_1}-t^{d_2}-t^{d_3}+t^{\alpha}+t^{\beta}}{(1-t^a)(1-t^b)}.$$
Note that this is a polynomial, since we have $\Dim_k(R/I)<\infty$ by the coprimeness of the $F_i$, hence $\Dim_k((R/I)_i)=0$ for $i\gg 0$. 
\end{rem}

\begin{prop}\label{deltadim} With the previous notations we get the following formulas.
\begin{enumerate}
\item $\alpha+\beta=d_1+d_2+d_3.$
\item Let $Q(r,s,t):=2rs+2rt+2st-r^2-s^2-t^2$. Then $$\Dim_k(R/I)=\frac{Q(d_1,d_2,d_3)+\sg(F_1,F_2,F_3)^2}{4ab}.$$
\end{enumerate}\end{prop}

\begin{proof}
Write $P(t)$ for the numerator of $\lK(R/I)(t)$. Differentiating both sides of $P(t)=\lK_{R/I}(t)(1-t^a)(1-t^b)$ one gets
\begin{align*}
 &  -d_1t^{d_1-1}-d_2t^{d_2-1}-d_3t^{d_3-1}+\alpha t^{\alpha-1}+\beta t^{\beta-1}\\
=& (\lK_{R/I}(t))'(1-t^a)(1-t^b)-a\cdot \lK_{R/I}(t)(1-t^{a-1})(1-t^b)-b\cdot \lK_{R/I}(t)(1-t^a)(1-t^{b-1}).
\end{align*}
The evaluation at 1 gives (i).
Similarly, differentiating $P(t)=\lK_{R/I}(t)(1-t^a)(1-t^b)$ twice and evaluating at 1 gives (ii).

Note that (i) can be seen alternatively by the additivity of the degree of vector bundles.
\end{proof}

The previous proposition gives $$\sg(F_1,F_2,F_3)=\alpha-\beta=d_1+d_2+d_3-2\beta,$$ where $\beta$ is the minimal degree of a non-trivial syzygy for $I$. This allows us to drop the coprimeness condition in the definition of $\sg$.

\begin{defi}
Let $m(F_1,F_2,F_3)$ be the \textit{minimal degree} of a non-trivial syzygy for the ideal generated by $F_1$, $F_2$ and $F_3$. Then define $$\sg(F_1,F_2,F_3):=d_1+d_2+d_3-2\cdot m(F_1,F_2,F_3).$$
\end{defi}

Some basic properties of $\sg$ are given in the next proposition.

\begin{prop}\label{propdelta}
With the previous notations and $F\in R$ non-zero and homogeneous the following (in-)equalities hold.

\begin{enumerate}
\item $\sg(FF_1,FF_2,FF_3)=\Deg(F)+\sg(F_1,F_2,F_3)$.
\item If $F$ and $F_3$ are coprime, we have $\sg(FF_1,FF_2,F_3)=\sg(F_1,F_2,F_3)$.
\item $|\sg(FF_1,F_2,F_3)-\sg(F_1,F_2,F_3)|\leq\Deg(F)$.
\item $\sg(F_1^p,F_2^p,F_3^p)=p\cdot\sg(F_1,F_2,F_3)$.
\end{enumerate}\end{prop}

\begin{proof}
See Propositions 2.12, 2.17, 2.19 resp. 2.16 of \cite{texdiss}.
\end{proof}

We now specify the situation:

\begin{setup}
Let $k$ be an algebraically closed field of characteristic $p>0$, let $R:=k[X,Y]$ equipped with the standard grading. For $a_1$, $a_2$, $a_3$, $e\in\N$ and $q:=p^e$ set 
$$\delta\left(\frac{a_1}{q},\frac{a_2}{q},\frac{a_3}{q}\right):=q^{-1}\cdot\sg(X^{a_1},Y^{a_2},(X+Y)^{a_3}).$$
Then Proposition \ref{propdelta} (iv) ensures, that this definition is independent of $q$ and part (iii) gives the inequality
$$\left|\delta\left(\frac{a_1}{q},\frac{a_2}{q},\frac{a_3}{q}\right)-\delta\left(\frac{b_1}{q},\frac{b_2}{q},\frac{b_3}{q}\right)\right|\leq q^{-1}\cdot\left\Vert(a_1-b_1,a_2-b_2,a_3-b_3)\right\Vert_1,$$
where $\Vert\_\Vert_1$ denotes the taxicap distance. Therefore $\delta$ is Lipschitz and extends to a unique continuous function $\delta$ on $[0,\infty)^3$ with 
$\delta\left(\frac{t}{q}\right)=q^{-1}\cdot\delta(t)$ and $|\delta(t)-\delta(s)|\leq \left\Vert t-s\right\Vert_1$.
\end{setup}

Our next goal is to explain how $\delta$ can be computed in this situation. We follow \cite{mason}, which provides a simpler proof than \cite{handiss}.

\begin{defi}
Let $$L_{\text{odd}}:=\left\{u\in\N^3|u_1+u_2+u_3\text{ is odd}\right\}$$ be the odd lattice.
\end{defi}

We are interested in the taxicap distance of elements of the form $\tfrac{t}{p^s}$ to $L_{\text{odd}}$, where $s$ is an integer and $t$ a three-tuple of non-negative real numbers.
Note that for given $\tfrac{t}{p^s}$ there is at most one $u\in L_{\text{odd}}$ satisfying 
$$\left\Vert \frac{t}{p^s}-u\right\Vert_1<1.$$
Moreover, the only candidates for $u_i$ are the roundups and rounddowns of $\tfrac{t_i}{p^s}$.

\begin{defi}
For $s\in\Z$ and $t\in[0,\infty)^3$ we define
$$E_s(t):=\left\{\begin{aligned}
 & 0 && \text{ if there is no } u\in L_{\text{odd}} \text{ with }\left\Vert \frac{t}{p^s}-u\right\Vert_1<1,\\
 & 1-\left\Vert \frac{t}{p^s}-u\right\Vert_1 && \text{ if }\left\Vert \frac{t}{p^s}-u\right\Vert_1<1. 
\end{aligned}\right.$$
\end{defi}

Monsky has proved in \cite[Lemma 21, Theorem 22]{mason} the following properties of the functions $E_s(t)$.

\begin{thm}[Monsky]\label{estthm}
Let $t=(t_1,t_2,t_3)\in[0,\infty)^3.$

\begin{enumerate}
\item If the $t_i$ do not satisfy the strict triangle inequality, say $t_1 \geq t_2+t_3$, then $$\Max_{s\in\Z}\left\{p^s\cdot E_s(t)\right\}=t_1-t_2-t_3.$$
\item If the $t_i$ satisfy the strict triangle inequality, then either $E_s(t)=0$ for all $s$ or there exists a maximal $s$ with $E_s(t)\neq 0$. Moreover, for this maximal $s$ we have $$p^s\cdot E_s(t)=\Max_{r\in\Z}\left\{p^r\cdot E_r(t)\right\}.$$
\item $$\delta(t)=\Max_{s\in\Z}\left\{p^s\cdot E_s(t)\right\}.$$
\end{enumerate}
\end{thm}

From this theorem one gets Han's Theorem \cite[Theorem 2.25 and Theorem 2.29]{handiss}.

\begin{thm}[Han]\label{hansthm}
Let $t=(t_1,t_2,t_3)\in[0,\infty)^3.$ If the $t_i$ do not satisfy the strict triangle inequality (w.l.o.g. $t_1\geq t_2+t_3$), we have $\delta(t)=t_1-t_2-t_3.$
If the $t_i$ satisfy the strict triangle inequality and there are $s\in \Z$ and $u\in L_{odd}$ with $\left\Vert p^st-u\right\Vert_1<1$, then there is such a pair $(s,u)$ with minimal $s$ and with this pair $(s,u)$ we get
$$\delta(t)=\frac{1}{p^s}\cdot\left(1-\left\Vert p^st-u\right\Vert_1\right).$$ Otherwise, we have $\delta(t)=0$.
\end{thm}

\begin{proof}
If the $t_i$ do not satisfy the strict triangle inequality, the statement is immediate from Theorem \ref{estthm} (i) and (iii).
If the $t_i$ satisfy the strict triangle inequality, then by Theorem \ref{estthm} (ii) there is a minimal $s$ such that $E_{-s}(t)\neq 0$ or all $E_s(t)$ vanish. Part (iii) gives the result.
\end{proof}

\begin{rem}
At least if $t$ is a triple of rational numbers, there is a finite range in which the $s$ in Theorem \ref{hansthm} has to lie. Namely, if the $t_i$ do not satisfy 
the strict triangle inequality, one has $p^lE_l(t)=0$ for all $l\geq\tfrac{3}{2}\cdot\Max(t_i)$ by \cite{mason}. This gives a lower bound for $s$. With 
$t_i:=z_i/n_i$, $z_i,n_i\in\N$, the distance of $t\cdot p^u$ for $u\geq 0$ depends only on the residue classes of $z_i\cdot p^u$ modulo $2\cdot n_i$. With 
$n:=2\cdot\text{lcm}(t_i)$, the series 
$$((\overline{z_1\cdot p^u},\overline{z_2\cdot p^u},\overline{z_3\cdot p^u}))_{u\in\N},$$
where $\overline{x}:=x\text{ mod }n$, is eventually periodic. The number of different values for the elements of this series gives an upper bound for $s$.

Note that there is an explicit implementation of Han's Theorem (for rational arguments) in Appendix \ref{app.myhkmthm}.
\end{rem}

\begin{exa}\label{deltahalb}
Let $b\leq a\in\N_{\geq 2}$ with $2(a+b)\geq ab$ and both prime to the odd prime $p$. For $e\in\N$ we set $q:=p^e$ for short. We want to show $$\delta\left(\frac{1}{2},\frac{1}{a},\frac{1}{b}\right)=0.$$
The first step is to show that the $s$ in Han's Theorem \ref{hansthm} has to be positive. The nearest elements in $L_{\text{odd}}$ to $(\tfrac{1}{2q},\tfrac{1}{aq},\tfrac{1}{bq})$ are $(1,0,0)$, $(0,1,0)$, $(0,0,1)$ and $(1,1,1)$.
The first three points give a taxicap distance of at least 1. For the point $(1,1,1)$ we get
$$\begin{aligned}
 & 1-\frac{1}{2q}+1-\frac{1}{aq}+1-\frac{1}{bq} && < 1\\
\Longleftrightarrow & \frac{1}{2q}+\frac{1}{aq}+\frac{1}{bq} && > 2\\
\Longleftrightarrow & 1+\frac{2}{a}+\frac{2}{b} && > 4q.
\end{aligned}$$
Since $a,b\geq 2$, this gives the contradiction $4q< 3$. Therefore $s>0$ in Theorem \ref{hansthm}.

Now consider the points $(\tfrac{q}{2},\tfrac{q}{a},\tfrac{q}{b})$ for $q\geq p$. Since the first coordinate yields a summand $\tfrac{1}{2}$ in the taxicap distance to $L_{\text{odd}}$ and we may choose $\tfrac{p-1}{2}$ or $\tfrac{p+1}{2}$, 
we get that the taxicap distance of $(\tfrac{q}{2},\tfrac{q}{a},\tfrac{q}{b})$ to $L_{\text{odd}}$ is exactly $\tfrac{1}{2}$ plus the taxicap distance 
of $(\tfrac{q}{a},\tfrac{q}{b})$ to $\N^2$.
The four points in $\N^2$ that we have to consider have the roundup or rounddown of $\tfrac{q}{a}$ in the first and of $\tfrac{q}{b}$ in the second 
component. Hence, the taxicap distance to $\N^2$ is of the form $\tfrac{n}{a}+\tfrac{m}{b}$ with $1\leq n< a$ and $1\leq m< b$, since $a$ and $b$ are 
prime to $p$. If this taxicap distance is smaller than $\tfrac{1}{2}$, we get the contradiction 
$$nb+ma<\frac{ab}{2}\leq a+b.$$ 
All in all, the taxicap distance of $p^s(\tfrac{1}{2},\tfrac{1}{a},\tfrac{1}{b})$ to $L_{\text{odd}}$ is greater than 1 for all $s\in\Z$. Therefore $\delta$ is zero.
\end{exa}

Before we connect the $\delta$ function with Hilbert-Kunz multiplicities, we prove a computation rule for a non-standard graded analogue.

\begin{setup}\label{setuptau}
Let $k$ be an algebraically closed field of characteristic $p>0$, let $R:=k[U,V]$ equipped with the positive grading $\Deg(U)=a$, $\Deg(V)=b$. Fix $c$, $d\in\N_{\geq 1}$ 
such that $ac=bd$. For $a_1$, $a_2$, $a_3$, $c$, $d$, $e\in\N$ with $q:=p^e$ set 
$$\tau\left(\frac{a_1}{q},\frac{a_2}{q},\frac{a_3}{q}\right):=q^{-1}\cdot\sg(U^{a_1},V^{a_2},(U^c+V^d)^{a_3}).$$ 
Again Proposition \ref{propdelta} (iv) ensures, that this definition is independent of $q$ and part (iii) gives the inequality
$$\left|\tau\left(\frac{a_1}{q},\frac{a_2}{q},\frac{a_3}{q}\right)-\tau\left(\frac{b_1}{q},\frac{b_2}{q},\frac{b_3}{q}\right)\right|\leq q^{-1}\cdot\left\Vert(a_1-b_1,a_2-b_2,a_3-b_3)\right\Vert_1.$$ 
Therefore $\tau$ is Lipschitz and extends to a unique continuous function $\tau$ on $[0,\infty)^3$ with $\tau(\tfrac{t}{q})=q^{-1}\cdot\tau(t)$ and $|\tau(t)-\tau(s)|\leq \left\Vert t-s\right\Vert_1$.
\end{setup}

\begin{lem}\label{deltaquasi} With the notations from Setup \ref{setuptau} and $t\in[0,\infty)^3$ we have 
$$\tau(t)=ac\delta\left(\frac{t_1}{c},\frac{t_2}{d},t_3\right).$$
\end{lem}

\begin{proof}
For $t\in\N^3$ both sides multiply by $p$ when $t$ is replaced by $pt$. Since rationals of the form $\tfrac{nc}{p^j}$, $\tfrac{nd}{p^j}$ are dense in $[0,\infty)$, it suffices 
to prove the statement for $t=(ct_1,dt_2,t_3)$ with $t_i\in\N$. Let $S$ be the standard-graded polynomial ring $k[X,Y]$. Consider the map 
$\phi:S\rightarrow R$, $X\mapsto U^c$, $Y\mapsto V^d$, which is homogeneous of degree $ac$. Denote by $\psi$ the induced map on the projective spectra. 
We obtain 
$$\xymatrix{
\psi\pb(\Syz_{\Proj(S)}(X^{t_1},Y^{t_2},(X+Y)^{t_3}))\ar[r]^{\cong}\ar[d]^{\cong} & \Syz_{\Proj(R)}(U^{ct_1},V^{dt_2},(U^c+V^d)^{t_3})\ar[d]^\cong\\
\psi\pb(\Oc_{\Proj(S)}(l)\oplus\Oc_{\Proj(S)}(l+\delta))\ar[r]^{\cong} & \Oc_{\Proj(R)}(acl)\oplus\Oc_{\Proj(S)}(acl+ac\delta)
}$$
for some $l\in\Z$ and $\delta=\delta(t_1,t_2,t_3)$. Now the claim follows, since we have the isomorphism
$$\Syz_{\Proj(R)}(U^{ct_1},V^{dt_2},(U^c+V^d)^{t_3})\cong\Oc_{\Proj(R)}(m)\oplus\Oc_{\Proj(R)}(m+\tau)$$
for some $m\in\Z$ and $\tau=\tau(ct_1,dt_2,t_3)$.
\end{proof}

We now prove a slightly more general version of \cite[Theorem 2.3]{irred} (see Corollary \ref{irredmon} below). We will skip the proofs that are still valid in our situation.

\begin{setup}\label{setupmu}
Let $R$ be a standard-graded domain of dimension $d$ and $J:=(x_1,\ldots,x_n)$ an $R_+$-primary ideal with all $x_i$ non-zero and homogeneous. For $l\in\N^n$ let 
$J(l):=\left(x_1^{l_1},\ldots,x_n^{l_n}\right)$. We denote $\HKM(J(l),R)$ by $\mu(l)$.
\end{setup}

We will show that $\mu$ extends to a unique continuous function on $[0,\infty)^n$ with the property $\mu(pt)=p^d\cdot\mu(t)$ 
(this is Remark 5 in the appendix of \cite{texnew} and we will lean on the proof given there in the case $l_i=l_1$ for all $i$).

\begin{lem}
In the situation of Setup \ref{setupmu} let $l\in\N^n$. If $l_i\geq 1$, let $l^{(i)}:=l-e_i$, where $e_i$ denotes the $i$-th canonical vector. Then $$\mu(l)-\mu\left(l^{(i)}\right)=O\left((\Max(l_j))^{d-1}\right)$$
\end{lem}

\begin{proof}
The $R$-module $J\left(l^{(i)}\right)/J(l)$ is annihilated by $x_i$, which gives it a natural $S:=R/x_i$-module structure. Therefore we get the inequality
$$\lambda_R\left(J\left(l^{(i)}\right)/J(l)\right)=\lambda_S\left(J\left(l^{(i)}\right)S/J(l)S\right)\leq \lambda_S(S/J(l)S)\leq \lambda_S\left(S/\mm^{n\cdot\Max(l_j)}\right).$$
But the last length is for large $\Max(l_j)$ given by a polynomial in $\Max(l_j)$ of degree $d-1$.

By Theorem \ref{watyorelhkm} we have
$$\mu(J(l))-\mu\left(J\left(l^{(i)}\right)\right)\leq \HKM(R)\cdot\lambda_R\left(J\left(l^{(i)}\right)/J(l)\right)=O\left((\Max(l_j))^{d-1}\right).$$
\end{proof}

\begin{lem}
The notations are the same as in Setup \ref{setupmu}. Let $\Ic\subsetneq[0,1]^n$ be the set of all elements of the form $(\tfrac{a_1}{q_1},\ldots,\tfrac{a_n}{q_n})$, where the $q_i\geq 1$ are powers of $p$ and the $a_i$ are natural numbers $\leq q_i$.
Then $\mu:\Ic\ra\R$ is Lipschitz.
\end{lem}

\begin{proof}
Given two elements from $\Ic$, we may expand each component by a power of $p$ and end up with two elements of the form $\tfrac{l}{q}$, $\tfrac{m}{q}$, where $l$ and $m$ are $n$-tuples of natural numbers $\leq q$.
$$\begin{aligned}
|\mu(l)-\mu(m)| =&\text{ } |\mu(l)-\mu(l_1,\ldots,l_{n-1},m_n)+\mu(l_1,\ldots,l_{n-1},m_n)-\mu(l_1,\ldots,l_{n-2},m_{n-1},m_n)\\
 &\text{ } +\ldots+\mu(l_1,m_2,\ldots,m_n)-\mu(m)|\\
\leq &\text{ } |\mu(l)-\mu(l_1,\ldots,l_{n-1},m_n)|+|\mu(l_1,\ldots,l_{n-1},m_n)-\mu(l_1,\ldots,l_{n-2},m_{n-1},m_n)|\\
  &\text{ } +\ldots+|\mu(l_1,m_2,\ldots,m_n)-\mu(m)|\\
 = &\text{ } \sum_{i=1}^n|\mu(l_1,\ldots,l_i,m_{i+1},\ldots,m_n)-\mu(l_1,\ldots,l_{i-1},m_i,\ldots,m_n)|.
\end{aligned}$$
Expanding the $i$-th summand in the same way (w.l.o.g. $l_i\geq m_i$), the previous lemma gives the inequality
\begin{align*}
 &\text{ } |\mu(l_1,\ldots,l_{i-1},l_i,m_{i+1},\ldots,m_n)-\mu(l_1,\ldots,l_{i-1},m_i,m_{i+1},\ldots,m_n)|\\
\leq &\text{ } |\mu(l_1,\ldots,l_{i-1},l_i,m_{i+1},\ldots,m_n)-\mu(l_1,\ldots,l_{i-1},l_i-1,m_{i+1},\ldots,m_n)|\\
 &\text{ } +\ldots +|\mu(l_1,\ldots,l_{i-1},m_i+1,m_{i+1},\ldots,m_n)-\mu(l_1,\ldots,l_{i-1},m_i,m_{i+1},\ldots,m_n)|\\
\leq &\text{ } \sum_{j=0}^{|l_{n+1-i}-m_{n+1-i}|-1}M_i\cdot\Max(l_1,\ldots,l_{i-1},l_i-j,m_{i+1},\ldots,m_n)^{d-1}\\
\leq &\text{ } M_i\cdot \Max(l_1,\ldots,l_i,m_{i+1},\ldots,m_n)^{d-1}\cdot|l_i-m_i|,
\end{align*}
where $M_i$ is the constant from the previous lemma. Summation over $i$ yields
$$ |\mu(l)-\mu(m)| \leq n\cdot \Max\{M_i\}\cdot (\Max\{l_1,m_1,\ldots,l_n,m_n\})^{d-1}\left\Vert l-m\right\Vert_1.$$
Dividing by $q^d$ gives 
$$\left|\mu\left(\frac{l}{q}\right)-\mu\left(\frac{m}{q}\right)\right|\leq \frac{n\cdot \Max\{M_i\}}{q}\left\Vert l-m\right\Vert_1= n\cdot \Max\{M_i\}\left\Vert \frac{l}{q}-\frac{m}{q}\right\Vert_1,$$
hence $\mu$ is Lipschitz on $\Ic$ with Lipschitz constant $n\cdot \Max\{M_i\}$.
\end{proof}

\begin{lem}\label{muconti}
The notations are the same as in Setup \ref{setupmu}. For $t\in[0,\infty)^n$ choose $e\in\N$ such that $p^e\leq\Max(t_i)<p^{e+1}=:q$. This gives 
$\tfrac{t}{q}\in [0,1]^n$ and we define $\mu(t):= q^d\cdot\mu(\tfrac{t}{q})$. Then $\mu(t)$ is continous.
\end{lem}

\begin{proof}
By the previous lemma we only have to take care of the "borders" $p^e$. We have 
$$\Lim{\Max{t_i}\downarrow p^e}\mu(t)=q^d\cdot\mu\left(\frac{1}{p},\ldots,\frac{1}{p}\right)=\left(\frac{q}{p}\right)^d\cdot \HKM(R).$$
For the "limit from below", we have to look for the limit $t\rightarrow (p^e,\ldots,p^e)$, $t_1,\ldots,t_n<p^e$. This limit is
$$\Lim{\stackrel{t\rightarrow (p^e,\ldots,p^e)}{t_1,\ldots,t_n<p^e}}\mu(t)=\left(\frac{q}{p}\right)^d\cdot\mu(1,\ldots,1)=\left(\frac{q}{p}\right)^d\cdot \HKM(R).$$
\end{proof}

Next, we want to use the function $\mu$ to compute the Hilbert-Kunz multiplicities of certain hypersurfaces $k[X,Y,Z]/(F)$ with respect to the ideals 
$(X^a,Y^b,Z^c)$. This is a slight generalization of Monsky's result from \cite{irred}, where he restricted himself to the case $a=b=c$.

\begin{defi}\label{defdegreg}
Let $k$ be an algebraically closed field of characteristic $p>0$. An irreducible, homogeneous (standard-graded) polynomial $F\in k[X,Y,Z]$ of degree $d$ is called \textit{regular} if each of the points $(1:0:0)$, $(0:1:0)$, $(0:0:1)$ has multiplicity $<\tfrac{d}{2}$ on the projective curve defined by $F$. Other\-wise we call $F$ \textit{irregular}.
\end{defi}

Monsky has shown in \cite[Lemma 2.2]{irred} that regular, irreducible trinomials $$F=X^{a_1}Y^{a_2}Z^{a_3}+X^{b_1}Y^{b_2}Z^{b_3}+X^{c_1}Y^{c_2}Z^{c_3}$$ appear in two families:

\begin{center}
\begin{tabular}{c|ccc|ccc|ccc}
 & $a_1$ & $a_2$ & $a_3$ & $b_1$ & $b_2$ & $b_3$ & $c_1$ & $c_2$ & $c_3$\\ \hline
type I & $>\tfrac{d}{2}$ & & 0 & 0 & $>\tfrac{d}{2}$ & & & 0 & $>\tfrac{d}{2}$ \\
type II & $d$ & 0 & 0 &  & $>\tfrac{d}{2}$ & & 0 & & $>\tfrac{d}{2}$ \\
\end{tabular}
\end{center}

An empty slot means, that there is no further restriction at the exponent (except making $F$ homogeneous).

The determinant of the matrix
\begin{equation}\label{deflambda}
\left(\begin{array}{ccc}
a_1 & b_1 & c_1\\
a_2 & b_2 & c_2\\
a_3 & b_3 & c_3
\end{array}\right)
\end{equation}
is $d\lambda$ with $0<\lambda\in\N$, because the rows sum up to $(d,d,d)$.

\begin{con}
Given a regular, irreducible trinomial $F$, compute $\lambda$ as above. Let $A:=k[r,s,t]/(r^{\lambda}+s^{\lambda}+t^{\lambda})$.
We denote by $R$, $S$, $T$ the elements $r^{\lambda}$, $s^{\lambda}$, $t^{\lambda}$ of $A$. Let $x$, $y$, $z$ be the elements 
$r^{b_3}s^{b_2}$, $s^{c_1}t^{c_3}$, $r^{a_1}t^{a_2}$ if $F$ is of type I
and $r^{c_3-b_3}s^{c_2}t^{b_3}$, $s^d$, $r^{b_1}t^{d-b_1}$ if $F$ is of type II.

Now we show where the ring $k[X,Y,Z]/(F)$ appears in these data. Let $B:=k[x,y,z]\subseteq A$ and consider the surjective map $\phi:k[X,Y,Z]\ra B$, 
$X\mapsto x$, $Y\mapsto y$, $Z\mapsto z$. Assume that $F$ is of type I. Replacing the third row in (\ref{deflambda}) by the sum of all three rows, 
one finds $\lambda=a_1b_2+a_2c_1-b_2c_1$. Similarly, one obtains the equalities $\lambda=b_3c_1+a_1c_3-a_1b_3$ and $\lambda=b_2c_3+a_2b_3-a_2c_3$. Then
\begin{align*}
x^{a_1}y^{a_2}+y^{b_2}z^{b_3}+x^{c_1}z^{c_3} & = (r^{b_3}s^{b_2})^{a_1}(s^{c_1}t^{c_3})^{a_2}+(s^{c_1}t^{c_3})^{b_2}(r^{a_1}t^{a_2})^{b_3}+(r^{b_3}s^{b_2})^{c_1}(r^{a_1}t^{a_2})^{c_3}\\
 & = r^{a_1b_3}s^{a_1b_2+a_2c_1}t^{a_2c_3}+r^{a_1b_3}s^{b_2c_1}t^{b_2c_3+a_2b_3}+r^{b_3c_1+a_1c_3}s^{b_2c_1}t^{a_2c_3}\\
 & = r^{a_1b_3}s^{b_2c_1}t^{a_2c_3}(s^{a_1b_2+a_2c_1-b_2c_1}+t^{b_2c_3+a_2b_3-a_2c_3}+r^{b_3c_1+a_1c_3-a_1b_3})\\
 & = r^{a_1b_3}s^{b_2c_1}t^{a_2c_3}(r^{\lambda}+s^{\lambda}+t^{\lambda})
\end{align*}
shows that $\phi$ factors through $k[X,Y,Z]/(F)$. Since $k[X,Y,Z]/(F)$ and $B$ are two-dimensional domains they have to be isomorphic.  A similar 
computation shows that $B$ and $k[X,Y,Z]/(F)$ are isomorphic if $F$ is of type II.

We also define the subrings $B':=k[R,S,T]$ and $C:=k[x^{\lambda},y^{\lambda},z^{\lambda}]$ of $A$.
Note that $B'$ is isomorphic to a polynomial ring in two variables.

Monsky has shown that $B'$ and $B$ are finite over $C$ of ranks $d$ resp. $\lambda^2$.
\end{con}

\begin{defi}\label{defmyhkmthm}
Let $F=X^{a_1}Y^{a_2}Z^{a_3}+X^{b_1}Y^{b_2}Z^{b_3}+X^{c_1}Y^{c_2}Z^{c_3}$ be regular and irreducible. For $t=(t_1,t_2,t_3)\in[0,\infty)^3$ set 
\begin{align*}
\alpha'(t) &:= t_1d-t_1b_3-t_1c_2-t_2c_1-t_3b_1,\\ 
\beta'(t) &:= t_2d-t_2c_1-t_3a_2-t_1c_2-t_2a_3,\\ 
\gamma'(t) &:= t_3d-t_3a_2-t_1b_3-t_3b_1-t_2a_3.
\end{align*}
Let $\alpha_+(t):=\alpha'(t)$ if $\alpha'(t)$ is $\geq 0$ and $0$ otherwise. Similarly, $\alpha_-(t):=-\alpha'(t)$ if $\alpha'(t)<0$ and $0$ otherwise. Do the same for $\beta'(t)$ and $\gamma'(t)$.
After all these preparations set
$$\begin{array}{ccc|ccc}
 & \text{type I} & & & \text{type II} & \\
\alpha(t) & \beta(t) & \gamma(t) & \alpha(t) & \beta(t) & \gamma(t)\\ \hline
 \alpha_+(t)+\gamma_-(t) & \beta_+(t)+\alpha_-(t) & \gamma_+(t)+\beta_-(t) & \alpha_+(t)+\beta_-(t)+\gamma_-(t) & \beta_+(t) & \gamma_+(t)+\alpha_-(t)
  \end{array}$$
\end{defi}

\begin{thm}\label{myhkmthm}
For $t=(t_1,t_2,t_3)\in[0,\infty)^3$ we have 
$$\mu(t)=\frac{d}{4}\cdot Q(t_1,t_2,t_3)+\frac{\lambda^2}{4d}\left(\delta\left(\frac{\alpha(t)}{\lambda},\frac{\beta(t)}{\lambda},\frac{\gamma(t)}{\lambda}\right)\right)^2.$$
\end{thm}

\begin{proof}
Since both sides multiply by $p^2$, when $t$ is replaced by $pt$ and since rationals of the form $\tfrac{n\lambda}{p^i}$ are dense in $[0,\infty)$, it suffices 
to prove the statement for $t=\lambda\cdot(a,b,c)$ with natural numbers $a$, $b$, $c$. Therefore we have to show
$$\mu(\lambda\cdot(a,b,c))=\frac{d\cdot\lambda ^2}{4}\cdot Q(a,b,c)+\frac{\lambda^2}{4d}\left(\delta\left(\frac{\alpha(t)}{\lambda},\frac{\beta(t)}{\lambda},\frac{\gamma(t)}{\lambda}\right)\right)^2.$$
Consider the ideal $J:=(x^{a\lambda},y^{b\lambda},z^{c\lambda})$ of $C$.
By Theorem \ref{watyorank} we have 
$$\HKM(J,C)=\frac{\HKM(JB,B)}{\lambda^2}\qquad\text{and}\qquad \HKM(J,C)=\frac{\HKM(JB',B')}{d}.$$
This yields 
$$\HKM(JB,B)=\frac{\lambda^2}{d}\cdot \HKM(JB',B')=\frac{\lambda^2}{d}\cdot\Dim_k(B'/JB'),$$
since $B'$ is regular (cf. \cite[1.4]{watyo1}).
 
If $F$ is of type I, we have $$JB'=((R^{b_3}S^{b_2})^a,(S^{c_1}T^{c_3})^b,(R^{a_1}T^{a_2})^c).$$
By Proposition \ref{propdelta} (ii) we have
\begin{align*}
\sg((R^{b_3}S^{b_2})^a,(S^{c_1}T^{c_3})^b,(R^{a_1}T^{a_2})^c) & = \sg(R^{ab_3}S^{ab_2-bc_1},T^{bc_3},R^{ca_1}T^{ca_2})\\
& = \sg(S^{ab_2-bc_1},T^{bc_3},R^{ca_1-ab_3}T^{ca_2})\\
& = \sg(S^{ab_2-bc_1},T^{bc_3-ca_2},R^{ca_1-ab_3})\\
& = \delta\left(\frac{\alpha(t)}{\lambda},\frac{\beta(t)}{\lambda},\frac{\gamma(t)}{\lambda}\right).
\end{align*}
Note that we made some assumptions in this computation, for example $ab_2\geq bc_1$ in the first equality. We defined $\alpha$, $\beta$ and $\gamma$ such that all 
cases that might appear in the above computation end up in $\delta\left(\tfrac{\alpha(t)}{\lambda},\tfrac{\beta(t)}{\lambda},\tfrac{\gamma(t)}{\lambda}\right)$.
By Proposition \ref{deltadim} (ii) we have 
\begin{align*}
\Dim_k(B'/JB') &= \frac{1}{4}\cdot\left(Q(ad,bd,cd)+\delta(ab_2-bc_1,bc_3-ca_2,ca_1-ab_3)^2\right)\\
&= \frac{1}{4}\cdot\left(d^2\cdot Q(a,b,c)+\left(\delta\left(\frac{\alpha(t)}{\lambda},\frac{\beta(t)}{\lambda},\frac{\gamma(t)}{\lambda}\right)\right)^2\right).
\end{align*}
Multiplying by $\tfrac{\lambda^2}{d}$ gives the result.

If $F$ is of type II we have 
$$JB'=(R^{ac_3-ab_3}S^{ac_2}T^{ab_3},S^{bd},R^{cb_1}T^{cd-cb_1})\quad\text{and}$$
\begin{align*}
\sg(JB') & = \sg(R^{ac_3-ab_3}S^{ac_2}T^{ab_3},S^{bd},R^{cb_1}T^{cd-cb_1})\\
& = \sg(R^{ac_3-ab_3}T^{ab_3},S^{bd-ac_2},R^{cb_1}T^{cd-cb_1})\\
& = \sg(R^{ac_3-ab_3},S^{bd-ac_2},R^{cb_1}T^{cd-cb_1-ab_3})\\
& = \sg(R^{ac_3-ab_3-cb_1},S^{bd-ac_2},T^{cd-cb_1-ab_3})\\
& = \delta\left(\frac{\alpha(t)}{\lambda},\frac{\beta(t)}{\lambda},\frac{\gamma(t)}{\lambda}\right).
\end{align*}
As in the case where $F$ is of type I we treated only one of the possible cases. Again all other cases end up in the same result by definition of $\alpha$, $\beta$ and $\gamma$.

By Proposition \ref{deltadim} (ii) we have
\begin{align*}
 \Dim_k(B'/JB') &= \frac{1}{4}\cdot\left(Q(ad,bd,cd)+\delta(ac_3-ab_3-cb_1,bd-ac_2,cd-cb_1-ab_3)^2\right)\\
&= \frac{1}{4}\cdot \left(d^2\cdot Q(a,b,c)+\left(\delta\left(\frac{\alpha(t)}{\lambda},\frac{\beta(t)}{\lambda},\frac{\gamma(t)}{\lambda}\right)\right)^2\right).
\end{align*}
Multiplying by $\tfrac{\lambda^2}{d}$ gives the result.
\end{proof}

We obtain \cite[Theorem 2.3]{irred} as a corollary.

\begin{cor}\label{irredmon}
$$\mu(s,s,s)=\frac{3d}{4}s^2+\frac{\lambda^2}{4d}\delta\left(\frac{(a_1+b_2-d)s}{\lambda},\frac{(a_1+c_3-d)s}{\lambda},\frac{(b_2+c_3-d)s}{\lambda}\right)^2.$$
\end{cor}

With $t=(a,b,c)\in\N^3$ in Theorem \ref{myhkmthm} one obtains

\begin{cor}\label{myhkmcor}
$$\HKM\left(\left(X^a,Y^b,Z^c\right),k[X,Y,Z]/(F)\right)=\frac{d}{4}\cdot Q(a,b,c)+\frac{\lambda^2}{4d}\left(\delta\left(\frac{\alpha(a,b,c)}{\lambda},\frac{\beta(a,b,c)}{\lambda},\frac{\gamma(a,b,c)}{\lambda}\right)\right)^2.$$ 
\end{cor}

\begin{rem}
An explicit implementation of Corollary \ref{myhkmcor} is given in Appendix \ref{app.myhkmthm}.
\end{rem}

\begin{cor}\label{deltastrongsemistable}
Let $N$ be the normalization of the projective curve $V_+(F)$. Then the pull-back of $\Syz_{V_+(F)}(X^a,Y^b,Z^c)$ to $N$ is strongly semistable
if and only if $$\delta\left(\frac{\alpha(a,b,c)}{\lambda},\frac{\beta(a,b,c)}{\lambda},\frac{\gamma(a,b,c)}{\lambda}\right)=0.$$
\end{cor}

\begin{proof}
By Theorem \ref{trivedi} the pull-back of $\Syz_{V_+(F)}(X^a,Y^b,Z^c)$ to the normalization $N$ is strongly semistable if and only if 
\begin{align*}
 \HKM\left(\left(X^a,Y^b,Z^c\right),k[X,Y,Z]/(F)\right) &= \frac{d}{2}\cdot\left(\frac{(a+b+c)^2}{2}-a^2-b^2-c^2\right)\\
 &= \frac{d}{4}\cdot Q(a,b,c).
\end{align*}
The claim follows by comparing this Hilbert-Kunz multiplicity with Corollary \ref{myhkmcor}.
\end{proof}

\begin{rem}\label{hkmstgnonstg} Let $R=k[U,V,W]/(f)$ be positively-graded, such that the image $F$ of $f$ under the map $k[U,V,W]\lra k[X,Y,Z],$ 
where $U$, $V$, $W$ are mapped to $X^{\Deg(U)}$, $Y^{\Deg(V)}$, $Z^{\Deg(W)}$ becomes of type I or II.
In this situation we can compute
$$\HKM\left(\left(X^{\Deg(U)},Y^{\Deg(V)},Z^{\Deg(W)}\right),k[X,Y,Z]/(F)\right)$$
with Corollary \ref{myhkmcor}. By Theorem \ref{watyorank} we have to divide this number by the degree of the extension $Q(S):Q(R)$ to get $\HKM(R)$. 
This degree is nothing else but the product $\Deg(U)\cdot\Deg(V)\cdot\Deg(W)$.
\end{rem}

In the next section we will use this remark to discuss the behaviour of the Hilbert-Kunz multiplicity, when one exponent in $f$ tends to infinity and the others are constant. 

\begin{defi}
 With the notations from the previous remark, we call the quasi-homo\-geneous trinomial $f$ \textit{regular} if the homogeneous trinomial $F$ is regular.
\end{defi}

In \cite[Theorem 2.30]{handiss}, Han computed the Hilbert-Kunz multiplicity of diagonal hypersurfaces of dimension two. In view of the last remark, we can 
compute them using Theorem \ref{myhkmthm}.

\begin{cor}\label{hkmdiag}
The Hilbert-Kunz multiplicity of $k[U,V,W]/(U^{d_1}+V^{d_2}+W^{d_3})$ equals
$$\frac{1}{4}\cdot\left(2d_1+2d_2+2d_3-\frac{d_1d_2}{d_3}-\frac{d_1d_3}{d_2}-\frac{d_2d_3}{d_1}+d_1d_2d_3\cdot\delta\left(\frac{1}{d_1},\frac{1}{d_2},\frac{1}{d_3}\right)^2\right).$$
Moreover, if the $\tfrac{1}{d_i}$ do not satisfy the strict triangle inequality, this is $\Min\left\{d_1,d_2,d_3\right\}$.
\end{cor}

\begin{proof}
The second statement is clear from the first. Let $\Deg(U)=d_2d_3$, $\Deg(V)=d_1d_3$ and $\Deg(W)=d_1d_2$. 
Then we get $d=d_1d_2d_3$, $\alpha=d_1d_2^2d_3^2$, $\beta=d_1^2d_2d_3^2$, $\gamma=d_1^2d_2^2d_3$ and $\lambda=d_1^2d_2^2d_3^2$. 
Dividing $$\HKM\left(\left(X^{d_2d_3},Y^{d_1d_3},Z^{d_1d_2}\right),k[X,Y,Z]/\left(X^d+Y^d+Z^d\right)\right)$$ by $\Deg(U)\cdot\Deg(V)\cdot\Deg(W)=d_1^2d_2^2d_3^2$, we get the result.
\end{proof}

\begin{defi}\label{defdegregnonst}
 For any $P$, $Q\in\N\setminus\{0,1\}$, we say that $k[U,V,W]/(U^P+V^Q+UVW)$ has a \textit{singularity of type $T_{PQ\infty}$}.
\end{defi}

\begin{exa}\label{lastexahan}
We compute the Hilbert-Kunz multiplicities of $R:=k[U,V,W]/(f)$ of type ADE and $T_{PQ\infty}$. We will denote the degrees of $U$, $V$, $W$ and $f$ by $a$, $b$, $c$ and $d$.
\begin{itemize}
\item[$A_n$:] Since we need a trinomial, we use $f:=U^{n+1}+V^2+W^2$ in this computation, hence we have to restrict to odd characteristics. The arguments 
of $\delta$ in Corollary \ref{hkmdiag} are $\tfrac{1}{2}$, $\tfrac{1}{2}$ and $\tfrac{1}{n+1}$. It is easy to see $\delta=0$ if $p$ is odd. This gives
$$
\HKM(R) = \frac{1}{4}\left(2n+10-2(n+1)-\frac{4}{n+1}\right) = 2-\frac{1}{n+1}.
$$
\item[$D_n$:] Let $f=U^2+V^{n-1}+VW^2$ with $n\geq 4$ and $a=n-1$, $b=2$, $c=n-2$. Then $d=2n-2$, $F=X^{2n-2}+Y^{2n-2}+Y^2Z^{2n-4}$, $\lambda=4(n-1)(n-2)$, $\alpha=2(n-1)(n-2)$, $\beta=2(n-1)$ and $\gamma=2(n-1)(n-2)$.
The arguments of $\delta$ in Corollary \ref{myhkmcor} are $\tfrac{1}{2},\tfrac{1}{2(n-2)},\tfrac{1}{2},$ hence $\delta=0$ as above (provided $p$ is odd). We get
\begin{align*}
\HKM(R) &= \frac{\frac{2n-2}{4}\left(2\cdot((n-1)(n-2)+2(2n-3))-((n-2)^2+2^2+(n-1)^2)\right)}{2(n-1)(n-2)}\\
 &= 2-\frac{1}{4(n-2)}.
\end{align*}
If $p=2$, we have $\delta\left(\tfrac{1}{2},\tfrac{1}{2(n-2)},\tfrac{1}{2}\right)=\tfrac{1}{2\cdot (n-2)}$ and $\HKM(R)=2$.
\item[$E_6$:] Let $f=U^2+V^3+W^4$. By Corollary \ref{hkmdiag} and Example \ref{deltahalb} we get
$$
\HKM(R) = \frac{1}{4}\left(18-\frac{3}{2}-\frac{8}{3}+24\cdot\delta\left(\frac{1}{2},\frac{1}{3},\frac{1}{4}\right)^2\right) = 2-\frac{1}{24},
$$
if $p\geq 5$. If $p\leq 3$, we have $\delta=\tfrac{1}{12}$ and $\HKM(R)=2$.
\item[$E_7$:] Let $f=U^2+V^3+VW^3$, $a=9$, $b=6$, $c=4$ and $d=18$. Then $F=X^{18}+Y^{18}+Y^6Z^{12}$, $\alpha=6\cdot 18$, $\beta=3\cdot 18$, $\gamma=4\cdot 18$ and $\lambda=12\cdot 18$.
The arguments of $\delta$ are $\tfrac{1}{2},\tfrac{1}{3},\tfrac{1}{4}$. By Example \ref{deltahalb} we get $\delta=0$ if $p\geq 5$ and
$$
\HKM(R) = \frac{\frac{18}{4}\left(2\cdot(9\cdot 6+9\cdot 4+6\cdot 4)-(9^2+6^2+4^2)\right)}{9\cdot 6\cdot 4} = 2-\frac{1}{48}.
$$
If $p\leq 3$, we have $\delta=\tfrac{1}{12}$ and $\HKM(R)=2$.
\item[$E_8$:] Let $f=U^2+V^3+W^5$. By Corollary \ref{hkmdiag} and Example \ref{deltahalb} for $p\geq 7$ we get
$$
\HKM(R) = \frac{1}{4}\left(20-\frac{6}{5}-\frac{10}{3}-\frac{15}{2}+30\cdot\delta\left(\frac{1}{2},\frac{1}{3},\frac{1}{5}\right)^2\right) = 2-\frac{1}{120}.
$$
If $p\leq 5$, one gets $\delta=\tfrac{1}{30}$ and $\HKM(R)=2$.
\item[$T_{PQ\infty}$:] Let $f=U^P+V^Q+UVW$, $a=Q$, $b=P$, $c=PQ-P-Q$ and $d=PQ$. To get $c\geq 1,$ we assume $P+Q>4$. Then $F=X^{PQ}+Y^{PQ}+X^QY^PZ^{PQ-P-Q}$, $\alpha=(Q-2)PQ$, $\beta=(P-2)PQ$, $\gamma=\lambda=(PQ-P-Q)PQ$. 
Because $$(P-2)(Q-2)\geq 0 \Longleftrightarrow \frac{P-2}{PQ-P-Q}+\frac{Q-2}{PQ-P-Q}\leq 1,$$
we see that the arguments of $\delta$ do not satisfy the strict triangle inequality and hence 
$$
\HKM(R) = \frac{3P^2Q^2-4P^2Q-4PQ^2+4PQ}{PQ(PQ-P-Q)} = 3-\frac{P+Q-4}{PQ-P-Q}
$$
in all characteristics. Since $4<P+Q$ and $\tfrac{P+Q-4}{PQ-P-Q}\leq 1$, this number is in the interval $[2,3)\cap\Q$.
\end{itemize}

The computations in the ADE situation recover Theorem \ref{watyogroup} with one exception: To compute the Hilbert-Kunz multiplicity of a surface ring of type 
$A_n$ we had to restrict to odd characteristics, which is not necessary in Theorem \ref{watyogroup} if $n+1$ is odd.
\end{exa}

We always assumed that $F$ is regular, meaning that all singularities of $V_+(F)$ have multiplicity strictly smaller than $\tfrac{d}{2}$. The Hilbert-Kunz 
multiplicity of an irregular trinomial was computed by Trivedi in \cite[Theorem 1.3]{tri2}.

\begin{thm}[Trivedi]
Let $C$ be a singular irreducible plane curve of degree $d$ defined over an algebraically closed field of characteristic $p>0$. Assume that $C$ has a singular 
point of multiplicity $r\geq \tfrac{d}{2}$. Let $R$ be the homogeneous coordinate ring of $C$. Then
$$\HKM(R)=\frac{3d}{4}+\frac{(2r-d)^2}{4d}.$$
\end{thm}

\section[The behaviour of the Hilbert-Kunz multiplicity in families]{The behaviour of the Hilbert-Kunz multiplicity in families \except{toc}{of surface rings given by certain quasi-homogeneous trinomials}}\label{quasitrinomials}
In this section we want to apply Theorem \ref{myhkmthm} to study the behaviour of the Hilbert-Kunz multiplicity in families of surface rings defined by irreducible, 
regular, quasi-homogeneous trinomials (IRQT for short).
To be more precise we study the Hilbert-Kunz multiplicity of rings $R:=k[U,V,W]/(F)$, where $F$ is a trinomial of the type
$$U^{a_1}V^{a_2}+V^{b_2}W^{b_3}+U^{c_1}W^{c_3}\qquad\text{or}\qquad U^{a_1}+U^{b_1}V^{b_2}W^{b_3}+V^{c_2}W^{c_3},$$
homogeneous of degree $d$ with respect to some positive grading $\Deg(U)=a$, $\Deg(V)=b$, $\Deg(W)=c$ and regular in the sense of Definition \ref{defdegregnonst}.
Now we are interested in the behaviour of the Hilbert-Kunz multiplicity of $R$, when $a_1$, $b_2$ or $c_3$ grows.

\begin{defi}
With the notations above let

\begin{tabular}{ll}
$F_1(a_2,b_1,b_3,c_2,c_3)$ & be the family of IRQT of the form $U^LV^{a_2}+V^{b_2}W^{b_3}+U^{c_1}W^{c_3}$,\\
$F_U(b_1,b_2,b_3,c_2,c_3)$ & be the family of IRQT of the form $U^L+U^{b_1}V^{b_2}W^{b_3}+V^{c_2}W^{c_3}$,\\
$F_V(a_1,b_1,b_3,c_2,c_3)$ & be the family of IRQT of the form $U^{a_1}+U^{b_1}V^LW^{b_3}+V^{c_2}W^{c_3}$,\\
$F_W(a_1,b_1,b_2,b_3,c_2)$ & be the family of IRQT of the form $U^{a_1}+U^{b_1}V^{b_2}W^{b_3}+V^{c_2}W^L$.
\end{tabular}

We will write $F_1$, $F_U$, $F_V$, $F_W$ for short.
\end{defi}

In what follows, we will always assume that the arguments of $\delta$ in Corollary \ref{myhkmcor} do not satisfy the strict triangle inequality for $L$ large enough.

\subsection{The family $F_1$}
The trinomials $$U^LV^{a_2}+V^{b_2}W^{b_3}+U^{c_1}W^{c_3}$$ are quasi-homogeneous in the grading
\begin{align*}
a & = b_2c_3-a_2c_3+a_2b_3\\
b & = (c_3-b_3)L+b_3c_1\\
c & = b_2L-b_2c_1+a_2c_1
\end{align*}
of degree $$d=b_2c_3L+a_2b_3c_1.$$

Note that $b_2>a_2$ and $c_3>b_3$ give $a$, $b>0$. Assume $L>c_1$ to get $c>0$. The regularity conditions $aL$, $bb_2$, $cc_3>d/2$ translate for large $L$ to $b_2c_3>2a_2(c_3-b_3)$ and $c_3>2b_3$.

We have (recall Definition \ref{defmyhkmthm})
$$\begin{aligned}
\alpha' & = ad-acb_3-abc_1 && = ab(b_2-c_1),\\
\beta' & = bd-bca_2-abc_1 && = bc(c_3-a_2),\\
\gamma' & = cd-acb_3-bca_2 && = ac(L-b_3).
\end{aligned}$$
Since $\gamma'$ is a degree two polynomial in $L$ with leading coefficient $(b_2-a_2)b_2c_3+a_2b_2b_3>0$ and $\alpha'$ depends linearly on $L$, we get 
$\gamma'>0$ and $\alpha'<\gamma'$ for large $L$. 
\begin{center}
\begin{table}[h]
$$\begin{array}{c|c||c|c|c||c}
\alpha' & \beta' & \alpha & \beta & \gamma & \delta\cdot\lambda \\ \hline
+ & + & \alpha' & \beta' & \gamma' & \gamma'-\alpha'-\beta'\text{ or }\beta'-\alpha'-\gamma' \\
+ & - & \alpha' & 0 & \gamma'-\beta' & \gamma'-\alpha'-\beta'\text{ or }\beta'-\alpha'-\gamma'\\
- & + & 0 & \beta'-\alpha' & \gamma' & \gamma'+\alpha'-\beta'\text{ or }\beta'-\alpha'-\gamma'\\
- & - & 0 & -\alpha' & \gamma'-\beta' & \gamma'+\alpha'-\beta'\text{ or }\beta'-\alpha'-\gamma'.
\end{array}$$
\caption{The table shows the possible values of $\delta\cdot\lambda$ depending the signs of $\alpha'$, $\beta'$.}
\label{tabdellamf1}
\end{table}
\end{center}
According to Table \ref{tabdellamf1} the value of $\delta\cdot \lambda$ in Theorem \ref{myhkmthm} has the shape $\gamma'-\beta'-\alpha'$ 
or $\pm(\beta'-\gamma'-\alpha')$. In the first case we have $\delta\cdot\lambda=cd-bd-ad+2abc_1$, which gives the Hilbert-Kunz multiplicity
\begin{align*}
\HKM(R) &= \frac{d\cdot(2ab+2ac+2bc-a^2-b^2-c^2)}{4abc}+\frac{(cd-bd-ad+2abc_1)^2}{4abcd}\\
 &= \frac{d}{c}+\frac{abc_1^2}{cd}+\frac{c_1}{c}\cdot(c-b-a)\\
 &= c_1+\frac{d-ac_1}{c}+\frac{abc_1^2}{cd}-\frac{bc_1}{c}\\
 &= c_1+c_3+c_1\frac{b_3-c_3}{b_2}-\frac{b_3c_1^2}{b_2}\frac{a}{d}.
\end{align*}
In this case $L$ occurs only in $d$, hence the Hilbert-Kunz multiplicity is increasing in $L$ and tends to $c_1+c_3+c_1\frac{b_3-c_3}{b_2}$.

Similarly, in the second case we have $\delta\cdot\lambda=\pm(bd-cd-ad+2acb_3)$ and
$$\HKM(R)=\frac{d+b_3(b-c)}{b}+\frac{acb_3^2-adb_3}{bd}=b_2+b_3-b_2b_3\frac{a}{d}.$$
Here the Hilbert-Kunz multiplicity is increasing in $L$ and tends to $b_2+b_3$.

By the symmetry of the problem we get in the case $b_2=L$:
$$\begin{aligned}
\HKM(R) & = c_1+c_3-c_1c_3\frac{b}{d} && \text{in the cases } \lambda\cdot\delta=\pm(\gamma'-\beta'-\alpha')\\
\HKM(R) & = a_1+a_2+a_2\frac{c_1-a_1}{c_3}-\frac{a_2^2c_1}{c_3}\frac{b}{d} && \text{in the case } \lambda\cdot\delta=\alpha'-\beta'-\gamma'
\end{aligned}$$
and in the case $c_3=L$:
$$\begin{aligned}
\HKM(R) & = a_1+a_2-a_1a_2\frac{c}{d} && \text{in the cases } \lambda\cdot\delta=\pm(\alpha'-\beta'-\gamma')\\
\HKM(R) & = b_2+b_3+b_3\frac{a_2-b_2}{a_1}-\frac{a_2b_3^2}{a_1}\frac{c}{d} && \text{in the case } \lambda\cdot\delta=\beta'-\alpha'-\gamma'.
\end{aligned}$$

\begin{exa}
Consider the family $(U^L+V^{d_2}+W^{d_3})_{L\gg 0}$ of polynomials defining diagonal hypersurfaces with $d_2\neq d_3$. The above computations show that their Hilbert-Kunz multiplicity tends to $\Min(d_2,d_3)$ for $L\ra\infty$.
This is exactly $\HKM\left(k[V,W]/\left(V^{d_2}+W^{d_3}\right)\right)$. We cannot apply the above computation in the case $d_2=d_3$, since then the triple 
$\left(\tfrac{1}{L},\tfrac{1}{d_2},\tfrac{1}{d_2}\right)$ always satisfies the strict triangle inequality.
\end{exa}

Next, we prove that the limits for $L\ra\infty$ above appear as the Hilbert-Kunz multiplicity of a binomial hypersurface attached to the family $F_1$.

\begin{defi} We set
$\lim_{L\ra\infty}F_1:=V^{b_2}W^{b_3}+U^{c_1}W^{c_3}=W^{b_3}(V^{b_2}+U^{c_1}W^{c_3-b_3}).$
\end{defi}

\begin{rem}
By Theorem \ref{thmhkmexists} (ii) and Remark \ref{kunzbsp} we have
$$\HKM\left(k[U,V,W]/(W^{b_3}\cdot(V^{b_2}+U^{c_1}W^{c_3-b_3})\right)=\begin{cases}c_1+c_3+c_1\frac{b_3-c_3}{b_2} & \text{if }c_1,c_3-b_3<b_2,\\
                                                                              b_2+b_3 & \text{otherwise.}
\end{cases}$$
\end{rem}

\begin{lem}
We obtain the following limit behaviour of the Hilbert-Kunz multiplicity in the family $F_1$
$$\lim_{L\ra\infty}\HKM\left(k[U,V,W]/(U^LV^{a_2}+V^{b_2}W^{b_3}+U^{c_1}W^{c_3})\right)=\HKM\left(\lim_{L\ra\infty}F_1\right).$$
\end{lem}

\begin{proof}
Since the two possible values for $\HKM\left(\lim_{L\ra\infty}F_1\right)$ coincide in the cases $c_1=b_2$ and $c_3-b_3=b_2$, we may assume that both 
equations do not hold. Then we only have to show the equivalence
$$\delta\cdot\lambda=\gamma'-\alpha'-\beta'\Leftrightarrow c_1,c_3-b_3<b_2.$$
By Table \ref{tabdellamf1} the condition $\delta\cdot\lambda=\gamma'-\alpha'-\beta'$ is equivalent to the conditions that $\alpha'$ and the leading 
coefficient of $\gamma'-\beta'$ are positive. Now, the condition $\alpha'>0$ gives $b_2>c_1$ and $\LC(\gamma'-\beta')=b_2c_3(b_2+b_3-c_3)>0$ 
gives $b_2>c_3-b_3$.
\end{proof}

\subsection{The family $F_U$}
The trinomials $$U^L+U^{b_1}V^{b_2}W^{b_3}+V^{c_2}W^{c_3}$$ are quasi-homogeneous in the grading
\begin{align*}
a & = b_2c_3-b_3c_2\\
b & = (c_3-b_3)L-c_3b_1\\
c & = (b_2-c_2)L+b_1c_2
\end{align*}
of degree $$d=aL.$$

Since $c_3>b_3$ and $b_2>c_2$ the degrees are positive. In this case the regularity conditions $aL,bb_2,cc_3>d/2$ translate to the condition 
$$b_2c_3+b_3c_2>\Max(2b_2b_3,2c_2c_3)$$ 
for large values of $L$. We have
\begin{align*}
\alpha' & = ad-ac(b_1+b_3)-abc_2,\\
\beta' & = bd-abc_2,\\
\gamma' & = cd-ac(b_1+b_3).
\end{align*}
In this case $\beta'$ and $\gamma'$ are polynomials in $L$ of degree two with positive leading coefficients and $\alpha'$ depends linearly on $L$. Therefore we may assume $L$ to be big enough such that $\beta',$ $\gamma'\geq 0$ and $\alpha'<\beta',$ $\gamma'$.
\begin{center}
\begin{table}[h]
$$\begin{array}{c||c|c|c||c}
\alpha' & \alpha & \beta & \gamma & \delta\cdot\lambda \\ \hline
+ & \alpha' & \beta' & \gamma' & \gamma'-\alpha'-\beta'\text{ or }\beta'-\alpha'-\gamma' \\
- & 0 & \beta' & \gamma'-\alpha' & \gamma'-\alpha'-\beta'\text{ or }-\gamma'+\alpha'+\beta'.
\end{array}$$
\caption{The table shows the possible values of $\delta\cdot\lambda$ depending the sign of $\alpha'$.}
\label{tabdellamfu}
\end{table}
\end{center}
Thus, $\delta\cdot \lambda$ has the shape $\pm(\gamma'-\beta'-\alpha')$ or $\beta'-\gamma'-\alpha'$. In the first case we have 
$\delta\cdot\lambda=\pm(cd-bd-ad+2abc_2)$, which gives the Hilbert-Kunz multiplicity
$$\HKM(R)=\frac{d+c_2(c-b)}{c}+\frac{abc_2^2-adc_2}{cd}=c_2+c_3-c_2c_3\frac{a}{d}.$$
This is an increasing function in $L$, tending to $c_2+c_3$.

Similarly, in the second case we have $\delta\cdot\lambda=bd-cd-ad+2ac(b_1+b_3)$, which gives the Hilbert-Kunz multiplicity
\begin{align*}
\HKM(R) &= \frac{d+(b_1+b_3)(b-c)}{b}+\frac{ac(b_1+b_3)^2-ad(b_1+b_3)}{bd}\\
 &= b_1+b_2+b_3+b_1\frac{c_2-b_2}{c_3-b_3}+\frac{b_1^2c_3}{b}\frac{c_2-b_2}{c_3-b_3} +\frac{(b_1+b_3)^2}{b}(b_2-c_2+\frac{b_1c_2}{L})-\frac{ab_3+b_1^2c_2}{b}.
\end{align*}
This function tends to $b_1+b_2+b_3+b_1\frac{c_2-b_2}{c_3-b_3}$.

\begin{exa}
 Consider the family $(U^L+UVW+W^Q)_{L\gg 0}$ of $T_{LQ\infty}$ singularities. Since $\alpha=(Q-2)LQ$, $\beta=(Q^2-Q)L^2-Q^2L$ and $\gamma=(L-2)LQ$, we are in the second case above and see that the Hilbert-Kunz multiplicity
is given by
$$3-\frac{1}{Q-1}+\left(3-Q-\frac{1}{Q-1}\right)\frac{1}{LQ-L-Q}.$$
In particular, the Hilbert-Kunz multiplicity is constant for $Q=2$ and strictly increasing in the cases $Q\geq 3$.
\end{exa}

One question concerning the behaviour of the Hilbert-Kunz multiplicity as a function in $L$ is left.

\begin{que}
Assume that we are in the case $\delta\cdot\lambda=bd-cd-ad+2ac(b_1+b_3)$. Are there Hilbert-Kunz functions that are (strictly) decreasing as functions in $L$?
\end{que}

As in the case of the family $F_1$, the limit of the Hilbert-Kunz multiplicity for $L\ra\infty$ appears as Hilbert-Kunz multiplicity of an attached binomial 
hypersurface.

\begin{defi}
Let $\lim_{L\ra\infty}F_U:=U^{b_1}V^{b_2}W^{b_3}+V^{c_2}W^{c_3}=V^{c_2}W^{b_3}(U^{b_1}V^{b_2-c_2}+W^{c_3-b_3}).$
\end{defi}

\begin{rem}
By Theorem \ref{thmhkmexists} (ii) and Remark \ref{kunzbsp} we have
$$\begin{aligned}
 & \HKM\left(k[U,V,W]/(V^{c_2}W^{b_3}\cdot(U^{b_1}V^{b_2-c_2}+W^{c_3-b_3}))\right)\\
= & \begin{cases}
b_1+b_2+b_3+b_1\frac{c_2-b_2}{c_3-b_3} & \text{if }b_1,b_2-c_2<c_3-b_3,\\ 
c_2+c_3 & \text{otherwise.}
\end{cases}
\end{aligned}$$
\end{rem}

\begin{lem}
The limit behaviour of the Hilbert-Kunz multiplicity in the family $F_U$ is given by
$$\lim_{L\ra\infty}\HKM\left(k[U,V,W]/(U^L+U^{b_1}V^{b_2}W^{b_3}+V^{c_2}W^{c_3})\right)=\HKM\left(\lim_{L\ra\infty}F_U\right).$$
\end{lem}

\begin{proof}
Since the two possible values of $\HKM\left(\lim_{L\ra\infty}F_U\right)$ coincide in the cases $b_1=c_3-b_3$ and $b_2-c_2=c_3-b_3$, we may assume that both 
equalities do not hold. According to Table \ref{tabdellamfu} we have
$$\delta\cdot\lambda=\beta'-\alpha'-\gamma'\Leftrightarrow \alpha'>0\text{ and }\Max\{\beta',\gamma'\}=\beta'.$$
The condition $\alpha'>0$ is equivalent to $b_1<c_3-b_3$ and the condition $\Max\{\beta',\gamma'\}=\beta'$ is equivalent to the positivity of the 
leading coefficient of 
$$\beta'-\gamma'=a\cdot (c_2+c_3-b_2-b_3)L^2+O(L).$$
\end{proof}

\subsection{The family $F_V$}
The trinomials $$U^{a_1}+U^{b_1}V^{L}W^{b_3}+V^{c_2}W^{c_3}$$ are quasi-homogeneous in the grading
\begin{align*}
a & = c_3L-b_3c_2\\
b & = a_1c_3-a_1b_3-b_1c_3\\
c & = a_1L-a_1c_2+b_1c_2
\end{align*}
of degree $$d=aa_1.$$

For $L>c_2$ we get positive values for $a$ and $c$. The condition $b>0$ is non-trivial. For large $L$ we get the condition 
$$(*)\qquad a_1c_3>2(a_1b_3+b_1c_3)$$ 
to get a regular polynomial. We have
\begin{align*}
\alpha' & = ad-ac(b_1+b_3)-abc_2,\\
\beta' & = bd-abc_2,\\
\gamma' & = cd-ac(b_1+b_3).
\end{align*}
In this case $\beta'$ depends only linearly on $L$ and $\alpha'$, $\gamma'$ are $O(L^2)$. 
\begin{center}
\begin{table}[h]
$$\begin{array}{c|c|c||c|c|c||c}
\alpha' & \beta' & \gamma' & \alpha & \beta & \gamma & \delta\cdot\lambda \\ \hline
+ & + & + & \alpha' & \beta' & \gamma' & \alpha'-\beta'-\gamma'\text{ or }\gamma'-\alpha'-\beta' \\
+ & + & - & \alpha'-\gamma' & \beta' & 0 & \alpha'-\beta'-\gamma'\text{ or }\beta'-\alpha'+\gamma' \\
+ & - & + & \alpha'-\beta' & 0 & \gamma' & \alpha'-\beta'-\gamma'\text{ or }\gamma'-\alpha'+\beta' \\
+ & - & - & \alpha'-\beta'-\gamma' & 0 & 0 & \alpha'-\beta'-\gamma' \\
- & + & + & 0 & \beta' & \gamma'-\alpha' & \alpha'+\beta'-\gamma'\text{ or }\gamma'-\alpha'-\beta' \\
- & + & - & -\gamma' & \beta' & -\alpha' & \alpha'-\beta'-\gamma'\text{ or }\gamma'-\alpha'-\beta' \\
- & - & + & -\beta' & 0 & \gamma'-\alpha' & \alpha'-\beta'-\gamma'\text{ or }\gamma'-\alpha'+\beta' \\
- & - & - & -\beta'-\gamma' & 0 & -\alpha' & \alpha'-\beta'-\gamma'\text{ or }\gamma'-\alpha'+\beta'.
\end{array}$$
\caption{The table shows the possible values of $\delta\cdot\lambda$ depending the signs of $\alpha'$, $\beta'$, $\gamma'$.}
\label{tabdellamfv}
\end{table}
\end{center}
As we can read off Table \ref{tabdellamfv} the product $\delta\cdot \lambda$ has always the shape 
$\pm(\gamma'-\beta'-\alpha')$ or $\pm(\alpha'-\gamma'-\beta')$. In the first case we have $\delta\cdot\lambda=\pm(cd-bd-ad+2abc_2)$. This gives the 
constant Hilbert-Kunz multiplicity
$$\HKM(R)=\frac{d+c_2(c-a)}{c}+\frac{abc_2^2-bdc_2}{cd}=c_2+c_3-\frac{c_2c_3}{a_1}.$$

Similarly, in the second case we have $\delta\cdot\lambda=\pm(ad-cd-bd)$, which gives the constant Hilbert-Kunz multiplicity
$$\HKM(R)=\frac{d}{a}=a_1.$$

\begin{defi}
Let $\lim_{L\ra\infty}F_V:=U^{a_1}+V^{c_2}W^{c_3}.$
\end{defi}

\begin{rem}
By Theorem \ref{thmhkmexists} (ii) and Remark \ref{kunzbsp} we have
$$\HKM\left(k[U,V,W]/(U^{a_1}+V^{c_2}W^{c_3})\right)=
\begin{cases}
c_2+c_3-\frac{c_2c_3}{a_1} & \text{if }c_2,c_3<a_1,\\
a_1 & \text{otherwise.}
\end{cases}$$
\end{rem}

\begin{lem}
The limit behaviour of the Hilbert-Kunz multiplicity in the family $F_V$ is given by
$$\lim_{L\ra\infty}\HKM\left(k[U,V,W]/(U^{a_1}+U^{b_1}V^{L}W^{b_3}+V^{c_2}W^{c_3})\right)=\HKM\left(\lim_{L\ra\infty}F_V\right).$$
\end{lem}

\begin{proof}
We may assume $c_2,c_3\neq a_1$. Then Table \ref{tabdellamfv} shows that $\delta\cdot\lambda=\pm(\gamma'-\alpha'-\beta')$ is equivalent to $\beta'>0$ and 
$\LC(\gamma'-\alpha')>0$. The condition $\beta'>0$ gives $a_1>c_2$ and $\LC(\gamma'-\alpha')>0$ gives $a_1>c_3.$
\end{proof}

\subsection{The family $F_W$}
The trinomials $$U^{a_1}+U^{b_1}V^{b_2}W^{b_3}+V^{c_2}W^L$$ are quasi-homogeneous in the grading
\begin{align*}
a & = b_2L-b_3c_2\\
b & = (a_1-b_1)L-a_1b_3\\
c & = (b_2-c_2)a_1+b_1c_2
\end{align*}
of degree $$d=aa_1.$$

The inequalities $a_1>b_1$ and $b_2>c_2$ give $b$, $c>0$. Assume $L>b_3$ to get $a$ positive. The regularity conditions give for large $L$ the 
following conditions on the exponents
$$a_1>2b_1 \qquad \text{and} \qquad a_1b_2>2c_2(a_1-b_1).$$
We have
\begin{align*}
\alpha' & = ad-ac(b_1+b_3)-abc_2,\\
\beta' & = bd-abc_2,\\
\gamma' & = cd-ac(b_1+b_3).
\end{align*}
At first we note that the equivalence $$\alpha'<0\Longleftrightarrow c<0$$
gives $\alpha'\geq 0$.
In this case $\gamma'$ depends linearly on $L$ and $\alpha'$, $\beta'$ are polynomials in $L$ of degree two.
\begin{center}
\begin{table}[h]
$$\begin{array}{c|c||c|c|c||c}
\beta' & \gamma' & \alpha & \beta & \gamma & \delta\cdot\lambda \\ \hline
+ & + & \alpha' & \beta' & \gamma' & \alpha'-\beta'-\gamma'\text{ or }\beta'-\alpha'-\gamma' \\
+ & - & \alpha'-\gamma' & \beta' & 0 & \alpha'-\beta'-\gamma'\text{ or }\beta'-\alpha'+\gamma'\\
- & + & \alpha'-\beta' & 0 & \gamma' & \alpha'-\beta'-\gamma'\text{ or }\gamma'-\alpha'+\beta'\\
- & - & \alpha'-\beta'-\gamma' & 0 & 0 & \alpha'-\beta'-\gamma'.
\end{array}$$
\caption{The table shows the possible values of $\delta\cdot\lambda$ depending the signs of $\beta'$, $\gamma'$.}
\label{tabdellamfw}
\end{table}
\end{center}
Therefore $\delta\cdot \lambda$ has the shape $\beta'-\alpha'-\gamma'$ or $\pm(\alpha'-\beta'-\gamma')$. In the first case we have 
$\delta\cdot\lambda=bd-cd-ad+2ac(b_1+b_3)$, giving
\begin{align*}
\HKM(R) &= \frac{d+(b_1+b_3)(b-a)}{b}+\frac{ac(b_1+b_3)^2-cd(b_1+b_3)}{bd}\\
 &= b_1+b_2+b_3-\frac{b_2b_3}{a_1-b_1}+\frac{1}{b}(a_1-b_1-b_3)\left(\frac{a_1b_2b_3}{a_1-b_1}-b_3c_2-\frac{b_1c+b_3c}{a_1}\right)\\
 &= b_1+b_2+b_3-\frac{b_2b_3}{a_1-b_1}+\frac{1}{b}(a_1-b_1-b_3)^2\left(\frac{b_1c_2}{a_1}-\frac{b_1b_2}{a_1-b_1}\right)
\end{align*}
as Hilbert-Kunz multiplicity. It tends to $b_1+b_2+b_3-\frac{b_2b_3}{a_1-b_1}$. 
Treating $\HKM(R)$ as a function in $L$, it is either constant or strictly increasing. It cannot be strictly decreasing, since then the term $\tfrac{b_1c_2}{a_1}-\tfrac{b_1b_2}{a_1-b_1}$ 
would have to be negative, which is equivalent to $-b_1c_2\geq a_1\cdot (b_2-c_2)>0$. In the second case we have $\delta\cdot\lambda=\pm(ad-cd-bd)$ 
and get the constant Hilbert-Kunz multiplicity
$$\HKM(R)=\frac{d}{a}=a_1.$$

We give two examples showing that both possibilities, $\HKM(R)$ constant or strictly inreasing in $L$, appear in the case $\delta\cdot \lambda=\beta'-\alpha'-\gamma'$.

\begin{exa}
With $F:=U^5+UV^3+W^L$ we have $a=3L$, $b=4L$, $c=15$ and $d=15L$. Since $\alpha'=45L^2-45L$, $\beta'=60L^2$ and $\gamma'=180L$, we have $\beta'-\alpha'-\gamma'\geq 0$ for $L\geq 15$, hence we are in the 
first case of the above computation and the Hilbert-Kunz multiplicity of $R$ is given by 
$$\HKM(R)=4+\frac{1}{4L}\cdot 4^2\cdot \left(-\frac{3}{4}\right)=4-\frac{3}{L},$$
which is strictly increasing in $L$.
\end{exa}

\begin{exa}
Now consider $F:=U^5+UV^3W^4+W^L$. In this case we have $a=3L$, $b=4L-20$, $c=15$ and $d=15L$. 
This gives $\alpha'=45L^2-225L$, $\beta'=60L^2-300L$ and $\gamma'=0$. Therefore $\beta'-\alpha'-\gamma'\geq 0$ for $L\geq 5$ and we are again in the 
first case of the above computation. Now, the Hilbert-Kunz multiplicity of $R$ is given by 
$$\HKM(R)=8-\frac{12}{4}=5.$$
\end{exa}

\begin{defi}
Let $\lim_{L\ra\infty}F_W:=U^{a_1}+U^{b_1}V^{b_2}W^{b_3}=U^{b_1}(U^{a_1-b_1}+V^{b_2}W^{b_3}).$
\end{defi}

\begin{rem}
By Theorem \ref{thmhkmexists} (ii) and Remark \ref{kunzbsp} we have
$$\HKM\left(k[U,V,W]/(U^{b_1}\cdot(U^{a_1-b_1}+V^{b_2}W^{b_3}))\right)=
\begin{cases}
b_1+b_2+b_3-\frac{b_2b_3}{a_1-b_1} & \text{if }b_2,b_3<a_1-b_1,\\
a_1 & \text{otherwise.}
\end{cases}$$
\end{rem}

\begin{lem}
The limit behaviour of the Hilbert-Kunz multiplicity in the family $F_W$ is given by
$$\lim_{L\ra\infty}\HKM\left(k[U,V,W]/(U^{a_1}+U^{b_1}V^{b_2}W^{b_3}+V^{c_2}W^L)\right)=\HKM\left(\lim_{L\ra\infty}F_W\right).$$
\end{lem}

\begin{proof}
We may assume $b_2,b_3\neq a_1-b_1$. According to Table \ref{tabdellamfw} we have
$$\delta\cdot\lambda=\beta'-\alpha'-\gamma'\Leftrightarrow \beta',\gamma'>0\text{ and }\Max\{\alpha',\beta'\}=\beta'.$$
The condition $\gamma'>0$ is equivalent to $b_3<a_1-b_1$ and the condition $\Max\{\alpha',\beta'\}=\beta'$ is equivalent to the positivity of the 
leading coefficient of 
$$\beta'-\alpha'=a_1\cdot (b_1+b_2-a_1)L^2+O(L).$$
\end{proof}

The results in this section lead to the following questions.

\begin{que}
What do the limits of the Hilbert-Kunz multiplicities tell us about the families of singularities?
\end{que}

\begin{que}
Fix $F\in k[X_1,\ldots,X_n]$ and $a_1,\ldots,a_{n-1}\in\N$. Under which conditions does the equality 
$$\lim_{L\ra\infty}\HKM\left(k[X_1,\ldots,X_n]/\left(F+X_1^{a_1}\cdot\ldots\cdot X_{n-1}^{a_{n-1}}X_n^L\right)\right)=\HKM(k[X_1,\ldots,X_n]/(F))$$
hold? Does it even hold for the Hilbert-Kunz functions?
\end{que}

Examples for which the limit behaviour holds for the Hilbert-Kunz functions are the diagonal hypersurfaces $R_{l,m}:=k[X,Y,Z]/(X^2+Y^l+Z^m)$ with $\tfrac{1}{2}+\tfrac{1}{l}+\tfrac{1}{m}\leq 1$. For those we have the equality 
$$\lim_{m\ra\infty}\HKF(R_{l,m})=\HKF(k[X,Y,Z]/(X^2+Y^l)).$$
Further examples for which the limit behaviour from the question holds for the Hilbert-Kunz functions are given in Sections two and three of Chapter seven.

\chapter[Matrix factorizations and representations as first syzygy modules]{Matrix factorizations and representations as first syzygy modules \except{toc}{of ideals}}\label{chapmatfac}
In this chapter we want to introduce the theory of matrix factorizations, their connection to maximal Cohen-Macaulay modules over hypersurfaces and explain how one 
gets with their help isomorphisms of maximal Cohen-Macaulay modules with "nice" first syzygy modules of ideals.

\section{Properties of maximal Cohen-Macaulay modules}
In this section we state some important properties of maximal Cohen-Macaulay modules.

\begin{defi}
Let $R$ be a ring and $M$ a finitely generated $R$-module. We call $M$ a \textit{maximal Cohen-Macaulay} $R$-module, if $M=0$ or
$$\Depth_{R_{\mm}}(M_{\mm})=\Dim(M_{\mm})=\Dim(R_{\mm})$$
holds for all maximal ideals $\mm$ of $R$. We say that $R$ is \textit{Cohen-Macaulay} if $R$ is maximal Cohen-Macaulay as an $R$-module.
\end{defi}

\begin{rem}
An equivalent definition of the maximal Cohen-Macaulay property using Grothendieck's vanishing theorem (cf. \cite[Theorem 3.5.7]{brunsherzog}) is that 
$$\lK_{\mm R_{\mm}}^i(M_{\mm})=0$$
holds for all $0\leq i\leq \Dim(R_{\mm})-1$ and all maximal ideals $\mm$ of $R$.
\end{rem}

\begin{exa} We give some examples.

\begin{enumerate}
 \item If $\Dim(R)=0$, we have $0\leq\Depth_{R_{\mm}}(M_{\mm})\leq \Dim(R_{\mm})=0$ for all $\mm\in\Spec(R)$, giving that every finitely generated $R$-module is maximal Cohen-Macaulay.
 \item Regular rings are Cohen-Macaulay and every maximal Cohen-Macaulay $R$-module $M$ is free, since $M_{\mm}$ has finite projective dimension and this 
is $$\projdim_{R_{\mm}}(M_{\mm})=\Depth_{R_{\mm}}(R_{\mm})-\Depth_{R_{\mm}}(M_{\mm})=0$$ by the Auslander-Buchsbaum formula.
 \item If $(R,\mm)$ has dimension one and at least one non-zero divisor, then a finitely generated $R$-module $M$ is maximal Cohen-Macaulay if and only if it is torsion-free. That is because 
 $\mm$ and all its powers contain a non-zero divisor, giving the inclusion $\lK_{\mm}^0(M)\subseteq \TT(M)$, where 
 $\TT(M)$ denotes the torsion-submodule of $M$. For the other inclusion, note that $\TT(M)$ is Artinian because $R$ is one-dimensional, hence
 $\TT(M)=\lK^0_{\mm}(\TT(M))\subseteq \lK^0_{\mm}(M).$
\end{enumerate}
\end{exa}

\begin{defi}
Let $(R,\mm)$ be a Cohen-Macaulay local ring. We say that $R$ has an \textit{isolated singularity} if $R_{\pp}$ is regular for 
all $\pp\in\Spec(R)\setminus\{\mm\}$.
\end{defi}

We can characterize the property of having an isolated singularity by the structure of the maximal Cohen-Macaulay modules.

\begin{lem}\label{isolated}
Let $(R,\mm)$ be a Cohen-Macaulay local ring. The following are equivalent.

\begin{enumerate}
 \item The ring $R$ has an isolated singularity.
 \item All maximal Cohen-Macaulay $R$-modules are locally free on the punctured spectrum of $R$.
\end{enumerate}
\end{lem}

\begin{proof}
The implication (i) $\Rightarrow$ (ii) is clear since the property maximal Cohen-Macaulay localizes (cf. \cite[Theorem 2.1.2]{brunsherzog}). See \cite[Lemma 3.3]{yobook} for the converse.
\end{proof}

Maximal Cohen-Macaulay modules appear in a natural way as kernels in free resolutions as the following proposition shows.

\begin{prop}\label{kermcm}
Let $(R,\mm)$ be a Cohen-Macaulay ring and
$$F_{n-1}\ra F_{n-2}\ra\ldots\ra F_1\ra F_0$$
an exact sequence of free $R$-modules. If $n\geq\Dim(R)$, then $\Ker(F_{n-1}\ra F_{n-2})$ is maximal Cohen-Macaulay.
\end{prop}

\begin{proof}
See \cite[Proposition 1.4]{yobook}.
\end{proof}

\begin{defi}
Let $R$ be a ring and $M$ an $R$-module. Let $M^{\vee}:=\Hom_R(M,R)$. The module $M^{\vee\vee}$ is called the \textit{reflexive hull} of $M$.
We say that $M$ is \textit{reflexive} if the natural map 
\begin{align*}
 M & \longrightarrow M^{\vee\vee}\\
 m & \longmapsto \ev_m
\end{align*}
is an isomorphism.
\end{defi}

Note that reflexive hulls are torsion-free, since $\Hom_R(M,N)$ is torsion-free if $N$ is. 
They appear for example as modules of sections of certain sheaves over the punctured spectrum. We give the precise statement.

\begin{lem}
Let $(R,\mm)$ be a reduced, two-dimensional Cohen-Macaulay ring. Let $M$ be a finitely generated, torsion-free $R$-module. 
We denote by $\widetilde{M}$ the $\Oc_{\Spec(R)}$-module associated to $M$ and we set $U:=\Spec(R)\setminus\{\mm\}$. If $R$ is Gorenstein in codimension
one, we have the isomorphism
$$\Gamma(U,\widetilde{M}|_U)\cong M^{\vee\vee}.$$
\end{lem}

\begin{proof}
See Propositions 3.7 and 3.10 of \cite{mcmsur}.
\end{proof}

\begin{lem}\label{kerrefl}
Let $(R,\mm)$ be a reduced, two-dimensional Cohen-Macaulay ring which is Gorenstein in codimension one. Then kernels of maps between finitely generated reflexive modules are again reflexive.
\end{lem}

\begin{proof}
Let $0\ra K\ra M_1\ra M_2$ be an exact sequence of finitely generated $R$-modules with $M_1$ and $M_2$ reflexive. From this sequence, we obtain the commutative diagram
$$\xymatrix{
0\ar[r] & K\ar[r] & M_1\ar[r]\ar[d]^{\cong} & M_2\ar[d]^{\cong}\\
0\ar[r] & \Gamma(U,\widetilde{K}_{|U})\ar[r] & \Gamma(U,\widetilde{M_1}_{|U})\ar[r] & \Gamma(U,\widetilde{M_2}_{|U}),
}$$
which induces an isomorphism $K\cong\Gamma(U,\widetilde{K}_{|U})\cong K^{\vee\vee}$, since $K$ - as a submodule of a torsion-free module - is itself torsion-free.
\end{proof}

Our next goal is to show that for a two-dimensional normal domain maximal Cohen-Macaulay is equivalent to reflexive. For this we will need the following 
proposition.

\begin{prop}\label{charrefl}
Let $R$ be a ring and $M$ a finitely generated $R$-module. Then $M$ is reflexive if and only if

\begin{enumerate}
  \item $M_{\pp}$ is reflexive for all $\pp\in\Spec(R)$ with $\Depth_{R_{\pp}}(R_{\pp})\leq 1$ and
  \item $\Depth_{R_{\pp}}(M_{\pp})\geq 2$ for all $\pp\in\Spec(R)$ with $\Depth_{R_{\pp}}(R_{\pp})\geq 2$.
\end{enumerate}
\end{prop}

\begin{proof} See \cite[Proposition 1.4.1]{brunsherzog}.\end{proof}

\begin{thm}\label{mcmrefl}
Let $R$ be a two-dimensional normal domain and $M$ a finitely generated $R$-module. Then $M$ is maximal Cohen-Macaulay if and only if it is reflexive.
\end{thm}

\begin{proof}
We give the proof of \cite[Proposition 1.28]{karroum}. First assume that $M$ is maximal Cohen-Macaulay. Then $M$ satisfies condition (ii) of Proposition \ref{charrefl}. 
Since $R$ is normal, we have $\Depth_{R_{\pp}}(R_{\pp})\leq 1$, whenever $\Dim(R_{\pp})\leq 1$ by Serre's criterion. Since $R_{\pp}$ is regular for those primes, the localization 
$M_{\pp}$ is free, hence reflexive. Note that this part of the proof does not need the condition $\Dim(R)=2$.

For the contrary, let $M$ be reflexive. Then, since $R$ is normal, by condition (ii) of Proposition \ref{charrefl}, the modules $M_{\pp}$ are maximal Cohen-Macaulay for all $\pp\in\Spec(R)$ 
such that $\Dim(R_{\pp})=2$. The modules $M_{\pp}$ for $\Dim(R_{\pp})=0$ are of course maximal Cohen-Macaulay. For the primes $\pp$ 
with $\Dim(R_{\pp})=1$, the modules $M_{\pp}$ are reflexive, hence torsion-free, hence maximal Cohen-Macaulay.
\end{proof}

\begin{defi}
Let $R$ be a ring. We say that $R$ is of \textit{finite Cohen-Macaulay type} if there are only finitely many isomorphism classes of indecomposable, maximal 
Cohen-Macaulay modules.
\end{defi}

\subsection{Maximal Cohen-Macaulay modules and Hilbert-Kunz theory}

\begin{thm}[Seibert]
Let $(R,\mm,k)$ be a local ring of dimension $d$ and characteristic $p>0$. Assume that $R$ is of finite Cohen-Macaulay type and that $R^p$ is a 
finite $R$-module.

\begin{enumerate}
 \item If $R$ is Cohen-Macaulay, then $\HKM(R)\in\Q$.
 \item If $N$ is maximal Cohen-Macaulay, then $\HKM(N)\in\Q$.
\end{enumerate}
\end{thm}

\begin{proof}
See \cite[Theorem 4.1]{seibert2}.
\end{proof}

Moreover, Seibert gives an algorithm to compute the Hilbert-Kunz functions in these cases. For example he obtains
$$\HKF(N,p^e)=2p^e,\text{ if }e\geq 1$$ 
for all maximal Cohen-Macaulay modules $N$ over $R:=k\llbracket X,Y\rrbracket/(X^n+Y^2),$ where $n\geq 3$ is odd (cf. \cite[Example 5.1]{seibert2}) and
$$\begin{aligned}
   \HKF(R,p^e) & = 2p^{3e}-1,\\
   \HKF(K,p^e) & = 2p^{3e}+1,\\
   \HKF(U,p^e) & = 4p^{3e}+4
  \end{aligned}$$
for all $e\geq 0$, where $R$ is the second Veronese subring of $k\llbracket X,Y,Z\rrbracket$ (with $\chara(k)>2$) and $K$, $U$ are representatives for the 
isomorphism classes of the non-free, 
indecomposable, maximal Cohen-Macaulay modules with $K=(XY,Y^2,YZ)R$ (this is the canonical module) and 
$$U=K^3/R\left(\begin{matrix}XY\\Y^2\\YZ\end{matrix}\right)$$ 
(cf. \cite[Example 5.2]{seibert2}).

Let $k$ be an algebraically closed field of prime characteristic such that $n\in\N_{\geq 2}$ is invertible in $k$. Recall that the $n$-th Veronese 
subring $R$ of $k[X,Y]$ is the ring of invariants under the action $X\mapsto \epsilon\cdot X$, $Y\mapsto \epsilon\cdot Y$, 
where $\epsilon$ denotes a primitive $n$-th root of unity. The group acting on $k[X,Y]$ is isomorphic to $\Z/(n)$. Since $\Z/(n)$ is commutative, all 
irreducible representations are one-dimensional and there are exactly $n$ different equivalence classes (cf. \cite[Theorem 8.1]{dornhoff}). 
By the correspondence between irreducible representations of the acting group and 
the indecomposable, maximal Cohen-Macaulay modules over the ring of invariants (cf. \cite[Corollary 6.4]{leuwie}), we obtain that the $n$-th Veronese 
subring of $k[X,Y]$ has up to isomorphism exactly $n$ different indecomposable, maximal Cohen-Macaulay modules of rank one and none of higher rank. This 
shows that the Picard group of the punctured spectrum of $R$ is isomorphic to $\Z/(n)$. In the next example we want to demonstrate how one can make use 
of the finite Cohen-Macaulay type in the application of Proposition \ref{hkgeomapp}. 

\begin{exa}\label{veronese}
Let $S:=k[X,Y]$ and $R$ its $n$-th Veronese subring, hence $R$ is generated as $k$-algebra by the monomials $Z_i:=X^iY^{n-i}$. To compute the Hilbert-Kunz 
function of $R$, we have to control the Frobenius pull-backs of $\Syz(Z_0,\ldots,Z_n)$ on the punctured spectrum of $R$. Since $\Syz_R(Z_0,\ldots,Z_n)$ is 
the second syzygy module of the quotient $R/(Z_0,\ldots,Z_n)$, it is maximal Cohen-Macaulay. Since this module has rank $n$, it has to split into $n$ indecomposable, 
maximal Cohen-Macaulay modules of rank one. Our first goal is to describe these indecomposable modules as syzygy modules of two-generated monomial ideals. 
To compute generators for syzygy modules of monomial ideals of $R$, we will work over the factorial domain $S$ as follows. Let $P:=\Syz_R(F_1,\ldots,F_m)$ 
be a syzygy module of a monomial ideal in $R$ and let $(A_1,\ldots,A_m)\in P$. Then treat the relation 
$\sum A_j\cdot F_j=0$ over $S$, meaning that we write the monomials $F_j$ in the variables $X$ and $Y$. Then one can cancel the common 
multiple of the $F_j$, written in $X$, $Y$. The result is a relation $\sum A_j\cdot G_j(X,Y)=0$, where the $G_j\in S$ are coprime. From this relation 
one can compute the $A_j$ (as elements in $R$), using the fact that $k[X,Y]$ is factorial. 
Let $i<j\in\{0,\ldots,n\}$ and let $(A,B)\in\Syz_R(Z_i,Z_j)$. Then 
$$\begin{aligned}
 && AZ_i+BZ_j &= 0 \\
\Leftrightarrow && AX^iY^{n-i} &= -BX^jY^{n-j} \\
\Leftrightarrow && AY^{j-i} &= -BX^{j-i} \\
\Leftrightarrow && (A,B) &= C(X^{j-i},-Y^{j-i})
\end{aligned}$$
for some $C\in k[X,Y]$. Since $A$ is an element in $R$, its degree as element in $S$ is $na$ for some $a\in\N$. Therefore, the degree of $C$ is $na+i-j$. 
If $a\geq 2$ we see that $A$ and $B$ have a common factor in $R$, hence if $(A,B)\in\Syz_R(Z_i,Z_j)$ is of minimal degree, the degree of $A\in S$ has to 
be $n$. But then the possibilities for $C$ of minimal degree are the monomials of degree $n-j+i$. Hence, the $R$-modules $\Syz_R(Z_i,Z_j)$ are (minimally) generated by 
$n-j+i+1$ elements. Explicitely, one has
$$\Syz_R(Z_i,Z_j)=\left\langle
\begin{pmatrix}Z_{j-i} \\ -Z_0\end{pmatrix},
\begin{pmatrix}Z_{j-i+1} \\ -Z_1\end{pmatrix},\ldots,
\begin{pmatrix}Z_{n-1} \\ -Z_{n-j+i-1}\end{pmatrix},
\begin{pmatrix}Z_{n} \\ -Z_{n-j+i}\end{pmatrix}
\right\rangle.$$
This shows 
$$\Syz_R(Z_i,Z_j)\cong\Syz_R(Z_l,Z_m)\Longleftrightarrow |i-j|=|l-m|.$$
Therefore, it is possible to distinguish the different indecomposable, maximal Cohen-Macaulay modules by their number of generators. Moreover, this shows 
that the number of generators of a maximal Cohen-Macaulay module of rank $r$ is bounded above by $n\cdot r$, since it has to split into a direct sum of $r$ 
indecomposable, maximal Cohen-Macaulay modules of rank one and each of these has at most $n$ generators. For $i\in\{0,\ldots,n-1\}$ let 
$$\{(s_{i,1},s_{i+1,1}),\ldots,(s_{i,n},s_{i+1,n})\}$$ 
be a system of generators of $\Syz_R(Z_i,Z_{i+1})$. These give elements 
$$(0,\ldots,0,s_{i,j},s_{i+1,j},0,\ldots,0)\in\Syz_R(Z_0,\ldots,Z_n).$$
We want to show that 
\begin{equation}\label{gen}\{(0,\ldots,0,s_{i,j},s_{i+1,j},0,\ldots,0)|1\leq j\leq n, 0\leq i\leq n-1\}\end{equation}
generates $\Syz_R(Z_0,\ldots,Z_n)$. Since all $s_{i,j}$ have degree one we have to show that the elements in (\ref{gen}) are linear independent. Assuming 
the contrary, there are $i_0\in\{0,\ldots,n-1\}$, $j_0\in\{1,\ldots,n\}$ such that
\begin{equation}\label{linrel}\begin{pmatrix} 0 \\ \vdots \\ 0 \\ s_{i_0,j_0} \\ s_{i_0+1,j_0} \\ 0 \\ \vdots \\ 0 \end{pmatrix} =
\sum_{j\neq j_0}\lambda_j\cdot\begin{pmatrix} 0 \\ \vdots \\ 0 \\ s_{i_0,j} \\ s_{i_0+1,j} \\ 0 \\ \vdots \\ 0 \end{pmatrix}
+
\sum_{i\neq i_0}\begin{pmatrix} 0 \\ \vdots \\ 0 \\ t_{i,i} \\ t_{i,i+1} \\ 0 \\ \vdots \\ 0 \end{pmatrix},\end{equation}
where $\lambda_j\in k$ and for each $i$ the element $(0,\ldots,0,t_{i,i},t_{i,i+1},0,\ldots,0)$ is a linear combination of the elements 
$(0,\ldots,0,s_{i,j},s_{i+1,j},0,\ldots,0)$ for $j\in\{1,\ldots,n\}$. Computing the first $i_0$ components of the right hand side in Equation (\ref{linrel}) (note that the 
summands for $i\geq i_0+1$ are not necessary), we obtain the following system of linear equations
$$\begin{array}{rcl}
t_{i_0-1,i_0}+\sum_{j\neq j_0}\lambda_j\cdot s_{i_0,j} &=& s_{i_0,j_0}\\
t_{i_0-1,i_0-1}+t_{i_0-2,i_0-1} &=& 0\\ 
 &\vdots& \\
t_{i_0-l,i_0-l}+t_{i_0-l-1,i_0-l} &=& 0\\ 
 &\vdots& \\
t_{1,1}+t_{0,1} &=& 0\\ 
t_{0,0} &=& 0.
\end{array}$$
Since $t_{0,0}=0$ and $t_{0,0}Z_0+t_{0,1}Z_1=0$, we obtain $t_{0,1}=0$. Solving the above linear equation system from the bottom, we obtain the contradiction
$s_{i_0,j_0}=\sum_{j\neq j_0}\lambda_j\cdot s_{i_0,j}$.

This shows that the $n\cdot n$ elements from (\ref{gen}) are part of a minimal system of generators of $\Syz_R(Z_0,\ldots,Z_n)$. Since this module is 
maximal Cohen-Macaulay of rank $n$, it has at most $n\cdot n$ generators. All in all, we obtain
$$\Syz_R(Z_0,\ldots,Z_n)\cong\bigoplus_{i=0}^{n-1}\Syz_R(Z_i,Z_{i+1}).$$
For $i\in\N$ write $i=n\cdot q+r$ with $q\in\N$ and $r\in\{0,\ldots,n-1\}$. Let $(A,B)\in\Syz_R(Z_0^i,Z_1^i)$. Then
$$\begin{aligned}
 && AZ_0^i+BZ_1^i &= 0 \\
\Leftrightarrow && AY^{ni} &= -BX^iY^{(n-1)i} \\
\Leftrightarrow && AY^i &= -BX^i \\
\Leftrightarrow && (A,B) &= C(X^i,-Y^i)
\end{aligned}$$
for some $C\in k[X,Y]$. An argumentation as above shows that the possibilities for $C$ of minimal degree are the monomials of degree $n-r$. This shows 
that $\Syz_R(Z_0^i,Z_1^i)$ has $n-r+1$ generators and hence
$$\Syz_R(Z_0^i,Z_1^i)\cong\Syz_R(Z_0,Z_r)(m)\cong\Syz_R(Z_0^r,Z_1^r)(l)$$
for some $m,l\in\Z$. All in all, we obtain for every $m\in\Z$ 
$$\Syz_R(Z_0^i,\ldots,Z_n^i)(m)\cong \Syz_R(Z_0^r,\ldots,Z_n^r)\left(m-(n+1)\frac{i-r}{n}\right).$$
If the characteristic of $k$ is $p$ and $i=p^e$, we obtain from Proposition \ref{hkgeomapp} as in Example \ref{hkfdreivar} by using $\Deg(\Proj(R))=n$ 
(cf. \cite[Exercise I.7.1]{hartshorne}) the formula
$$\HKF(R,q)=\frac{n+1}{2}\cdot(q^2-r^2)+D_r,$$
with $D_r:=\Dim_k\left( R/\left(Z_0^r,\ldots,Z_n^r\right)\right)$.

Next, we give an explicit formula for $D_r$. Of course, we have 
\begin{equation}\label{dr}
D_r = \sum_{i=0}^r\Dim_k\left( R/\left(Z_0^r,\ldots,Z_n^r\right)\right)_i+\sum_{i\geq r+1}\Dim_k\left( R/\left(Z_0^r,\ldots,Z_n^r\right)\right)_i.
\end{equation}
Since monomials in $R$ of degree $i$ correspond to monomials in $k[X,Y]$ of degree $ni$, the first sum above equals 
$$\sum_{i=0}^r(ni+1)-(n+1)=n\cdot\frac{r(r+1)}{2}+r-n.$$
We will show that the second summand in (\ref{dr}) vanishes by showing that a monomial in $R$ of degree $r+1$ has a factor of the form $Z_j^r$ for some $j$. 
Let $M\in R$ be a monomial of degree $r+1$. This means that $M=X^iY^{n(r+1)-i}\in k[X,Y]$ for some $i\in\{0,\ldots,n(r+1)\}$. Assume that 
$i\in [ar,(a+1)r)$ for some $a\in\N$. We obtain
$$M=X^iY^{n(r+1)-i}=X^{i-ar+ar}Y^{nr+n-i+ar-ar}=X^{ar}Y^{nr-ar}X^{i-ar}Y^{ar+n-i}=Z_a^rX^{i-ar}Y^{ar+n-i},$$
showing that all monomials in $R_{r+1}$ vanish modulo $(Z_0^r,\ldots,Z_n^r)$.

Finally, under the assumption $\gcd(p,n)=1$ the Hilbert-Kunz function of $R$ is given by
$$\HKF(R,p^e)=\frac{n+1}{2}\cdot \left(p^{2e}-r^2\right)+n\cdot\frac{r(r+1)}{2}+r-n,$$
where $r$ is the smallest non-negative representative of the class of $p^e$ in $\Z/(n)$.
\end{exa}

\section{Generalities on matrix factorizations}
Let $(S,\mm,k)$ be a regular local ring, $f\in \mm^2\setminus\{0\}$ and let $R:=S/(f)$.
The starting point for the theory of matrix factorizations is the observation that by the Auslander-Buchsbaum formula every maximal Cohen-Macaulay $R$-module $M$ has - viewed as a $S$-module - a free resolution of length one:
$$0\ra S^n\stackrel{\phi}{\ra} S^n\ra M\ra 0.$$

Because $\phi$ is injective and $f\cdot S^n$ is contained in the image of $\phi$, one can construct a linear map $\psi:S^n\ra S^n$ satisfying $\phi\circ\psi=f\cdot\Id$.
Composing this equation from the right by $\phi$ and using again that $\phi$ is injective, one gets $\psi\circ\phi=f\cdot\Id$. This motivates the following definition.

\begin{defi}
For a given non-zero $f\in\mm^2$, a pair of $n\times n$-matrices $(\phi,\psi)$ over $S$ with
$$\phi\circ\psi=\psi\circ\phi=f\cdot\Id_n$$
is called a \textit{matrix factorization} for $f$ of size $n$.

A matrix factorization $(\phi,\psi)$ is called \textit{reduced} if all entries are non-units.
\end{defi}

\begin{exa}
The pairs $(f,1)$ and $(1,f)$ are non-reduced matrix factorizations for $f$ of size one. If $f=g\cdot h$ with $g,h\in\mm$, then $(g,h)$ and $(h,g)$ are reduced matrix factorizations for $f$ of size one.
\end{exa}

We denote the set of matrix factorizations for $f$ by $\MF_S(f)$. The following two definitions will enrich its structure to an additive category.

\begin{defi}
 Let $(\phi_1,\psi_1)$ and $(\phi_2,\psi_2)$ be two matrix factorizations for $f$ of size $n$ resp. $m$. A \textit{morphism} from $(\phi_1,\psi_1)$ to $(\phi_2,\psi_2)$ is a pair of $m\times n$-matrices $(\alpha,\beta)$ over $S$ such that the following diagram commutes
$$\xymatrix{
S^n\ar[r]^{\psi_1}\ar[d]^{\alpha} & S^n\ar[r]^{\phi_1}\ar[d]^{\beta} & S^n\ar[d]^{\alpha} \\
S^m\ar[r]^{\psi_2} & S^m\ar[r]^{\phi_2} & S^m.
}$$
We call $(\phi_1,\psi_1)$ and $(\phi_2,\psi_2)$ \textit{equivalent} if there is a morphism $(\alpha,\beta)$ between them, where $\alpha$ and $\beta$ are isomorphisms. 
We will denote a morphism $(\alpha,\beta)$ from $(\phi_1,\psi_1)$ to $(\phi_2,\psi_2)$ by $$(\alpha,\beta):(\phi_1,\psi_1)\ra(\phi_2,\psi_2).$$
\end{defi}

\begin{rem}
Note that the commutativity of the second square is enough: Multiplying $\psi_2\alpha$ by $f\cdot\Id$ from the right gives
$$\begin{aligned}
(\psi_2\alpha)(f\cdot\Id) =& (\psi_2\alpha)(\phi_1\psi_1)\\
 =& \psi_2(\alpha\phi_1)\psi_1\\
 =& \psi_2(\phi_2\beta)\psi_1\\
 =& (\psi_2\phi_2)(\beta\psi_1)\\
 =& (f\cdot\Id)(\beta\psi_1).
\end{aligned}$$
\end{rem}

\begin{defi}
 Let $(\phi_1,\psi_1)$ and $(\phi_2,\psi_2)$ be two matrix factorizations for $f$ of size $n$ resp. $m$. We define their direct sum as
$$(\phi_1,\psi_1)\oplus (\phi_2,\psi_2):=\left(\begin{pmatrix}\phi_1&0\\0&\phi_2\end{pmatrix},\begin{pmatrix}\psi_1&0\\0&\psi_2\end{pmatrix}\right).$$ 
\end{defi}

\begin{defi}
A matrix factorization $(\phi,\psi)$ is called \textit{indecomposable} if it is not equivalent to a direct sum of matrix factorizations. 
\end{defi}

\begin{exa}\label{d4split}
Let $f=X^2+Y^3+YZ^2$. Assume that the ground field is not of characteristic two and contains $i$. Then $(\phi_0,\phi_0)$, $(\phi_1,\phi_1)$ and $(\phi_2,\phi_2)$ with $$\phi_0=\begin{pmatrix} -X & Y^2 & YZ & 0\\ Y & X & 0 & -Z\\ Z & 0 & X & Y\\ 0 & -YZ & Y^2 & -X\end{pmatrix},$$
$$\phi_1=\begin{pmatrix} X & iY^2+YZ\\ Z-iY & -X\end{pmatrix}\qquad\text{and}\qquad \phi_2=\begin{pmatrix} X & -iY^2+YZ\\ Z+iY & -X\end{pmatrix}$$
are matrix factorizations for $f$. In this case $(\alpha,-\alpha):(\phi_0,\phi_0)\lra(\phi_1,\phi_1)\oplus (\phi_2,\phi_2)$ with 
$$\alpha=\begin{pmatrix} i & 0 & 0 & -1\\ 0 & -1 & -i & 0\\ -i & 0 & 0 & -1\\ 0 & -1 & i & 0 \end{pmatrix}$$
is an equivalence of matrix factorizations for $f$.
\end{exa}

Since $f\cdot S^n\subseteq \phi(S^n)$, the cokernel of $\phi$ is annihilated by $f$ and therefore an $R$-module. Taking the sequence of maps and modules
$$\ldots\stackrel{\phi}{\longrightarrow} S^n\stackrel{\psi}{\longrightarrow} S^n\stackrel{\phi}{\longrightarrow} S^n\lra\coKer(\phi)\lra 0$$
modulo $f$, gives a two-periodic free resolution of $\coKer(\phi)$ as an $R$-module. Because of this, $\coKer(\phi)$ is isomorphic to all of its $2i$-th 
syzygy modules that are maximal Cohen-Macaulay by Proposition \ref{kermcm} if $2i$ is at least $\Dim(R)$.
All in all we get a functor 
$$\coKer((\phi,\psi)):=\coKer(\phi):\MF_S(f)\lra\mathfrak{C}(R),$$ 
where $\mathfrak{C}(R)$ denotes the category of maximal Cohen-Macaulay $R$-modules.
As described at the beginning of this section, we have a functor $\Gamma:\mathfrak{C}(R)\longrightarrow\MF_S(f)$ as well.

\begin{thm}[Eisenbud]\label{eismatfac}
The maps $\coKer$ and $\Gamma$ induce a bijection between the sets of reduced matrix factorizations for $f$ up to equivalence and non-free, maximal Cohen-Macaulay $R$-modules up to isomorphism. Moreover, this bijection respects direct sums.
\end{thm}

\begin{proof}See \cite[Corollary 6.3]{eismat}.\end{proof}

\begin{rem}
Note that the non-reduced matrix factorizations $(f,1)$ and $(1,f)$ of $f$ are mapped under $\coKer$ to $R$ \mbox{resp. $0$}. Using the above theorem one can show that any 
matrix factorization $(\phi,\psi)$ is equivalent to a direct sum of matrix factorizations of the form $(\phi_0,\psi_0)\bigoplus\oplus_{i=1}^a(f,1)\bigoplus\oplus_{i=1}^b(1,f)$, where the natural numbers $a$, $b$ are uniquely determined and $(\phi_0,\psi_0)$ is reduced and unique up to equivalence (cf. \cite[Remark 7.5]{yobook}).
\end{rem}

Theorem \ref{eismatfac} generalizes to certain quotients of categories given by the following construction.

\begin{con}
Let $\Ac$ be an Abelian category and $\Bc$ a subclass of objects of $\Ac$. For any $A$, $B\in\Ac$, we define $\Bc(A,B)\subseteq\Hom_{\Ac}(A,B)$ as the set of morphisms that factor through direct sums of objects in $\Bc$. 
In fact, $\Bc(A,B)$ is a normal subgroup of $\Hom_{\Ac}(A,B)$. We define the category $\Ac/\Bc$, whose objects are the objects of $\Ac$, but the morphisms from $A$ to $B$ are given by the quotient
$$\Hom_{\Ac/\Bc}(A,B):=\Hom_{\Ac}(A,B)/\Bc(A,B).$$
\end{con}

\begin{defi}
 Let $\underline{\MF}_S(f):=\MF_S(f)/\{(1,f),(f,1)\}$ be the category of equivalence classes of reduced matrix factorizations and $\underline{\mathfrak{C}}(R):=\mathfrak{C}(R)/\{R\}$ the category of isomorphism classes of non-free maximal Cohen-Macaulay \mbox{$R$-modules}.
\end{defi}

Now Theorem \ref{eismatfac} generalizes as follows.

\begin{thm}[Eisenbud]
The functors $\coKer$ and $\Gamma$ induce an equivalence of the categories $\underline{\MF}_S(f)$ and $\underline{\mathfrak{C}}(R).$
\end{thm}

\begin{proof}See \cite[Theorem 7.4]{yobook}.\end{proof}

\begin{rem}\label{rankofmatfac}
 If $f$ is prime and $(\phi,\psi)$ is a matrix factorization for $f$ of size $n$, then there is a unit $u\in S$ and a number $m\in\{0,\ldots,n\}$ such that $\Det(\phi)=u\cdot f^m$ and $\Det(\psi)=u^{-1}\cdot f^{n-m}$. In this case $\coKer(\phi)$ has rank $m$ and $\coKer(\psi)$ has rank $n-m$ as $R$-modules.
This can be seen by localizing at the prime ideal $(f)$. Since $S_{(f)}$ is a discrete valuation ring, the map $\phi_{(f)}$ on the localization is equivalent to the map $f\cdot\Id_m\oplus\Id_{n-m}$. Obviously, $\coKer(\phi_{(f)})$ has rank $m$ over the field $Q(R)\cong S_{(f)}/fS_{(f)}$.
Note that $m\in\{1,\ldots,n-1\}$, if $(\phi,\psi)$ is reduced. Compare with \cite[Proposition 5.6]{eismat}.
\end{rem}

The next lemma describes the connection of a maximal Cohen-Macaulay $R$-module and its dual in terms of matrix factorizations.

\begin{lem}
Given a matrix factorization $(\phi,\psi)$ for $f$, the pair $\left(\phi\trans,\psi\trans\right)$ is again a matrix factorization for $f$. Moreover, we have $$\coKer\left(\phi\trans\right)\cong(\coKer(\phi))^{\vee}.$$
\end{lem}

\begin{proof}
The first statement is clear. For a matrix $M$ over $S$, we denote the matrix over $R$ given by taking every entry from $M$ modulo $f$ by $\overline{M}$. Applying $\Hom(\_,R)$ to $$R^n\stackrel{\overline{\phi}}{\lra}R^n\lra\coKer(\overline{\phi})\lra 0,$$
produces the exact sequence
$$0\lra(\coKer(\overline{\phi}))^{\vee}\lra R^n\stackrel{\overline{\phi}\trans}{\lra}R^n\lra\coKer\left(\overline{\phi}\trans\right)\lra 0.$$
This shows $$\coKer\left(\overline{\phi}\trans\right)\cong R^n/\Ker\left(\overline{\psi}\trans\right)\cong \Ima\left(\overline{\psi}\trans\right)\cong \Ker\left(\overline{\phi}\trans\right)\cong(\coKer(\overline{\phi}))^{\vee}.$$
Since $\coKer(\phi)$ and $\coKer\left(\phi\trans\right)$ have a natural $R$-module structure, they identify with the $R$-modules $\coKer(\overline{\phi})$ resp. $\coKer\left(\overline{\phi}\trans\right)$.
\end{proof}

Due to Yoshino (cf. \cite{yotensor}), there is the notion of the tensor product of matrix factorizations for power series in disjoint sets of variables.

\begin{defi}
Let $(\phi_f,\psi_f)$ be a matrix factorization for $f$ in $k\llbracket X_1,\ldots,X_r \rrbracket$ of size $n$ and $(\phi_g,\psi_g)$ a matrix factorization for $g$ in $k\llbracket Y_1,\ldots,Y_s\rrbracket$ of size $m$.
Then 
$$(\phi_f,\psi_f)\hat{\otimes} (\phi_g,\psi_g):=\left(
\begin{pmatrix}\phi_f\otimes\id_m & \id_n\otimes\phi_g\\ -\id_n\otimes\psi_g & \psi_f\otimes\id_m\end{pmatrix},
\begin{pmatrix}\psi_f\otimes\id_m & -\id_n\otimes\phi_g\\ \id_n\otimes\psi_g & \phi_f\otimes\id_m\end{pmatrix}
\right)$$
is a matrix factorization for $f+g$ in $k\llbracket X_1,\ldots,X_r,Y_1,\ldots,Y_s\rrbracket$ of size $2nm$.

If $(\alpha,\beta):(\phi_1,\psi_1)\ra (\phi_2,\psi_2)$ is a morphism in $\MF(f)$, then 
$$(\alpha,\beta)\hat{\otimes}(\phi_g,\psi_g)\coloneqq \left(\begin{pmatrix}\alpha\otimes\id_m & 0\\ 0 & \beta\otimes\id_m\end{pmatrix},\begin{pmatrix}\beta\otimes\id_m & 0\\ 0 & \alpha\otimes\id_m\end{pmatrix}\right)$$
is a morphism from $(\phi_1,\psi_1)\hat{\otimes}(\phi_g,\psi_g)$ to $(\phi_2,\psi_2)\hat{\otimes}(\phi_g,\psi_g)$ in $\MF(f+g)$. Similarly, if $(\gamma,\delta):(\phi_3,\psi_3)\ra(\phi_4,\psi_4)$ is a morphism in $\MF(g)$, then
$$(\phi_f,\psi_f)\hat{\otimes}(\gamma,\delta)\coloneqq \left(\begin{pmatrix}\id_n\otimes\gamma & 0\\ 0 & \id_n\otimes\delta\end{pmatrix},\begin{pmatrix}\id_n\otimes\delta & 0\\ 0 & \id_n\otimes\gamma\end{pmatrix}\right)$$
is a morphism from $(\phi_f,\psi_f)\hat{\otimes}(\phi_3,\psi_3)$ to $(\phi_f,\psi_f)\hat{\otimes}(\phi_4,\psi_4)$ in $\MF(f+g)$.
\end{defi}

\begin{exa}\label{diagrank2matfac}
Let $d_1$, $d_2$, $d_3\in\N_{\geq 2}$. Then for each $1\leq a\leq d_1-1$ the pair $(X^a,X^{d_1-a})$ is a reduced matrix factorization for $X^{d_1}$. Similarly, $(Y^b,Y^{d_2-b})$ and $(Z^c,Z^{d_3-c})$ with $1\leq b\leq d_2-1$, $1\leq c\leq d_3-1$ are reduced matrix factorizations for $Y^{d_2}$ resp. $Z^{d_3}$.
With these matrix factorizations we can construct matrix factorizations for $X^{d_1}+Y^{d_2}$ and for $X^{d_1}+Y^{d_2}+Z^{d_3}$. For given $a$, $b$, $c$ as above, we get
\begin{align*}
 & ((X^a,X^{d_1-a})\hat{\otimes}(Y^b,Y^{d_2-b}))\hat{\otimes}(Z^c,Z^{d_3-c})\\
=& \left(\begin{pmatrix}
X^a & Y^b \\
-Y^{d_2-b} & X^{d_1-a}
\end{pmatrix},
\begin{pmatrix}
X^{d_1-a} & -Y^b \\
Y^{d_2-b} & X^a 
\end{pmatrix}\right)\hat{\otimes}(Z^c,Z^{d_3-c})\\
= &
\left(\begin{pmatrix}
X^a & Y^b & Z^c & 0\\
-Y^{d_2-b} & X^{d_1-a} & 0 & Z^c\\
-Z^{d_3-c} & 0 & X^{d_1-a} & -Y^b\\
0 & -Z^{d_3-c} & Y^{d_2-b} & X^a
\end{pmatrix},
\begin{pmatrix}
X^{d_1-a} & -Y^b & -Z^c & 0\\
Y^{d_2-b} & X^a & 0 & -Z^c\\
Z^{d_3-c} & 0 & X^a & Y^b\\
0 & Z^{d_3-c} & -Y^{d_2-b} & X^{d_1-a}
\end{pmatrix}\right).
\end{align*}
Denote the result by $(\phi,\psi)$. Then $(\phi,\psi)$ and $\left(\phi\trans,\psi\trans\right)$ are equivalent. Let $$\eta:=\begin{pmatrix} 0 & -1 \\ 1 & 0 \end{pmatrix}.$$
Then $(\alpha,\alpha):(\phi,\psi)\lra \left(\phi\trans,\psi\trans\right)$ given by $$\alpha=\begin{pmatrix} 0 & \eta\\ \eta & 0\end{pmatrix}$$
is an equivalence of matrix factorizations. Note that the corresponding modules $\coKer(\phi)$ have rank two by Remark \ref{rankofmatfac}, since $\Det(\phi)=(X^{d_1}+Y^{d_2}+Z^{d_3})^2$.
\end{exa}

This tensor product behaves much better than the tensor product of modules.

\begin{thm}[Yoshino]\label{proptensor}
With the previous notations $$\_\hat{\otimes}(\phi_g,\psi_g):\MF(f)\lra\MF(f+g)$$ is an exact functor. Moreover, if $(\phi_g,\psi_g)$ is reduced, this functor is faithful.
\end{thm} 

\begin{proof} See \cite[Lemmata 2.8 and 2.11]{yotensor}.\end{proof}

\begin{rem}\label{splitrem}
 In Chapter 3 of \cite{yotensor} Yoshino discusses how the indecomposability of matrix factorizations behaves under tensor products. If $\chara(k)\neq 2$, then by \cite[Example 3.8]{yotensor} the matrix factorizations computed in Example \ref{diagrank2matfac}
are indecomposable if and only if at most one of the following equalities holds: $2a=d_1$, $2b=d_2$, $2c=d_3$.
\end{rem}

Yoshino found also the following decomposition behaviour.

\begin{lem}\label{tensorsplit}
Suppose $\chara(k)\neq 2$ and $i\in k$. For matrix factorizations $(\phi,\phi)$ of size $n$ and $(\psi,\psi)$ of size $m$ the tensor product $(\phi,\phi)\hat{\otimes}(\psi,\psi)$ decomposes as the direct sum $(\xi,\zeta)\oplus (\zeta,\xi)$ with
$$(\xi,\zeta)=(\phi\otimes\id_m-i(\id_n\otimes\psi),\phi\otimes\id_m+i(\id_n\otimes\psi)).$$
\end{lem}

\begin{proof} See \cite[Lemma 3.2]{yotensor}.\end{proof}

\begin{exa}
Consider the matrix factorizations $\left(X^{d_1'},X^{d_1'}\right)$ for $X^{2d_1'}$ and $\left(Y^{d_2'},Y^{d_2'}\right)$ for $Y^{2d_2'}$. Then $$\left(X^{d_1'},X^{d_1'}\right)\hat{\otimes}\left(Y^{d_2'},Y^{d_2'}\right) \cong \left(X^{d_1'}-iY^{d_2'},X^{d_1'}+iY^{d_2'}\right)\oplus \left(X^{d_1'}+iY^{d_2'},X^{d_1'}-iY^{d_2'}\right).$$
If we are in the case $a=d_1'=d_1/2$ and $b=d_2'=d_2/2$ of Example \ref{diagrank2matfac},
$$\left(\left(X^{d_1'},X^{d_1'}\right)\hat{\otimes}\left(Y^{d_2'},Y^{d_2'}\right)\right)\hat{\otimes}\left(Z^c,Z^{d_3-c}\right)$$ splits by Remark \ref{splitrem} (or Theorem \ref{proptensor}) as the tensor products of the two direct summands above with $(Z^c,Z^{d_3-c})$.
\end{exa}

\section[Constructing isomorphisms of maximal Cohen-Macaulay modules]{Constructing isomorphisms \except{toc}{between}\for{toc}{of} maximal Cohen-Macaulay modules \except{toc}{and first syzygy modules of ideals}}
In this section we construct isomorphisms between maximal Cohen-Macaulay modules and first syzygy modules of ideals over rings of dimension two.
To be more precise, given a full set $M_1,\ldots,M_n$ of representatives for the isomorphism classes of indecomposable, maximal Cohen-Macaulay 
$R$-modules, we want to find for each $i$ non-zero elements $F_1,\ldots,F_m\in R$ such that
$$M_i\cong\Syz_R(F_1,\ldots,F_m).$$
This has to be possible, since first syzygy modules of ideals are second syzygy modules of quotient rings.

From now on we assume that $f$ is irreducible, that $R=S/(f)$ is a local Cohen-Macaulay normal domain of dimension two and that $R$ has no 
singularities outside the origin. In this situation all maximal Cohen-Macaulay modules are locally free on the punctured spectrum of $R$ 
(cf. Lemma \ref{isolated}). We will denote the punctured spectrum of $R$ by $U$.

Now we describe a way to find a representation of a maximal Cohen-Macaulay module, given by a matrix factorization, as a syzygy bundle on $U$. 
At the end of the chapter we discuss an algebraic approach to find an isomorphism between maximal 
Cohen-Macaulay modules and first syzygy modules of ideals that reflects the main idea of restricting $\phi$ to a submodule of lower rank better than 
the rather sheaf-theoretic approach. We will see that this algebraic approach does not work very properly.

\begin{con}
For a (reduced) matrix factorization $(\phi_S,\psi_S)$ of size $n$, we have a two periodic free resolution of $M:=\coKer(\phi_S)$ as $R$-module:
$$\ldots\stackrel{\psi}{\lra} R^n\stackrel{\phi}{\lra} R^n\stackrel{\psi}{\lra} R^n\stackrel{\phi}{\lra} R^n\lra M\lra 0,$$
where $\phi$ and $\psi$ arise from $\phi_S$ and $\psi_S$ by taking their entries modulo $f$. Note that we can identify $M$ with $\coKer(\phi)$.

Therefore we get $$M\cong R^n/\Ima(\phi)\cong R^n/\Ker(\psi)\cong \Ima(\psi).$$

Later we will see that in the case of rank one modules the isomorphism $M\cong\Ima(\psi)$ is already enough to describe $M$ as a syzygy module.

For any subset $J$ of $\{1,\ldots,n\}$ we denote by $\psi^J$ the matrix obtained from $\psi$ by keeping all columns whose index belongs to $J$ (and 
deleting those with index in $J^{\text{C}}$).

If $\coKer(\phi)\cong \Ima(\psi)$ has rank $m$ as an $R$-module, the sheaf $\widetilde{\Ima(\psi)}|_U$ gives a locally free $\Oc_U$-module of rank $m$. Therefore a 
representing first syzygy module of an ideal has to be a syzygy module of an $(m+1)$-generated ideal.

\label{supposethat}Suppose that we can choose a $J\subseteq\{1,\ldots,n\}$ of cardinality $m+1$ such that $\widetilde{\Ima(\psi)}|_U$ and $\Fc\coloneqq\widetilde{\Ima(\psi^J)}|_U$
are isomorphic as sheaves of $\Oc_U$-modules\footnote{Note that we have no argument why this should be possible, but it works in all examples.}, where the isomorphism is the map induced by the natural inclusion $\Ima\left(\psi^J\right)\subseteq\Ima(\psi)$. 
A necessary condition for this is that the matrix $\psi^J$ has full rank - namely $m$ - in every non-zero point. 
Otherwise there would be at least one point $u\in U$ such that the stalk $\widetilde{\Ima(\psi^J)}_u$ has rank at most $m-1$.

Since $\widetilde{\Ima(\psi)}|_U$ comes from a maximal Cohen-Macaulay $R$-module, it is locally free as $\Oc_U$-module, giving that 
$$\Oc_U^{m+1}\stackrel{\psi^J}{\lra}\Fc\lra 0$$
is a surjection of locally free $\Oc_U$-modules of ranks $m+1$ and $m$. Therefore the kernel $\Lc$ has to be locally free of rank one. This gives an isomorphism 
$$\Lc\otimes\Det(\Fc)\cong\Oc_U$$
of the determinants (cf. Proposition \ref{propdet} (iv)). We distinguish whether the determinant of $\Fc$ is trivial or not.

\textit{Case 1.} Assume $\Det(\Fc)\cong\Oc_U$. Then the line bundle $\Lc$ has to be trivial. Hence
\begin{equation}\label{dettrivial}
0\lra\Oc_U\stackrel{\eta}{\lra}\Oc_U^{m+1}\stackrel{\psi^J}{\lra}\Fc\lra 0
\end{equation}
is a short exact sequence and $\eta$ has to be the multiplication by a vector $(G_1,\ldots,G_{m+1})\trans$ with $G_i\in R$. But then Sequence 
\eqref{dettrivial} is exactly the dual of the presenting sequence of $\Syz_U(G_1,\ldots,G_{m+1})$, hence there is an isomorphism
$$\Syz_U(G_1,\ldots,G_{m+1})\cong\Fc^{\vee}.$$
Moreover, the ideal generated by the $G_i$ has to be $\mm$-primary, since the determinant of the syzygy bundle is trivial. Note that this case cannot appear if $\Fc$ has rank one.

\textit{Case 2.} Now assume that $\Lc$ is a non-trivial line bundle, hence $\Det(\Fc)\cong\Lc^{\vee}$. Dualizing the sequence 
$0\ra\Lc\ra\Oc_U^{m+1}\ra\Fc\ra 0$ yields
\begin{equation}\label{detnontrivial}0\lra\Fc^{\vee}\stackrel{(\psi^J)\trans}{\lra}\Oc_U^{m+1}\ra\Lc^{\vee}\lra 0.\end{equation}
Since $\Lc^{\vee}$ is a line bundle it corresponds to a Cartier divisor $D$. There is a finite cover $(U_i)$ of $U$ such that $D|_{U_i}$ 
is principal, say $D|_{U_i}=(h_i)$. The multiplication with the common denominator of the $h_i$ defines an embedding of $\Lc^{\vee}$ into $\Oc_U$. With 
this embedding, we can extend Sequence \eqref{detnontrivial} to 
\begin{equation}\label{nontrivial}0\lra\Fc^{\vee}\stackrel{(\psi^J)\trans}{\lra}\Oc_U^{m+1}\stackrel{\mu}{\lra}\Oc_U\lra\Tc\lra 0,\end{equation}
where $\Tc$ is a non-zero torsion sheaf and the map $\mu$ is the multiplication by a row vector $(G_1,\ldots,G_{m+1})$. In this case, Sequence 
\eqref{nontrivial} is the presenting sequence of the syzygy bundle $\Syz_U(G_1,\ldots,G_{m+1})$, where the ideal $(G_1,\ldots,G_{m+1})$ is not 
$\mm$-primary, since $\Tc$ is non-zero.

In both cases the column vector $(G_1,\ldots,G_{m+1})\trans$ generates the kernel of $\psi^J$. In order to compute the $G_i$, we only need to find an element $(F_1,\ldots,F_{m+1})\trans\in\Ker\left(\psi^J\right)$.
Then $(F_1,\ldots,F_{m+1})\trans$ is just a mutiple of $(G_1,\ldots,G_{m+1})\trans$ and equality (up to multiplication by a unit) holds if we can prove that the $F_i$ are coprime. 
For example this condition is automatically fullfilled if the ideal generated by the $F_i$ is $\mm$-primary. Anyway, the syzygy modules of the two ideals are isomorphic.

All in all, our task is to find a vector $(F_1,\ldots,F_{m+1})\trans$ in the kernel of $\psi^J$. This gives an isomorphism 
$$\widetilde{M}^{\vee}\cong\Fc^{\vee}\cong\Syz_U(F_1,\ldots,F_{m+1}).$$
Since we have the isomorphisms 
$$M^{\vee}\cong\coKer(\phi)^{\vee}\cong\coKer\left(\phi\trans\right)\cong\Ima\left(\psi\trans\right),$$
we will apply this construction to $\psi$, if $\coKer(\phi)$ is selfdual, or to $\psi\trans$, if $\coKer(\phi)$ is not selfdual. In both cases we end up with an 
isomorphism $$\widetilde{M}\cong\Syz_U(F_1,\ldots,F_{m+1}).$$
\end{con}

\begin{rem}
Since the underlying ring $R$ is a two-dimensional normal domain the isomorphisms of $\Oc_U$-modules constructed above, extend to isomorphisms of $R$-modules, since all involved modules are reflexive (cf. \cite[Lemma 3.6]{mcmsur}).
\end{rem}

We will illustrate this construction by an explicit computation.

\begin{exa}
 In Example \ref{diagrank2matfac} we found the matrix factorizations
$$(\phi,\psi):=\left(\begin{pmatrix}
X^a & Y^b & Z^c & 0\\
-Y^{d_2-b} & X^{d_1-a} & 0 & Z^c\\
-Z^{d_3-c} & 0 & X^{d_1-a} & -Y^b\\
0 & -Z^{d_3-c} & Y^{d_2-b} & X^a
\end{pmatrix},
\begin{pmatrix}
X^{d_1-a} & -Y^b & -Z^c & 0\\
Y^{d_2-b} & X^a & 0 & -Z^c\\
Z^{d_3-c} & 0 & X^a & Y^b\\
0 & Z^{d_3-c} & -Y^{d_2-b} & X^{d_1-a}
\end{pmatrix}\right)$$
for $X^{d_1}+Y^{d_2}+Z^{d_3}$ with $1\leq a\leq d_1-1$, $1\leq b\leq d_2-1$ and $1\leq c\leq d_3-1$.

The corresponding modules $\coKer(\phi)$ have rank two and are selfdual, hence we can apply the above construction to $\psi$. The computations
\begin{align*}
\frac{Y^{d_2-b}}{X^a}\cdot \begin{pmatrix} -Y^b \\ X^a \\ 0 \\ Z^{d_3-c} \end{pmatrix}+\frac{Z^{d_3-c}}{X^a}\cdot \begin{pmatrix} -Z^c \\ 0 \\ X^a \\ -Y^{d_2-b} \end{pmatrix} = &  \begin{pmatrix} X^{d_1-a} \\ Y^{d_2-b} \\ Z^{d_3-c} \\ 0 \end{pmatrix},\\
-\frac{X^{d_1-a}}{Y^b}\cdot \begin{pmatrix} -Y^b \\ X^a \\ 0 \\ Z^{d_3-c} \end{pmatrix}+\frac{Z^{d_3-c}}{Y^b}\cdot \begin{pmatrix} 0 \\ -Z^c \\ Y^b \\ X^{d_1-a} \end{pmatrix} = &  \begin{pmatrix} X^{d_1-a} \\ Y^{d_2-b} \\ Z^{d_3-c} \\ 0 \end{pmatrix}
\end{align*}
show that $\widetilde{\Ima(\psi)}|_U^{\vee}$ is locally generated by the columns two, three and four.
Now we have to find three non-zero elements $F_1$, $F_2$ and $F_3$ such that $(F_1,F_2,F_3)\trans$ belongs to the kernel of $\psi^{\{2,3,4\}}$.

The first row $(-Y^b,-Z^c,0)$ of $\psi^{\{2,3,4\}}$ gives the relation $-Y^b\cdot F_1-Z^c\cdot F_2=0$ from which we conclude (since $Y^b$ and $Z^c$ are coprime in $R$) that $F_1=Z^c\cdot G$ and $F_2=-Y^b\cdot G$ for some non-zero $G\in R$.
From the second row $(X^a,0,-Z^c)$ we get the relation $Z^c\cdot X^a\cdot G-Z^c\cdot F_3=0$. This gives $F_3=X^a\cdot G$. Looking at the relations given by the rows three and four of $\psi^{\{2,3,4\}}$, we see that $G$ 
can be chosen as one. We get the isomorphism
$$\widetilde{\Ima(\psi)}|_U^{\vee}\cong\widetilde{\Ima(\psi^{\{2,3,4\}})}|_U^{\vee}\cong\Syz_U(Z^c,-Y^b,X^a),$$
extending to the global isomorphism
$$\Ima(\psi)^{\vee}\cong\Syz_R(Z^c,-Y^b,X^a).$$
In fact, one can delete any column from $\psi$, giving the isomorphisms
\begin{align*}
 \Ima(\psi)^{\vee} \cong & \Syz_R(Z^c,-Y^b,X^a)\\
 \cong & \Syz_R(Z^c,X^{d_1-a},Y^{d_2-b})\\
 \cong & \Syz_R(Y^b,X^{d_1-a},-Z^{d_3-c})\\
 \cong & \Syz_R(X^a,-Y^{d_2-b},-Z^{d_3-c}).
\end{align*}
In all four cases the ideal $(F_1,F_2,F_3)$ is $\mm$-primary, hence $(F_1,F_2,F_3)\trans$ generates the kernel of $\psi^J$ and $\Det(\widetilde{\coKer(\phi)}|_U)$ is always trivial.
\end{exa}

In the next sections we give a complete list of representatives of reduced, indecomposable matrix factorizations up to equivalence in the ADE case 
and compute first syzygy modules of ideals isomorphic to the maximal Cohen-Macaulay modules given by the matrix factorizations. Proofs for the 
completeness of the lists of matrix factorizations can be found in \cite{matfacade} and \cite[Chapter 9, \S 4]{leuwie}. We use their enumeration of 
the matrix factorizations, which corresponds to the enumeration of the corresponding Dynkin-diagrams. In all cases the matrices are various 
$2n\times 2n$- matrices and the corresponding modules have rank $n$. Note that for an $R$-module $M$ we will denote the corresponding $\Oc_{\Spec(R)}$-module again by $M$.
\section{The case $A_n$ with $n\geq 0$}
Let $k$ be an algebraically closed field of odd characteristic not dividing $n+1$ and let 
$$R:=k\llbracket X,Y,Z\rrbracket /(X^{n+1}+YZ)$$ 
with maximal ideal $\mm$. Denote by $U:=\punctured{R}{\mm}$ the punctured spectrum of $R$.
The category of reduced, indecomposable matrix factorizations up to equivalence has the objects $(\phi_m,\psi_m)$ for $m\in\{1,\ldots,n\}$ with
$$(\phi_m,\psi_m)=
\left(\begin{pmatrix}
Y & X^{n+1-m}\\
X^m & -Z
\end{pmatrix},
\begin{pmatrix}
Z & X^{n+1-m}\\
X^m & -Y
\end{pmatrix}\right).$$
We always have $\Rank(\coKer(\phi_m))=1$. The modules
$$\Syz_R(Y,X^{n+1-m})\qquad\text{and}\qquad \Syz_R(X^m,-Z)$$
are generated by the columns of $\psi_m$, hence they are isomorphic to $\coKer(\phi_m)$. Note that $(\psi_m,\phi_m)\cong (\phi_{n+1-m},\psi_{n+1-m})$ by the morphism 
$$\left(\begin{pmatrix}0&1\\-1&0\end{pmatrix},\begin{pmatrix}0&-1\\1&0\end{pmatrix}\right)$$
and $\coKer(\phi_m)^{\vee}=\coKer(\phi_{n+1-m})$, because $(\phi_m^T,\psi_m^T)=(\phi_{n+1-m},\psi_{n+1-m})$. Hence, the Picard-group 
of $R$ is isomorphic to $\Z/(n+1)$.

Since the $\Oc_U$-modules $\Syz_U(X^m,Z)$ are line bundles, they have a description as the line bundle of a reflexive ideal. From the short exact sequence 
of $\Oc_U$-modules 
$$0\ra\Syz_U(X^m,Z)\ra \Oc_U^2\ra (X^m,Z)|_U\ra 0,$$ 
we get by taking determinants $$(X^m,Z)|_U\cong\Syz_U(X^m,Z)^{\vee}\cong\Syz_U(X^m,Y).$$

Taking global sections $\Gamma(U,\_)$ gives the isomorphism
$$\Syz_R(X^m,Y)\cong (X^m,Z)^{\vee\vee}.$$
It remains to compute the reflexive hulls of the ideals $(X^m,Z)$. The reflexive hull of a finitely generated module over a normal ring is the intersection over all localizations of 
that module at primes of height one (cf. the corollary to Theorem 1 in Chapter VII, \S 4, no. 2 of \cite{bourbaki}). 
All $(X^m,Z)$ have exactly one minimal prime (of height one), namely $(X,Z)$, hence we get
$$(X^m,Z)^{\vee\vee}=\{f\in R| f\in(X^m,Z)R_{(X,Z)}\}=\{f\in R| f\in(X^m)R_{(X,Z)}\},$$
where the last equality holds, since $Y$ gets invertible and $Z=-X^{n+1}/Y\in(X^m)R_{(X,Z)}$.

We show by induction on $m$, that the ideals $(X^m,Z)$ are reflexive. This is clear for $m=1$, since $(X,Z)$ is a prime ideal of height one.
Let $m>1$ and assume $(X^l,Z)^{\vee\vee}=(X^l,Z)$ for all $l<m$. The inclusion $(X^m,Z)\subseteq(X^m,Z)^{\vee\vee}$ is clear. 
Let $f\in (X,Z)$, such that $f\in (X^m)R_{(X,Z)}$. Then there are $a,b\in R$ with $f=aX+bZ$. 
Since $bZ\in (X^m)R_{(X,Z)}$, this gives $f-bZ=aX\in(X^m)R_{(X,Z)}$. From this we deduce $a\in(X^{m-1})R_{(X,Z)}$, 
which means $a\in (X^{m-1},Z)$ by the induction hypothesis.

\begin{thm}\label{syzan}
The pairwise non-isomorphic modules $M_m:=\Syz_R(X^m,Z)$ for $m=1,\ldots,n$ give a complete list of representatives of the isomorphism classes of indecomposable, non-free, maximal Cohen-Macaulay modules. 
Moreover, we have $M_m^{\vee}\cong M_{n+1-m}$ and $M_m\cong (X^m,Y)$.
\end{thm}

\section{The case $D_n$ with $n\geq 4$}
Let $k$ be an algebraically closed field of odd characteristic not dividing $n-2$, let 
$$R:=k\llbracket X,Y,Z\rrbracket/(X^2+Y^{n-1}+YZ^2)$$ 
with maximal ideal $\mm$. Denote by $U:=\punctured{R}{\mm}$ the punctured spectrum of $R$.
In this case the category of reduced, indecomposable matrix factorizations up to equivalence has $n$ objects $(\phi_i,\psi_i)$, $i\in\{1,\ldots,n\}$, given by
\begin{align*}
\phi_1=\psi_1= &
\begin{pmatrix}
X & Y^{n-2}+Z^2\\
Y & -X
\end{pmatrix},
\end{align*}
\begin{align*}
\phi_{n-1}=\psi_{n-1}= &
\begin{pmatrix}
X & Y(iY^{(n-2)/2}+Z)\\
Z-iY^{(n-2)/2} & -X
\end{pmatrix}\quad\text{and}\\
\phi_n=\psi_n= &
\begin{pmatrix}
X & Y(-iY^{(n-2)/2}+Z)\\
Z+iY^{(n-2)/2} & -X
\end{pmatrix}\quad\text{if }n\text{ is even,}\\
\phi_{n-1}=&
\begin{pmatrix}
X+iY^{(n-1)/2} & YZ\\
Z & -X+iY^{(n-1)/2}
\end{pmatrix},\\
\psi_{n-1}= &
\begin{pmatrix}
X-iY^{(n-1)/2} & YZ\\
Z & -X-iY^{(n-1)/2}
\end{pmatrix}\quad\text{and}\\
(\phi_n,\psi_n)=&
(\psi_{n-1},\phi_{n-1})\quad\text{if }n\text{ is odd.}
\end{align*}
Independent of the parity of $n$ we have
$$\phi_m=\psi_m =
\begin{pmatrix}
-X & 0 & YZ & Y^{m/2}\\
0 & -X & Y^{n-1-m/2} & -Z\\
Z & Y^{m/2} & X & 0\\
Y^{n-1-m/2} & -YZ & 0 & X
\end{pmatrix}$$
if $m\in\{2,\ldots,n-2\}$ is even, and
$$\phi_m=\psi_m =
\begin{pmatrix}
-X & Y^{n-1-(m-1)/2} & YZ & 0\\
Y^{(m-1)/2} & X & 0 & -Z\\
Z & 0 & X & Y^{n-2-(m-1)/2}\\
0 & -YZ & Y^{(m+1)/2} & -X
\end{pmatrix}$$ 
if $m\in\{2,\ldots,n-2\}$ is odd.

The corresponding modules $\coKer(\phi_m)$ have rank one for $m\in\{1,n-1,n\}$ and rank two otherwise. Moreover, $\coKer(\phi_1)$ is selfdual, 
but $\coKer(\phi_{n-1})$ and $\coKer(\phi_n)$ are selfdual only when $n$ is even. For $n$ odd they are dual to each other. In any case the 
equivalences are given by $(\alpha,-\alpha)$ with
 $$\alpha:=\begin{pmatrix}0&-1\\1&0\end{pmatrix}.$$

The modules $\coKer(\phi_m)$ are selfdual, since $(\eta,\xi):(\phi_m,\psi_m)\lra(\phi_m\trans,\psi_m\trans)$ with
\begin{align*}
\eta=\xi &= \begin{pmatrix}-\alpha & 0\\ 0 & \alpha\end{pmatrix} \quad \text{if }m\text{ is even and}\\
\eta=-\xi &= \begin{pmatrix}\alpha & 0\\ 0 & \alpha\end{pmatrix} \quad \text{if }m\text{ is odd}
\end{align*}
are equivalences of matrix factorizations. This shows that the Picard-group of $R$ is isomorphic to $\Z/(2)\times\Z/(2)$ if $n$ is even and isomorphic 
to $\Z/(4)$ if $n$ is odd.

We now compute representing first syzygy modules of ideals. In the rank one case we only have to look at the columns of $\psi_j$. The columns of 
$\psi_1$ generate the modules $\Syz_R(-Y,X)$ and $\Syz_R(X,Y^{n-2}+Z^2)$. Similarly, for $n$ even we get
$$\begin{aligned}
\coKer(\phi_{n-1}) & \cong \Syz_R(-iY^{(n-2)/2}+Z,-X) && \cong \Syz_R\left(X,Y\left(Z+iY^{(n-2)/2}\right)\right) \qquad\text{and}\\
\coKer(\phi_n) & \cong \Syz_R(iY^{(n-2)/2}+Z,-X) && \cong \Syz_R\left(X,Y\left(Z-iY^{(n-2)/2}\right)\right)
\end{aligned}$$
and for $n$ odd we get
$$\begin{aligned}
\coKer(\phi_{n-1}) & \cong \Syz_R\left(-Z,X+iY^{(n-1)/2}\right) &&\cong \Syz_R(X-iY^{(n-1)/2},YZ) \qquad\text{and}\\
\coKer(\phi_n) & \cong \Syz_R\left(-Z,X-iY^{(n-1)/2}\right) &&\cong \Syz_R(X+iY^{(n-1)/2},YZ).
\end{aligned}$$
The ideals $(X,Y)$, $(X,Z\pm iY^{(n-2)/2})$ and $(X\pm iY^{(n-1)/2},Z)$ are prime ideals of height one, hence reflexive and we get the ideal representations
$$\begin{aligned}
   \coKer(\phi_1) &\cong (X,Y), && \\
   \coKer(\phi_{n-1}) &\cong (X,Z-iY^{(n-2)/2}) && \text{and}\\
   \coKer(\phi_n) &\cong (X,Z+iY^{(n-2)/2}) && \text{if }n\text{ is even,}\\
   \coKer(\phi_{n-1}) &\cong (Z,X-iY^{(n-2)/2}) && \text{and}\\
   \coKer(\phi_n) &\cong (Z,X+iY^{(n-2)/2}) && \text{if }n\text{ is odd.}
  \end{aligned}$$
Dealing with the rank two case, we have to omit one column from $\psi_m$ (we can work with $\psi_m$ instead of $\psi_m\trans$, since $\coKer(\phi_m)$ 
is selfdual) such that the rank of the $\Oc_U$-module corresponding to the image of the reduced matrix remains two.

If we erase the first or fourth column in $\psi_m$ for $m$ even, the two-minors of the new matrix vanish in the point $(0,0,1)\in U$. Hence, we cannot 
omit these columns, since we would get a sheaf whose rank in $(0,0,1)$ would be at most one.
The following computation shows, that we can omit the second column of $\psi_m$ and still have a generating set for $\Ima(\psi_m)|_U$:
\begin{align*}
 \frac{Y^{m/2}}{X}\cdot \begin{pmatrix} YZ \\ Y^{n-1-m/2} \\ X \\ 0 \end{pmatrix}-\frac{YZ}{X}\cdot \begin{pmatrix} Y^{m/2} \\ -Z \\ 0 \\ X \end{pmatrix} & = \begin{pmatrix} 0 \\ -X \\ Y^{m/2} \\-YZ \end{pmatrix}\\
 \frac{Y^{m/2}}{Z}\cdot \begin{pmatrix} -X \\ 0 \\ Z \\ Y^{n-1-m/2} \end{pmatrix}+\frac{X}{Z}\cdot \begin{pmatrix} Y^{m/2} \\ -Z \\ 0 \\ X \end{pmatrix} & = \begin{pmatrix} 0 \\ -X \\ Y^{m/2} \\-YZ \end{pmatrix}.
\end{align*}
The rows of $\psi_m^{\{1,3,4\}}$ generate the $R$-module $\Syz_R(-X,Z,Y^{n-1-m/2})$, hence we get an isomorphism of this syzygy module with $\coKer(\phi_m)$. If we choose the coefficients 
$$\left(-\frac{YZ}{X},-\frac{Y^{n-1-m/2}}{X},0\right)\qquad \text{and} \qquad \left(\frac{X}{Z},0,-\frac{Y^{n-1-m/2}}{Z}\right)$$
on $\open(X)$ resp. $\open(Z)$ for the columns one, two, four, we see that we can leave out the third column. Then the rows of $\psi_m^{\{1,2,4\}}$ 
generate $\Syz_R(Y^{m/2},-Z,X)$, which gives an isomorphism of this syzygy module with $\coKer(\phi_m)$. Both syzygy modules are syzygy modules of 
$\mm$-primary ideals, hence the determinant of $\coKer(\phi_m)|_U$ is trivial.

Now, only the case where $m$ is odd is left. As in the case where $m$ is even, erasing the first or fourth column in $\psi_m$ gives a matrix with all 
two-minors vanishing in the point $(0,0,1)\in U$. Again we may omit the second or the third column, where the corresponding coefficients are given in 
the table below (each row of the table is a syzygy for the columns of $\psi_m$ on the local piece $\open(X)$ or $\open(Z)$).
$$\begin{array}{c|cccc}
   & 1 & 2 & 3 & 4 \\ \hline
\open(X) & -\frac{Y^{n-1-\frac{m-1}{2}}}{X} & -1 & 0 & \frac{YZ}{X}\\
\open(Z) & 0 & -1 & \frac{Y^{n-1-\frac{m-1}{2}}}{Z} & -\frac{X}{Z}\\
\open(X) & -\frac{YZ}{X} & 0 & -1 & -\frac{Y^{\frac{m+1}{2}}}{X}\\
\open(Z) & \frac{X}{Z} & -\frac{Y^{\frac{m-1}{2}}}{Z} & -1 & 0
  \end{array}$$
The rows of $\psi_m^{\{1,3,4\}}$ generate $\Syz_R(YZ,X,Y^{(m+1)/2})$ and the rows of $\psi_m^{\{1,2,4\}}$ generate the module $\Syz_R(Y^{n-1-(m-1)/2},X,-YZ)$. The ideals are not $\mm$-primary, hence the determinant of $\Dc\coloneqq\coKer(\phi_m)|_U$ is not trivial. 
Since the $R$-module $\coKer(\phi_m)$ is selfdual, the determinant $\Dc$ has to be selfdual, too. Hence, if $n$ is odd, we get directly $$\Dc\cong\coKer(\phi_1)|_U.$$
If $n$ is even, we need to compute the reflexive hull of $(X,Y^{(m+1)/2},YZ)$. This ideal has only the minimal prime $(X,Y)$. Since in the localization at $(X,Y)$ 
the equality $$(X,Y^{(m+1)/2},YZ)_{(X,Y)}=(X,Y)_{(X,Y)}$$ holds, the reflexive hull of $(X,Y^{(m+1)/2},YZ)$ is $(X,Y)$ and we obtain $$\Dc\cong\coKer(\phi_1)|_U.$$

\begin{thm}\label{syzdn}
The pairwise non-isomorphic modules 
$$\begin{aligned}
M_1 &= \Syz_R(X,Y), &&\\
M_m &= \Syz_R(X,Y^{m/2},Z) && \text{if }m\in\{2,\ldots,n-2\}\text{ is even},\\
M_m &= \Syz_R(X,Y^{(m+1)/2},YZ) && \text{if }m\in\{2,\ldots,n-2\}\text{ is odd},\\
M_{n-1} &= \Syz_R(X,Z-iY^{(n-2)/2}) && \text{and}\\
M_n &= \Syz_R(X,Z+iY^{(n-2)/2}) && \text{if }n\text{ is even, or}\\
M_{n-1} &= \Syz_R(Z,X+iY^{(n-1)/2}) && \text{and}\\
M_n &= \Syz_R(Z,X-iY^{(n-1)/2}) && \text{if }n\text{ is odd}
\end{aligned}$$
give a complete list of representatives of the isomorphism classes of indecomposable, non-free, maximal Cohen-Macaulay modules. Moreover, these modules 
are selfdual with the exception $M_{n-1}^{\vee}\cong M_n$ if $n$ is odd. The determinants of $M_{n-1}|_U$ and $M_n|_U$ are $M_1|_U$. 
For the rank one modules we have the isomorphisms $M_1\cong (X,Y)$, $M_{n-1}\cong (X,Z-iY^{(n-2)/2})$ and $M_n\cong (X,Z+iY^{(n-2)/2})$ if $n$ is even and 
$M_{n-1}\cong (Z,X-iY^{(n-1)/2})$ and $M_n\cong (Z,X+iY^{(n-1)/2})$ if $n$ is odd.
\end{thm}

\section{The case $E_6$}
Let $k$ be an algebraically closed field of characteristic at least five, let 
$$R:=k\llbracket X,Y,Z\rrbracket/(X^2+Y^3+Z^4)$$
with maximal ideal $\mm$. Denote by $U:=\punctured{R}{\mm}$ the punctured spectrum of $R$.
In this case the category of reduced, indecomposable matrix factorizations up to equivalence has six objects, which are given by
$$\phi_5=\psi_6=
\begin{pmatrix}
-Z^2+iX & Y\\
Y^2 & Z^2+iX
\end{pmatrix}\qquad\text{and}\qquad
\phi_6=\psi_5=
\begin{pmatrix}
-Z^2-iX & Y\\
Y^2 & Z^2-iX
\end{pmatrix},$$
\begin{align*}
 \phi_1=\psi_1 &=
\begin{pmatrix}
-X & 0 & Y^2 & Z^3\\
0 & -X & Z & -Y\\
Y & Z^3 & X & 0\\
Z & -Y^2 & 0 & X
\end{pmatrix},\\
\phi_3=\psi_4 &=
\begin{pmatrix}
-Z^2+iX & 0 & YZ & Y\\
-YZ & Z^2+iX & Y^2 & 0\\
0 & Y & iX & Z\\
Y^2 & -YZ & Z^3 & iX
\end{pmatrix},\\
\phi_4=\psi_3 &=
\begin{pmatrix}
-Z^2-iX & 0 & YZ & Y\\
-YZ & Z^2-iX & Y^2 & 0\\
0 & Y & -iX & Z\\
Y^2 & -YZ & Z^3 & -iX
\end{pmatrix},\\
\phi_2 &=
\begin{pmatrix}
-iX & -Z^2 & YZ & 0 & Y^2 & 0\\
-Z^2 & -iX & 0 & 0 & 0 & Y\\
0 & 0 & -iX & -Y & 0 & Z\\
0 & YZ & -Y^2 & -iX & Z^3 & 0\\
Y & 0 & 0 & Z & -iX & 0\\
0 & Y^2 & Z^3 & 0 & YZ^2 & -iX
\end{pmatrix},\\
\psi_2 &=
\begin{pmatrix}
iX & -Z^2 & YZ & 0 & Y^2 & 0\\
-Z^2 & iX & 0 & 0 & 0 & Y\\
0 & 0 & iX & -Y & 0 & Z\\
0 & YZ & -Y^2 & iX & Z^3 & 0\\
Y & 0 & 0 & Z & iX & 0\\
0 & Y^2 & Z^3 & 0 & YZ^2 & iX
\end{pmatrix},
\end{align*}
where $\coKer(\phi_j)$ has rank one for $j\in\{5,6\}$, rank two for $j\in\{1,3,4\}$ and rank three for $j=2$. 
At first we clarify the dualities. The modules $\coKer(\phi_5)$ and $\coKer(\phi_6)$ are dual to each other by the equivalence 
$(\alpha,-\alpha):(\phi_5,\psi_5)\ra(\phi_6^T,\psi_6^T)$ with
$$\alpha=\begin{pmatrix}0&-1\\1&0\end{pmatrix}.$$
Since the Picard-group of $R$ is given by the isomorphism classes of $R$, $\coKer(\phi_5)$ and $\coKer(\phi_6)$, it is isomorphic to $\Z/(3)$.
The modules $\coKer(\phi_3)$ and $\coKer(\phi_4)$ are dual to each other by the equivalence $(\eta,-\eta):(\phi_3,\psi_3)\ra(\phi_4^T,\psi_4^T)$ with
$$\eta=\begin{pmatrix}-\alpha & 0\\ 0 & -\alpha\end{pmatrix}.$$
The modules $\coKer(\phi_1)$ and $\coKer(\phi_2)$ are the only modules of rank two resp. three left, hence they are selfdual.

Now we compute representations as first syzygy modules of ideals.

By looking at the columns of $\psi_5$ and $\psi_6$ (since they give modules of rank one) we get
$$\begin{aligned}
 \coKer(\phi_5) & \cong \Syz_R(Y^2,iX+Z^2) && \cong \Syz_R(Z^2-iX,-Y)\\
 \coKer(\phi_6) & \cong \Syz_R(Y^2,Z^2-iX) && \cong \Syz_R(Z^2+iX,-Y)
\end{aligned}$$
The two ideals $(\pm iX+Z^2,Y)$ are prime ideals of height one, hence reflexive. We obtain the ideal representations
\begin{align*}
\coKer(\phi_5) & \cong (Z^2+iX,Y)\\
 \coKer(\phi_6) & \cong (Z^2-iX,Y).
\end{align*}
Dealing with the matrix factorizations $(\phi_j,\psi_j)$ corresponding to modules of rank two, we have to delete one column from $\psi_j$.

Since $(\phi_1,\psi_1)$ is equivalent to the matrix factorization from Example \ref{diagrank2matfac} with the values $a=b=c=1$ and $d_i=1+i$ for the parameters, we get the following representations as first syzygy modules of $\mm$-primary ideals:
\begin{align*}
\coKer(\phi_1) & \cong \Syz_R(X,Y^2,Z^3)\\ 
 & \cong \Syz_R(X,Y,Z)\\
 & \cong \Syz_R(X,Y,Z^3)\\
 & \cong \Syz_R(X,Y^2,Z).
\end{align*}
Fix a point of the form $(x,0,z)\in U$, hence either $ix+z^2=0$ or $ix-z^2=0$. Deleting the first column from $\psi_4$, all two-minors of the reduced matrix vanish iff $z^2+ix=0$. If $z^2=ix$ and the rows of $\psi_4^{\{2,3,4\}}$ generate $\Syz_R(F_1,F_2,F_3)$, 
we must have $F_3=-Z\cdot F_2$ by the first row. The third row gives the condition $0=Y\cdot F_1+(iX-Z^2)\cdot F_2=Y\cdot F_1$, hence $F_1=0$. This is a contradiction, since $\Syz_R(0,F_2,F_3)$ splits. Deleting the second column from $\psi_4$ or the first column from $\psi_3$ if $ix=-z^2$ or the second column from $\psi_3$ if $ix=z^2$, one gets similar contradictions. By deleting the third columns we get isomorphisms
$$\Ima(\psi_3)\cong\Syz_R(-iX+Z^2,YZ,-Y^2)\qquad\text{and}\qquad\Ima(\psi_4)\cong\Syz_R(iX+Z^2,YZ,-Y^2).$$
Deleting the fourth columns we get the isomorphisms
$$\Ima(\psi_3)\cong\Syz_R(iXZ-Z^3,iXY,Y^2)\qquad\text{and}\qquad\Ima(\psi_4)\cong\Syz_R(iXZ+Z^3,iXY,-Y^2).$$
Since $\coKer(\phi_3)$ and $\coKer(\phi_4)$ are dual to each other, we get the isomorphisms
$$\begin{aligned}
 \coKer(\phi_3) & \cong \Syz_R(iX+Z^2,YZ,-Y^2)\\
 & \cong \Syz_R(iXZ+Z^3,iXY,-Y^2),\\
 \coKer(\phi_4) & \cong \Syz_R(-iX+Z^2,YZ,-Y^2)\\
 & \cong \Syz_R(iXZ-Z^3,iXY,Y^2),
\end{aligned}$$
where the appearing ideals are not $\mm$-primary. The ideals $(\pm iX+Z^2,Y^2,YZ)$ have only one minimal prime, namely $(\pm iX+Z^2,Y)$.
Since in the localization at $(iX+Z^2,Y)$ the equality $$(iX+Z^2,Y^2,YZ)_{(iX+Z^2,Y)}=(iX+Z^2,Y)_{(iX+Z^2,Y)}$$ holds (and similarly for the negative sign),
the minimal primes have to be the reflexive hulls. Hence, we obtain
\begin{align*}
 \Det(\coKer(\phi_3)|_U) &\cong \coKer(\phi_5)|_U\\
 \Det(\coKer(\phi_4)|_U) &\cong \coKer(\phi_6)|_U.
\end{align*}
Considering $\psi_2$ we have to erase two columns, since $\coKer(\phi_2)$ has rank three. Not all choices of two columns are possible, since there are 
choices $\{m,l\}$ such that all three-minors of the matrix $\psi^{\{1,\ldots,6\}\setminus\{m,l\}}$ vanish. The non-possible choices $\{m,l\}$ for columns 
of $\psi_2$ and points showing that these choices are not possible are given in the next table.
$$\begin{array}{c|c}
   \text{deleted columns} & \text{point} \\ \hline
 \{1,2\}\vee \{3,6\}\vee \{4,5\} & (i,0,1)\\
 \{1,5\}\vee \{2,6\}\vee \{3,4\} & (i,1,0)\\
 \{1,4\}\vee \{1,6\}\vee \{4,6\} & (0,-1,1).
  \end{array}$$
Deleting all other pairs of columns is possible and yields isomorphisms
$$\Ima(\psi_2)\cong \Syz_R(F_1,\ldots,F_4),$$
where the polynomials are given as follows
$$\begin{array}{c|cccc}
   \text{deleted columns} & F_1 & F_2 & F_3 & F_4 \\ \hline
1,3 & Y^2 & -iXZ & Z^2 & -iXY \\
2,3 & Y^2 & Z^3 & -iX & YZ^2 \\
2,4 & iXY & -Z^3 & -Y^2 & iXZ^2 \\
2,5 & YZ & -iX & -Y^2 & Z^3 \\
3,5 & Z^2 & iX & -YZ & -Y^2 \\
5,6 & iXZ & Z^3 & -Y^2 & -iXY.
  \end{array}$$
Note that all the ideals $(F_1,F_2,F_3,F_4)$ are $\mm$-primary.

\begin{thm}\label{syze6}
The pairwise non-isomorphic modules 
\begin{align*}
M_1 &= \Syz_R(X,Y,Z),\\
M_2 &= \Syz_R(X,Y^2,YZ,Z^2),\\
M_3 &= \Syz_R(iX+Z^2,Y^2,YZ),\\
M_4 &= \Syz_R(-iX+Z^2,Y^2,YZ),\\
M_5 &= \Syz_R(-iX+Z^2,Y),\\
M_6 &= \Syz_R(iX+Z^2,Y)
\end{align*}
give a complete list of representatives of the isomorphism classes of indecomposable, non-free, maximal Cohen-Macaulay modules. Moreover, $M_1$ and $M_2$ are 
selfdual, $M_3^{\vee}\cong M_4$ and $M_5^{\vee}\cong M_6$. The Determinants of $M_3|_U$ resp. $M_4|_U$ are $M_5|_U$ resp. $M_6|_U$. 
For the rank one modules we have the isomorphisms $M_5\cong (iX+Z^2,Y)$ and $M_6\cong (-iX+Z^2,Y)$.
\end{thm}

\section{The case $E_7$}
Let $k$ be an algebraically closed field of characteristic at least five, let 
$$R:=k\llbracket X,Y,Z\rrbracket/(X^2+Y^3+YZ^3)$$
with maximal ideal $\mm$. Denote by $U:=\punctured{R}{\mm}$ the punctured spectrum of $R$.
In this case the category of reduced, indecomposable matrix factorizations up to equivalence has seven objects, which are given by
\begin{align*}
\phi_7=\psi_7 &=
\begin{pmatrix}
X & Y\\
Y^2+Z^3 & -X
\end{pmatrix},\\
\phi_1=\psi_1 &=
\begin{pmatrix}
X & 0 & -Y^2 & Z\\
0 & X & YZ^2 & Y\\
-Y & Z & -X & 0\\
YZ^2 & Y^2 & 0 & -X
\end{pmatrix},\\
\phi_4=\psi_4 &=
\begin{pmatrix}
-X & Z^2 & 0 & Y\\
YZ & X & -Y^2 & 0\\
0 & -Y & -X & Z\\
Y^2 & 0 & YZ^2 & X
\end{pmatrix},\\
\phi_6=\psi_6 &=
\begin{pmatrix}
X & 0 & -YZ & Y\\
0 & X & Y^2 & Z^2\\
-Z^2 & Y & -X & 0\\
Y^2 & YZ & 0 & -X
\end{pmatrix},\\
\phi_2=\psi_2 &=
\begin{pmatrix}
-X & Z^2 & YZ & 0 & Y^2 & 0\\
YZ & X & 0 & 0 & 0 & -Y\\
0 & 0 & X & -Y & 0 & Z\\
0 & -YZ & -Y^2 & -X & YZ^2 & 0\\
Y & 0 & 0 & Z & X & 0\\
0 & -Y^2 & YZ^2 & 0 & Y^2Z & -X
\end{pmatrix},\\
\phi_5=\psi_5 &=
\begin{pmatrix}
-X & 0 & YZ & 0 & 0 & Y\\
-YZ & X & 0 & -Z^2 & -Y^2 & 0\\
Z^2 & 0 & X & -Y & YZ & 0\\
0 & -YZ & -Y^2 & -X & 0 & 0\\
0 & -Y & 0 & 0 & -X & -Z\\
Y^2 & 0 & 0 & YZ & -YZ^2 & X
\end{pmatrix},\\
\phi_3=\psi_3 &=
\begin{pmatrix}
-X & 0 & YZ & -Z^2 & 0 & 0 & Y^2 & 0\\
0 & -X & 0 & Z^2 & 0 & 0 & 0 & Y\\
Z^2 & Z^2 & X & 0 & 0 & -Y & 0 & 0\\
0 & YZ & 0 & X & -Y^2 & 0 & 0 & 0\\
0 & 0 & 0 & -Y & -X & 0 & 0 & Z\\
0 & 0 & -Y^2 & 0 & 0 & -X & YZ^2 & Z^2\\
Y & 0 & 0 & 0 & -Z^2 & Z & X & 0\\
0 & Y^2 & 0 & 0 & YZ^2 & 0 & 0 & X
\end{pmatrix},
\end{align*}
where $\coKer(\phi_j)$ has rank one for $j=7$, rank two for $j\in\{1,4,6\}$, rank three for $j\in\{2,5\}$ and rank four for $j=3$. Since there is only one 
non-trivial isomorphism class of indecomposable, maximal Cohen-Macaulay modules of rank one, the Picard-group of $R$ is isomorphic to $\Z/(2)$.

Again we start by computing the dual modules. Clearly, $\coKer(\phi_7)$ and $\coKer(\phi_3)$ are selfdual. We will see that all other modules are 
selfdual, too. Let 
$$\alpha:=\begin{pmatrix}0 & 1\\ -1 & 0\end{pmatrix}.$$
The morphisms $(\alpha_i,\beta_i):(\phi_i,\psi_i)\ra(\phi_i\trans,\psi_i\trans)$ given by
\begin{align*}
 \alpha_1=\beta_1 &= \begin{pmatrix}\alpha & 0\\ 0 & \alpha \end{pmatrix},\\
 \alpha_4=-\beta_4 &= \begin{pmatrix}-\alpha & 0\\ 0 & \alpha \end{pmatrix},\\
 \alpha_6=\beta_6 &= \begin{pmatrix}\alpha & 0\\ 0 & -\alpha \end{pmatrix},\\
 \alpha_2=-\beta_2 &= \begin{pmatrix}\alpha & 0 & 0\\ 0 & \alpha & 0 \\ 0 & 0 & \alpha\end{pmatrix},\\
 \alpha_5=-\beta_5 &= \begin{pmatrix}\alpha & 0 & 0\\ 0 & \alpha & 0 \\ 0 & 0 & -\alpha\end{pmatrix}
\end{align*}
are equivalences of matrix factorizations.

Computing representations as first syzygy modules of ideals, we start again with the rank one case. The columns of $\psi_7$ generate $\Syz_R(X,Y)$ and $\Syz_R(Y^2+Z^3,-X)$. Therefore $\coKer(\phi_7)$ is isomorphic to both of these syzygy modules.

In the rank two cases, we have to delete one column from the matrices $\psi_j$.

Deleting the second or fourth column from $\psi_1$, we get a matrix with vanishing two-minors in $(0,0,1)$.
By deleting the first resp. the third column, we get isomorphisms
$$\coKer(\phi_1)\cong\Syz_R(X,Z,Y^2)\cong\Syz_R(Z,Y,-X).$$
Deleting the second or fourth column from $\psi_4$, we get a matrix with vanishing two-minors in $(0,0,1)$.
By deleting the first resp. the third column, we get isomorphisms
$$\coKer(\phi_4)\cong\Syz_R(Y^2,X,-YZ^2)\cong\Syz_R(-X,YZ,Y^2).$$
Note that both ideals are not $\mm$-primary (since $\{X,Y\}$ is not a system of parameters).

Deleting the first or fourth column from $\psi_6$, we get a matrix with vanishing two-minors in $(0,0,1)$.
By deleting the second resp. the third column, we get isomorphisms
$$\coKer(\phi_6)\cong\Syz_R(X,-Z^2,Y^2)\cong\Syz_R(Y,Z^2,-X).$$
We now deal with the rank three cases. To represent $\coKer(\phi_2)$, we have to erase two columns from $\psi_2$. If one of those two is an even column, 
we get a matrix with all three-minors vanishing in $(0,0,1)$. In $(1,-1,0)$ all three-minors of the matrix obtained from $\psi_2$ by deleting the columns 
one and five vanish. Deleting the columns one and three resp. three and five, we obtain the isomorphisms
$$\coKer(\phi_2)\cong\Syz_R(Y^2,XZ,-Z^2,XY)\cong\Syz_R(Z^2,X,-YZ,-Y^2).$$
Similarly, to represent $\coKer(\phi_5)$, we have to erase two columns from $\psi_5$. If one of those two is the column one, four or six, we get a 
matrix with all three-minors vanishing in $(0,0,1)$. Deleting the columns two and three, all three-minors vanish in $(0,-1,-1)$ and all three-minors 
vanish in $(1,-1,0)$, when erasing the columns two and five. In this case we get only an isomorphism by deleting the columns three and five:
$$\coKer(\phi_5)\cong\Syz_R(Y^2,-XZ,YZ^2,XY).$$
The ideal is not $\mm$-primary (since $\{XZ,Y\}$ is not a system of parameters).

The last module to consider is $\coKer(\phi_3)$ of rank four, hence we have to omit three columns in $\psi_3$. The following choices of columns are not 
possible, since all four-minors of the resulting matrix vanish in the given point (a question mark indicates an arbitrary choice)
$$\begin{array}{cc}
 \text{deleted columns} & \text{point}\\ \hline
 4,?,? & (0,0,1)\\
 8,?,? & (0,0,1)\\
 1,2,? & (0,0,1)\\
 5,6,? & (0,0,1)\\
 1,7,? & (1,-1,0)\\
 3,6,? & (1,-1,0)\\
 2,5,? & (0,-1,-1).
\end{array}$$
The other four possibilities ((1,3,5), (2,3,7), (2,6,7) and (3,5,7)) give the isomorphisms
\begin{align*}
\coKer(\phi_3) & \cong \Syz_R(-XY,Y^2Z,-XZ^2,Z^3,Y^3)\\
 & \cong \Syz_R(YZ^2,-XY,Y^2+Z^3,Z^4,XZ^2)\\
 & \cong \Syz_R(XZ^2,-Z^4,Y^3,XY,-Y^2Z^2)\\
 & \cong \Syz_R(Z^3,-Y^2-Z^3,-XZ,-YZ^2,-XY).
\end{align*}
Note that the second ideal is not monomial. Multiplying all generators by $Y$ yields a monomial ideal and changes the syzygy module only 
by an isomorphism.

\begin{thm}\label{syze7}
The pairwise non-isomorphic modules 
\begin{align*}
M_1 &= \Syz_R(X,Y,Z),\\
M_2 &= \Syz_R(X,Y^2,YZ,Z^2),\\
M_3 &= \Syz_R(XY,XZ,Y^2,YZ^2,Z^3),\\
M_4 &= \Syz_R(X,Y^2,YZ),\\
M_5 &= \Syz_R(XY,XZ,Y^2,YZ^2),\\
M_6 &= \Syz_R(X,Y,Z^2),\\
M_7 &= \Syz_R(X,Y)
\end{align*}
give a complete list of representatives of the isomorphism classes of indecomposable, non-free, maximal Cohen-Macaulay modules. Moreover, all $M_j$ 
are selfdual and $\Det(M_j|_U)\cong M_7|_U$ for $j\in\{4,5,7\}$. The rank one module $M_7$ is isomorphic to the ideal $(X,Y)$.
\end{thm}

\section{The case $E_8$}
Let $k$ be an algebraically closed field of characteristic at least seven, let
$$R:=k\llbracket X,Y,Z\rrbracket/(X^2+Y^3+Z^5)$$ 
with maximal ideal $\mm$. Denote by $U:=\punctured{R}{\mm}$ the punctured spectrum of $R$.
In this case the category of reduced, indecomposable matrix factorizations up to equivalence has eight objects, given by
\begin{align*}
\phi_1=\psi_1 &=
\begin{pmatrix}
X & 0 & Y & Z\\
0 & X & Z^4 & -Y^2\\
Y^2 & Z & -X & 0\\
Z^4 & -Y & 0 & -X
\end{pmatrix},\\
\phi_8=\psi_8 &=
\begin{pmatrix}
X & 0 & Y & Z^2\\
0 & X & Z^3 & -Y^2\\
Y^2 & Z^2 & -X & 0\\
Z^3 & -Y & 0 & -X
\end{pmatrix},\\
\phi_2=\psi_2 &=
\begin{pmatrix}
X & -Z^2 & YZ & 0 & -Y^2 & 0\\
-Z^3 & -X & 0 & 0 & 0 & Y\\
0 & 0 & -X & Y & 0 & Z\\
0 & -YZ & Y^2 & X & Z^4 & 0\\
-Y & 0 & 0 & Z & -X & 0\\
0 & Y^2 & Z^4 & 0 & -YZ^3 & X
\end{pmatrix},\\
\phi_6=\psi_6 &=
\begin{pmatrix}
-X & 0 & 0 & Z^2 & 0 & Y\\
YZ & X & -Z^3 & 0 & -Y^2 & 0\\
0 & -Z^2 & -X & Y & 0 & 0\\
Z^3 & 0 & Y^2 & X & -YZ^2 & 0\\
0 & -Y & 0 & 0 & -X & Z\\
Y^2 & 0 & -YZ^2 & 0 & Z^4 & X
\end{pmatrix},\\
\phi_3=\psi_3 &=
\begin{pmatrix}
-X & 0 & -YZ & Z^2 & 0 & 0 & Y^2 & 0\\
0 & -X & Z^3 & 0 & 0 & 0 & 0 & Y\\
0 & Z^2 & X & 0 & 0 & -Y & 0 & 0\\
Z^3 & YZ & 0 & X & -Y^2 & 0 & 0 & 0\\
0 & 0 & 0 & -Y & -X & 0 & Z^3 & Z\\
0 & 0 & -Y^2 & 0 & 0 & -X & 0 & Z^2\\
Y & 0 & 0 & 0 & Z^2 & -Z & X & 0\\
0 & Y^2 & 0 & 0 & 0 & Z^3 & 0 & X
\end{pmatrix},
\end{align*}
\begin{align*}
\phi_7=\psi_7 &=
\begin{pmatrix}
X & 0 & 0 & 0 & -Z^3 & 0 & 0 & -Y\\
YZ & -X & 0 & 0 & 0 & Z^2 & Y^2 & 0\\
0 & 0 & -X & Z^2 & 0 & Y & -Z^3 & 0\\
0 & 0 & 0 & X & -Y^2 & 0 & 0 & Z^2\\
-Z^2 & 0 & 0 & -Y & -X & 0 & 0 & 0\\
0 & Z^3 & Y^2 & 0 & YZ^2 & X & 0 & 0\\
0 & Y & -Z^2 & 0 & 0 & 0 & X & Z\\
-Y^2 & 0 & 0 & Z^3 & 0 & 0 & 0 & -X
\end{pmatrix},\\
\phi_4=\psi_4 &=
\begin{pmatrix}
X & 0 & YZ & 0 & 0 & -Z^2 & Z^3 & 0 & -Y^2 & 0\\
0 & -X & 0 & 0 & 0 & 0 & 0 & -Z^2 & 0 & Y\\
0 & 0 & -X & Z^2 & 0 & 0 & 0 & Y & 0 & 0\\
0 & YZ & Z^3 & X & 0 & 0 & -Y^2 & 0 & 0 & 0\\
0 & Z^2 & 0 & 0 & X & -Y & 0 & 0 & Z^3 & 0\\
-Z^3 & 0 & 0 & 0 & -Y^2 & -X & 0 & 0 & 0 & Z^2\\
0 & 0 & 0 & -Y & 0 & 0 & -X & 0 & 0 & Z\\
0 & -Z^3 & Y^2 & 0 & 0 & 0 & YZ^2 & X & 0 & 0\\
-Y & 0 & 0 & 0 & Z^2 & 0 & 0 & Z & -X & 0\\
0 & Y^2 & YZ^2 & 0 & 0 & 0 & Z^4 & 0 & 0 & X
\end{pmatrix},\\
\phi_5=\psi_5 &=
\begin{pmatrix}
-X & 0 & 0 & 0 & 0 & 0 & 0 & Z^2 & 0 & 0 & 0 & Y\\
0 & -X & -YZ & 0 & 0 & 0 & Z^3 & -Z^2 & 0 & 0 & Y^2 & 0\\
0 & 0 & X & 0 & 0 & -Z^2 & 0 & 0 & Z^3 & -Y & 0 & 0\\
YZ & 0 & 0 & X & -Z^3 & 0 & 0 & 0 & -Y^2 & 0 & 0 & 0\\
0 & 0 & 0 & -Z^2 & -X & 0 & 0 & Y & 0 & 0 & 0 & 0\\
0 & 0 & -Z^3 & 0 & 0 & -X & -Y^2 & 0 & 0 & 0 & YZ^2 & Z^2\\
Z^2 & Z^2 & 0 & 0 & 0 & -Y & X & 0 & 0 & 0 & 0 & 0\\
Z^3 & 0 & 0 & 0 & Y^2 & 0 & 0 & X & -YZ^2 & 0 & 0 & 0\\
0 & 0 & 0 & -Y & 0 & 0 & 0 & 0 & -X & 0 & 0 & Z\\
0 & 0 & -Y^2 & -Z^3 & 0 & 0 & YZ^2 & 0 & 0 & -X & -Z^4 & 0\\
0 & Y & 0 & 0 & Z^2 & 0 & 0 & 0 & 0 & -Z & X & 0\\
Y^2 & 0 & 0 & 0 & -YZ^2 & 0 & 0 & 0 & Z^4 & 0 & 0 & X
\end{pmatrix},
\end{align*}
where $\coKer(\phi_j)$ has rank two for $j\in\{1,8\}$, rank three for $j\in\{2,6\}$, rank four for $j\in\{3,7\}$, rank five for $j=4$ and rank six for 
$j=5$. Since there are no non-trivial modules of rank one, the Picard-group of $R$ is trivial and all syzygy modules will be syzygy modules of 
$\mm$-primary ideals.

Note that $(\phi_1,\psi_1)$ and $(\phi_8,\psi_8)$ are equivalent to the matrix factorizations from Example \ref{diagrank2matfac} with $d_1=2$, $d_2=3$, 
$d_3=5$, $a=b=1$ and $c=1$ resp. $c=3$. We get the isomorphisms
\begin{align*}
 \coKer(\phi_1) & \cong \Syz_R(X,Y,Z)\\
 & \cong \Syz_R(X,Y^2,Z^4)\\
 & \cong \Syz_R(X,Y^2,Z)\\
 & \cong \Syz_R(X,Y,Z^4),
\end{align*}
\begin{align*}
\coKer(\phi_8) & \cong \Syz_R(X,Y,Z^2)\\
 & \cong \Syz_R(X,Y^2,Z^3)\\
 & \cong \Syz_R(X,Y^2,Z^2)\\
 & \cong \Syz_R(X,Y,Z^3).
\end{align*}
Now we show that all other modules $\coKer(\phi_j)$ are selfdual. This is obvious for $j=4,5$. Let 
$\alpha:=\left(\begin{smallmatrix}0 & 1\\ -1 & 0\end{smallmatrix}\right).$
Then the morphisms $(\alpha_j,\beta_j):(\phi_j,\psi_j)\lra(\phi_j\trans,\psi_j\trans)$ given by
\begin{align*}
 \alpha_2=-\beta_2 &= \begin{pmatrix}\alpha & 0 & 0\\ 0 & \alpha & 0 \\ 0 & 0 & \alpha\end{pmatrix},\\
 \alpha_6=-\beta_6 &= \begin{pmatrix}-\alpha & 0 & 0\\ 0 & \alpha & 0 \\ 0 & 0 & \alpha\end{pmatrix},\\
 \alpha_3=\beta_3 &= \begin{pmatrix}-\alpha & 0 & 0 & 0\\ 0 & \alpha & 0 & 0 \\ 0 & 0 & -\alpha & 0\\ 0&0&0& -\alpha\end{pmatrix},\\
 \alpha_7=-\beta_7 &= \begin{pmatrix}\alpha & 0 & 0 & 0\\ 0 & -\alpha & 0 & 0 \\ 0 & 0 & \alpha & 0\\ 0&0&0& -\alpha\end{pmatrix}
\end{align*}
are equivalences of matrix factorizations.

We continue with the computations of representations as syzygy modules. 

Considering $\psi_2$ and $\psi_6$ we have to delete two columns. The following choices $\{l,m\}$ are not possible, since they give 
matrices with all three-minors vanishing in the given point:
$$\begin{array}{c|c|c}
   \text{point} & \text{deleted columns from }\psi_2 & \text{deleted columns from }\psi_6\\ \hline
 (0,1,-1) & \{1,4\}\vee \{1,6\}\vee \{4,6\} & \{2,4\}\vee \{2,6\}\vee \{4,6\}\\
 (1,-1,0) & \{1,5\}\vee \{2,6\}\vee \{3,4\} & \{1,6\}\vee \{2,5\}\vee \{3,4\}\\
 (1,0,-1) & \{1,2\}\vee \{3,6\}\vee \{4,5\} & \{1,4\}\vee \{2,3\}\vee \{5,6\}.
  \end{array}$$
All other choices give isomorphisms $\coKer(\phi_2)\cong\Syz_R(F_1,F_2,F_3,F_4)$ with
$$\begin{array}{ccccc}
 \text{deleted columns} & F_1 & F_2 & F_3 & F_4\\ \hline
1,3 & Y^2 & -XZ & -Z^2 & XY\\
2,3 & Y^2 & -Z^4 & X & YZ^3\\
2,4 & XY & Z^4 & -Y^2 & XZ^3\\
2,5 & YZ & -X & Y^2 & Z^4\\
3,5 & Z^2 & X & YZ & -Y^2\\
5,6 & XZ & -Z^4 & Y^2 & XY
\end{array}$$
and $\coKer(\phi_6)\cong\Syz_R(F_1,F_2,F_3,F_4)$ with
$$\begin{array}{ccccc}
 \text{deleted columns} & F_1 &F_2 &F_3 &F_4\\ \hline
1,2 & -Y^2 & -XY & Z^3 & XZ^2\\
1,3 & Y^2 & YZ^2 & X & -Z^4\\
1,5 & Z^3 & X & -Y^2 & YZ^2\\
3,5 & -X & YZ & Z^3 & Y^2\\
3,6 & Z^4 & XY & XZ^2 & -Y^2\\
4,5 & Y^2 & XZ & -Z^3 & XY.
\end{array}$$
We get isomorphisms $\coKer(\phi_3)\cong\Syz_R(F_1,\ldots,F_5)$ with
$$\begin{array}{cccccc}
  \text{deleted columns} & F_1 & F_2 & F_3 & F_4 & F_5\\ \hline
1,3, 5 & XY &-Y^2Z &XZ^2 &Z^3 &X^2\\
1,5,8 & XZ^3 &X^2 &-YZ^4 &-XY^2 &-Y^2Z\\
2,4,7 & Y^2Z &-XY &Z^4 &-X^2 &XZ^3\\
2,5,7 & XZ &Y^2 &-Z^4 &XY &-YZ^3\\
3,5,7 & Z^3 &Y^2 &XZ &YZ^2 &XY\\
4,6,7 & YZ^3 &X^2 &-XZ^2 &-Y^2Z &-XY^2.
\end{array}$$
We can represent $\coKer(\phi_7)$ as $\Syz_R(F_1,\ldots,F_5)$ with
$$\begin{array}{cccccc}
  \text{deleted columns} & F_1 & F_2 & F_3 & F_4 & F_5\\ \hline
1,2,3 & -X^2 &XY &-Y^2Z^2 &Z^4 &-XZ^3\\
1,3,7 & Z^4 &-XY &Y^2 &XZ^2 &-YZ^3\\
2,3,4 & X^2 &-XZ^2 &YZ^4 &Y^2Z &-XY^2\\
3,5,7 & Y^2 &-XZ &-YZ^2 &Z^4 &XY\\
5,6,7 & XY &Y^2Z &-Z^4 &-XZ^2 &X^2\\
6,7,8 & XZ^3 &YZ^4 &Y^2Z^2 &XY^2 &X^2.
\end{array}
$$
In all other cases the resulting matrix has vanishing four-minors in (at least) one of the ``test points'' $(0,1,-1)$, $(1,0,-1)$ and $(1,-1,0)$.

Similarly, we have $\coKer(\phi_4)\cong\Syz_R(F_1,\ldots,F_6)$ with
$$\begin{array}{ccccccc}
  \text{deleted columns} & F_1 & F_2 & F_3 & F_4 & F_5 & F_6\\ \hline
1,2,3,5 & -XY^2 &Z^6 &-X^2 &XYZ^2 &YZ^3 &XZ^4\\
1,2,3,6 & -Y^4 &Z^6 &-XY^2 &Y^3Z^2 &-XZ^3 &Y^2Z^4\\
1,2,4,6 & Y^4 &XZ^4 &Y^2Z^3 &XY^3 &-X^2Z &XY^2Z^2\\
1,2,5,7 & -X^2 &XZ^3 &YZ^4 &XY^2 &Y^2Z &XYZ^2\\
1,3,5,7 & X^2 &-XYZ &-Y^2Z^2 &XZ^3 &Z^4 &-XY^2\\
3,5,7,9 & Z^4 &XY &-Y^2Z &XZ^2 &YZ^3 &-Y^3\\
3,5,9,10 & -XZ^3 &YZ^4 &XY^2 &Z^6 &-Y^3 &-XYZ^2\\
3,6,9,10 & Y^2Z^3 &XZ^4 &X^2Y &-Z^6 &-XY^2 &-X^2Z^2\\
6,7,8,9 & -YZ^4 &XY^2 &XZ^3 &X^2Z &-Y^2Z^2 &X^2Y.
\end{array}
$$
In all other cases the resulting matrix has vanishing five-minors in one of the test points.

Finally, we can represent $\coKer(\phi_5)$ as $\Syz_R(F_1,\ldots,F_7)$ with
$$\begin{array}{cccccccc}
  \text{deleted columns} & F_1 & F_2 & F_3 & F_4 & F_5 & F_6 & F_7\\ \hline
1,3,4,6,11 & -X^2Z^2 &-Y^4 &XZ^4 &-XY^3 &Y^2Z^3 &YZ^6 &XY^2Z^2\\
1,3,5,6,11 & -XZ^4 &-Y^4 &Z^6 &-Y^3Z^2 &-XY^2 &-XYZ^3 &Y^2Z^4\\
1,3,5,7,11 & -YZ^4 &XY^2 &-Z^6 &XYZ^2 &X^2 &-Y^2Z^3 &-XZ^4\\
1,3,7,9,11 & Y^2Z^2 &XZ^3 &X^2 &YZ^4 &-XY^2 &-Z^6 &XYZ^2\\
1,7,9,10,11 & XYZ^2 &Z^6 &-Y^2Z^3 &-XY^2 &XZ^4 &Y^4 &-Y^3Z^2\\
2,3,5,7,9 & -XY^2 &Y^3Z &-XYZ^2 &Y^2Z^3 &XZ^4 &Z^5 &Y^4\\
2,3,5,7,12 & Y^2Z^4 &XY^3 &YZ^6 &XY^2Z^2 &-Y^4 &X^2Z^3 &XZ^4\\
3,5,7,9,11 & X^2 &Z^5 &-XYZ &-Y^2Z^2 &-XZ^3 &YZ^4 &-XY^2\\
7,8,9,10,11 & Y^4 &YZ^5 &-XZ^4 &XY^2Z &-Y^2Z^3 &-X^2Z^2 &XY^3.
\end{array}
$$
In all other cases the resulting matrix has vanishing six-minors in one of test points.

\begin{thm}\label{syze8}
The pairwise non-isomorphic modules 
\begin{align*}
M_1 &= \Syz_R(X,Y,Z),\\
M_2 &= \Syz_R(X,Y^2,YZ,Z^2),\\
M_3 &= \Syz_R(XY,XZ,Y^2,YZ^2,Z^3),\\
M_4 &= \Syz_R(XY,XZ^2,Y^3,Y^2Z,YZ^3,Z^4),\\
M_5 &= \Syz_R(XY^2,XYZ^2,XZ^4,Y^4,Y^3Z,Y^2Z^3,Z^5),\\
M_6 &= \Syz_R(X,Y^2,YZ,Z^3),\\
M_7 &= \Syz_R(XY,XZ,Y^2,YZ^2,Z^4),\\
M_8 &= \Syz_R(X,Y,Z^2)
\end{align*}
give a complete list of representatives of the isomorphism classes of indecomposable, non-free, maximal Cohen-Macaulay modules. Moreover, all $M_j$ are selfdual.
\end{thm}

\section{An idea of a purely algebraic construction of the isomorphisms}
In this section we explain an idea how to find the isomorphisms from the previous sections with purely algebraic methods.

Given a matrix factorization $(\phi,\psi)$ of size $n$ and rank $m$, we want an isomorphism from $\coKer(\phi)$ to $\Syz_R(I)$, where $I:=(F_1,\ldots,F_{m+1})$ is an ideal (with unknown generators).
Let $f: R^{m+1}\lra R$ be the map $\langle(F_1,\ldots,F_{m+1})\trans,\_\rangle$. Our goal is to find the $F_i$.
Since we want to restrict the map $\phi$ to $R^{m+1}$, the idea is to factor $f$ as $f=p\circ\phi\circ\iota$ with linear maps $\iota: R^{m+1}\lra R^n$ and $p:R^n\lra R$. We restrict ourselves to inclusions $\iota$ that send the standard base vectors of $R^{m+1}$ to (pairwise different) standard base vectors of $R^n$.
To simplify notation, lets assume $\Ima(\iota)=R^{m+1}\times\{0\}^{n-m-1}$.
In this case $\iota\trans$ is right-inverse to $\iota$, giving $p\phi=f\iota\trans$. But $f\iota\trans$ is $f$ on the first $m+1$ components and zero on the last $n-1-m$ components.
Therefore $p$ has to be a syzygy for the last $n-1-m$ columns of $\phi$ and for the first $m+1$ columns of $\phi$ we get $p\phi^{c,i}=f_i$, where $\phi^{c,i}$ denotes the $i$-th column of $\phi$ ($i\in\{1,\ldots,m+1\}$).
From the commutative diagram
$$\xymatrix{
0\ar[r] & \Ker(\phi)\ar[r] & R^n\ar[r]^{\phi}\ar[d]^{\iota\trans} & R^n\ar[d]^p\\
0\ar[r] & \Ker(f)\ar[r] & R^{m+1}\ar[r]^f & R
}$$
we get an induced map $\Ker(\phi)\lra\Ker(f)$, which is in fact an inclusion. Then one has to show that this inclusion is surjective, which is equivalent to the condition that $\Ker(\phi)\lra\Ker(f)$ splits, since $\Ker(\phi)$ is indecomposable and has the same rank as $\Ker(f)$.

\begin{exa}
Let $R:=k[X,Y,Z]/(X^{d_1}+Y^{d_2}+Z^{d_3})$ with $d_i\in\N_{\geq 2}$ and consider the matrix
$$\phi=\begin{pmatrix}
X^a & Y^b & Z^c & 0\\
-Y^{d_2-b} & X^{d_1-a} & 0 & Z^c\\
-Z^{d_3-c} & 0 & X^{d_1-a} & -Y^b\\
0 & -Z^{d_3-c} & Y^{d_2-b} & X^a
\end{pmatrix}$$
from Example \ref{diagrank2matfac}. Choosing $p:=(1,0,0,0)$ and
$$\iota:=\begin{pmatrix}
1 & 0 & 0\\       
0 & 1 & 0\\
0 & 0 & 1\\
0 & 0 & 0
\end{pmatrix},$$
we get $p\phi\iota=(X^a,Y^b,Z^c)$. Therefore we get the inclusion $\Ker(\phi)\lra\Syz_R(X^a,Y^b,Z^c)$. We have to show that this map splits. That means
we have to show that every $(A,B,C)\in\Syz_R(X^a,Y^b,Z^c)$ extends uniquely to an element $s\in\Ker(\phi)$, which has to be of the form $(A,B,C,D)$, 
since $\iota\trans(s)=(A,B,C)$ (this already shows that there is at most one extension). 

Multiplying $AX^a+BY^b+CZ^c=0$ by $-Y^{d_2-b}$ gives
$$\begin{aligned}
 & -AX^aY^{d_2-b}-BY^{d_2}-CY^{d_2-b}Z^c && =0\\
\Longleftrightarrow & -AX^aY^{d_2-b}+B(X^{d_1}+Z^{d_3})-CY^{d_2-b}Z^c && =0\\
\Longleftrightarrow & X^a(-AY^{d_2-b}+BX^{d_1-a})-Z^c(-BZ^{d_3-c}+CY^{d_2-b}) && =0.
\end{aligned}$$
Since $X^a$ and $Z^c$ are coprime in $R$ this gives $$-AY^{d_2-b}+BX^{d_1-a}=-Z^cD\qquad\text{and}\qquad -BZ^{d_3-c}+CY^{d_2-b}=-X^aD$$
for some $D\in R$.

Similarly, multiplying $AX^a+BY^b+CZ^c=0$ by $-Z^{d_3-c}$ gives
$$\begin{aligned}
 & -AX^aZ^{d_3-c}-BY^bZ^{d_3-c}-CZ^{d_3} && =0\\
\Longleftrightarrow & -AX^aY^{d_2-b}+BY^bZ^{d_3-c}+C(X^{d_1}+Y^{d_2}) && =0\\
\Longleftrightarrow & X^a(-AY^{d_2-b}+CX^{d_1-a})+Y^b(-BZ^{d_3-c}+CY^{d_2-b}) && =0\\
\Longleftrightarrow & X^a(-AY^{d_2-b}+CX^{d_1-a})-X^aY^bD && =0.
\end{aligned}$$
The last line is equivalent to $-AY^{d_2-b}+CX^{d_1-a}=Y^bD$, since $X^a$ is a non-zero divisor. Now we can extend $(A,B,C)\in\Syz_R(X^a,Y^b,Z^c)$ uniquely to $(A,B,C,D)\in\Ker(\phi)$.
\end{exa}

Note that $\phi$ and $\phi\iota$ had the same rank in this example. But this is not true in general.

\begin{exa}
Let $R:=k[X,Y,Z]/(X^2+Y^{n-1}+YZ^2)$ and consider
$$\phi_m:=\begin{pmatrix}
-X & 0 & YZ & Y^{m/2}\\
0 & -X & Y^{n-1-m/2} & -Z\\
Z & Y^{m/2} & X & 0\\
Y^{n-1-m/2} & -YZ & 0 & X
\end{pmatrix}$$
for $m$ even.
In this case we may choose $p$ to be either $(1,0,0,0)$ or $(0,0,0,1)$. In these cases $\iota$ is the map $(r_1,r_2,r_3)\mapsto (0,r_1,r_2,r_3)$ resp. $(r_1,r_2,r_3)\mapsto(r_1,r_2,r_3,0)$.
For the first pair $(p,\iota)$ we get $p\phi_m\iota=(-X,Y^{n-1-m/2},-Z)$ and for the second pair we get $p\phi_m\iota=(Z,Y^{m/2},X)$.
In both cases we already saw that $\Ker(\phi_m)\cong\coKer(\phi_m)$ is isomorphic to the first syzygy modules of these ideals.
But the rank of $\phi_m\iota$ in the point $(0,0,1)$ is only one for both choices.
Choosing $\iota$ to embed $R^3$ inside $R^4$ outside the second resp. the third component (in these cases the rank of $\phi_m\iota$ would be two in every point), one gets $f=(-X,YZ,Y^{m/2})$ resp. $f=(Y^{n-1-m/2},-YZ,X)$, whose first syzygy modules
are not isomorphic to $\Ker(\phi_m)$: In the first case we have to find a $B$ which extends $(0,Y^{m/2-1},-Z)$ to $(0,B,Y^{m/2-1},-Z)\in\Ker(\phi_m)$. The third row of $\phi_m$ gives the condition
$BY^{m/2}+Y^{m/2-1}X=0$, which is equivalent to $X=-BY$. This is a contradiction. Similarly, if the syzygy $(Z,Y^{n-2-m/2},0)$ in the second case would extend to $(Z,Y^{n-2-m/2},C,0)\in\Ker(\phi_m)$, the second row of $\phi_m$ gives $-XY^{n-2-m/2}+CY^{n-1-m/2}=0$, which is equivalent to $X=CY$.
\end{exa}

The last example already showed that this approach causes a lot of troubles:

\begin{enumerate}
 \item There is no canonical choice for $\iota$. Even the ranks of $\phi$ and $\phi\iota$ might be different in some points.
 \item To find $p$ one has to find a simultaneous syzygy for up to five column vectors of length up to twelve.
 \item Even if one finds possible maps $\iota$ and $p$, the induced inlusion $\Ker(\phi)\lra\Syz_R(f)$ might not be surjective.
\end{enumerate}

All those disadvantages have analogues in the geometric approach from the third section. But in the geometric approach these analogues are advantages 
(at least in all our explicit examples):

\begin{enumerate}
 \item For every choice of columns such that the rank of the reduced matrix remains the same, one gets a representation.
 \item The generators of the ideal can be computed from the rows of the reduced matrices.
 \item There are no ``false friends``, meaning that everything that looks like a solution is in fact a solution.
\end{enumerate}

\section{A comment on the graded situation}
Since we want to use the results from this chapter to control the Frobenius pull-backs of the sheaves $\Syz_C(X,Y,Z)$, where $C\coloneqq\Proj(R)$, we need a graded version of our results. Therefore we summarize briefly Chapter 15 of \cite{yobook} (the original paper is \cite{ausreit}).

Let $R$ be a positively-graded, Cohen-Macaulay ring with $R_0=k$ a field. Denote by $\mm$ the graded maximal ideal and by $\hat{R}$ the $\mm$-adic completion of $R$. We will need the following categories.
$$\begin{array}{lcc}
\text{symbol} & \text{objects} & \text{morphisms}\\ \hline
\mathfrak{grC}(R) & \text{finitely generated graded }R\text{-modules} & \text{degree-preserving homomorphisms}\\
\mathfrak{grM}(R) & \text{graded maximal C-M }R\text{-modules} & \text{degree-preserving homomorphisms}\\
\mathfrak{C}(\hat{R}) & \text{finitely generated } \hat{R}\text{-modules} & \text{homomorphisms}\\
\mathfrak{M}(\hat{R}) & \text{maximal C-M }\hat{R}\text{-modules} & \text{homomorphisms}
\end{array}$$
We say that $R$ is of \textit{finite graded Cohen-Macaulay type} if there are - up to degree shift - only finitely many isomorphism classes of graded, indecomposable, non-free, maximal Cohen-Macaulay modules. The goal is to find a relation between the (graded) indecomposable, maximal Cohen-Macaulay modules over $R$ and over $\hat{R}$.

Recall that taking the $\mm$-adic completion $\hat{-}$ is a functor from $\mathfrak{grM}(R)$ to $\mathfrak{M}(\hat{R})$ that restricts to a functor from $\mathfrak{grC}(R)$ to $\mathfrak{C}(\hat{R})$. 

We get the following lemma (cf. \cite[Lemma 15.2]{yobook}).

\begin{lem}
\begin{enumerate}
\item If $M\in\mathfrak{grM}(R)$ is indecomposable, then $\hat{M}\in\mathfrak{M}(\hat{R})$ is indecomposable.
\item If $M,N\in\mathfrak{grM}(R)$ are indecomposable such that $\hat{M}\cong\hat{N}$ in $\mathfrak{M}(\hat{R})$, then $M\cong N$ in $\mathfrak{grM}(R)$ up to degree shift.
\end{enumerate}
\end{lem}

An immediate consequence of the previous lemma is, that $R$ is of finite graded Cohen-Macaulay type if $\hat{R}$ is of finite Cohen-Macaulay type. Moreover, the graded 
analogous of the Theorems \ref{syzan} - \ref{syze8} still give a full list of representatives of the isomorphism classes of graded, indecomposable, 
non-free, maximal Cohen-Macaulay $R$-modules. But so far we might have repetitions in the lists. At least if $R_0=k$ is perfect, we can get rid of this. 
First we need a definition.

\begin{defi}
\begin{enumerate}
\item We call a finitely generated $\hat{R}$-module $M$ \textit{graduable} if there exists a finitely generated graded $R$-module $A$ such that $\hat{A}\cong M$.
\item An $\hat{R}$-homomorphism $f:M\ra N$ with $M$, $N$ graduable is called a \textit{graduable homomorphism} if there is a graded homomorphism of finitely generated graded $R$-modules $A$, $B$, such that the following diagram of $\hat{R}$-modules and homomorphisms commutes.
$$\xymatrix{
M\ar[r]^{\cong}\ar[d]^f & \hat{A}\ar[d]^{\hat{g}}\\
N\ar[r]^{\cong} & \hat{B}.
}$$
\end{enumerate}
\end{defi}

The next theorem, proven by Auslander and Reiten, gives the promised converse.

\begin{thm}[Auslander, Reiten]\label{ausreitgraded}
Suppose $R_0=k$ is perfect. If $R$ is of finite graded Cohen-Macaulay type, then $\hat{R}$ is of finite Cohen-Macaulay type. Moreover, in this case all maximal 
Cohen-Macaulay $\hat{R}$-modules are graduable.
\end{thm}

\begin{proof}See \cite[Theorem 15.14]{yobook}.\end{proof}

As a consequence of the last theorem, we obtain that the following definition makes sence.

\begin{defi} Let $S=k[X_1,\ldots,X_m]$ be positively-graded with $k$ a perfect field. Let $f\in S_+^2$ be homogeneous of degree $d$. Then a \textit{graded matrix factorization} for $f$ of size $m$ 
is a pair of $m\times m$ matrices $\phi$, $\psi$ with entries in $S$, such that $\phi\circ\psi=\psi\circ\phi=f\cdot\id_m$ and the morphisms given by $\phi$ resp. $\psi$ are homogeneous of degree zero resp. $d$. A morphism from a matrix factorization into another 
is again a pair of matrices $(\alpha,\beta)$ such that the corresponding diagram commutes and the morphism given by $\alpha$ and $\beta$ are homogeneous of degree zero.
\end{defi}

By Theorem \ref{ausreitgraded}, the isomorphisms $\coKer(\phi)\cong\Syz_R(F_1,\ldots,F_{m+1})$ of (ungraded) $R$-modules constructed in the last sections, 
are graduable in the sence that they hold in the graded situation up to a degree shift. But this shift cannot be computed with the definition of a graded matrix factorization, 
since the modules $\coKer(\phi)$ have no ``natural'' grading. We illustrate this by an example.

\begin{exa}\label{nonaturalgrading}
Consider the hypersurface $R:=k[X,Y,Z]/(X^n+YZ)$ of type $A_{n-1}$ with $n\geq 1$, graded by $\Deg(X)=2$ and $\Deg(Y)=\Deg(Z)=n$. Let $S:=k[X,Y,Z]$.
From the matrix factorizations 
$$(\phi_m,\psi_m)=\left(\begin{pmatrix} Y & X^{n-m}\\ X^m & -Z\end{pmatrix},\begin{pmatrix}Z & X^{n-m}\\ X^m & -Y\end{pmatrix}\right)$$
we obtained the isomorphisms 
\begin{equation}\label{syzrepr}
\begin{array}{ccl}
\coKer(\phi_m) &\cong & \Syz_R(X^m,-Z)\\
 &\cong& \Syz_R(Y,X^{n-m}).
\end{array}
\end{equation}
For all $\eta\in\Z$ the sequence
$$S(-\eta-2m+n)\oplus S(-\eta)\stackrel{\psi_m}{\lra}S(-\eta-2m+2n)\oplus S(-\eta+n)\stackrel{\phi_m}{\lra}S(-\eta-2m+n)\oplus S(-\eta)$$
shows that $(\phi_m,\psi_m)$ are graded matrix factorizations for $X^n+YZ$. Reducing the above sequence by $X^n+YZ$, we obtain the short exact sequence
$$0\ra\Ima(\psi_m)\ra R(-\eta-2m+n)\oplus R(-\eta)\ra \Ima(\phi_m)\ra 0.$$
The presenting sequences of the syzygy modules in (\ref{syzrepr}) are
$$\xymatrix{
   0\ar[r] & \Syz_R(X^m,-Z)\ar[r] & R(-2m)\oplus R(-n)\ar[r] & (X^m,Z)\ar[r] & 0,\\
   0\ar[r] & \Syz_R(Y,X^{n-m})\ar[r] & R(-n)\oplus R(-2n+2m)\ar[r] & (Y,X^{n-m})\ar[r] & 0.
}$$
Choosing $\eta=n$, we obtain $\coKer(\phi_m)\cong\Syz_R(X^m,-Z)$ as graded modules and $\coKer(\phi_m)\cong\Syz_R(Y,X^{n-m})$ is an isomorphism of graded modules, 
if we choose $\eta=2n-2m$. 

Since every value of $\eta$ is as good as any other, the ungraded isomorphisms $\coKer(\phi)\cong\Syz_R(F_1,\ldots,F_{m+1})$ have no natural graded analogues.
\end{exa}

\chapter[Hilbert-series of first syzygy modules of monomial ideals]{The Hilbert-series of first syzygy modules of \except{toc}{certain} monomial ideals \except{toc}{in $\N$-graded rings of the form $k[X,Y_1,\ldots,Y_n]/(X^d-F(Y_1,\ldots,Y_n))$}}\label{syzses}
The main results of this chapter are the Theorems \ref{ses} and \ref{HilbSer}, which allow us to compute the Hilbert-series of certain syzygy modules. 
In the second section we will compute the Hilbert-series of the syzygy modules representing the isomorphism classes of the non-free, indecomposable, 
maximal Cohen-Macaulay modules over surface rings of type ADE.
We will see that in many cases these Hilbert-series carry enough information to detect the isomorphism class of a given indecomposable, maximal Cohen-Macaulay module. 
It contains also enough information to exclude many possible splitting behaviours of a given maximal Cohen-Macaulay module. Note that we will always work 
with the description of the Hilbert-series as rational functions und call these rational functions again Hilbert-series.

\section{A short exact sequence}
The goal of this section is to provide an invariant for the modules $\Syz_R(F_1,\ldots,F_n)$ from Theorems \ref{syzan}-\ref{syze8}. This invariant should be 
efficiently computable for the modules $\Syz_R(F_1^q,\ldots,F_n^q)$ at least in the case where $n=3$ and the $F_i$ are monomials. One idea was to use the Hilbert-series as invariant. But how to 
compute it for various powers $p^e$ of various primes? The solution to this question was inspired by a work of Brenner who showed that on a standard-graded projective curve $C:=\Proj(R)$ 
with $R:=k[X,Y,Z]/(X^d-F(Y,Z))$ all generators of 
$\Syz_{C}(X^a,Y^b,Z^c)$ (with $a,b,c\in\N_{\geq 1}$) come from $(\Prim^1_k)^2=(\Proj(k[Y,Z]))^2$ (cf. \cite[Lemma 1.1]{miyaoka}). We give the precise statement.

\begin{lem}
Let $k$ be a field and let $F(Y,Z)\in k[Y,Z]$ be a homogeneous polynomial of degree $d$. Assume that the projective curve $C:=\Proj(k[X,Y,Z]/(X^d-F(Y,Z)))$ is smooth. Let $a,b,c\in\N_{\geq 1}$ and let $a=dq+r$ with $0\leq r<d$ and $q\in\N$. Then there exists for every $m\in\Z$ a surjective morphism
$$\Syz_C(F^q,Y^b,Z^c)(m-r)\oplus\Syz_C(F^{q+1},Y^b,Z^c)(m)\ra\Syz_C(X^a,Y^b,Z^c)(m).$$
\end{lem}

Let $k$ be a field and $$R:=k[X,Y_1,\ldots,Y_n]/(X^d-F(Y_1,\ldots,Y_n))$$ with $\Deg(X)=\alpha$, $\Deg(Y_i)=\beta_i$ and $F$ homogeneous of degree 
$d\alpha$ ($d$, $n$, $\alpha$, $\beta_i\in\N_{\geq 1}$). Let $V_1,\ldots,V_{m+l}$ be monomials in the $Y_i$ with $m,l\geq 1$ and let $a\in\N_{\geq 1}$.
The goal of this section is to compute the Hilbert-series of $R$-modules of the form
$$\Syz_R(X^a\cdot V_1,\ldots,X^a\cdot V_m,V_{m+1},\ldots,V_{m+l}).$$

\begin{defi}
With the previous notations we define for $j\in\N$ the module
$$\Sc_j:=\Syz_R(F^j\cdot V_1,\ldots,F^j\cdot V_m,V_{m+1},\ldots,V_{m+l}).$$
\end{defi}

\begin{lem}
Let $a=d\cdot q +r$ with $0\leq r\leq d-1$ and $q\in\N$. Then the homogeneous map
$$\phi_s:\Sc_q(s-\alpha\cdot r)\dirsum\Sc_{q+1}(s)\longrightarrow \Syz_R(X^a\cdot V_1,\ldots,X^a\cdot V_m,V_{m+1},\ldots,V_{m+l})(s),$$
which sends $(f_1,\ldots,f_{m+l}),(g_1,\ldots,g_{m+l})$ to
$$(f_1+X^{d-r}\cdot g_1,\ldots,f_m+X^{d-r}\cdot g_m,X^r\cdot f_{m+1}+g_{m+1},\ldots,X^r\cdot f_{m+l}+g_{m+l})$$
is surjective for all $s\in\Z$.
\end{lem}

\begin{proof}
We check that the restriction of $\phi_s$ to each direct summand is \mbox{well-defined}. Let $(f_1,\ldots,f_{m+l})\in \Sc_q(s-\alpha\cdot r)$. The computation
$$\begin{aligned}
 & \text{ } f_1\cdot X^a\cdot V_1+\ldots+f_m\cdot X^a\cdot V_m+X^r\cdot f_{m+1}\cdot V_{m+1}+\ldots+ X^r\cdot f_{m+l}\cdot V_{m+l}&& \\
= & \text{ } X^r\cdot(f_1\cdot X^{dq}\cdot V_1+\ldots+f_m\cdot X^{dq}\cdot V_m+f_{m+1}\cdot V_{m+1}+\ldots+f_{m+l}\cdot V_{m+l})&& \\ 
= & \text{ } X^r\cdot(f_1\cdot F^q\cdot V_1+\ldots+f_m\cdot F^q\cdot V_m+f_{m+1}\cdot V_{m+1}+\ldots+f_{m+l}\cdot V_{m+l})&&=0
\end{aligned}$$
shows that the restriction of $\phi_s$ to $\Sc_q(s-\alpha\cdot r)$ is well-defined (it has the correct degree, since every summand of the term in brackets in the last line is homogeneous of degree $s-\alpha\cdot r$ and $X^r$ has degree $\alpha\cdot r$).

Now let $(g_1,\ldots,g_{m+l})\in \Sc_{q+1}(s)$. In this case we compute
$$\begin{aligned}
 & \text{ } X^{d-r}\cdot g_1\cdot X^a\cdot V_1+\ldots+X^{d-r}\cdot g_m\cdot X^a\cdot V_m+g_{m+1}\cdot V_{m+1}+\ldots+g_{m+l}\cdot V_{m+l}&&\\
= & \text{ } g_1\cdot F^{q+1}\cdot V_1+\ldots+g_m\cdot F^{q+1}\cdot V_m+g_{m+1}\cdot V_{m+1}+\ldots+g_{m+l}\cdot V_{m+l} &&= 0,
\end{aligned}$$
which shows that the restriction of $\phi_s$ to $\Sc_{q+1}(s)$ is well-defined (it has the correct degree, since every summand in the last line has degree $s$).

We now check the surjectivity. Let
$$h:=(h_1,\ldots,h_{m+l})\in \Syz_R(X^a\cdot V_1,\ldots,X^a\cdot V_m,V_{m+1},\ldots,V_{m+l})(s).$$
We write
$$h_i  :=  h_{i,0}+h_{i,1}X+\ldots+h_{i,d-2}X^{d-2}+h_{i,d-1}X^{d-1}$$
with $h_{i,j}\in k[Y_1,\ldots,Y_n]$ for $i=1,\ldots,m+l$ and $j=0,\ldots,d-1$. From the equation 
$$h_1\cdot X^a\cdot V_1+\ldots+h_m\cdot X^a\cdot V_m+h_{m+1}\cdot V_{m+1}+\ldots+h_{m+l}\cdot V_{m+l}=0$$
we get a system of equations by considering the $k[Y_1,\ldots,Y_n]$-coefficients corresponding to $X^j$ for $j=0,\ldots,d-1$. Explicitly, these equations are
\begin{equation}\begin{split}
0 &= h_{1,\sigma(j)}\cdot X^{\sigma(j)}\cdot X^a\cdot V_1+\ldots+h_{m,\sigma(j)}\cdot X^{\sigma(j)}\cdot X^a\cdot V_m\\
&\quad +h_{m+1,j}\cdot X^j\cdot V_{m+1}+\ldots+ h_{m+l,j}\cdot X^j\cdot V_{m+l},
\end{split}
\label{syzsummand}
\end{equation}
where $\sigma(j)\in\{0,\ldots,d-1\}$ with $\sigma(j)\equiv j-r$ modulo $d$. Since we have
$$h=\sum_{j=0}^{d-1}(h_{1,\sigma(j)}\cdot X^{\sigma(j)},\ldots,h_{m,\sigma(j)}\cdot X^{\sigma(j)},h_{m+1,j}\cdot X^{j},\ldots,h_{m+l,j}\cdot X^{j}),$$
we may assume $h$ to be one of those summands for a fixed $j$, hence 
$$h:=(h_{1,\sigma(j)}\cdot X^{\sigma(j)},\ldots,h_{m,\sigma(j)}\cdot X^{\sigma(j)},h_{m+1,j}\cdot X^j,\ldots,h_{m+l,j}\cdot X^j).$$
We will show that $h$ comes either from $\Sc_{q}(s-\alpha\cdot r)$ or from $\Sc_{q+1}(s)$.

For $j<r$, we have $\sigma(j)=j-r+d$. Factoring out $X^{j}$ in equation \eqref{syzsummand} and using $X^a=F^qX^r$, we get
\begin{align*}
\begin{split}
 0 & = X^{j}\cdot (h_{1,\sigma(j)}\cdot X^{d-r}\cdot X^r\cdot F^q\cdot V_1+\ldots+h_{m,\sigma(j)}\cdot X^{d-r}\cdot X^r\cdot F^q\cdot V_m\\
&\quad +h_{m+1,j}\cdot V_{m+1}+\ldots+h_{m+l,j}\cdot V_{m+l})
\end{split}\\
 &= X^{j}\cdot (h_{1,\sigma(j)}\cdot F^{q+1}\cdot V_1+\ldots+h_{m,\sigma(j)}\cdot F^{q+1}\cdot V_m+h_{m+1,j}\cdot V_{m+1}+\ldots+h_{m+l,j}\cdot V_{m+1}).
\end{align*}
This shows that $$X^{j}(h_{1,\sigma(j)},\ldots,h_{m,\sigma(j)},h_{m+1,j},\ldots,h_{m+l,j})$$
belongs to $\Sc_{q+1}(s)$. Under $\phi_s$ it is mapped to $h$.

For $j\geq r$, we have $\sigma(j)=j-r$. Again by using $X^a=F^qX^r$, we get from equation \eqref{syzsummand}
\begin{align*}
\begin{split}
 0 & = h_{1,\sigma(j)}\cdot X^{\sigma(j)}\cdot X^r\cdot F^q\cdot V_1+\ldots+h_{m,\sigma(j)}\cdot X^{\sigma(j)}\cdot X^r\cdot F^q\cdot V_m\\
&\quad + h_{m+1,j}\cdot X^{\sigma(j)+r}\cdot V_{m+1}+\ldots+h_{m+l,j}\cdot X^{\sigma(j)+r}\cdot V_{m+l}
\end{split}\\
\begin{split}
 &= X^r\cdot (h_{1,\sigma(j)}\cdot X^{\sigma(j)}\cdot F^q\cdot V_1+\ldots+h_{m,\sigma(j)}\cdot X^{\sigma(j)}\cdot F^q\cdot V_m\\
&\quad +h_{m+1,j}\cdot X^{\sigma(j)}\cdot V_{m+1}+\ldots+h_{m+l,j}\cdot X^{\sigma(j)}\cdot V_{m+l}).
\end{split}
\end{align*}
This shows that $$X^{\sigma(j)}\cdot(h_{1,\sigma(j)},\ldots,h_{m,\sigma(j)},h_{m+1,j},\ldots,h_{m+l,j})$$
belongs to $\Sc_{q}(s-\alpha\cdot r)$. Under $\phi_s$ it is mapped to $h$.
\end{proof}

\begin{lem}
The homogeneous map 
$$\psi_s:\Syz_R(X^{a+d-2r}\cdot V_1,\ldots,X^{a+d-2r}\cdot V_m,V_{m+1},\ldots,V_{m+l})(s-\alpha\cdot r)\longrightarrow\Sc_q(s-\alpha\cdot r)\dirsum\Sc_{q+1}(s),$$
which sends $(h_1,\ldots,h_{m+l})$ to
$$(X^{d-r}\cdot h_1,\ldots,X^{d-r}\cdot h_m,h_{m+1},\ldots,h_{m+l}),(-h_1,\ldots,-h_m,-X^r\cdot h_{m+1},\ldots,-X^r\cdot h_{m+l})$$
is injective for all $s\in\Z$.
\end{lem}

\begin{proof}
Clearly, these maps are injective, provided they exist. The power $X^{a+d-2r}$ is well-defined, since $a+d-2r=dq+r+d-2r=dq+d-r$ is positive for $r<d$. Let
$$h:=(h_1,\ldots,h_{m+l})\in \Syz_R(X^{a+d-2r}\cdot V_1,\ldots,X^{a+d-2r}\cdot V_m,V_{m+1}\ldots,V_{m+l})(s-\alpha\cdot r).$$
Then the first component of the image of $h$ satisfies
$$\begin{aligned}
& X^{d-r}\cdot h_1\cdot F^q\cdot V_1+\ldots+X^{d-r}\cdot h_m\cdot F^q\cdot V_m+h_{m+1}\cdot V_{m+1}+\ldots +h_{m+l}\cdot V_{m+l} &&\\
=& h_1\cdot X^{a+d-2r}\cdot V_1+\ldots+h_m\cdot X^{a+d-2r}\cdot V_m+h_{m+1}\cdot V_{m+1}+\ldots +h_{m+l}\cdot V_{m+l} &&=0.
\end{aligned}$$
The second component of the image of $h$ satisfies (use $F^q=X^{a-r}$)
$$\begin{aligned}
& -h_1\cdot F^{q+1}\cdot V_1-\ldots-h_m\cdot F^{q+1}\cdot V_m -X^r\cdot h_{m+1}\cdot V_{m+1}-\ldots -X^r\cdot h_{m+l}\cdot V_{m+l} && \\
=& -X^r\cdot(h_1\cdot X^{a+d-2r}\cdot V_1+\ldots+h_m\cdot X^{a+d-2r}\cdot V_m+h_{m+1}\cdot V_{m+1}+\ldots +h_{m+l}\cdot V_{m+l})&& =0.
\end{aligned}$$
\end{proof}

\begin{thm}\label{ses}
For all $s\in\Z$ we have a short exact sequence
\begin{align*}
0 & \lra \Syz_R(X^{a+d-2r}\cdot V_1,\ldots,X^{a+d-2r}\cdot V_m,V_{m+1},\ldots,V_{m+l})(s-\alpha\cdot r)\\
 & \stackrel{\psi_s}{\lra} \Sc_q(s-\alpha\cdot r)\dirsum\Sc_{q+1}(s)\\
 & \stackrel{\phi_s}{\lra} \Syz_R(X^a\cdot V_1,\ldots,X^a\cdot V_m,V_{m+1},\ldots,V_{m+l})(s)\lra 0.
\end{align*}
\end{thm}

\begin{proof}
By the two previous lemmatas we only have to check the exactness at the middle spot.
At first we show $\phi_s\circ\psi_s=0$. With $$h:=(h_1,\ldots,h_{m+l})\in \Syz_R(X^{a+d-2r}\cdot V_1,\ldots,X^{a+d-2r}\cdot V_m,V_{m+1},\ldots,V_{m+l})(s-\alpha\cdot r)$$ we have
\begin{align*}
\begin{split}
 \phi_s(\psi_s(h)) & = \phi_s((X^{d-r}\cdot h_1,\ldots,X^{d-r}\cdot h_m,h_{m+1},\ldots,h_{m+l}),\\
&\quad (-h_1,\ldots,-h_m,-X^r\cdot h_{m+1},\ldots,-X^r\cdot h_{m+l}))
\end{split}\\
\begin{split}
&= (X^{d-r}\cdot h_1+X^{d-r}\cdot(-h_1),\ldots,X^{d-r}\cdot h_m+X^{d-r}\cdot(-h_m),\\
&\quad X^r\cdot h_{m+1}+(-X^r\cdot h_{m+1}),\ldots,X^r\cdot h_{m+l}+(-X^r\cdot h_{m+l}))
\end{split}\\
&= (0,\ldots,0).
\end{align*}
Let $t:=(f_1,\ldots,f_{m+l}),(g_1,\ldots,g_{m+l})\in \Ker\phi_s$. This yields
\begin{align*}
f_i+X^{d-r}g_i &= 0 \text{ for all } i=1,\ldots,m\\
X^rf_i+g_i &=0 \text{ for all } i=m+1,\ldots,m+l
\end{align*}
and we get 
\begin{align*}
 t &= (-X^{d-r}g_1,\ldots,-X^{d-r}g_m,f_{m+1},\ldots,f_{m+l}),(g_1,\ldots,g_m,-X^rf_{m+1},\ldots,-X^rf_{m+l})\\
 &= \psi_s(-g_1,\ldots,-g_m,f_{m+1},\ldots,f_{m+l}).
\end{align*}
\end{proof}

\begin{rem}
 Note that the syzygy modules $\Sc_q$ and $\Sc_{q+1}$ appearing in the middle spot are already defined over $k[Y_1,\ldots,Y_n]$. If $n=2$ they split by Hilberts syzygy theorem as a direct sum of $m+l-1$ (degree shifted) copies of $R$.
\end{rem}

\begin{thm}\label{HilbSer} Let $M:=\Syz_R(X^a\cdot V_1,\ldots,X^a\cdot V_m,V_{m+1},\ldots,V_{m+l})$. Then the Hilbert-series of $M$ is given by
\begin{equation*}
\lK_{M}(t) = \frac{(t^{\alpha\cdot r}-t^{\alpha\cdot d})\cdot \lK_{\Sc_q}(t)+(1-t^{\alpha\cdot r})\lK_{\Sc_{q+1}}(t)}{1-t^{\alpha\cdot d}}.
\end{equation*}
\end{thm}

\begin{proof}
 Let $a':=a+d-2r=dq+d-r$, $r':=d-r$ and $$M':=\Syz_R(X^{a'}\cdot V_1,\ldots,X^{a'}\cdot V_m,V_{m+1},\ldots,V_{m+l}).$$ Since $a'+d-2r'=a$, Theorem \ref{ses} yields
\begin{align*}
 \lK_M(t) & = t^{\alpha\cdot r}\lK_{\Sc_q}(t)+\lK_{\Sc_{q+1}}(t)-t^{\alpha\cdot r}\lK_{M'}(t)\\
 \lK_{M'}(t) & = t^{\alpha\cdot r'}\lK_{\Sc_q}(t)+\lK_{\Sc_{q+1}}(t)-t^{\alpha\cdot r'}\lK_{M}(t).
\end{align*}
Substituting $\lK_{M'}(t)$ in the first formula and solving for $\lK_M(t)$ gives the result.
\end{proof}

\section{Examples}
\sectionmark{Examples}
Using the representation of the isomorphism classes of indecomposable, non-free, maximal Cohen-Macaulay modules over surface rings of type ADE as first syzygy modules of ideals from Chapter \ref{chapmatfac}, 
Theorem \ref{HilbSer} enables us to compute their Hilbert-series if the ideal in the representation is monomial. If the ideal in the representation is not monomial, 
we compute the Hilbert-series of the syzygy module directly, using the short exact sequences
$$\begin{array}{ccccccccc}
 0 &\ra& \Syz_R(F_1,\ldots,F_{m+1}) & \ra & \bigoplus_{i=1}^{m+1} R(-\Deg(F_i)) & \ra & (F_1,\ldots,F_{m+1}) & \ra & 0,\\
 0 &\ra& (F_1,\ldots,F_{m+1})&\ra & R & \ra & R/(F_1,\ldots,F_{m+1}) & \ra & 0.
\end{array}$$
From these sequences one obtains
\begin{equation}\label{hilbserdirect}
\lK_{\Syz_R(F_1,\ldots,F_{m+1})}(t)=\left(\sum_{i=1}^{m+1}t^{\Deg(F_i)}-1\right)\cdot\lK_R(t)+\lK_{R/(F_1,\ldots,F_{m+1})}(t).
\end{equation}
We start with a few remarks on the expected structure of the Hilbert-series. This structure is given by the following lemma (cf. \cite[Exercise 4.4.12]{brunsherzog}).

\begin{lem}\label{hilbsercm}
Let $R$ be a positively-graded algebra over a field $k$. Let $M\neq 0$ be a finitely generated $R$-module and $S$ a graded Noether normalization of $R/\Ann(M)$ 
generated by elements of degree $a_1,\ldots,a_d$ with $d=\Dim(M)$. Then the following holds.
\begin{enumerate}
 \item There is a polynomial $Q(t)\in\Z[t,t^{-1}]$ such that $$\lK_M(t)=\frac{Q(t)}{\prod_{i=1}^d(1-t^{a_i})}.$$
 \item We have $Q(1)=\Rank_S(M)>0$.
 \item If $M$ is Cohen-Macaulay, the coefficients of $Q(t)$ are non-negative.
\end{enumerate}
\end{lem}

Since the modules in question are maximal Cohen-Macaulay, the coefficients in $Q(t)$ will be non-negative and $d=2$. Since $\Ann(M)=0$ in all cases, the Noether 
normalization $S$ of $R/\Ann(M)$ is in fact a Noether normalization of $R$. This can be chosen to be generated by two parameters from $R$. Since the equation of the $A_n$ singularities 
is isomorphic to $X^2+Y^2+Z^{n+1}=0$ (if $\chara(k)\neq 2$), we may choose $Y$, $Z$ as a parameter system in all cases. Then the rank of $R$ over $S$ is two (since $R\cong S[X]/(X^2+F(Y,Z))$) and we get
$$Q(1)=\Rank_S(M)=2\cdot\Rank_R(M).$$ 
We will always express the Hilbert-series as a rational function, since the explicit descriptions as power series are not very enlighting. Note that the isomorphisms $M\cong\Syz_R(F_1,\ldots,F_n)$ from Chapter \ref{chapmatfac} are only isomorphisms of $R$-modules. By Theorem \ref{ausreitgraded}, these isomorphisms extend after a degree shift to isomorphisms of graded $R$-modules, say $M(-l)\cong\Syz_R(F_1,\ldots,F_n)$ as graded $R$-modules for some $l\in\Z$. In this situation, we have
$$\lK_M(t)=t^l\cdot\lK_{\Syz_R(F_1,\ldots,F_n)}(t).$$
Note that we will always compute the Hilbert-series of the syzygy module.

\subsection{The case $A_n$}
Let $R:=k[X,Y,Z]/(X^{n+1}-YZ)$ with $\Deg(X)=2,\Deg(Y)=\Deg(Z)=n+1$ and $n\geq 0$. For any $i\in\{1,\ldots,n\}$, let $M_i:=\Syz_R(X^i,Z)$ (compare Theorem \ref{syzan}).

With the notations from Theorem \ref{HilbSer}, we have $a=i$ and $d=2n+2$. In all cases for $i$, we get $q=0$ and $r=i$. 
Hence the syzygy modules $S_0$ and $S_1$ are $\Syz_R(1,Z)$ resp. $\Syz_R(YZ,Z)$ for all $i$. They are free with basis $(Z,-1)$ resp. $(1,-Y)$ of total degrees $n+1$ and $2n+2$. By Theorem \ref{HilbSer} we get
$$\begin{aligned}
 \lK_{M_i}(t) &= \frac{(t^{2i}-t^{2(n+1)})t^{n+1}+(1-t^{2i})t^{2n+2}}{(1-t^2)(1-t^{n+1})^2}\\
 &= \frac{t^{2i+n+1}+t^{2n+2}}{(1-t^2)(1-t^{n+1})}
\end{aligned}$$
Note that all these Hilbert-series are different. Since the Hilbert-series of a free $R$-module of rank one is given by
$$\lK_{R(-l)}(t)=t^l\cdot\frac{1-t^{2n+2}}{(1-t^2)(1-t^{n+1})^2}=t^l\cdot\frac{1+t^{n+1}}{(1-t^2)(1-t^{n+1})},$$
we see that the Hilbert-series of the $M_i$ is different from the Hilbert-series of a free $R$-module, since the difference of the exponents in the numerator is 
$|n+1-2i|$ for $M_i$ and $n+1$ for $R(-l)$.

Therefore, the Hilbert-series of a maximal Cohen-Macaulay module of rank one tells us not to which isomorphism class the module belongs, since in this situation the dual modules $\Syz_R(X^i,Z)$ and $\Syz_R(X^i,Y)$ have the same Hilbert-series. In such a situation we will simply say that the Hilbert-series detects the isomorphism class \textit{up to dualizing}. Note that we have a homogeneous isomorphism 
\begin{align*}
 \Syz_R(X^j,Y)\oplus\Syz(X^j,Z) &\lra \Syz_R(X^j,Y,Z),\\
 ((A,B),(C,D)) &\longmapsto (A+C,B,D),
\end{align*}
for all $j\in\N$, which gives directly the Hilbert-series of $\Syz_R(X^j,Y,Z)$ as
$$\begin{aligned}
\lK_{\Syz_R(X^j,Y,Z)}(t) &= \lK_{\Syz_R(X^j,Y)}(t)+\lK_{\Syz_R(X^j,Z)}(t)\\
&=
\left\{\begin{aligned}
& 2\cdot\frac{t^{2n+2}+t^{2j+n+1}}{(1-t^2)(1-t^{n+1})} && \text{if }1\leq j\leq n\text{ and}\\
& 2\cdot t^{2j}\cdot\frac{1+t^{n+1}}{(1-t^2)(1-t^{n+1})} && \text{otherwise}.
\end{aligned}\right.
\end{aligned}$$

\subsection{The case $D_n$}
Let $R:=k[X,Y,Z]/(X^2+Y^{n-1}+YZ^2)$ with $n\geq 4$ and the degrees of the variables are $\Deg(X)=n-1$, $\Deg(Y)=2$, $\Deg(Z)=n-2$. According to Theorem \ref{syzdn} let 
$$\begin{aligned}
 M_1 &:= \Syz_R(X,Y), && \\
 M_j &:= \Syz_R(X,Y^{j/2},Z) && \text{if }j\in\{2,\ldots,n-2\} \text{ is even},\\
 M_j &:= \Syz_R(X,Y^{(j+1)/2},YZ) && \text{if }j\in\{2,\ldots,n-2\}\text{ is odd,}\\
 M_{n-1} &:= \Syz_R(X,Z-iY^{(n-2)/2}) && \text{and}\\
 M_n &:= \Syz_R(X,Z+iY^{(n-2)/2}) && \text{if }n\text{ is even, or}
 \end{aligned}$$
$$\begin{aligned}
 M_{n-1} &:= \Syz_R(Z,X+iY^{(n-1)/2}) && \text{and}\\
 M_n &:= \Syz_R(Z,X-iY^{(n-1)/2}) && \text{if }n\text{ is odd}.
\end{aligned}$$
Using Theorem \ref{HilbSer}, we get
\begin{align*}
\lK_{M_1}(t) &= \frac{(t^{n-1}-t^{2n-2})\cdot t^2+(1-t^{n-1})t^{2(n-1)}}{(1-t^{n-1})(1-t^{n-2})(1-t^2)}\\
 &= \frac{t^{n+1}+t^{2n-2}}{(1-t^{n-2})(1-t^2)},\\
\lK_{M_{j}}(t) &= \frac{(t^{n-1}-t^{2n-2})(t^j+t^{n-2})+(1-t^{n-1})(t^{2(n-1)}+t^{n-2+j})}{(1-t^{n-1})(1-t^{n-2})(1-t^2)}\\
 &= \frac{t^{n+j-2}+t^{n+j-1}+t^{2n-3}+t^{2n-2}}{(1-t^{n-2})(1-t^2)} \qquad\text{if }j\text{ is even},\\
\lK_{M_{j}}(t) &= \frac{(t^{n-1}-t^{2n-2})(t^{j+1}+t^n)+(1-t^{n-1})(t^{2(n-1)}+t^{n+j-1})}{(1-t^{n-1})(1-t^{n-2})(1-t^2)}\\
 &= \frac{t^{n+j-1}+t^{n+j}+t^{2n-2}+t^{2n-1}}{(1-t^{n-2})(1-t^2)} \qquad\text{if }j\text{ is odd}.\\
\end{align*}
To compute the Hilbert-series of $M_{n-1}$ and $M_n$, we use Formula (\ref{hilbserdirect}). Since the quotients
$$R/\left(X,Z\pm iY^{\frac{n-2}{2}}\right)\text{ and }R/\left(Z,X\pm iY^{\frac{n-1}{2}}\right)$$
are isomorphic to $k[Y]$ as graded $R$-modules, we obtain for $j\in\{n-1,n\}$
\begin{align*}
\lK_{M_{j}}(t) &= \left(t^{n-1}+t^{n-2}-1\right)\cdot\frac{1+t^{n-1}}{(1-t^2)(1-t^{n-2})}+\frac{1-t^{n-2}}{(1-t^2)(1-t^{n-2})}\\
&= \frac{t^{2n-3}+t^{2n-2}}{(1-t^{n-2})(1-t^2)}.
\end{align*}
Moreover, we have 
$$\lK_{R(-l)}(t)=t^l\cdot\frac{1+t^{n-1}}{(1-t^2)(1-t^{n-2})}.$$
The Hilbert-series can distinguish the class of $M_1$ from the classes of $M_{n-1}$ and $M_n$ if and only if $n\geq 5$, since in these cases the difference of the exponents in the numerator of $\lK_{M_1}(t)$ is at least two, while the difference in $\lK_{M_{n-1}}(t)$ and $\lK_{M_n}(t)$ is one. 
Moreover, the Hilbert-series can distinguish the classes of the $M_m$ with $m\in\{2,\ldots,n-2\}$.
Let $M$ be a maximal Cohen-Macaulay module of rank two, whose Hilbert-series is $t^l\cdot\lK_{M_j}(t)$ with $2\leq j\leq n-2$ and $l\in\Z$.
If $n=4$, we cannot say anything about the direct sum decomposition of $M$ - except that it has no free direct summands. For $n\geq 5$ the module $M$ might have $M_1$ as a direct summand only for $j=2$. In this case we get $M\cong M_1(-l-1)\oplus M_1(-l)$. We cannot 
exclude a splitting of $M$ into degree shifted copies of $M_{n-1}$ and $M_n$ only by looking at the Hilbert-series. But we know at least that $M$ has no free direct summands.

\subsection{The case $E_6$}
Let $R:=k[X,Y,Z]/(X^2+Y^3+Z^4)$ with $\Deg(X)=6$, $\Deg(Y)=4$ and $\Deg(Z)=3$. We define the modules
\begin{align*}
 M_1 &:= \Syz_R(X,Y,Z),\\
 M_2 &:= \Syz_R(X,Y^2,YZ,Z^2),\\
 M_3 &:= \Syz_R(iX+Z^2,Y^2,YZ),\\
 M_4 &:= \Syz_R(-iX+Z^2,Y^2,YZ),\\
 M_5 &:= \Syz_R(-iX+Z^2,Y),\\
 M_6 &:= \Syz_R(iX+Z^2,Y).
\end{align*}
Using Theorem \ref{HilbSer}, we get
\begin{align*}
 \lK_{M_1}(t) &= \frac{t^{7}+t^{9}+t^{10}+t^{12}}{(1-t^4)(1-t^3)},\\
 \lK_{M_2}(t) &= \frac{t^{10}+t^{11}+2t^{12}+t^{13}+t^{14}}{(1-t^4)(1-t^3)}.
\end{align*}
To compute the Hilbert-series of the $M_j$, $j\in\{3,4,5,6\}$, we use Formula (\ref{hilbserdirect}).

For $j=3,4$ the quotient $$R/(\pm iX+Z^2,Y^2,YZ)$$ is isomorphic as a graded $R$-module to $k[Z]\oplus Y\cdot k$ and for $j=5,6$ there is an isomorphism
$$R/(\pm iX+Z^2,Y)\cong k[Z]$$ as graded $R$-modules. Hence we obtain 
\begin{align*}
\lK_{M_3}(t)=\lK_{M_4}(t) &= \left(t^8+t^7+t^6-1\right)\cdot\frac{1+t^6}{(1-t^4)(1-t^3)}+\frac{1-t^4}{(1-t^4)(1-t^3)}+t^4\\
 &= \frac{t^{11}+t^{12}+t^{13}+t^{14}}{(1-t^4)(1-t^3)},\\
\lK_{M_5}(t)=\lK_{M_6}(t) &= \left(t^6+t^4-1\right)\frac{1+t^6}{(1-t^4)(1-t^3)}+\frac{1-t^4}{(1-t^4)(1-t^3)}\\
 &= \frac{t^{10}+t^{12}}{(1-t^4)(1-t^3)}.
\end{align*}
For completeness, the Hilbert-series of a free $R$-module is given by
$$\lK_{R(-l)}(t)=t^l\cdot\frac{1+t^6}{(1-t^4)(1-t^3)}.$$
In this case the Hilbert-series can distinguish the class of $M_1$ from the classes of $M_3$ and $M_4$. Let $M$ be a maximal Cohen-Macaulay module, 
whose Hilbert-series is $t^l\cdot\lK_{M_i}(t)$ for some $l\in\Z$ and some $i\in\{1,\ldots,6\}$. Then $M$ has no free direct summands. Moreover, only the 
following splittings are possible (with $\epsilon\in\{3,4\}$ and $\delta,\delta'\in\{5,6\}$)
$$\begin{aligned}
& i=1: & M &\cong M_{\delta}(-l-3)\oplus M_{\delta'}(-l),\\
& i=2: & M &\cong M_{\epsilon}(-l)\oplus M_{\delta}(-l) \text{ or}\\
& & &\cong M_{\epsilon}(-l+1)\oplus M_{\delta}(-l-2),\\
& i\in\{3,4\}: & M &\cong M_{\delta}(-l-1)\oplus M_{\delta'}(-l-2).
\end{aligned}$$

\subsection{The case $E_7$}
Let $R:=k[X,Y,Z]/(X^2+Y^3+YZ^3)$ with $\Deg(X)=9$, $\Deg(Y)=6$ and $\Deg(Z)=4$. We want to compute the Hilbert-series of the modules
\begin{align*}
 M_1 &:= \Syz_R(X,Y,Z),\\
 M_2 &:= \Syz_R(X,Y^2,YZ,Z^2),\\
 M_3 &:= \Syz_R(XY,XZ,Y^2,YZ^2,Z^3),\\
 M_4 &:= \Syz_R(X,Y^2,YZ),\\
 M_5 &:= \Syz_R(XY,XZ,Y^2,YZ^2),\\
 M_6 &:= \Syz_R(X,Y,Z^2),\\
 M_7 &:= \Syz_R(X,Y)
\end{align*}
appearing in Theorem \ref{syze7}. These are
\begin{align*}
 \lK_{M_1}(t) &= \frac{t^{10}+t^{13}+t^{15}+t^{18}}{(1-t^6)(1-t^4)},\\
 \lK_{M_2}(t) &= \frac{t^{14}+t^{16}+t^{17}+t^{18}+t^{19}+t^{21}}{(1-t^6)(1-t^4)},\\
 \lK_{M_3}(t) &= \frac{t^{18}+t^{19}+t^{20}+2t^{21}+t^{22}+t^{23}+t^{24}}{(1-t^6)(1-t^4)},\\
 \lK_{M_4}(t) &= \frac{t^{16}+t^{18}+t^{19}+t^{21}}{(1-t^6)(1-t^4)},\\
 \lK_{M_5}(t) &= \frac{t^{19}+t^{20}+t^{21}+t^{22}+t^{23}+t^{24}}{(1-t^6)(1-t^4)},\\
 \lK_{M_6}(t) &= \frac{t^{14}+t^{15}+t^{17}+t^{18}}{(1-t^6)(1-t^4)},\\
 \lK_{M_7}(t) &= \frac{t^{15}+t^{18}}{(1-t^6)(1-t^4)},\\
 \lK_{R(-l)}(t) &= t^l\cdot\frac{1+t^9}{(1-t^6)(1-t^4)}.\\
 \end{align*}
The Hilbert-series can distinguish the indecomposable, maximal Cohen-Macaulay modules. Given a maximal Cohen-Macaulay module $M$, whose Hilbert-series 
is $t^l\cdot\lK_{M_i}(t)$ for some $l\in\Z$ and some $i\in\{1,\ldots,6\}$, we cannot exclude a splitting of $M$ into 
(various) degree shifted copies of $M_7$ only by looking at $\lK_M(t)$. But we see that $M$ cannot have free direct summands and if $\Rank(M)\geq 3$ there are a lot of 
possible splittings excluded by the Hilbert-series. We list all possible splittings including at least one module of rank at least two.
$$\begin{aligned}
& i=2: & M &\cong M_4(-l+2)\oplus M_7(-l-3) \text{ or}\\
& & &\cong M_4(-l)\oplus M_7(-l+1),\\
& i=3: & M &\cong M_4(-l-5)\oplus M_4(-l-6) \text{ or}\\
& & &\cong M_6(-l-7)\oplus M_6(-l-9) \text{ or}\\
& & &\cong M_5(-l-2)\oplus M_7(-l-9) \text{ or}\\
& & &\cong M_5(-l-3)\oplus M_7(-l-6),\\
& i=5: & M &\cong M_4(-l-3)\oplus M_7(-l-5) \text{ or}\\
& & &\cong M_6(-l-5)\oplus M_7(-l-6) \text{ or}\\
& & &\cong M_6(-l-6)\oplus M_7(-l-4).
\end{aligned}$$
Negatively spoken, for $i=2$ the module $M$ cannot have $M_1$ or $M_6$ as a direct summand, for $i=3$ the modules $M_1$ and $M_2$ cannot appear as 
direct summands of $M$ and for $i=5$ again $M_1$ is not a direct summand of $M$.

\subsection{The case $E_8$}
Let $R:=k[X,Y,Z]/(X^2+Y^3+Z^5)$ with $\Deg(X)=15$, $\Deg(Y)=10$ and $\Deg(Z)=6$. As in the $E_7$ case all modules from Theorem \ref{syze8}
\begin{align*}
 M_1 &:= \Syz_R(X,Y,Z),\\
 M_2 &:= \Syz_R(X,Y^2,YZ,Z^2),\\
 M_3 &:= \Syz_R(XY,XZ,Y^2,YZ^2,Z^3),\\
 M_4 &:= \Syz_R(XY,XZ^2,Y^3,Y^2Z,YZ^3,Z^4),\\
 M_5 &:= \Syz_R(XY^2,XYZ^2,XZ^4,Y^4,Y^3Z,Y^2Z^3,Z^5),\\
 M_6 &:= \Syz_R(X,Y^2,YZ,Z^3),\\
 M_7 &:= \Syz_R(XY,XZ,Y^2,YZ^2,Z^4),\\
 M_8 &:= \Syz_R(X,Y,Z^2)
\end{align*}
are syzygy modules of monomial ideals. Using Theorem \ref{HilbSer}, we compute their Hilbert-series as
\begin{align*}
 \lK_{M_1}(t) &= \frac{t^{16}+t^{21}+t^{25}+t^{30}}{(1-t^{10})(1-t^6)},\\
 \lK_{M_2}(t) &= \frac{t^{27}+t^{31}+t^{32}+t^{35}+t^{36}+t^{40}}{(1-t^{10})(1-t^6)},\\
 \lK_{M_3}(t) &= \frac{t^{28}+t^{31}+t^{32}+t^{33}+t^{35}+t^{36}+t^{37}+t^{40}}{(1-t^{10})(1-t^6)},\\
 \lK_{M_4}(t) &= \frac{t^{34}+t^{36}+t^{37}+t^{38}+t^{39}+t^{40}+t^{41}+t^{42}+t^{43}+t^{45}}{(1-t^{10})(1-t^6)},
\end{align*}
\begin{align*}
 \lK_{M_5}(t) &= \frac{t^{46}+t^{47}+t^{48}+t^{49}+t^{50}+2t^{51}+t^{52}+t^{53}+t^{54}+t^{55}+t^{56}}{(1-t^{10})(1-t^6)},\\
 \lK_{M_6}(t) &= \frac{t^{26}+t^{28}+t^{30}+t^{31}+t^{33}+t^{35}}{(1-t^{10})(1-t^6)},\\
 \lK_{M_7}(t) &= \frac{t^{31}+t^{32}+t^{34}+t^{35}+t^{36}+t^{37}+t^{39}+t^{40}}{(1-t^{10})(1-t^6)},\\
 \lK_{M_8}(t) &= \frac{t^{22}+t^{25}+t^{27}+t^{30}}{(1-t^{10})(1-t^6)},\\
 \lK_{R(-l)}(t) &= t^l\cdot\frac{1+t^{15}}{(1-t^{10})(1-t^6)}.
\end{align*}
We see that the Hilbert-series of a given indecomposable, maximal Cohen-Macaulay module detects its isomorphism class. Let $M$ be a maximal 
Cohen-Macaulay module, whose Hilbert-series is $t^l\cdot\lK_{M_i}(t)$ for some fixed $i$ and $l\in\Z$. Then $M$ has no free direct summands.
For $i\in\{1,2,4,6,8\}$ the module $M$ cannot split and we get $M\cong M_i(-l)$. For $i=3$ the only possible splitting is 
$M\cong M_8(-l-6)\oplus M_8(-l-10)$, for $i=7$ only $M\cong M_8(-l-9)\oplus M_8(-l-10)$ is possible and for $i=5$ there is only the possible splitting 
$M\cong M_6(-l-19)\oplus M_6(-l-20)$.

\chapter[Relations between the maximal Cohen-Macaulay modules]{Relations between the maximal Cohen-Macaulay modules \except{toc}{over surface rings of type ADE}}\label{pbmaxcm}
Considering the surface rings of type ADE as rings of invariants under finite subgroups of $\SL_2(k)$, 
we get ringinclusions among them by the actions of normal subgroups.
Especially, we are interested in the pull-backs of the maximal Cohen-Macaulay modules along these inclusions. From these pull-backs we will deduce 
that the Frobenius pull-backs of $\Syz(\mm)$ over surface rings of type DE are indecomposable for almost all primes.

During this chapter the surface rings of type ADE are named by their type of singularity, e.g. $E_8:=k[X,Y,Z]/(X^2-Y^3-1728Z^5)$ (we use the hypersurface 
equations as they arise directly from the computation of the rings of invariants) and the corresponding punctured spectra get the same name but with 
curly letters, e.g. $\Ec_8:=\Spec(E_8)\setminus\{\mm\}$. For a module $M$ we will denote the restriction of its sheafification to the punctured 
spectrum by $\Mcc$.

We use the generators of the finite subgroups of $\SL_2(\C)$, given in \cite{mcmsur}. Note that these representations are well-defined in positive 
characteristics, if the group order is invertible.
To get explicit generators of the normal subgroups, one might use \cite{GAP4} and \cite{DGPS} to compute the corresponding rings of invariants 
(see Appendix \ref{app.group} for the computations). 

Let us very briefly recall some facts from invariant theory. If $G$ is a group, acting on a ring $R$ by ringautomorphisms and $H$ is a normal subgroup 
of $G$, the ring of invariants $R^G$ is a subring of $R^H$, the quotient $G/H$ acts on $R^H$ and one has $R^G=\left(R^H\right)^{G/H}$ (cf. \cite[Proposition 5.1 (3)]{holgerinvarianten}).
Let $X:=\Spec\left(R^G\right)$ and $Y=\Spec\left(R^H\right)$ (or the punctured resp. projective spectra). The inclusion $R^G\inj R^H$ induces a morphism $\iota:Y\ra X$. 
Given an $\Oc_X$-module $\Fc$, its pull-back $\iota\pb(\Fc)$ is an $\Oc_Y$-module. If $R=k[x,y]$ and $G\subsetneq\SL_2(k)$ finite, we want to understand 
what happens to the sheaves associated to the indecomposable, non-free, maximal Cohen-Macaulay modules under these pull-backs. For this purpose let 
$$i:R:=k[x,y]^G=k[U,V,W]/H_1(U,V,W)\inj S:=k[x,y]^H=k[X,Y,Z]/H_2(X,Y,Z)$$
be an inclusion between two (different) surface rings of type ADE, where $H$ is a normal subgroup of $G$. This inclusion is given by sending $U$, $V$, $W$ to 
homogeneous polynomials in the variables $X$, $Y$, $Z$ and induces a morphism 
$$\iota:\Sc:=\punctured{S}{(X,Y,Z)}\lra\Rc:=\punctured{R}{(U,V,W)}.$$
Let $M:=\Syz_R(F_1,\ldots,F_{m+1})$ be one of the indecomposable, non-free, maximal Cohen-Macaulay $R$-modules from Theorem \ref{syzan}-\ref{syze8}. 
We obtain from Proposition \ref{syzprop} the isomorphism
$$\iota\pb(\Mcc)\cong\Syz_{\Sc}(i(F_1),\ldots,i(F_{m+1})).$$
Now, the $\Oc_{\Sc}$-module $\Syz_{\Sc}(i(F_1),\ldots,i(F_{m+1}))$ has to be the sheaf associated to a (not necessarily indecomposable) maximal Cohen-Macaulay 
$S$-module $N$ and we want to have an explicit description of $N$. To get this description, we do the following manipulations on $P:=\Syz_S(i(F_1),\ldots,i(F_{m+1}))$ (cf. Lemma \ref{manipulationlemma}).

\begin{itemize}
 \item[Step 1:] If the $i(F_j)$ have a common factor $A$, replace all $i(F_j)$ by $i(F_j)/A$, which changes $P$ only by an isomorphism.
 \item[Step 2:] If $B$ is a common factor of $i(F_1),\ldots,\check{i(F_j)},\ldots,i(F_{m+1})$ and the ideal $(B,i(F_j))$ is $(X,Y,Z)$-primary, replace all 
		$i(F_k)$, $k\neq j$, by $i(F_k)/B$, which changes $P$ again only by an isomorphism.
\end{itemize}

In both steps, we allow also to replace $i(F_j)$ by $i(F_j)+G\cdot i(F_l)$ for $G\in S$ and $j\neq l$ or by a constant multiple $\lambda\cdot i(F_j)$ with $\lambda\in k^{\times}$.

In surprisingly many cases, we end up with an syzygy module appearing explicitly in Theorem \ref{syzan}-\ref{syze8} (or in a syzygy module that appeared in the 
computations in Chapter \ref{chapmatfac} as an alternative description).

If this is not the case, we compute generators for $P$ and look for relations among these generators. These relations will detect the module $P^{\vee}$. 
To compute generators for $P$, we will work over the factorial domain $k[x,y]$ (the ring on which the group acts) as follows. Let $(A_1,\ldots,A_{m+1})\in P$. Then treat the relation 
$\sum A_j\cdot i(F_j)=0$ over $k[x,y]$, meaning that we write the polynomials $i(F_j)$ in the variables $x$, $y$. Then one can again cancel the common 
multiple of the $i(F_j)$, written in $x$, $y$. The result is a relation $\sum A_j\cdot G_j(x,y)=0$, where the $G_j$ are coprime. From this relation 
one can compute the $A_j$ (as elements in $S$!), using the fact that $k[x,y]$ is factorial. We give an example.

\begin{exa}
Let $S=k[X,Y,Z]/(X^5-YZ)$ with $X=xy$, $Y=x^5$ and $Z=y^5$. Consider $(A_1,A_2)\in P:=\Syz_S(X^6,Y^2)$. Substituting $X=xy$ and $Y=x^5$ in the relation 
$A_1\cdot X^6+A_2\cdot Y^2=0$, one obtains $A_1\cdot x^6y^6+A_2\cdot x^{10}=0$, which is equivalent to
\begin{equation}\label{relation}A_1\cdot y^6+A_2\cdot x^4=0.\end{equation}
Since $k[x,y]$ is factorial, there has to be a $B\in k[x,y]$ such that 
\begin{equation}\label{syzygy}(A_1,A_2)=(-x^4,y^6)\cdot B.\end{equation}
For all choices of $B$, the tuple $(A_1,A_2)$ in (\ref{syzygy}) solves (\ref{relation}). On the other hand, for each $(A_1,A_2)\in P$, there has to be a 
$B\in k[x,y]$, such that (\ref{syzygy}) holds. 
All in all, to compute $S$-module generators for $P$, we only have to compute the polynomials $B\in k[x,y]$ with the property that $-x^4\cdot B$ and $y^6\cdot B$ are 
polynomials in $xy$, $x^5$, $y^5$.
Obviously, $B=x$ and $B=y^4$ have this property and are minimal with respect to this property, in the sence that every homogeneous monomial $B'\notin k$ 
having this property, has to be a multiple of either $x$ or $y^4$.
This shows that $P$ is generated by $(-x^4,y^6)\cdot x=(-Y,XZ)$ and $(-x^4,y^6)\cdot y^4=(-X^4,Z^2)$. Since $P$ does not appear in Theorem \ref{syzan}, 
we compute relations among the two generators. Since $X\cdot (-X^4,Z^2)-Z\cdot (-Y,XZ)=0$, we obtain 
$$P\cong\Syz_S(X,Z)^{\vee}\cong\Syz_S(X,Y).$$
\end{exa}

Of course, the computation gets much more complicated for higher ranks or if the polynomials $X=X(x,y)$, $Y=Y(x,y)$, $Z=Z(x,y)$ have high degrees or are not monomial. For this reason, we stop in the exceptional cases $E_6$, $E_7$, $E_8$ after Step 2 and compute generators for the obtained syzygy modules with Macaulay2 \cite{M2}.

\begin{center}
\begin{table}[h]
\begin{tabular}{c}
\xymatrix{
 & \Z/(4n)\ar@^{(->}[r] & \D_{2n} & & & & D_{2n+2}\ar@^{(->}[dd]\ar[r] & A_{4n-1}\ar@^{(->}[dd]\ar@^{(->}[rd]^{t|4n} & \\
\Z/(t)\ar@^{(->}[ru]^{t|4n}\ar@^{(->}[rd]^{t|2n} & & & & & & & & A_{t-1}\ar@^{(->}[dd]^{2|t} \\
 & \Z/(2n)\ar@^{(->}[uu]\ar@^{(->}[r]& \D_n\ar@^{(->}[uu] & & & & D_{n+2}\ar@^{(->}[r] & A_{2n-1}\ar@^{(->}[ru]^{t|2n}\ar@^{(->}[dd]^{2|n}\ar@^{(->}[rd] & \\
\Z/(2)\ar@^{(->}[uu]^{2|t}\ar@^{(->}[ur]\ar@^{(->}[dr]\ar@^{(->}[ddr] & & & & & & & & A_1 \\
 & \Z/(4)\ar@^{(->}[uu]^{2|n}\ar@^{(->}[r]& \D_2\ar@^{(->}[r] & \T\ar@^{(->}[d] & & E_6\ar@^{(->}[r] & D_4\ar@^{(->}[r] & A_3\ar@^{(->}[ru] & \\
 & \I & & \Obb & & E_7\ar@^{(->}[u] & & E_8\ar@^{(->}[ruu] & 
}
\end{tabular}
\caption{Inclusions of the finite subgroups (up to conjugation) of $\SL_2(\C)$ (left part) and of the corresponding rings of invariants (right part). 
All direct inclusions in the left part are inclusions of normal subgroups.}
\end{table}
\end{center}
 
\vspace*{-1.5cm}
\section{The Case $A_{n-1}$ with $n\geq 1$}
Consider natural numbers $t$ and $n$, where $t$ is a divisor of $n$. Let $\epsilon$ be a primitive $n$-th root of unity. On $k[x,y]$
the group action of $\Z/(n)$, generated by $\left(\begin{smallmatrix}\epsilon & 0\\ 0&\epsilon^{-1}\end{smallmatrix}\right),$ has the invariants $U:=xy$, $V:=x^n$, $W:=y^n$ and the group action of $\Z/(t)$, generated by $\left(\begin{smallmatrix}\epsilon & 0\\ 0&\epsilon^{-1}\end{smallmatrix}\right)^{n/t},$
has the invariants $X:=xy$, $Y:=x^t$, $Z:=y^t$. Since $\Z/(t)$ is a normal subgroup of $\Z/(n)$, invariant theory tells us that $A_{n-1}$ is a subring of $A_{t-1}$, where the inclusion is given by $U\mapsto X$, $V\mapsto Y^{n/t}$, $W\mapsto Z^{n/t}$, because the equalities $U=X$, $V=Y^{n/t}$, $W=Z^{n/t}$ hold in $k[x,y]$. In the grading $(2,t,t)$ resp. $(2,n,n)$ on $A_{t-1}$ resp. $A_{n-1}$ the inclusion is homogeneous. Denote the inclusion by $\phi_t$. The goal is to find an easier description of the module on the right hand side in the isomorphism
$$\phi_t\pb(\Syz_{\Ac_{n-1}}(U^l,V))\cong\Syz_{\Ac_{t-1}}(X^l,Y^{n/t}).$$

\begin{lem}\label{lempban}
Let $l\in\{1,\ldots,n-1\}$ with $l=at+r$ and $0\leq r\leq t-1$. Let $m\geq a+1$. There is an isomorphism $$\Syz_{A_{t-1}}(X^l,Y^m)\cong\Syz_{A_{t-1}}(X^r,Y).$$
\end{lem}

\begin{proof}
Let $(A,B)\in\Syz_{A_{t-1}}(X^l,Y^m)$. Then over $k[x,y]$ the relation $A\cdot X^l+B\cdot Y^m=0$ becomes $A\cdot (xy)^l+B\cdot x^{tm}=0$, which is equivalent to
\begin{equation*}A\cdot y^{at+r}+B\cdot x^{t(m-a-1)+t-r}=0.\end{equation*}
We obtain the existence of a $C\in k[x,y]$ such that
\begin{equation*} (A,B)=(x^{t(m-a-1)}x^{t-r},-y^{at}y^r)\cdot C. \end{equation*}
Since $A$ and $B$ belong to $A_{t-1}$, all possibilities for $C$ are $A_{t-1}$-combinations of $x^r$ and $y^{t-r}$. We obtain
$$\Syz_{A_{t-1}}(X^l,Y^m)=\left\langle\begin{pmatrix} Y^{m-a} \\ -X^r\cdot Z^a \end{pmatrix},\begin{pmatrix} X^{t-r}\cdot Y^{m-a-1} \\ -Z^{a+1}\end{pmatrix}\right\rangle.$$
Since the generators fullfill the relation
\begin{equation*}
-Z\cdot \begin{pmatrix} Y^{m-a} \\ -X^r\cdot Z^a \end{pmatrix}+X^r\cdot\begin{pmatrix} X^{t-r}\cdot Y^{m-a-1} \\ -Z^{a+1}\end{pmatrix}=0,
\end{equation*}
we get the isomorphisms
$$\Syz_{A_{t-1}}(X^l,Y^m)\cong\left(\Syz_{A_{t-1}}(X^r,Z)\right)^{\vee}\cong\Syz_{A_{t-1}}(X^r,Y).$$
\end{proof}

Since $l\leq n-1=\left(\tfrac{n}{t}-1\right)\cdot t +t-1$, we have $a\leq\tfrac{n}{t}-1$, hence by applying the above lemma, we obtain
$$\Syz_{A_{t-1}}(X^l,Y^{n/t})\cong \Syz_{A_{t-1}}(X^r,Y)$$
with $r\in\{0,\ldots,t-1\}$.

\begin{rem}
Using ideals to represent the indecomposable, maximal Cohen-Macaulay $A_{n-1}$-modules, we have by Theorem \ref{syzan} the isomorphism
$\phi_t\pb((U^l,W))\cong (X^r,Z)$.
\end{rem}

\section{The Case $D_{n+2}$ with $n\geq 2$}
The binary dihedral group $\mathbb{D}_{n}$ is generated by $\zeta:=\left(\begin{smallmatrix}\epsilon & 0\\ 0&\epsilon^{-1}\end{smallmatrix}\right),$ where $\epsilon$ is 
a primitive $2n$-th root of unity and $\tau:=\left(\begin{smallmatrix}0 & 1\\ -1&0\end{smallmatrix}\right)$.
The generating invariants are 
$$X:=xy(x^{2n}-y^{2n}), \text{ } Y:=x^2y^2 \text{ and } Z:=x^{2n}+y^{2n},$$ fullfilling the equation 
$$X^2+4Y^{n+1}-YZ^2=0.$$
If $n$ is odd, the normal subgroups of $\mathbb{D}_n$ are exactly the normal subgroups of $<\zeta>$. If $n$ is even there are two additional normal subgroups.
The first one is generated by $\tau$ and $\zeta^2$. The second normal subgroup is generated by $\zeta\tau$ and $\zeta^2$ (cf. \cite[Lemma 4.3]{dolgagroup}). 
These two last groups lead to the same rings of invariances as a direct computation shows. This ring is a $D_{n/2+2}$-singularity. Recall from Theorem \ref{syzdn} that 
$$\begin{aligned}
M_1 &= \Syz_{D_{n+2}}(X,Y), && \\
M_{2m} &= \Syz_{D_{n+2}}(X,Y^m,Z) && \left(\text{with }1\leq m\leq \frac{n}{2}\right),\\
M_{2m-1} &= \Syz_{D_{n+2}}(X,Y^m,YZ) && \left(\text{with }2\leq m\leq \frac{n+1}{2}\right),\\
M_{n+1} &= \Syz_{D_{n+2}}(X,Z-2Y^{n/2}) && (\text{with }n\text{ even}),\\
M_{n+2} &= \Syz_{D_{n+2}}(X,Z+2Y^{n/2}) && (\text{with }n\text{ even}),
\end{aligned}$$
$$\begin{aligned}
M_{n+1} &= \Syz_{D_{n+2}}(X+2iY^{\frac{n+1}{2}},Z) && (\text{with }n\text{ odd}),\\
M_{n+2} &= \Syz_{D_{n+2}}(X-2iY^{\frac{n+1}{2}},Z) && (\text{with }n\text{ odd})
\end{aligned}$$
is a complete list of generators for the isomorphism classes of indecomposable, non-free, maximal Cohen-Macaulay modules.

\subsection{Pull-backs to $\Ac_{t-1}$ with $t|2n$}

It is enough to deal with the pull-backs to $\Ac_{2n-1}$, because every pull-back from $\Dc_{n+2}$ to $\Ac_{t-1}$ (with $t|2n$) factors 
through $\Ac_{2n-1}$. The pull-back is induced by the map 
$$\phi:k[X,Y,Z]/(X^2+4Y^{n+1}-YZ^2)\longrightarrow k[R,S,T]/(T^{2n}-RS),$$
given by $X\mapsto T(R-S)$, $Y\mapsto T^2$, $Z\mapsto R+S$, where
$T=xy$, $R=x^{2n}$, $S=y^{2n}$.

Except for the modules $M_{n+1}$ and $M_{n+2}$, applying the steps 1 and 2 from the strategy described at the beginning of this chapter is already 
enough to compute the pull-backs.
\begin{align*}
\phi\pb(\Mcc_1) &= \phi\pb(\Syz_{\Dc_{n+2}}(X,Y))\\
&\cong \Syz_{\Ac_{2n-1}}(TR-TS,T^2)\\
&\cong \Syz_{\Ac_{2n-1}}(R-S,T)\\
&\cong \Oc_{\Ac_{2n-1}},\\
\phi\pb(\Mcc_{2m}) &= \phi\pb(\Syz_{\Dc_{n+2}}(X,Y^m,Z))\\
&\cong \Syz_{\Ac_{2n-1}}(TR-TS,T^{2m},R+S)\\
&\cong \Syz_{\Ac_{2n-1}}(R-S,T^{2m-1},R+S)\\
&\cong \Syz_{\Ac_{2n-1}}(R,S,T^{2m-1})\\
&\cong \Syz_{\Ac_{2n-1}}(R,T^{2m-1})\oplus \Syz_{\Ac_{2n-1}}(S,T^{2m-1}),\\
\phi\pb(\Mcc_{2m-1}) &= \phi\pb(\Syz_{\Dc_{n+2}}(X,Y^m,YZ))\\
&\cong \Syz_{\Ac_{2n-1}}(TR-TS,T^{2m},T^2R+T^2S)\\
&\cong \Syz_{\Ac_{2n-1}}(R-S,T^{2m-1},TR+TS)\\
&\cong \Syz_{\Ac_{2n-1}}(R-S,T^{2m-2},R+S)\\
&\cong \Syz_{\Ac_{2n-1}}(R,S,T^{2m-2})\\
&\cong \Syz_{\Ac_{2n-1}}(R,T^{2m-2})\oplus\Syz_{\Ac_{2n-1}}(S,T^{2m-2}).
\end{align*}
For $i=n+1$ and $i=n+2$ we have $\phi\pb(\Mcc_i)\cong\Syz_{\Ac_{2n-1}}(T^n,R)$, independently of the cardinality of $n$. We only compute $\phi\pb(\Mcc_{n+1})$ 
in the two cases corresponding to the the cardinality of $n$. To compute $\phi\pb(\Mcc_{n+2})$ one has only to change some signs. Let $n$ be even. In this case we get
\begin{align*}
\phi\pb(\Mcc_{n+1}) &= \phi\pb\left(\Syz_{\Dc_{n+2}}\left(X,Z-2Y^{\frac{n}{2}}\right)\right)\\
&\cong \Syz_{\Ac_{2n-1}}(TR-TS,R+S-2T^n)\\
&\cong \Syz_{\Ac_{2n-1}}(R-S,R+S-2T^n)\\
&\cong \Syz_{\Ac_{2n-1}}(R-S,R-T^n).
\end{align*}
For $(A,B)\in\Syz_{A_{2n-1}}(R-S,R-T^n)$ the corresponding relation over $k[x,y]$ is equivalent to $$A\cdot(x^n+y^n)+B\cdot x^n=0.$$
This gives $(A,B)=(-x^n,x^n+y^n)\cdot C$ with $C\in k[x,y]$. The minimal possibilities for $C$ are $x^n$ and $y^n$, giving the generators $s_1:=(-R,R+T^n)$ and $s_2:=(-T^n,T^n+S)$ with the relation $T^n\cdot s_1-R\cdot s_2=0$. The claim follows, since the syzygy module of $(T^n,R)$ is selfdual.

Now let $n$ be odd. In this case we obtain
\begin{align*}
\phi\pb(\Mcc_{n+1}) &= \phi\pb(\Syz_{\Dc_{n+2}}(X+2iY^{\frac{n+1}{2}},Z))\\
&\cong \Syz_{\Ac_{2n-1}}(TR-TS+2iT^{n+1},R+S)\\
&\cong \Syz_{\Ac_{2n-1}}(R-S+2iT^n,R+S)\\
&\cong \Syz_{\Ac_{2n-1}}(R+iT^n,R+S).
\end{align*}
Taking a syzygy $(A,B)\in\Syz_{A_{2n-1}}(R+iT^n,R+S)$, the corresponding relation over $k[x,y]$ is equivalent to $$A\cdot x^n+B\cdot(x^n-iy^n)=0.$$
This gives $(A,B)=(x^n-iy^n,-x^n)\cdot C$ with $C\in k[x,y]$. The minimal possibilities for $C$ are again $x^n$ and $y^n$, giving the generators $s_1:=(R-iT^n,-R)$ and $s_2:=(T^n-iS,-T^n)$ with the relation $T^n\cdot s_1-R\cdot s_2=0$.

\subsection{Pull-backs to $\Dc_{n/2+2}$}

Consider the map 
$$\psi: k[X,Y,Z]/(X^2+4Y^{n+1}-YZ^2)\rightarrow k[R,S,T]/(R^2+4S^{\frac{n}{2}+1}-ST^2),$$
given by $X\mapsto RT$, $Y\mapsto S$, $Z\mapsto T^2-2S^{n/2}$, where $R=xy(x^n-y^n)$, $S=x^2y^2$, $T=x^n+y^n$.

For all pull-backs the steps 1 and 2 from the beginning of this chapter are enough to detect the isomorphism classes of the pull-backs.
\begin{align*}
\psi\pb(\Mcc_1) &\cong \Syz_{\Dc_{\frac{n}{2}+2}}(RT,S)\\
&\cong \Syz_{\Dc_{\frac{n}{2}+2}}(R,S),\\
\psi\pb(\Mcc_{2m}) &\cong \Syz_{\Dc_{\frac{n}{2}+2}}\left(RT,S^m,T^2-2S^{\frac{n}{2}}\right)\\
 &\cong \Syz_{\Dc_{\frac{n}{2}+2}}(RT,S^m,T^2)\\
 &\cong \Syz_{\Dc_{\frac{n}{2}+2}}(R,S^m,T),\\
\psi\pb(\Mcc_{2m-1}) &\cong \Syz_{\Dc_{\frac{n}{2}+2}}\left(RT,S^m,ST^2-2S^{\frac{n}{2}+1}\right)\\
 &\cong \Syz_{\Dc_{\frac{n}{2}+2}}(RT,S^m,ST^2)\\
 &\cong \Syz_{\Dc_{\frac{n}{2}+2}}(R,S^m,ST),\\
\psi\pb(\Mcc_{n+1}) &\cong \Syz_{\Dc_{\frac{n}{2}+2}}\left(RT,T^2-4S^{\frac{n}{2}}\right)\\
&\cong \Syz_{\Dc_{\frac{n}{2}+2}}\left(R,T^2-4S^{\frac{n}{2}}\right)\\
&\cong \Syz_{\Dc_{\frac{n}{2}+2}}(R,S),
\end{align*}
\begin{align*}
\psi\pb(\Mcc_{n+2}) &\cong \Syz_{\Dc_{\frac{n}{2}+2}}(RT,T^2)\\
&\cong \Syz_{\Dc_{\frac{n}{2}+2}}(R,T)\\
&\cong \Oc_{\Dc_{\frac{n}{2}+2}}.
\end{align*}

\section{The Case $E_6$}
The binary tetrahedral group $\mathbb{T}$ of order 24 is generated by $\sigma$, $\tau$ and $\mu$, where
$$\sigma=\begin{pmatrix} i & 0\\ 0&-i\end{pmatrix}, \quad \tau=\begin{pmatrix}0 & 1\\ -1&0\end{pmatrix}, \quad \mu=\frac{1}{\sqrt{2}}\begin{pmatrix}\epsilon^7 & \epsilon^7\\ \epsilon^5&\epsilon\end{pmatrix}$$
with a primitive eighth root of unity $\epsilon$. The generating invariants are 
$$X:=x^{12}-33x^8y^4-33x^4y^8+y^{12}, Y:=x^8+14x^4y^4+y^8, Z:=x^5y-xy^5,$$ 
hence the equation of the $E_6$-singularity is 
$$X^2-Y^3+108X^4=0.$$
There are two normal subgroups of $\mathbb{T}$. The first one is generated by $\sigma$ and $\tau$ and the ring of invariants is a $D_4$-singularity. 
The second one is generated by $\tau^2$ and the ring of invariants is an $A_1$-singularity. Recall from Theorem \ref{syze6} that
\begin{align*}
M_1 &= \Syz_{E_6}(X,Y,Z),\\
M_2 &= \Syz_{E_6}(X,Y^2,YZ,Z^2),\\
M_3 &= \Syz_R(iX+6\sqrt{3}Z^2,Y^2,YZ),\\
M_4 &= \Syz_R(-iX+6\sqrt{3}Z^2,Y^2,YZ),\\
M_5 &= \Syz_R(-iX+6\sqrt{3}Z^2,Y),\\
M_6 &= \Syz_R(iX+6\sqrt{3}Z^2,Y)
\end{align*}
is a complete list of representatives of the isomorphism classes of indecomposable, non-free, maximal Cohen-Macaulay modules.

We consider the map 
$$\psi:k[X,Y,Z]/(X^2-Y^3+108Z^4)\longrightarrow k[U,V,W]/(UV-W^2),$$ 
that is given by 
$$X\mapsto U^6-33U^2W^4-33V^2W^4+V^6,\quad Y\mapsto U^4+V^4+14W^4,\quad Z\mapsto (U^2-V^2)W.$$
A direct computation shows that $\psi$ is the composition 
$$(A_3\ra A_1)\circ (D_4\ra A_3)\circ (E_6\ra D_4).$$ 
Therefore, it is enough to compute the pull-backs to $\Dc_4$.

To reduce the number of generators of the syzygy modules corresponding to the pull-backs of $\Mcc_i$, $i\geq 3$, we first prove 
the isomorphisms
\begin{align*}
M_5\oplus M_6 &\cong L:=\Syz_{E_6}(X,Y,Z^2)\quad\text{and}\\
M_3\oplus M_4 &\cong N:=\Syz_{E_6}(XY,XZ,Y^2,YZ^2,Z^3)\\
&\cong \Syz_{E_6}(XY,XZ^2,Y^3,Y^2Z,Z^3)\quad\text{(for later use).}
\end{align*}
The proofs use matrix factorizations. To avoid ugly constant coefficients in the matrices, we turn back to the equation $X^2+Y^3+Z^4=0$ for these proofs. 
Note that under this isomorphism, we only have to change the coefficients $6\sqrt{3}$ in $M_3,\ldots,M_6$ into 1. The matrix factorization 
$$\left(\left(\begin{pmatrix}-Z^2 & Y\\ Y^2 & Z^2\end{pmatrix}\right),\left(\begin{pmatrix}-Z^2 & Y\\ Y^2 & Z^2\end{pmatrix}\right)\right)\hat{\otimes}(X,X)$$ 
splits by Lemma \ref{tensorsplit} as $(\phi_5,\psi_5)\oplus (\phi_6,\psi_6)$. Obviously, the module given by this matrix factorization is selfdual. 
Computing the tensor product and deleting the first column (in the second matrix), one obtains the claimed isomorphism.

Let $\Omega=\left(\left(\begin{smallmatrix}
XE_4 & \eta\\
\eta & -XE_4\end{smallmatrix}\right),
\left(\begin{smallmatrix}
XE_4 & \eta\\
\eta & -XE_4\end{smallmatrix}\right)\right)$ with
$$\eta=\begin{pmatrix}
-Z^2 & 0 & YZ & Y\\
-YZ & Z^2 & Y^2 & 0\\
0 & Y & 0 & Z\\
Y^2 & -YZ & Z^3 & 0
\end{pmatrix}.$$
Because the map 
$$\left(E_8,\left(\begin{smallmatrix}0 & E_4\\E_4 & 0\end{smallmatrix}\right)\right):(\eta,\eta)\hat{\otimes}(X,X)\rightarrow \Omega$$ 
is an isomorphism of matrix factorizations, we see by Lemma \ref{tensorsplit} that $\Omega$ splits as $(\phi_3,\psi_3)\oplus (\phi_4,\psi_4)$. 
Again $\coKer(\Omega)$ is selfdual and deleting the columns one, three and seven from the second matrix of $\Omega$, the rows of the reduced matrix 
generate the module $\Syz_{E_6}(-XZ,XY,Y^2,-Z^3,YZ^2)$. For the other isomorphism delete the columns two, three and five. 

Now, we compute the pull-backs to $\Dc_4$. The computations are done with Macaulay2 (with coefficients in $\Z$) and can be found in Appendix 
\ref{m2.mcmpb.e6}. The map 
$$\phi:k[X,Y,Z]/(X^2-Y^3+108Z^4)\longrightarrow k[U,V,W](U^2+4V^3-VW^2)$$ 
is given by 
$$X\mapsto W^3-36V^2W,\quad Y\mapsto W^2+12V^2,\quad Z\mapsto U.$$
Therefore, the pull-back 
$$\phi\pb(\Syz_{\Ec_6}(X,Y,Z))\cong \Syz_{\Dc_4}(W^3-36V^2W,W^2+12V^2,U)$$
is generated by
$$s_1:=\begin{pmatrix}0 \\ -U \\ 12V^2+W^2\end{pmatrix},\text{ } 
s_2:=\begin{pmatrix}-V \\ -2VW \\ 3UW\end{pmatrix}, \text{ }
s_3:=\begin{pmatrix}W \\ 12V^2-W^2 \\ 36UV\end{pmatrix},\text{ } 
s_4:=\begin{pmatrix}-U \\ -3UW \\ 4W^3\end{pmatrix}.$$
The total degrees of these generators are seven, eight, eight and nine. Since there are no elements of degree one and two and $s_2$ and $s_3$ are linearly independent, all generators are necessary. We have the following global relations among these generators
\begin{align*}
-3U\cdot s_1+W\cdot s_2+V\cdot s_3 &= 0,\\
-4UW\cdot s_1+4V^2\cdot s_2+VW\cdot s_3+U\cdot s_4 &= 0,\\
(12V^2-4W^2)\cdot s_1+U\cdot s_3 +W\cdot s_4 &= 0.
\end{align*}
The second and third relation show that we can ommit the fourth generator locally on the covering $D(U)\cup D(W)=\Dc_4$. Then the first relation gives the 
isomorphism $$\phi\pb(\Syz_{\Ec_6}(\Mcc_1))\cong\Syz_{\Dc_4}(U,V,W).$$

The pull-back 
$$\phi\pb(\Mcc_5\oplus\Mcc_6)\cong\phi\pb(\Lc)\cong\phi\pb(\Syz_{\Ec_6}(X,Y,Z^2))\cong \Syz_{\Dc_4}(W^3-36V^2W,W^2+12V^2,U^2)$$ 
is generated by
$$s_1:=\begin{pmatrix}V \\ 2VW \\ -3W\end{pmatrix},\text{ }
s_2:=\begin{pmatrix}W \\ -W^2 + 12V^2 \\ 36V\end{pmatrix},\text{ }
s_3:=\begin{pmatrix}0 \\ U^2 \\ -12V^2-W^2\end{pmatrix},\text{ }
s_4:=\begin{pmatrix}U^2 \\ 0 \\ 36V^2W-W^3\end{pmatrix}.$$
Because $s_1$ and $s_2$ are linearly independent and of the same (minimal) total degree, both are necessary as generators. The relations 
$$s_3=\frac{1}{3}(W\cdot s_1-V\cdot s_2)\quad \text{and} \quad s_4=\frac{1}{3}((W^2-12V^2)\cdot s_1+2VW\cdot s_2)$$
show that $s_1$ and $s_2$ generate the whole module if the characteristic is not three, hence $\Mcc_5$ and $\Mcc_6$ become trivial after the pull-back.

To get the pull-back of $\Mcc_2$ we compute
\begin{align*}
 & \phi\pb(\Syz_{\Ec_6}(X,Y^2,YZ,Z^2))\\
 \cong & \Syz_{\Dc_4}(W^3-36V^2W,W^4+24V^2W^2+144V^4,UW^2+12UV^2,U^2).
\end{align*}
A set of generators for $\Syz_{D_4}(W^3-36V^2W,W^4+24V^2W^2+144V^4,UW^2+12UV^2,U^2)$ is given by
\begin{align*}
s_1 &:= (0,0,-U,12V^2+W^2),\\
s_2 &:= (0,-U,12V^2+W^2,0),\\
s_3 &:= (2VW,V,3U,-6W^2),\\
s_4 &:= (12V^2-W^2,W,0,-72VW),\\
s_5 &:= (UV,0,2VW,-3UW),\\
s_6 &:= (UW,U,-2W^2,36UV),\\
s_7 &:= (-U^2,0,-3UW,4W^3).
\end{align*}
Because we can write $3\cdot s_7=6W\cdot s_1-W\cdot s_2+V\cdot s_3$, we can omit $s_7$ as a generator (since $\chara(k)\neq 3$). The remaining generators are linearly independent and 
have total degree ten or eleven. Since there are no elements of degree one, all of them are necessary to generate the module. We have the relations
\begin{align*}
U\cdot (-6s_1-s_3)+V\cdot (s_2+2s_6) &= 0,\\
U\cdot (6s_1-2s_3+s_4)+(W-2V)\cdot (-2s_2+6s_5+s_6) &= 0,\\
U\cdot (-6s_1+2s_3+s_4)+(W+2V)\cdot (2s_2-6s_5+s_6) &= 0.
\end{align*}
Since the characteristic of our base field is at least five, the module generated by $s_1,\ldots,s_6$ is isomorphic to the module generated by the elements 
in the brackets. We get $$\phi\pb(\Mcc_2)\cong\Syz_{\Dc_4}(U,V)\oplus\Syz_{\Dc_4}(U,W-2V)\oplus\Syz_{\Dc_4}(U,W+2V).$$ 

Generators of $$\phi\pb(\Mcc_3\oplus \Mcc_4)\cong\phi\pb(\Nc)=\phi\pb(\Syz_{\Ec_6}(XY,XZ,Y^2,YZ^2,Z^3))$$ are given by
\begin{align*}
s_1 &:= (0, 0, 0, -U, W^2 + 12V^2),\\
s_2 &:= (-U, 12V^2+W^2, 0, 0, 0),\\
s_3 &:= (V, 0, 2VW, -3W, 0),\\
s_4 &:= (W, 0, -W^2 + 12V^2, 36V, 0),\\
s_5 &:= (0, 0, U^2, -12V^2-W^2, 0),\\
s_6 &:= (0, 2VW, UV, 3U, -6W^2),\\
s_7 &:= (U, -2W^2, UW, 0, -72VW),\\
s_8 &:= (0, UV, 0, 2VW, -3UW),\\
s_9 &:= (0, UW, 0, 12V^2-W^2, 36UV),\\
s_{10} &:= (0, -U^2, 0, -3UW, 4W^3),\\
s_{11} &:= (-U^2W, 3UW^3, 0, -2W^4+27U^2V, 81U^3).
\end{align*}
The equations
\begin{align*}
3\cdot s_5 &= -W\cdot s_3+V\cdot s_4,\\
3\cdot s_{10} &= 6W\cdot s_1+V\cdot (s_2+s_7)-W\cdot s_7,\\
s_{11} &= -27UV\cdot s_1+UW\cdot s_2-12VW\cdot s_8+2W^2\cdot s_9
\end{align*}
show that $s_5$, $s_{10}$ and $s_{11}$ can be omitted (since $\chara(k)\neq 3$).

The remaining generators are linearly independent and have total degrees twelve, thirteen or fourteen. The only possibility to have one generator to much, is an 
equation of the form $a_{14}=V\cdot b_{12}+W\cdot c_{12}$, where $a_{14}$, $b_{12}$ and $c_{12}$ are linear combinations of the generators of total degree given by the index. 
It's easy to verify that there is no such equation.

Moreover, if $\chara(k)\geq 5$, we can omit $s_8$ and $s_9$ on the covering $D(U)\cup D(W)$ as the following relations show
\begin{align*}
3U\cdot s_8 &= 3VW\cdot s_1-V^2\cdot s_2-UV\cdot s_3+2VW\cdot s_6,\\
W\cdot s_8 &= -3U\cdot s_1+V\cdot s_9,\\
6U\cdot s_9 &= (18W^2-72V^2)\cdot s_1-UW\cdot s_3+UV\cdot s_4+(3W^2-12V^2)\cdot s_6,\\
6W\cdot s_9 &= (2W^2-12V^2)\cdot s_3+VW\cdot s_6-3U\cdot s_7.
\end{align*}
Finally, we have
\begin{align*}
 W\cdot (6s_1+s_6)-U\cdot s_3+V\cdot s_7 &= 0,\\
 W\cdot (2S_2+s_6)+U\cdot s_4-12V\cdot s_6 &= 0,
\end{align*}
showing $$\phi\pb(\Mcc_3\oplus\Mcc_4)\cong\Syz_{\Dc_4}(U,V,W)^2.$$

\section{The Case $E_7$}
The binary octahedral group $\Obb$ of order 48 is generated by $\sigma$, $\tau$, $\mu$ and $\kappa$, where
$$\sigma=\begin{pmatrix} i & 0\\ 0&-i\end{pmatrix}, \quad \tau=\begin{pmatrix}0 & 1\\ -1&0\end{pmatrix}, 
\quad \mu=\frac{1}{\sqrt{2}}\begin{pmatrix}\epsilon^7 & \epsilon^7\\ \epsilon^5&\epsilon\end{pmatrix}, \quad 
\kappa=\begin{pmatrix}\epsilon & 0\\ 0& \epsilon^7\end{pmatrix},$$ 
where $\epsilon$ is again a primitive eighth root of unity. The generating invariants are 
$$X:=x^{17}y-34x^{13}y^5+34x^5y^{13}-xy^{17}, Y:=x^{10}y^2-2x^6y^6+x^2y^{10}, Z:=x^8+14x^4y^4+y^8.$$ 
The equation of an $E_7$-singularity is given by 
$$X^2+108Y^3-YZ^3=0.$$
There are three normal subgroups of $\Obb$. The first one is generated by $\sigma$ and $\tau$, which is isomorphic to $\D_2$ and has $D_4$ as ring of invariants. 
The second normal subgroup is $\T=<\sigma,\tau,\mu>$, since $\kappa^2=\sigma$. The corresponding ring of invariants is an $E_6$-singularity. 
The third normal subgroup is generated by $\tau^2$ with ring of invariants $A_1$.

Recall from Theorem \ref{syze7} that there are seven classes of non-free, indecomposable, maximal Cohen-Macaulay modules and represantatives are given by
\begin{align*}
M_1 &= \Syz_{E_7}(X,Y,Z),\\
M_2 &= \Syz_{E_7}(X,Y^2,YZ,Z^2),\\
M_3 &= \Syz_{E_7}(XY,XZ,Y^2,YZ^2,Z^3),\\
M_4 &= \Syz_{E_7}(X,Y^2,YZ),\\
M_5 &= \Syz_{E_7}(XY,XZ,Y^2,YZ^2),\\
M_6 &= \Syz_{E_7}(X,Y,Z^2),\\
M_7 &= \Syz_{E_7}(X,Y).
\end{align*}
Note that we only need to compute their pull-backs to $\Ec_6$, since the maps $\Dc_4\ra\Ec_7$ and $\Ac_1\ra\Ec_7$ factor through $\Ec_6$. The map 
$$\eta:k[X,Y,Z]/(X^2+108Y^3-YZ^3)\longrightarrow k[U,V,W]/(U^2-V^3+108W^4)$$ 
is given by $X\mapsto UW$, $Y\mapsto W^2$, $Z\mapsto V.$ With $l\in\{1,2\}$ we get 
\begin{align*}
\eta\pb(\Syz_{\Ec_7}(X,Y,Z^l)) &\cong \Syz_{\Ec_6}(UW,W^2,V^l)\\
 &\cong \Syz_{\Ec_6}(U,W,V^l)\\
 &\cong \Syz_{\Ec_6}(U,V,W),
\end{align*}
showing $\eta\pb(\Mcc_1)\cong\eta\pb(\Mcc_6)\cong\Syz_{\Ec_6}(U,V,W)$. Similarly, one has
\begin{align*}
\eta\pb(\Syz_{\Ec_7}(\Mcc_4)) &\cong \Syz_{\Ec_6}(UW,W^4,VW^2)\\
 &\cong \Syz_{\Ec_6}(U,W^3,VW)\\
 &\cong \Syz_{\Ec_6}(U,W^2,V)\\
 &\cong \Syz_{\Ec_6}(iU+6\sqrt{3}W^2,V)\oplus\Syz_{\Ec_6}(-iU+6\sqrt{3}W^2,V),
\end{align*}
where the last isomorphism is one of those, we have computed at the beginning of this section.

The pull-back $\eta\pb(\Syz_{\Ec_7}(X,Y))\cong\Syz_{\Ec_6}(U,W)$ is free.

We have the isomorphisms
\begin{align*}
\eta\pb(\Mcc_2) &= \eta\pb(\Syz_{\Ec_7}(X,Y^2,YZ,Z^2))\\
 &\cong \Syz_{\Ec_6}(UW,W^4,VW^2,V^2)\\
 &\cong \Syz_{\Ec_6}(U,V^2,VW,W^3)\\
 &\cong \Syz_{\Ec_6}(U,V^2,VW,W^2)\text{ and}\\
\eta\pb(\Mcc_5) &= \eta\pb(\Syz_{\Ec_7}(XY,XZ,Y^2,YZ^2))\\
 &\cong \Syz_{\Ec_6}(UW^3,UVW,W^4,V^2W^2)\\
 &\cong \Syz_{\Ec_6}(UW^2,UV,W^3,V^2W)\\
 &\cong \Syz_{\Ec_6}(UW,UV,W^2,V^2)\\
 &\cong \Syz_{\Ec_6}(U,V^2,VW,W^2).
\end{align*}
The last module to consider is $\Mcc_3$. The pull-back splits by the following computation
\begin{align*}
 &\quad \eta\pb(\Syz_{\Ec_7}(XY,XZ,Y^2,YZ^2,Z^3))\\
&\cong \Syz_{\Ec_6}(UW^3,UVW,W^4,V^2W^2,V^3)\\
&\cong \Syz_{\Ec_6}(UW^2,UV,W^3,V^2W,V^3)\\
&\cong \Syz_{\Ec_6}(iU+6\sqrt{3}W^2,V^2,VW)\oplus \Syz_{\Ec_6}(-iU+6\sqrt{3}W^2,V^2,VW),
\end{align*}
where the last isomorphism is one of those, we have seen at the beginning of this section.

\section{The Case $E_8$}
The binary icosahedral group $\I$ of order 120 has two generators $\phi$ and $\psi$ with
$$\phi=-\left(\begin{array}{cc}\epsilon^3 & 0\\ 0& \epsilon^2\end{array}\right)\quad\text{and}\quad\psi=\frac{1}{\sqrt{5}}\left(\begin{array}{cc}-\epsilon+\epsilon^4 & \epsilon^2-\epsilon^3\\ \epsilon^2-\epsilon^3& \epsilon-\epsilon^4\end{array}\right),$$
where $\epsilon$ is a primitive fifth root of unity. 
The generating invariants are 
\begin{align*}
X&:= x^{30}+522x^{25}y^5-10005x^{20}y^{10}-10005x^{10}y^{20}-522x^5y^{25}+y^{30},\\ 
Y&:= x^{20}-228x^{15}y^5+494x^{10}y^{10}+228x^5y^{15}+y^{20},\\
Z&:= x^{11}y+11x^6y^6-xy^{11}. 
\end{align*}
An $E_8$-singularity is given by the equation 
$$X^2-Y^3-1728Z^5=0.$$ 
The only normal subgroup of $\mathbb{I}$ is $\Z/2$, generated by $\phi^5$. The corresponding map 
$$\iota:k[X,Y,Z]/(X^2-Y^3-1728Z^5)\rightarrow k[R,S,T]/(RS-T^2)$$ 
is given by
\begin{align*}
X &\mapsto R^{15}+522R^{10}T^5-10005R^5T^{10}-10005S^5T^{10}-522S^{10}T^5+S^{15},\\
Y &\mapsto R^{10}-228R^5T^5+494T^{10}+228S^5T^5+S^{10},\\
Z &\mapsto R^5T-S^5T+11T^6.
\end{align*}
There are eight isomorphism classes of non-free, indecomposable, maximal Cohen-Ma\-caulay modules, represented by (cf. Theorem \ref{syze8})
\begin{align*}
M_1 &= \Syz_{E_8}(X,Y,Z),\\
M_2 &= \Syz_{E_8}(X,Y^2,YZ,Z^2),\\
M_3 &= \Syz_{E_8}(XY,XZ,Y^2,YZ^2,Z^3),\\
M_4 &= \Syz_{E_8}(XY,XZ^2,Y^3,Y^2Z,YZ^3,Z^4),\\
M_5 &= \Syz_{E_8}(XY^2,XYZ^2,XZ^4,Y^4,Y^3Z,Y^2Z^3,Z^5),\\
M_6 &= \Syz_{E_8}(X,Y^2,YZ,Z^3),\\
M_7 &= \Syz_{E_8}(XY,XZ,Y^2,YZ^2,Z^4),\\
M_8 &= \Syz_{E_8}(X,Y,Z^2).
\end{align*}
Since there are only the indecomposable, maximal Cohen-Macaulay modules $A_1$ and $\Syz_{A_1}(R,T)$ on $A_1$, we only need the minimal numbers of generators 
of $\iota\pb(\Syz_{\Ec_8}(\Mcc_i))$ to understand its splitting behaviour. Explicitely, the module $A_1$ is one-generated, while $\Syz_{A_1}(R,T)$ is two-generated.
Now, a maximal Cohen-Macaulay $A_1$-module of rank $r$, that is minimally generated by $\mu$ elements, has exactly $\mu-r$ direct summands isomorphic to 
$\Syz_{A_1}(R,T)$ and $2r-\mu$ free direct summands. Therefore, we only explain the Macaulay2 computation, given in Appendix \ref{m2.mcmpb.e8}.

First consider the ring $A:=\Z[R,S,T]/(RS-T^2)$. Since the syzygy modules, we are interested in, have far to many generators, when working over $A$, we also work 
over $B:=\Q[R,S,T]/(RS-T^2)$. The Macaulay2 command for the canonical inclusion $A\inj B$ is \verb map(B,A)  and for a matrix $M$ over $A$ 
the command \verb map(B,A)**M  treats $M$ as a matrix over $B$. We denote the $A$-matrix, whose columns generate $\iota\pb(\Syz_{\Ec_8}(\Mcc_i))$ 
with coefficients in $\Z$ by \verb Si  and the corresponding $B$-matrix by \verb Ti  (cf. \ref{m2.mcmpb.e8} up to \verb i21 ).

Using \verb mingens  \verb image  \verb Ti , we produce matrices \verb M1 $,\ldots,$ \verb M8 , such that the columns of \verb Mi  and \verb Ti  generate the same module, but 
the \verb Mi  have less columns than the \verb Ti  (cf. \ref{m2.mcmpb.e8} \verb i22-i29 ).

In \ref{m2.mcmpb.e8} \verb i30-i37  we compute the canonical map from \verb image  \verb Ti  to \verb image  \verb Mi . 
The outputs are matrices, whose $j-th$ column contains the coefficients necessary to write the $j-th$ column of \verb Ti  in terms of the columns of \verb Mi . 
Working in positive characteristic, we may use the \verb Mi  instead of the \verb Ti , whenever all the coefficients are well-defined.
We always have to invert two and three, which is no problem, since our characteristic has to be at least seven to write down the generators of the group $\I$. 
In the case of \verb T3  we additionally need to invert five and seven and for \verb T5  we have to invert five and 59. 
The computations with coefficients in $\Z/(p)$, $p\in\{7,59\}$ are done in \ref{m2.mcmpb.e8} \verb i38-i43 .

In all cases we end up with a matrix \verb Mi , whose columns generate $\iota\pb(\Syz_{\Ec_8}(\Mcc_i))$. For each $i$ these generators are linearly independent and have 
the same total degree, hence they give a minimal system of generators.

The final results are gathered together in Table \ref{tabpbe8}. 

\section{Overview}
We end this section with an overview. The modules in the columns are modules over the scheme in the head-spot of the column. The rows tell us to which 
module the module in the first column pulls back.
\begin{center}
\begin{table}[h]
$$\begin{array}{c||c|c|c}
\Dc_{n+2} & \Ac_{2n-1} & \Ac_{t-1},\text{ }t|2n & \Dc_{n/2+2},\text{ ($n$ even)}\\ \hline
\Mcc_1 & \Oc_{\Ac_{2n-1}} & \Oc_{\Ac_{t-1}} & \Mcc_1\\
\Mcc_{2m} & \Mcc_{2m-1}\oplus \Mcc_{2(n-m)+1} & \Mcc_r\oplus \Mcc_{t-r},\text{ }r=2m-1\text{ mod } t & \Mcc_{2m}\\
\Mcc_{2m-1} & \Mcc_{2m-2}\oplus \Mcc_{2(n-m)+2} & \Mcc_r\oplus \Mcc_{t-r},\text{ }r=2m-2\text{ mod } t & \Mcc_{2m-1}\\
\Mcc_{n+1},\text{ $n$ even} & \Mcc_n & \Mcc_{r},\text{ }r=n\text{ mod } t & \Mcc_1\\
\Mcc_{n+2},\text{ $n$ even} & \Mcc_n & \Mcc_{r},\text{ }r=n\text{ mod } t &\Oc_{\Dc_{n/2+2}}\\
\Mcc_{n+1},\text{ $n$ odd} & \Mcc_n & \Mcc_{r},\text{ }r=n\text{ mod } t & -\\
\Mcc_{n+2},\text{ $n$ odd} & \Mcc_n & \Mcc_{r},\text{ }r=n\text{ mod } t & - 
\end{array}$$
\caption{The table shows the pull-back behaviour of the $\Mcc_i$ on $\Dc_{n+2}$.}
\label{tabpbd}
\end{table}
\end{center}
In the third column of Table \ref{tabpbd} the number $r$ is the minimal non-negative integer satisfying the modulo condition.
\begin{center}
\begin{table}[h]
$$\begin{array}{c||c|c}
\Ec_6 & \Dc_4 & \Ac_1 \\ \hline
\Mcc_1 & \Mcc_2 & \Mcc_1^2\\
\Mcc_2 & \Mcc_1\oplus \Mcc_3\oplus \Mcc_4 & \Oc_{\Ac_1}^3\\
\Mcc_3 & \Mcc_2 & \Mcc_1^2 \\
\Mcc_4 & \Mcc_2 & \Mcc_1^2 \\
\Mcc_5 & \Oc_{\Dc_4} & \Oc_{\Ac_1} \\
\Mcc_6 & \Oc_{\Dc_4} & \Oc_{\Ac_1}
\end{array}$$
\caption{The table shows the pull-back behaviour of the $\Mcc_i$ on $\Ec_{6}$.}
\label{tabpbe6}
\end{table}
\end{center}
\begin{center}
\begin{table}[h]
$$\begin{array}{c||c|c|c}
\Ec_7 & \Ec_6 & \Dc_4 & \Ac_1\\ \hline
\Mcc_1 & \Mcc_1 & \Mcc_2 & \Mcc_1^2\\
\Mcc_2 & \Mcc_2 & \Mcc_1\oplus \Mcc_3\oplus \Mcc_4 & \Oc_{\Ac_1}^3\\
\Mcc_3 & \Mcc_3\oplus\Mcc_4 & \Mcc_2^2 & \Mcc_1^4\\
\Mcc_4 & \Mcc_5\oplus \Mcc_6 & \Dc_4^2 & \Oc_{\Ac_1}^2\\
\Mcc_5 & \Mcc_2 & \Mcc_1\oplus \Mcc_3\oplus \Mcc_4 & \Oc_{\Ac_1}^3\\
\Mcc_6 & \Mcc_1 & \Mcc_2 & \Mcc_1^2\\
\Mcc_7 & \Ec_6 & \Dc_4 & \Oc_{\Ac_1}
\end{array}$$
\caption{The table shows the pull-back behaviour of the $\Mcc_i$ on $\Ec_{7}$.}
\label{tabpbe7}
\end{table}
\end{center}
\vspace*{-1cm}
\begin{center}
\begin{table}[h]
$$\begin{array}{c||c}
\Ec_8 & \Ac_1 \\ \hline
\Mcc_1 & \Mcc_1^2\\
\Mcc_2 & \Oc_{\Ac_1}^3\\
\Mcc_3 & \Mcc_1^4\\
\Mcc_4 & \Oc_{\Ac_1}^5\\
\Mcc_5 & \Mcc_1^6\\
\Mcc_6 & \Oc_{\Ac_1}^3\\
\Mcc_7 & \Oc_{\Ac_1}^4\\
\Mcc_8 & \Mcc_1^2\\
\end{array}$$
\caption{The table shows the pull-back behaviour of the $\Mcc_i$ on $\Ec_{8}$.}
\label{tabpbe8}
\end{table}
\end{center}
\vspace*{-1cm}
Using these direct computations, we can show that the Frobenius pull-backs of the syzygy modules of the maximal ideals stay indecomposable in almost all 
characteristics if the singularity is of type $D$ or $E$.

\begin{lem}\label{fpbindec}
Let $U$ be the punctured spectrum of a surface ring $R$ of type $D$ or $E$ and let $\TF:U\ra U$ be the (restricted) Frobenius morphism. Assume that the 
order of the group corresponding to the singularity is invertible in the ground field. Then $\fpb{e}(\Syz_U(\mm))$ is indecomposable for all $e\in\N$.
\end{lem}

\begin{proof}
Recall from Chapter \ref{chapmatfac} that $\Syz_U(\mm)$ is isomorphic to $\Mcc_2$ if $R$ is of type $D$ and $\Mcc_1$ if $R$ is of type $E$. Assume first that $R$ is of type 
$D_{2n+2}$ or $E$. If $\fpb{e}(\Syz_U(\mm))$ splits for some $e\geq 1$, its pull-back to $\Ac_1$ (along the morphism $\iota$, induced by the inclusion into $A_1$) gets trivial, 
since all line bundles pull back to the structure sheaf by Tables \ref{tabpbd} - \ref{tabpbe8}. But the pull-back of $\Syz_U(\mm)$ along $\iota$ is 
non-trivial (in fact, it is always $\Mcc_1^2$ by the direct computations) and its Frobenius pull-backs stay non-trivial in odd 
characteristics. This gives a contradiction, since the two types of pull-backs commute.

Assume that $R$ is of type $D_{n+2}$ with $n$ odd. Since $\fpb{e}(\Syz_U(\mm))$ might split as $\Mcc_{n+1}\oplus\Mcc_{n+2}$ and 
both pull back to $\Mcc_1$ on $\Ac_1$, the argument above gives no contradiction. In this case, we have to compare the pull-backs to $\Ac_{n-1}$. Then 
all line bundles from $\Dc_{n+2}$ pull back to the structure sheaf of $\Ac_{n-1}$, but $\iota\pb(\Mcc_2)\cong\Mcc_1\oplus\Mcc_{n-1}$. The Frobenius pull-backs 
on the right hand side do not get trivial, since $n$ is invertible modulo $p$.

Moreover, if $R$ is of type $D$, the Frobenius pull-backs of $\Mcc_2$ are again of the form $\Mcc_{2m}$, since the $\Mcc_{2m-1}$ get trivial after pulling them 
back to $\Ac_1$.
\end{proof}

By the previous lemma all Frobenius pull-backs of $\Syz_U(\mm)$ are again indecomposable of rank two. Since in all cases $\Det(\fpb{e}(\Syz_U(\mm)))\cong\Oc_U$, 
we obtain directly $\fpb{e}(\Syz_U(\mm))\cong\Syz_U(\mm)$ if $R$ is of type $E_6$, since in this case there is only one rank two module, whose representation 
is a syzygy module of an $\mm$-primary ideal. In the next chapter we will see that Lemma \ref{fpbindec} leads always to a $(0,e_0)$-Frobenius periodicity 
$\fpb{e_0}(\Syz_U(\mm))\cong\Syz_U(\mm)$. For example if $R$ is of type $E_7$ or $E_8$, the length $e_0$ of this periodicity is - depending 
on the characteristic - one or two.

\chapter{The Hilbert-Kunz functions of surface rings of type ADE}\label{chaphkfade}
In this chapter we compute the Hilbert-Kunz functions of surface rings of type ADE. We recall very briefly the situation. Let $R:=k[U,V,W]/(f)$ be a surface 
ring of type ADE with $\Deg(U)=\alpha$, $\Deg(V)=\beta$, $\Deg(W)=\gamma$ and let $S:=k[X,Y,Z]/(F)$ be the associated standard-graded ring. Denote their 
projective spectra by $C_R$ resp. $C_S$ and their punctured spectra by $\Rc$ resp. $\Sc$. Then by Remark \ref{hkfflat} the Hilbert-Kunz function of $R$ 
is the Hilbert-Kunz function of $S$ with respect 
to the ideal $(X^{\alpha},Y^{\beta},Z^{\gamma})$ divided by $\alpha\cdot\beta\cdot\gamma$, since the map $R\ra S$, sending $U$, $V$ and $W$ to 
$X^{\alpha}$, $Y^{\beta}$ resp. $Z^{\gamma}$, is flat by \cite[Theorem 22.2]{matsu} and the extension $Q(S):Q(R)$ has degree $\alpha\cdot\beta\cdot\gamma$. 
In the standard-graded situation we may use the geometric approach 
of Brenner/ Trivedi, which is possible, if we find for each $q=p^e$ an isomorphism of the form
$$\Syz_{S}(X^{\alpha\cdot q},Y^{\beta\cdot q},Z^{\gamma\cdot q})\cong\Syz_{S}(X^a,Y^b,Z^c)(n),$$
where the vector $(a,b,c)$ and the degree shift $n\in\Z$ depend on $q$ and at least one of the inequalities $\alpha\cdot q<a$, 
$\beta\cdot q<b$, $\gamma\cdot q<c$ holds. In each case we will show that there exists a 
finite list of triples $(a',b',c')$ such that there are isomorphisms
$$\Syz_{R}(U^q,V^q,W^q)\cong\Syz_{R}\left(U^{a'},V^{b'},W^{c'}\right)(m)$$
for some $m\in\Z$ and for all $q$. These isomorphisms induce isomorphisms in the standard-graded case of the desired form.
After all we use the formula
\begin{equation}\label{formnontr}
\begin{split}
 & \HKF(R,q) \\
=& \frac{d\cdot Q(\alpha,\beta,\gamma)\cdot q^2-d\cdot Q(a'\cdot\alpha,b'\cdot\beta,c'\cdot\gamma)+4\cdot\Dim_k(S/(X^{a'\cdot\alpha},Y^{b'\cdot\beta},Z^{c'\cdot\gamma}))}{4\cdot\alpha\cdot\beta\cdot\gamma},
\end{split}
\end{equation}
with $d=\deg(f)$ and $Q(r,s,t)=2(rs+rt+st)-r^2-s^2-t^2$ from Lemma \ref{hkfdreivar} (together with Remark \ref{hkfflat}).

In small characteristics, it happens sometimes that the syzygy modules $\Syz_R(U^q,V^q,W^q)$ split as $R(-m+\delta)\oplus R(-m)$ for some $m,\delta\in\Z$. 
In those cases the Hilbert-Kunz function can be computed by the formula (cf. Lemma \ref{hkffrei} and Remark \ref{hkfflat})
\begin{equation}\label{formtr}
\HKF(R,q)=\frac{d\cdot Q(\alpha,\beta,\gamma)\cdot q^2+d\cdot \delta^2}{4\cdot\alpha\cdot\beta\cdot\gamma}.
\end{equation}

\section{The case $A_{n-1}$ with $n\geq 1$}
Let $R:=k[U,V,W]/(U^n-VW)$ with $\Deg(U)=2$, $\Deg(V)=n$, $\Deg(W)=n$ and $\chara(k)=p\nmid n$. In Chapter \ref{chapmatfac} we saw that we may choose $\Syz_{\Rc}(U,V)$ 
as generator of the Picard group of $\Rc$. The $e$-th Frobenius pull-back is just the $q$-th tensor power, which is isomorphic to the tensor power of 
order $q$ modulo $n$, hence $\Syz_{\Rc}(U^q,V^q)\cong\Syz_{\Rc}(U^r,V)$ with $r\equiv q$ mod $n$, $r\in\{1,\ldots,n-1\}$, by Lemma \ref{lempban}. 
Using the isomorphism 
$$\Syz_R(U^r,V,W)\cong\Syz_R(U^r,V)\oplus\Syz_R(U^r,W)\cong\Syz_R(U^r,V)\oplus\Syz_R(U^{n-r},V)^{\vee},$$
we find 
$$\Syz_R(U^q,V^q,W^q)\cong\Syz_R(U^r,V,W)(m)$$
for some $m\in\Z$.

This yields with Formula (\ref{formnontr})
\begin{align*}
\HKF_R(R,q) &= \frac{2n\cdot Q(2,n,n)\cdot q^2-2n\cdot Q(2r,n,n)+8rn^2}{8n^2}\\
&= \frac{(8n-4)q^2-8rn+4r^2}{4n}+r\\
&= \left(2-\frac{1}{n}\right)q^2-r+\frac{r^2}{n}.
\end{align*}
Note that we can define all appearing syzygy modules in the case $p|n$ as well. 
Since the computations in Section 4.1 are correct in all characteristics, the above computation of the Hilbert-Kunz function remains correct in the case
$p|n$, if we would know that our list of maximal Cohen-Macaulay modules is complete.

By a theorem of Herzog (cf. \cite[Theorem 6.3]{leuwie}) we only need that $R$ is a direct summand of the polynomial ring $k[u,v]$, on which the group $\Z/(n)$ acts, to 
have an one to one correspondence between indecomposable, maximal Cohen-Macaulay $R$-modules and direct summands of $k[u,v]$ as an $R$-module.

Leuschke and Wiegand point out in \cite[Remark 6.22]{leuwie} that $R$ is a direct summand of $k[u,v]$ and that the direct summands of $k[u,v]$ as an $R$-module are 
the same as in the non-modular case (meaning that the algebraic definition is the same, although the modules in the modular case might not come from representations of the group $\Z/n$). 
Therefore, the formula above is a full description of the Hilbert-Kunz function (compare with Example \ref{genhkfan} or Remark \ref{kunzbsp}).

By Theorem \ref{fsighkf}, the F-signature function of $R$ is given by
$$\FS(R,q)=\frac{1}{n}\cdot q^2+r-\frac{r^2}{n}.$$
\vspace*{-1cm}
\section{The case $D_{n+2}$ with $n\geq 2$}
Let $R:=k[U,V,W]/(U^2+V^{n+1}+VW^2)$ with $\Deg(U)=n+1$, $\Deg(V)=2$, $\Deg(W)=n$ and $\chara(k)=p\nmid 4n$.

Recall from Chapter \ref{pbmaxcm} that there is an embedding 
$$\phi:R\lra A_{2n-1}:=k[X,Y,Z]/(X^{2n}-YZ)$$ 
with the following pull-back behaviour 
$$\begin{array}{c||c}
\Rc & \Ac_{2n-1}\\ \hline
\Mcc_1 & \Oc_{\Ac_{2n-1}}\\
\Mcc_{2m} & \Nc_{2m-1}\oplus\Nc_{2(n-m)+1}\\
\Mcc_{2m-1} & \Nc_{2m-2}\oplus\Nc_{2(n-m)+2}\\
\Mcc_{n+1} & \Nc_n\\
\Mcc_{n+2} & \Nc_n,
\end{array}$$
where the $\Nc_i$ are the $\Mcc_i$ on $\Ac_{2n-1}$ from Chapter \ref{pbmaxcm}. By Lemma \ref{fpbindec} we have 
\begin{equation}
\fpb{e}(\Syz_{\Rc}(\mm))=\fpb{e}(\Mcc_2)\cong \Mcc_{2m}  
\end{equation}
for all $e\in\N$ and $m=m(q)\in\{1,\ldots,\tfrac{n}{2}\}$. The parameter $m=m(q)$ can be computed by considering the pull-back via $\phi$: 
$\Mcc_2$ pulls back to $\Nc_1\oplus\Nc_{2n-1}$ on $\Ac_{2n-1}$. The $e$-th Frobenius pull-back of this module is $\Nc_r\oplus\Nc_{2n-r}$ with 
$r\equiv q$ mod $2n$ (and $1\leq r\leq 2n-1$). If $r\leq n-1$, only $\Mcc_{r+1}$ pulls back to $\Nc_r\oplus\Nc_{2n-r}$ via $\phi$. 
If $r\geq n$, only $\Mcc_{2n-r-1}$ pulls back to $\Nc_r\oplus\Nc_{2n-r}$ via $\phi$. But in Chapter \ref{chapmatfac} we saw that $\Mcc_{r+1}$ and $\Mcc_{2n-r-1}$ 
are isomorphic (on $\Rc$). We find 
$$\Syz_{\Rc}(U^q,V^q,W^q)\cong\Syz_{\Rc}(U,V^{(r+1)/2},W),$$ 
where $r\equiv q$ mod $2n$ with $1\leq r\leq 2n-1$. This isomorphism induces a global isomorphism
$$\Syz_R(U^q,V^q,W^q)\cong\Syz_R(U,V^{(r+1)/2},W)(m),$$
for some $m\in\Z$. Using Formula (\ref{formnontr}), we get
\begin{align*}
\HKF(R,q) &= \frac{(2n+2)\cdot Q(n+1,2,n)\cdot q^2-(2n+2)\cdot Q(n+1,r+1,n)+4(r+1)(n^2+n)}{8(n^2+n)}\\
&= \frac{(8n-1)q^2-4nr-4n+r^2}{4n}+\frac{r+1}{2}\\
&= \left(2-\frac{1}{4n}\right)q^2-\frac{r+1}{2}+\frac{r^2}{4n}.
\end{align*}
Dealing with characteristics dividing $4n$, the computations in Section 4.2 (and 4.1) still hold if $p$ is odd. By a similar argument as in the previous section (just change the groups, the references are the same) 
the indecomposable, maximal Cohen-Macaulay modules (defined algebraically) are the same in the modular and the non-modular case.
To get a full description of the Hilbert-Kunz function, we only need to take a closer look at the case $p=2$. In this case the ring $D_{n+2}$ in Section 4.2 is not a $D_{n+2}$-singularity, since it is a binomial hypersurface.

Let $q=2^e$ with $e\geq 1$. Using Lemma \ref{manipulationlemma}, we obtain
\begin{align*}
\Syz_R(U^q,V^q,W^q) &\cong \Syz_R((V^{n+1}+VW^2)^{\frac{q}{2}},V^q,W^q)\\
&\cong \Syz_R(V^{(n+1)\cdot\frac{q}{2}}+V^{\frac{q}{2}}W^q,V^q,W^q)\\
&\cong \Syz_R(V^{\frac{q}{2}}W^q,V^q,W^q)\\
&\cong \Syz_R(V^{\frac{q}{2}},V^q,1)(-nq)\\
&\cong \Syz_R(1,V^{\frac{q}{2}},1)(-(n+1)q)\\
&\cong R(-(n+1)q)\oplus R(-(n+2)q).
\end{align*}
Using Formula (\ref{formtr}), we get 
$$\HKF(R,2^e) = \frac{Q(n+1,2,n)\cdot 4^e+4^e}{4n}=2\cdot 4^e.$$

\begin{thm}\label{hkfdn}
 The Hilbert-Kunz function of $D_{n+2}$ is given by (with $n\geq 2$, $q=p^e$ and $r\equiv q$ mod $2n$ for $0\leq r\leq 2n-1$)
$$e\longmapsto\left\{\begin{aligned}
 & 2\cdot q^2 && \text{if } e\geq 1,p=2,\\
 & \left(2-\frac{1}{4n}\right)q^2-\frac{r+1}{2}+\frac{r^2}{4n} && \text{otherwise.}
          \end{aligned}\right.$$
\end{thm}

Using Theorem \ref{fsighkf}, one obtains

\begin{cor}\label{fsigndn}
The F-signature function of $R$ is given by
$$\FS(R,q)=\left\{\begin{aligned}
 & 0 && \text{if } e\geq 1,p=2,\\
 & \frac{1}{4n}\cdot q^2+\frac{r+1}{2}-\frac{r^2}{4n} && \text{otherwise.}
          \end{aligned}\right.$$
\end{cor}

\section{The case $E_6$}
Let $R:=k[U,V,W]/(U^2+V^3+W^4)$ with $\deg(U)=6$, $\Deg(V)=4$, $\Deg(W)=3$ and $\chara(k)=p\geq 5$. 
The Frobenius pull-backs of $\Syz_{\Rc}(U,V,W)$ are indecomposable by Lemma \ref{fpbindec}. Since they have rank two and trivial determinant, we obtain a 
$(0,1)$-Frobenius periodicity, extending globally to
$$\Syz_R(U^p,V^p,W^p)\cong\Syz_R(U,V,W)\left(-\frac{13}{2}\cdot (p-1)\right).$$
Iterating yields the isomorphism
$$\Syz_R(U^q,V^q,W^q) \cong \Syz_R(U,V,W)\left(-\frac{13}{2}\cdot (q-1)\right)$$
for all $q$ and we get (with Formula (\ref{formnontr}))
\begin{align*}
\HKF(R,q) &= \frac{12\cdot Q(6,4,3)\cdot q^2-12\cdot Q(6,4,3)+12\cdot 24}{12\cdot 24}\\
&= \frac{47\cdot q^2-47+24}{24}\\
&= \left(2-\frac{1}{24}\right)q^2-\frac{23}{24}.
\end{align*}
It remains to compute the Hilbert-Kunz function in characteristics two and three. Let $p=2$. For $q\geq 2$ we have by using Lemma \ref{manipulationlemma}
\begin{align*}
\Syz_R(U^q,V^q,W^q) &= \Syz_R\left(V^q,W^q,(V^3+W^4)^{\frac{q}{2}}\right)\\
&= \Syz_R(V^q,W^q,V^{\frac{3q}{2}}+W^{2q})\\
&\cong \Syz_R(V^q,W^q,W^{2q})\\
&\cong \Syz_R(V^q,1,W^q)(-3q)\\
&\cong R(-6q)\oplus R(-7q).
\end{align*}
Using Formula (\ref{formtr}), we get
$$\HKF(R,q)=\frac{12\cdot Q(6,4,3)\cdot q^2+12\cdot q^2}{4\cdot 6\cdot 4\cdot 3}=2\cdot q^2.$$
The analogue computation for $p=3$ and $q\geq 3$ gives $\HKF(R,q)=2\cdot q^2$.

\begin{thm}\label{hkfe6}
The Hilbert-Kunz function of $E_6$ is given by
$$e\longmapsto\left\{\begin{aligned}
 & 2\cdot q^2 && \text{if } e\geq 1\text{ and }p\in\{2,3\},\\
 & \left(2-\frac{1}{24}\right)q^2-\frac{23}{24} && \text{otherwise.}
\end{aligned}\right.$$
\end{thm}

From this we get by Theorem \ref{fsighkf} the F-signature function.

\begin{cor}\label{fsigne6}
The F-signature function of $R$ is given by
$$\FS(R,q)=\left\{\begin{aligned}
 & 0 && \text{if } e\geq 1\text{ and }p\in\{2,3\},\\
 & \frac{1}{24}\cdot q^2+\frac{23}{24} && \text{otherwise.}
\end{aligned}\right.$$
\end{cor}

\section{The case $E_7$}
Let $R:=k[U,V,W]/(U^2+V^3+VW^3)$ with $\Deg(U)=9$, $\Deg(V)=6$, $\Deg(W)=4$ and let $\chara(k)=p\geq 5$. We get from Lemma \ref{fpbindec} that the 
Frobenius pull-backs of $\Syz_{\Rc}(U,V,W)$ stay indecomposable and have to be isomorphic to either $\Syz_{\Rc}(U,V,W)$ or $\Syz_{\Rc}(U,V,W^2)$.
Recall from Chapter \ref{syzses} that the Hilbert-series can distinguish these two classes. The Hilbert-series of the corresponding $R$-modules are given by
\begin{align*}
\lK_{\Syz_R(U,V,W)}(t) &= \frac{t^{10}+t^{13}+t^{15}+t^{18}}{(1-t^6)(1-t^4)}\\
\lK_{\Syz_R(U,V,W^2)}(t) &= \frac{t^{14}+t^{15}+t^{17}+t^{18}}{(1-t^6)(1-t^4)}.
\end{align*}
With the help of Theorem \ref{HilbSer} and Han's Theorem \ref{hansthm} we compute the Hilbert-series of $\Syz_R(U^p,V^p,W^p)$ and $\Syz_R(U^p,V^p,W^{2p})$.
In the first case we need to compute the degrees of the generators of 
$$\Syz_R\left(V^p,W^p,(V^3+VW^3)^{\frac{p-1}{2}}\right)\qquad\text{and}\qquad\Syz_R\left(V^p,W^p,(V^3+VW^3)^{\frac{p+1}{2}}\right).$$
We have the isomorphisms
\begin{align*}
\Syz_R\left(V^p,W^p,(V^3+VW^3)^{\frac{p-1}{2}}\right) &\cong \Syz_R\left(V^{\frac{p+1}{2}},W^p,(V^2+W^3)^{\frac{p-1}{2}}\right)(-3(p-1))\text{ and}\\
\Syz_R\left(V^p,W^p,(V^3+VW^3)^{\frac{p+1}{2}}\right) &\cong \Syz_R\left(V^{\frac{p-1}{2}},W^p,(V^2+W^3)^{\frac{p+1}{2}}\right)(-3(p+1)).
\end{align*}
Their syzygy gaps are given by 
\begin{equation}\label{deltae71}
12\cdot\delta\left(\frac{p+1}{4},\frac{p}{3},\frac{p-1}{2}\right)\qquad\text{resp.}\qquad 12\cdot\delta\left(\frac{p-1}{4},\frac{p}{3},\frac{p+1}{2}\right). 
\end{equation}
Similarly, in the second case the syzygy gaps of the corresponding $\Sc_j$ from Theorem \ref{HilbSer} are given by 
\begin{equation}\label{deltae72}
12\cdot\delta\left(\frac{p+1}{4},\frac{2p}{3},\frac{p-1}{2}\right)\qquad\text{resp.}\qquad 12\cdot\delta\left(\frac{p-1}{4},\frac{2p}{3},\frac{p+1}{2}\right).
\end{equation}
The various arguments of $\delta$ do not satisfy the strict triangle inequality if and only if $p\leq 9$ and the arguments are $\tfrac{p-1}{4}$, $\tfrac{p}{3}$, $\tfrac{p+1}{2}$. In this case 
we have $12\cdot\delta=9-p$. In all other cases the $s$ in Han's Theorem has to be non-negative. For $s=0$, $p=12l+r$, $r\in\{1,5,7,11\}$ and $\tfrac{p}{3}$ as second argument, the values 
of $12\cdot\delta$ are given in Table \ref{tabe71}.
\begin{center}
\begin{table}[h]
$$\begin{array}{c|c|c|c|c}
r & +,l\text{ even} & +,l\text{ odd} & -,l\text{ even} & -,l\text{ odd}\\ \hline
1 & 2 & 2 & 8 & 4 \\
5 & 2 & 2 & 4 & 8 \\
7 & 8 & 4 & 2 & 2 \\
11 & 4 & 8 & 2 & 2.
\end{array}$$
\caption{The table shows the values of (\ref{deltae71}) for all primes $p\geq 5$ depending on their decomposition as $p=12l+r$. The sign in the headline is 
the same as in the first argument of (\ref{deltae71}).}
\label{tabe71}
\end{table}
\end{center}
\vspace*{-1cm}
With $\tfrac{2p}{3}$ as second argument, the values of $12\cdot\delta$ are given in Table \ref{tabe72}.
\begin{center}
\begin{table}[h]
$$\begin{array}{c|c|c|c|c}
r & +,l\text{ even} & +,l\text{ odd} & -,l\text{ even} & -,l\text{ odd}\\ \hline
1 & 2 & 2 & 4 & 8 \\
5 & 2 & 2 & 8 & 4 \\
7 & 4 & 8 & 2 & 2 \\
11 & 8 & 4 & 2 & 2.
\end{array}$$
\caption{The table shows the values of (\ref{deltae72}) for all primes $p\geq 5$ depending on their decomposition as $p=12l+r$. The sign in the headline is 
the same as in the first argument of (\ref{deltae72}).}
\label{tabe72}
\end{table}
\end{center}
Computing the Hilbert-series, we find
$$\begin{aligned}
\lK_{\Syz_R(U^p,V^p,W^p)}(t) &= \left\{\begin{aligned} & t^{19\cdot\frac{p-1}{2}}\cdot \lK_{\Syz_R(U,V,W)}(t) && \text{if } p\text{ mod }24\in \{\pm 1,\pm 7\},\\ 
 & t^{19\cdot\frac{p-1}{2}-2}\cdot \lK_{\Syz_R(U,V,W^2)}(t) && \text{if } p\text{ mod }24\in \{\pm 5,\pm 11\},\end{aligned}\right.\\
\lK_{\Syz_R(U^p,V^p,W^{2p})}(t) &= \left\{\begin{aligned} & t^{23\cdot\frac{p-1}{2}}\cdot \lK_{\Syz_R(U,V,W^2)}(t) && \text{if } p\text{ mod }24\in \{\pm 1,\pm 7\},\\  & t^{23\cdot\frac{p-1}{2}+2}\cdot \lK_{\Syz_R(U,V,W)}(t) && \text{if } p\text{ mod }24\in \{\pm 5,\pm 11\}.\end{aligned}\right.
\end{aligned}$$
For the higher pull-backs we get 
$$\Syz_R(U^q,V^q,W^q)\cong \left\{\begin{aligned} & \Syz_R(U,V,W)\left(-\frac{19(q-1)}{2}\right) && \text{if } q\text{ mod }24\in \{\pm 1,\pm 7\},\\ & \Syz_R(U,V,W^2)\left(-\frac{19(q-1)}{2}+2\right) && \text{if } q\text{ mod }24\in \{\pm 5,\pm 11\}.\end{aligned}\right.$$
Therefore, the Hilbert-Kunz function is given by (use Formula (\ref{formnontr}))
\begin{align*}
\HKF(R,q) &= \frac{18\cdot Q(9,6,4)\cdot q^2-18\cdot Q(9,6,4)+18\cdot 48}{18\cdot 48}\\
&= \frac{95\cdot q^2-95+48}{48}\\
&= \left(2-\frac{1}{48}\right)q^2-\frac{47}{48},
\end{align*}
if $q\text{ mod }24\in\{\pm 1,\pm 7\}$ and (again with Formula (\ref{formnontr}))
\begin{align*}
\HKF(R,q) &= \frac{18\cdot Q(9,6,4)\cdot q^2-18\cdot Q(9,6,8)+18\cdot 96}{18\cdot 48}\\
&= \frac{95\cdot q^2-167+96}{48}\\
&= \left(2-\frac{1}{48}\right)q^2-\frac{71}{48},
\end{align*}
if $q\text{ mod }24\in\{\pm 5,\pm 11\}$.

Now let $p=2$ and $q\geq 2$. Then (by Lemma \ref{manipulationlemma})
\begin{align*}
\Syz_R(U^q,V^q,W^q) &\cong \Syz_R\left(V^q,W^q,\left(V^3+VW^3\right)^{\frac{q}{2}}\right)\\
&\cong \Syz_R\left(V^q,W^q,V^{\frac{3q}{2}}+V^{\frac{q}{2}}W^{\frac{3q}{2}}\right)\\
&\cong \Syz_R\left(V^q,W^q,V^{\frac{q}{2}}W^{\frac{3q}{2}}\right)\\
&\cong \Syz_R\left(V^{\frac{q}{2}},W^q,W^{\frac{3q}{2}}\right)(-3q)\\
&\cong \Syz_R\left(V^{\frac{q}{2}},1,W^{\frac{q}{2}}\right)(-7q)\\
&\cong R(-9q)\oplus R(-10q).
\end{align*}
The corresponding value of the Hilbert-Kunz function is $2\cdot q^2$ by Formula (\ref{formtr}).

For $p=3$ and $q\geq 3$ we have (with Lemma \ref{manipulationlemma})
\begin{align*}
 \Syz_R(U^q,V^q,W^q) &\cong \Syz_R(U^q,(U^2+VW^3)^{\frac{q}{3}},W^q)\\
 &\cong \Syz_R(U^q,U^{\frac{2q}{3}}+V^{\frac{q}{3}}W^q,W^q)\\
 &\cong \Syz_R(U^q,U^{\frac{2q}{3}},W^q)\\
 &\cong \Syz_R(U^{\frac{q}{3}},1,W^q)(-6q)\\
 &\cong R(-9q)\oplus R(-10q).
\end{align*}
Again we get $\HKF(R,q)=2\cdot q^2$.

\begin{thm}\label{hkfe7}
 The Hilbert-Kunz function of $E_7$ is given by
$$e\longmapsto\left\{\begin{aligned}
 & 2\cdot q^2 && \text{if } e\geq 1\text{ and }p\in\{2,3\},\\
 & \left(2-\frac{1}{48}\right)q^2-\frac{47}{48} && \text{if }q\text{ mod }24\in\{\pm1,\pm7\},\\
 & \left(2-\frac{1}{48}\right)q^2-\frac{71}{48} && \text{otherwise.}
\end{aligned}\right.$$
\end{thm}

Reformulating the cases in the above theorem in terms of $p$, one obtains the function
$$e\longmapsto\left\{\begin{aligned}
 & 2\cdot q^2 && \text{if } e\geq 1\text{ and }p\in\{2,3\},\\
 & \left(2-\frac{1}{48}\right)q^2-\frac{71}{48} && \text{if }p\text{ mod }24\in\{\pm 5,\pm 11\}\text{ and }e\text{ is odd},\\
 & \left(2-\frac{1}{48}\right)q^2-\frac{47}{48} && \text{otherwise.}
\end{aligned}\right.$$
By using Theorem \ref{fsighkf}, we can compute the F-signature function.

\begin{cor}\label{fsigne7}
The F-signature function of $R$ is given by
$$\FS(R,q)=\left\{\begin{aligned}
 & 0 && \text{if } e\geq 1\text{ and }p\in\{2,3\},\\
 & \frac{1}{48}\cdot q^2+\frac{47}{48} && \text{if }q\text{ mod }24\in\{\pm1,\pm7\},\\
 & \frac{1}{48}\cdot q^2+\frac{71}{48} && \text{otherwise.}
\end{aligned}\right.$$
\end{cor}

\section{The case $E_8$}\label{sectionhkfe8}
Let $R:=k[U,V,W]/(U^2+V^3+W^5)$ with $\Deg(U)=15$, $\Deg(V)=10$, $\Deg(W)=6$ and let $\chara(k)=p\geq 7$. 
By Lemma \ref{fpbindec} the Frobenius pull-back of $\Syz_{\Rc}(U,V,W)$ is either
$\Syz_{\Rc}(U,V,W)$ or $\Syz_{\Rc}(U,V,W^2)$. Moreover, their classes can be distinguished by the Hilbert-series. We have 
\begin{align*}
\lK_{\Syz_R(U,V,W)}(t) &= \frac{t^{16}+t^{21}+t^{25}+t^{30}}{(1-t^{10})(1-t^6)}\\
\lK_{\Syz_R(U,V,W^2)}(t) &= \frac{t^{22}+t^{25}+t^{27}+t^{30}}{(1-t^{10})(1-t^6)}.
\end{align*}
We compute the Hilbert-series of $\Syz_R(U^p,V^p,W^p)$ and $\Syz_R(U^p,V^p,W^{2p})$, using Theorem \ref{HilbSer} and Han's Theorem \ref{hansthm}.

In the first case we need to compute 
\begin{equation}\label{deltae81}
30\cdot\delta\left(\frac{p}{3},\frac{p}{5},\frac{p\mp 1}{2}\right) 
\end{equation}
and in the second case we are interested in the numbers 
\begin{equation}\label{deltae82}
30\cdot\delta\left(\frac{p}{3},\frac{2p}{5},\frac{p\mp 1}{2}\right).
\end{equation}
Note that for primes at most 13 the numbers $\tfrac{p}{3}$, $\tfrac{p}{5}$ and $\tfrac{p+1}{2}$ do not satisfy the strict triangle inequality, in these cases we get $30\cdot\delta=15-p$. 
In all other cases the $s$ in Han's Theorem is non-negative. With $s=0$, $p=30l+r$ and $r\in\{1,7,11,13,17,19,23,29\}$ one gets for $30\cdot\delta$ with $\tfrac{p}{5}$ as second argument the values
 shown in Table \ref{tabe81} and with $\tfrac{2p}{5}$ as second argument the values are given in Table \ref{tabe82}.
\begin{center}
\begin{table}[h]
$$\begin{array}{c|c|c|c|c}
r & -,l\text{ odd} & -,l\text{ even} & +,l\text{ odd} & +,l\text{ even}\\ \hline
1 & 14 & 4 & 4 & 14 \\
7 & 8 & 2 & 2 & 8 \\
11 & 4 & 14 & 14 & 4 \\
13 & 2 & 8 & 8 & 2 \\
17 & 2 & 8 & 8 & 2 \\
19 & 4 & 14 & 14 & 4\\
23 & 8 & 2 & 2 & 8 \\
29 & 14 & 4 & 4 & 14
\end{array}$$
\caption{The table shows the values of (\ref{deltae81}) for all primes $p\geq 7$ depending on their decomposition as $p=30l+r$. The sign in the headline is 
the same as in the last argument of (\ref{deltae81}).}
\label{tabe81}
\end{table}
\end{center}
\begin{center}
\begin{table}[h]
$$\begin{array}{c|cc|cc}
r & -,l\text{ odd} & -,l\text{ even} & +,l\text{ odd} & +,l\text{ even}\\ \hline
1 & 8 & 2 & 2 & 8 \\
7 & 14 & 4 & 4 & 14 \\
11 & 2 & 8 & 8 & 2 \\
13 & 4 & 14 & 14 & 4 \\
17 & 4 & 14 & 14 & 4 \\
19 & 2 & 8 & 8 & 2 \\
23 & 14 & 4 & 4 & 14 \\
29 & 8 & 2 & 2 & 8
\end{array}.$$
\caption{The table shows the values of (\ref{deltae82}) for all primes $p\geq 7$ depending on their decomposition as $p=30l+r$. The sign in the headline is 
the same as in the last argument of (\ref{deltae82}).}
\label{tabe82}
\end{table}
\end{center}
Computing the Hilbert-series, we find
\begin{align*}
\lK_{\Syz_R(U^p,V^p,W^p)}(t) &= \left\{\begin{aligned} & t^{31\cdot\frac{p-1}{2}}\cdot \lK_{\Syz_R(U,V,W)}(t) && \text{if } p\text{ mod }30\in \{\pm 1,\pm 11\},\\ & t^{31\cdot\frac{p-1}{2}-3}\cdot \lK_{\Syz_R(U,V,W^2)}(t) && \text{if } p\text{ mod }30\in \{\pm 7,\pm 13\},\end{aligned}\right.\\
\lK_{\Syz_R(U^p,V^p,W^{2p})}(t) &= \left\{\begin{aligned} & t^{37\cdot\frac{p-1}{2}}\cdot \lK_{\Syz_R(U,V,W^2)}(t) && \text{if } p\text{ mod }30\in \{\pm 1,\pm 11\},\\ & t^{37\cdot\frac{p-1}{2}+3}\cdot \lK_{\Syz_R(U,V,W)}(t) && \text{if } p\text{ mod }30\in \{\pm 7,\pm 13\}.\end{aligned}\right.
\end{align*}
For the higher pull-backs we get
$$\Syz_R(U^q,V^q,W^q)\cong \left\{\begin{aligned} & \Syz_R(U,V,W)\left(-\frac{31(q-1)}{2}\right) && \text{if } q\text{ mod }30\in \{\pm 1,\pm 11\},\\ & \Syz_R(U,V,W^2)\left(-\frac{31(q-1)}{2}+3\right) && \text{if } q\text{ mod }30\in \{\pm 7,\pm 13\}.\end{aligned}\right.$$
Using Formula (\ref{formnontr}), the Hilbert-Kunz function is given by
\begin{align*}
\HKF(R,q) &= \frac{30\cdot Q(15,10,6)\cdot q^2-30\cdot Q(15,10,6)+30\cdot 120}{30\cdot 120}\\
&= \frac{239\cdot q^2-239+120}{120}\\
&= \left(2-\frac{1}{120}\right)q^2-\frac{119}{120},
\end{align*}
if $q\text{ mod }30\in\{\pm 1,\pm 11\}$ and
\begin{align*}
\HKF(R,q) &= \frac{30\cdot Q(15,10,6)\cdot q^2-30\cdot Q(15,10,12)+30\cdot 240}{30\cdot 120}\\
&= \frac{239\cdot q^2-431+240}{120}\\
&= \left(2-\frac{1}{24}\right)q^2-\frac{191}{120},
\end{align*}
if $q\text{ mod }30\in\{\pm 7,\pm 13\}$.

In characteristics two, three and five, all Frobenius pull-backs split and one obtains (similarly to the computations for $E_6$) $\HKF_R(R,q)=2\cdot q^2$ (with $q>1$).

\begin{thm}\label{hkfe8}
 The Hilbert-Kunz function of $E_8$ is given by
$$e\longmapsto\left\{\begin{aligned}
 & 2\cdot q^2 && \text{if } e\geq 1\text{ and }p\in\{2,3,5\},\\
 & \left(2-\frac{1}{120}\right)q^2-\frac{119}{120} && \text{if }q\text{ mod }30\in\{\pm1,\pm11\},\\
 & \left(2-\frac{1}{120}\right)q^2-\frac{191}{120} && \text{otherwise.}
\end{aligned}\right.$$
\end{thm}

Reformulating the cases in the above theorem in terms of $p$, one obtains the function
$$e\longmapsto\left\{\begin{aligned}
 & 2\cdot q^2 && \text{if } e\geq 1\text{ and }p\in\{2,3,5\},\\
 & \left(2-\frac{1}{120}\right)q^2-\frac{191}{120} && \text{if }p\text{ mod }30\in\{\pm 7,\pm 13\}\text{ and }e\text{ is odd},\\
 & \left(2-\frac{1}{120}\right)q^2-\frac{119}{120} && \text{otherwise.}
\end{aligned}\right.$$
With the help of Theorem \ref{fsighkf} one gets the F-signature function.
          
\begin{cor}\label{fsigne8}
The F-signature function of $R$ is given by
$$\FS(R,q)=\left\{\begin{aligned}
 & 0 && \text{if } e\geq 1\text{ and }p\in\{2,3,5\},\\
 & \frac{1}{120}\cdot q^2+\frac{119}{120} && \text{if }q\text{ mod }30\in\{\pm1,\pm11\},\\
 & \frac{1}{120}\cdot q^2+\frac{191}{120} && \text{otherwise.}
\end{aligned}\right.$$
\end{cor}

In any case the Hilbert-Kunz functions of surface rings of type ADE have (at least for $p\geq 7$) a periodic $O(1)$-term. One possible value of this term is 
always $1-\tfrac{1}{|G|}$. It is not known whether the other possible values of the $O(1)$-term can be expressed in terms of invariants of the corresponding 
groups. Another question is whether the length of the Frobenius periodicity of $\Syz(\mm)$ is encoded in properties of the group $G$.

\chapter{Extensions and further examples}\label{chapextfurexa}
In this chapter we want to discuss further examples of the methods developed so far. More explicitly, we will compute the Hilbert-Kunz function of a surface 
ring of type $E_8$ with respect to different ideals $I\neq\mm$ and the Hilbert-Kunz functions of the degenerations $A_{\infty}:=k[X,Y,Z]/(XY)$ and 
$D_{\infty}:=k[X,Y,Z]/(X^2Y-Z^2)$ of the singularities of type $A_n$ and $D_n$. In the fourth section we will combine our approach with a theorem of Kustin, 
Rahmati and Vraciu to compute the Hilbert-Kunz functions of the homogeneous coordinate ring of Fermat curves under the condition that $\Syz(\mm)$ is strongly semistable on the projective 
spectrum. Finally, in section five we will state some open questions and discuss possible generalizations of various theorems.

\section{Hilbert-Kunz functions of $E_8$ with respect to various ideals $I\neq\mm$}
In this section we will demonstrate how the developed methods are helpfull to compute the Hilbert-Kunz function of $E_8$ with respect to the ideals 
$(X,Y,Z^2)$ resp. $(X,Y^2,YZ,Z^2)$.

\begin{exa}\label{exa1chap7}
Let $k$ be an algebraically closed field of characteristic $p\geq 7$, 
$$R:=k[X,Y,Z]/(X^2+Y^3+Z^5),$$ 
$U:=\punctured{R}{(X,Y,Z)}$ and $I:=(X,Y,Z^2)$. By 
Lemma \ref{fpbindec} the Frobenius pull-backs of $\Syz_U(I)$ are indecomposable, hence of the form $\Syz_U(I)$ or $\Syz_U(X,Y,Z)$ by Theorem \ref{syze8}. 
By the explicit computations of the Hilbert-series of $\Syz_R(X^p,Y^p,Z^p)$ and $\Syz_R(X^p,Y^p,Z^{2p})$ in Section \ref{sectionhkfe8} we obtain
$$\Syz_R(X^q,Y^q,Z^{2q})\cong \left\{\begin{aligned} & \Syz_R(X,Y,Z^2)\left(-\frac{37(q-1)}{2}\right) && \text{if } q\text{ mod }30\in \{\pm 1,\pm 11\},\\ & \Syz_R(X,Y,Z)\left(-\frac{37(q-1)}{2}-3\right) && \text{if } q\text{ mod }30\in \{\pm 7,\pm 13\}.\end{aligned}\right.$$
Using Lemma \ref{hkfdreivar} and Remark \ref{hkfflat} one obtains that $\HKF(I,R,q)$ is given by
$$\frac{30\cdot Q(15,10,12)\cdot q^2-30\cdot Q(15,10,6c)}{4\cdot 15\cdot 10\cdot 6}+c,$$
where $c\in\{1,2\}$ such that $\Syz_R(X^q,Y^q,Z^{2q})\cong\Syz_R(X,Y,Z^c)(m)$ for some $m\in\Z$. All in all, the function
$$e\longmapsto\begin{cases}
\left(3+\frac{71}{120}\right)q^2-\frac{119}{120} & \text{if }p\text{ mod }30\in\{\pm 7,\pm 13\}\text{ and }e\text{ is odd},\\
\left(3+\frac{71}{120}\right)q^2-\frac{191}{120} & \text{if }p\text{ mod }30\in\{\pm 7,\pm 13\}\text{ and }e\text{ is even},\\
\left(3+\frac{71}{120}\right)q^2-\frac{191}{120} & \text{if }p\text{ mod }30\in\{\pm 1,\pm 11\}
\end{cases}$$
is the Hilbert-Kunz function of $R$ with respect to $I$.
\end{exa}

\begin{exa}\label{exa2chap7}
Let $k$, $R$ and $U$ be as in the previous example and $I:=(X,Y^2,YZ,Z^2)$. Then $\Syz_R(I)$ is the module $M_2$ from Theorem \ref{syze8}. Since there are 
no non-trivial line bundles on $U$ the Frobenius pull-backs of $\Syz_U(I)$ either stay indecomposable or become free. In the case where $\Syz_U(I\qpot)$ is 
indecomposable for some $q$ it has to be isomorphic to either $\Syz_U(I)$ or $\Syz_U(X,Y^2,YZ,Z^3)$, since these are representatives for the two indecomposable 
vector bundles of rank three (cf. Theorem \ref{syze8}). Which of the possibilities is true can be checked by the Hilbert-series. In view of Theorem 
\ref{HilbSer} we need to compute the degree of the generators of 
$$S_{q-1}:=\Syz_R(X^{q-1},Y^{2q},Y^qZ^q,Z^{2q}) \text{ and } S_{q+1}:=\Syz_R(X^{q+1},Y^{2q},Y^qZ^q,Z^{2q})$$ 
to compute the Hilbert-series of $\Syz_R(X^q,Y^{2q},Y^qZ^q,Z^{2q})$. Since there is no effective way to 
compute these we restrict to $p=7$. An explicit computation using CoCoA shows that $S_{7-1}$ is generated by elements of degree 140, 142 and 144, while 
$S_{7+1}$ is generated by elements of degree 150, 152 and 154. This shows 
\begin{equation}\label{hsm2}
\lK_{\Syz_R(X^7,Y^{14},Y^7Z^7,Z^{14})}(t)=t^{124}\lK_{\Syz_R(X,Y^2,YZ,Z^3)}(t).
\end{equation}
The module $\Syz_R(X^6,Y^{14},Y^7Z^7,Z^{21})$ is generated by elements of degree 152, 156 and 160 and the module $\Syz_R(X^8,Y^{14},Y^7Z^7,Z^{21})$ is 
generated by elements of degree 162, 166 and 170. This gives 
\begin{equation}\label{hsm6}
\lK_{\Syz_R(X^7,Y^{14},Y^7Z^7,Z^{21})}(t)=t^{135}\lK_{\Syz_R(X,Y^2,YZ,Z^2)}(t).
\end{equation}
Combining (\ref{hsm2}) and (\ref{hsm6}) one finds
$$\Syz_R\left(X^{7^e},Y^{2\cdot 7^e},Y^{7^e}Z^{7^e},Z^{2\cdot 7^e}\right)\cong
\left\{\begin{aligned}
& \Syz_R(X,Y^2,YZ,Z^3)(-21q+23) && \text{if }e\text{ is odd,}\\
& \Syz_R(Y,Y^2,YZ,Z^2)(-21q+21) && \text{if }e\text{ is even.}\\
\end{aligned}\right.$$
With the help of the standard-graded ring $S:=k[U,V,W]/(U^{30}+V^{30}+W^{30})$ and the map $R\ra S$, $X\mapsto U^{15}$, $Y\mapsto V^{10}$, $Z\mapsto W^6$, 
we can compute 
$$\HKF\left(\left(U^{15},V^{20},V^{10}W^6,W^{12}\right),S,7^e\right)$$
as in Lemma \ref{hkfdreivar}. Let $q:=7^e$, $a\in\{2,3\}$, $n:=21q-17-2a$ and 
$$D:=\Dim_k(S/(U^{15},V^{20},V^{10}W^{6},W^{6a})=
\begin{cases}
2700 & \text{if }a=2,\\
3600 & \text{if }a=3,
\end{cases}$$
where $a$ and $n$ are chosen such that 
$$\Syz_R(X^q,Y^{2q},Y^qZ^q,Z^{2q})\cong\Syz_R(X,Y^2,YZ,Z^a)(-n).$$ 
With these preparations we obtain
\begin{align*}
 &\quad \HKF\left(\left(U^{15},V^{20},V^{10}W^6,W^{12}\right),S,q\right)\\
 &= \lim_{x\ra\infty}\left[\sum_{m=0}^x\dimglo{\Oc_C(m)}-\sum_{m=0}^x\dimglo{\Oc_C(m-15 q)}-\sum_{m=0}^x\dimglo{\Oc_C(m-20 q)}\right.\\
 &\quad -\sum_{m=0}^x\dimglo{\Oc_C(m-16q)}-\sum_{m=0}^x\dimglo{\Oc_C(m-12q)}+\sum_{m=0}^x\dimglo{\Oc_C(m-n-15)}\\
 &\quad +\sum_{m=0}^x\dimglo{\Oc_C(m-n-20)}+\sum_{m=0}^x\dimglo{\Oc_C(m-n-16)}+\sum_{m=0}^x\dimglo{\Oc_C(m-n-6a)}\\
 &\quad \left. -\sum_{m=0}^x\dimglo{\Oc_C(m-n)}\right]+D\\
 &= \lim_{x\ra\infty}\left[\sum_{m=0}^x\dimglo{\Oc_C(m)}-\sum_{m=0}^{x-15 q}\dimglo{\Oc_C(m)}-\sum_{m=0}^{x-20 q}\dimglo{\Oc_C(m}-\sum_{m=0}^{x-16 q}\dimglo{\Oc_C(m)}\right.\\
 &\quad -\sum_{m=0}^{x-12 q}\dimglo{\Oc_C(m)}+\sum_{m=0}^{x-n-15}\dimglo{\Oc_C(m)}+\sum_{m=0}^{x-n-20}\dimglo{\Oc_C(m)}+\sum_{m=0}^{x-n-16}\dimglo{\Oc_C(m)}\\
 &\quad \left. +\sum_{m=0}^{x-n-6a}\dimglo{\Oc_C(m)}-\sum_{m=0}^{x-n}\dimglo{\Oc_C(m)}\right]+D\\
 &= 30\cdot (149q^2 + 12a^2 - 102a + 7)+D.
\end{align*}
Since the degree of the extension 
$$k[U,V,W]/\left(U^{15},V^{10},W^{6}\right):k$$
is $15\cdot 10\cdot 6=30^2$ we obtain
\begin{align*}
\HKF(IS,S,7^e) &= \begin{cases}
30\cdot 149\cdot 49^e-2130 & \text{if }e\text{ is odd,}\\
30\cdot 149\cdot 49^e-1770 & \text{if }e\text{ is even}
\end{cases}\quad\text{and}\\
\HKF(I,R,7^e) &= \left\{\begin{aligned}
& \frac{149}{30}\cdot 49^e-\frac{71}{30} && \text{if }e\text{ is odd,}\\
& \frac{149}{30}\cdot 49^e-\frac{59}{30} && \text{if }e\text{ is even.}
\end{aligned}\right.
\end{align*}
By Theorem \ref{trivedi} we see that the vector bundle $\Syz_{\Proj(S)}(IS)$ is strongly semistable.
\end{exa}

\section{The Hilbert-Kunz function of $A_{\infty}$}
In this and the following section we will compute the Hilbert-Kunz functions of the rings $A_{\infty}$ and $D_{\infty}$. These examples will show, that 
we (at least in these cases) can drop the assumptions

\begin{enumerate}
 \item isolated singularity,
 \item normal domain,
 \item finite Cohen-Macaulay type,
 \item there are finitely many isomorphism classes of indecomposable vector bundles on the punctured spectrum.
\end{enumerate}

Note that we discussed the idea to contsruct first syzygy modules of ideals representing the indecomposable, maximal Cohen-Macaulay modules only under the 
assumption of a normal ring. The condition normal was only required to ensure that there has to be such a representation. If the ring is not normal, first 
syzygy modules of ideals are still reflexive, but they might not be maximal Cohen-Macaulay.

We choose the standard-grading for the ring $A_{\infty}:=k[X,Y,Z]/(XY)$.

\begin{thm}[Burban, Drozd] The indecomposable, non-free, maximal Cohen-Macau\-lay $A_{\infty}$-modules are described up to isomorphism (and degree shift) by the 
following matrix factorizations of $XY$

\begin{enumerate}
 \item $(X,Y)$ and $(Y,X)$,
 \item $$\left(\begin{pmatrix}Y & Z^n\\ 0 & X\end{pmatrix},\begin{pmatrix}X & -Z^n\\ 0 & Y\end{pmatrix}\right)\text{ and }\left(\begin{pmatrix}X & -Z^n\\ 0 & Y\end{pmatrix},\begin{pmatrix}Y & Z^n\\ 0 & X\end{pmatrix}\right),$$
	for all $n\geq 1$.
\end{enumerate}

Moreover, the modules corresponding to the matrix factorizations in (i) are the two components of the normalization $k[X,Z]\times k[Y,Z]$ of 
$A_{\infty}$ and are not locally free on the punctured spectrum of $A_{\infty}$. The modules corresponding to the matrix factorizations in (ii) are 
locally free of rank one on the punctured spectrum.
\end{thm}

\begin{proof}
See \cite[Theorem 5.3 and Remark 5.5]{mcmsur}.
\end{proof}

The syzygy modules $\Syz_{A_{\infty}}(X)$ and $\Syz_{A_{\infty}}(Y)$ correspond to the matrix factorizations in (i). Obviously, they are selfdual. 
The syzygy modules $\Syz_{A_{\infty}}(Y,Z^n)$ and $\Syz_{A_{\infty}}(X,-Z^n)$ correspond to the matrix factorizations in (ii). They are dual to each other, 
since 
$$(\alpha,\alpha):\left(\begin{pmatrix}Y & 0\\ Z^n & X\end{pmatrix},\begin{pmatrix}X & 0\\ -Z^n & Y\end{pmatrix}\right)\ra\left(\begin{pmatrix}X & -Z^n\\ 0 & Y\end{pmatrix},\begin{pmatrix}Y & Z^n\\ 0 & X\end{pmatrix}\right),$$

with $\alpha:=\left(\begin{smallmatrix} 0 & -1 \\ 1 & 0 \end{smallmatrix}\right)$ is an equivalence of matrix factorizations. 
Using Formula (\ref{hilbserdirect}) on page \pageref{hilbserdirect}, one obtains
$$\begin{aligned}
\lK_{A_{\infty}}(t) &= \frac{1+t}{(1-t)^2},\\
\lK_{\Syz_{A_{\infty}}(X)}(t) = \lK_{\Syz_{A_{\infty}}(Y)}(t) &= \frac{t^2}{(1-t)^2},\\
\lK_{\Syz_{A_{\infty}}(Y,Z^n)}(t) = \lK_{\Syz_{A_{\infty}}(X,-Z^n)}(t) &= \frac{t^2+t^{n+1}}{(1-t)^2}.
\end{aligned}$$
In this case the Hilbert-series of a maximal Cohen-Macaulay module that is locally free on the punctured spectrum can detect its isomorphism class up 
to dualizing with one exception. Namely, the modules $\Syz_{A_{\infty}}(X,Z^2)$, $\Syz_{A_{\infty}}(Y,Z^2)$ and $A_{\infty}(-2)$ have the same 
Hilbert-series.

From the splitting $$\Syz_{A_{\infty}}(X,Y,Z)=\Syz_{A_{\infty}}(X,Z)\oplus\Syz_{A_{\infty}}(Y,Z),$$
we deduce the splittings $$\Syz_{A_{\infty}}(X^q,Y^q,Z^q)=\Syz_{A_{\infty}}(X^q,Z^q)\oplus\Syz_{A_{\infty}}(Y^q,Z^q),$$
where by the symmetry of the problem $\Syz_{A_{\infty}}(X^q,Y^q,Z^q)$ is either free or of the form 
$\Syz_{A_{\infty}}(X,Z^n)(m)\oplus\Syz_{A_{\infty}}(Y,Z^n)(m)$ for some $n\in\N$ and some $m\in\Z$. All in all, we see that it is enough to detect the 
isomorphism class of $\Syz_{A_{\infty}}(X^q,Z^q)$ up to dualizing.

By Formula (\ref{hilbserdirect}) on page \pageref{hilbserdirect}, we have
\begin{equation}\label{HS1}
\lK_{\Syz_{A_{\infty}}(X^q,Z^q)}(t)=\frac{(2t^q-1)(1-t^2)}{(1-t)^3}+\lK_{k[X,Y,Z]/(X^q,Z^q,XY)}(t).
\end{equation}
With $M:=k[X,Y,Z]/(X^q,Z^q)$ one has by \cite[Proposition 5.2.16]{cocoa2} 
\begin{align}\label{HS2}
\lK_{k[X,Y,Z]/(X^q,Z^q,XY)}(t)=\lK_{M/(XY)M}(t) &= \lK_M(t)-t^2\lK_{M/(0:_M(XY))}(t)\nonumber \\
&= \lK_M(t)-t^2\lK_{M/(X^{q-1})M}(t)\nonumber \\
&= \frac{(1-t^q)^2}{(1-t)^3}-t^2\cdot\frac{(1-t^{q-1})(1-t^q)}{(1-t)^3}.
\end{align}
Substituting Equation (\ref{HS2}) in Equation (\ref{HS1}), one obtains
$$\lK_{\Syz_{A_{\infty}}(X^q,Z^q)}(t)=\frac{t^{q+1}+t^{2q}}{(1-t)^2}=t^{q-1}\cdot \frac{t^2+t^{q+1}}{(1-t)^2}.$$
If $q\neq 2$, this gives directly 
$$\Syz_{A_{\infty}}(X^q,Z^q)\cong \Syz_{A_{\infty}}(X,Z^q)(-q+1) \text{ or } \Syz_{A_{\infty}}(X^q,Z^q)\cong \Syz_{A_{\infty}}(Y,Z^q)(-q+1).$$
In the case $q=2$, we might have $\Syz_{A_{\infty}}(X^2,Z^2)\cong A_{\infty}(-3)$. But $\Syz_{A_{\infty}}(X^2,Z^2)$ is obviously not free.

All in all, we see
$$\Syz_{A_{\infty}}(X^q,Y^q,Z^q)=\Syz_{A_{\infty}}(X,Y,Z^q)$$
and get by Lemma \ref{hkfdreivar}
$$\HKF(A_{\infty},q)=\frac{2\cdot Q(1,1,1)\cdot q^2-2\cdot Q(1,1,q)+4q}{4}=2q^2-q.$$
Comparing this result with Example \ref{genhkfan}, we see
$$\HKF(A_{\infty},q)=\lim_{n\ra\infty}\HKF(A_n,q).$$

\section{The Hilbert-Kunz function of $D_{\infty}$}
For the ring $D_{\infty}:=k[X,Y,Z]/(X^2Y-Z^2)$ we choose the grading $\Deg(X)=\Deg(Y)=2$ and $\Deg(Z)=3$.

\begin{thm}[Burban, Drozd] The indecomposable, non-free, maximal Cohen-Macau\-lay $D_{\infty}$-modules are described up to isomorphism (and degree shift) by the 
following matrix factorizations of $X^2Y-Z^2$

\begin{enumerate}
 \item $$\left(\begin{pmatrix}Z & XY\\ X & Z\end{pmatrix},\begin{pmatrix}-Z & XY\\ X & -Z\end{pmatrix}\right),$$ 
 \item $$\left(\begin{pmatrix}X^2 & Z\\ Z & Y\end{pmatrix},\begin{pmatrix}Y & -Z\\ -Z & X^2\end{pmatrix}\right),$$
 \item $$\left(\begin{pmatrix}Z & XY & 0 & -Y^{n+1}\\ X & Z & Y^n & 0\\ 0 & 0 & Z & XY\\ 0 & 0 & X & Z\end{pmatrix},
	       \begin{pmatrix}-Z & -XY & 0 & Y^{n+1}\\ X & Z & Y^n & 0\\ 0 & 0 & -Z & -XY\\ 0 & 0 & X & Z\end{pmatrix}\right),$$
	for all $n\geq 1$ and
 \item $$\left(\begin{pmatrix}Z & XY & -Y^n & 0\\ X & Z & 0 & Y^n\\ 0 & 0 & Z & XY\\ 0 & 0 & X & Z\end{pmatrix},
	       \begin{pmatrix}-Z & -XY & -Y^n & 0\\ X & Z & 0 & -Y^n\\ 0 & 0 & -Z & -XY\\ 0 & 0 & X & Z\end{pmatrix}\right),$$
	for all $n\geq 1$.
\end{enumerate}

Moreover, all corresponding modules are selfdual, the module corresponding to the matrix factorization (i) has rank one, it is the normalization of $D_{\infty}$ 
and is not locally free on the punctured spectrum. The modules corresponding to the matrix factorizations (ii), (iii) and (iv) are locally free on the 
punctured spectrum and have rank one in case (ii) and rank two in the cases (iii) and (iv).
\end{thm}

\begin{proof}
See \cite[Theorem 5.7 and Remark 5.9]{mcmsur}.
\end{proof}

Corresponding first syzygy modules of ideals are given by

\begin{enumerate}
 \item $\Syz_{D_{\infty}}(X,Z)$ in case (i),
 \item $\Syz_{D_{\infty}}(Z,Y)$ in case (ii),
 \item $\Syz_{D_{\infty}}(Y^{n+1},-XY,Z)$ in case (iii) and
 \item $\Syz_{D_{\infty}}(Y^n,-Z,X)$ in case (iv),
\end{enumerate}

where one has to remove the second column in the second matrices to get the representations (iii) and (iv).

Although we do not need the Hilbert-series of the above modules to compute the Hilbert-Kunz function of $D_{\infty}$, we compute them for completeness, 
using Theorem \ref{HilbSer}. They are given by
$$\begin{aligned}
\lK_{D_{\infty}}(t) &= \frac{1+t^3}{(1-t^2)^2},\\
\lK_{\Syz_{D_{\infty}}(X,Z)}(t) &= \frac{t^5+t^6}{(1-t^2)^2},\\
\lK_{\Syz_{D_{\infty}}(Z,Y)}(t) &= \frac{t^5+t^6}{(1-t^2)^2},\\
\lK_{\Syz_{D_{\infty}}(Y^{n+1},-XY,Z)}(t) &= \frac{t^6+t^7+t^{2n+4}+t^{2n+5}}{(1-t^2)^2}\text{ and}\\
\lK_{\Syz_{D_{\infty}}(Y^n,-Z,X)}(t) &= \frac{t^5+t^6+t^{2n+2}+t^{2n+3}}{(1-t^2)^2}.
\end{aligned}$$
One should mention the phenomenon that we have two indecomposable, maximal Cohen-Macaulay $D_{\infty}$-modules with the same Hilbert-series, but one is locally free on 
the punctured spectrum, while the other one is not.

To compute the Hilbert-Kunz function of $D_{\infty}$, we start with the case $p=2$ and $q\geq 2$. Using Lemma \ref{manipulationlemma}, we obtain 
\begin{align*}
\Syz_{D_{\infty}}(X^q,Y^q,Z^q) &\cong \Syz_{D_{\infty}}\left(X^q,Y^q,X^qY^{\frac{q}{2}}\right)\\
&\cong \Syz_{D_{\infty}}\left(1,Y^q,Y^{\frac{q}{2}}\right)(-2q)\\
&\cong \Syz_{D_{\infty}}\left(1,Y^{\frac{q}{2}},1\right)(-3q)\\
&\cong D_{\infty}(-3q)\oplus D_{\infty}(-4q).
\end{align*}
Using Formula (\ref{formtr}) on page \pageref{formtr}, one gets
$$\HKF(D_{\infty},2^e)=\frac{6\cdot Q(2,2,3)\cdot 4^e+6\cdot 4^e}{48}=2\cdot 4^e.$$
If $p$ is odd and $q>1$, we get by Lemma \ref{manipulationlemma}
\begin{align*}
\Syz_{D_{\infty}}(X^q,Y^q,Z^q) &\cong \Syz_{D_{\infty}}\left(X^q,Y^q,ZX^{q-1}Y^{\frac{q-1}{2}}\right)\\
&\cong \Syz_{D_{\infty}}\left(X,Y^q,ZY^{\frac{q-1}{2}}\right)(-2(q-1))\\
&\cong \Syz_{D_{\infty}}\left(X,Y^{\frac{q+1}{2}},Z\right)(-3(q-1)).
\end{align*}
Using Formula (\ref{formnontr}) on page \pageref{formnontr}, one obtains
$$\HKF(D_{\infty},q)=\frac{6\cdot Q(2,2,3)\cdot q^2-6\cdot Q(2,q+1,3)+24\cdot (q+1)}{48}=2\cdot q^2-\frac{1}{2}\cdot (q+1).$$
To stay in our approach, we compute the Hilbert-series of $\Syz_{D_{\infty}}(X^q,Y^q,Z^q)$ to identify its isomorphism class. By Theorem \ref{HilbSer}, we have
$$\lK_{\Syz_{D_{\infty}}(X^q,Y^q,Z^q)}(t)=\frac{(t^3-t^6)\lK_{\Syz_{D_{\infty}}(X^q,Y^q,Z^{q-1})}(t)+(1-t^3)\lK_{\Syz_{D_{\infty}}(X^q,Y^q,Z^{q+1})}(t)}{1-t^6}.$$
It is not hard to see that 
$$\begin{aligned}
\Syz_{D_{\infty}}(X^q,Y^q,Z^{q-1}) &= \left\langle\begin{pmatrix}Y^{\frac{q-1}{2}}\\0\\-X\end{pmatrix},\begin{pmatrix}0\\ X^{q-1}\\ -Y^{\frac{q+1}{2}}\end{pmatrix}\right\rangle &&\cong D_{\infty}(-3q+1)\oplus D_{\infty}(-4q+2),\\
\Syz_{D_{\infty}}(X^q,Y^q,Z^{q+1}) &= \left\langle\begin{pmatrix}Y^q\\-X^q\\0\end{pmatrix},\begin{pmatrix}XY^{\frac{q+1}{2}}\\ 0\\ -1\end{pmatrix}\right\rangle &&\cong D_{\infty}(-4q)\oplus D_{\infty}(-3q-3).
\end{aligned}$$
Therefore,
\begin{align*}
\lK_{\Syz_{D_{\infty}}(X^q,Y^q,Z^q)}(t) &= \frac{t^{3q+2}+t^{3q+3}+t^{4q}+t^{4q+1}}{(1-t^2)^2}\\
&= t^{3q-3}\cdot\frac{t^5+t^6+t^{q+3}+t^{q+4}}{(1-t^2)^2}.
\end{align*}
Since the Hilbert-series of an indecomposable, maximal Cohen-Macaulay $D_{\infty}$-module, which is locally free on the punctured spectrum, detects 
its isomorphism class, we obtain 
$$\Syz_{D_{\infty}}(X^q,Y^q,Z^q)\cong\Syz_{D_{\infty}}\left(X,Y^{\frac{q+1}{2}},Z\right)(-3q+3),$$
provided we can show that the module in question is indecomposable, since from the Hilbert-series and the fact that $\Syz_{D_{\infty}}(X^q,Y^q,Z^q)$ is 
locally free on the punctured spectrum, the splitting
$$\Syz_{D_{\infty}}(X^q,Y^q,Z^q)\cong\Syz_{D_{\infty}}(Z,Y)(-3q+3)\oplus \Syz_{D_{\infty}}(Z,Y)(-4q+5)$$
is still possible.

A comparison with Theorem \ref{hkfdn} yields
$$\HKF(D_{\infty},q)=\lim_{n\ra\infty}\HKF(D_n,q).$$

\section{On the Hilbert-Kunz functions of two-dimensional Fermat rings}
In this section we want to compute the Hilbert-Kunz functions of the Fermat rings
$$R:=k[X,Y,Z]/(X^n+Y^n+Z^n),$$ 
where $k$ is algebraically closed of positive characteristic $p$ with $\gcd(p,n)=1$, 
under the condition that $\Syz_C(X,Y,Z)$ is strongly semistable on $C:=\Proj(R)$.  
For this computation we need a result of \cite{vraciu}. For completeness, we explain how Kaid computed in his thesis 
(cf. \cite{almar}) the values of the Hilbert-Kunz function of $R$ for large $q$ if $\Syz_C(X,Y,Z)$ is not strongly semistable.

We start with the case, where $\Syz_C(X,Y,Z)$ is not strongly semistable. Recall that this is equivalent to 
$$\delta\left(\frac{1}{n},\frac{1}{n},\frac{1}{n}\right)\neq 0$$ 
by Corollaries \ref{deltastrongsemistable} and \ref{hkmdiag}.

\begin{thm}[Kaid]\label{almarfilt} Assume $\gcd(p,n)=1$. 
If $\delta\left(\tfrac{1}{n},\tfrac{1}{n},\tfrac{1}{n}\right)\neq 0$ and $s$ is the integer of Han's Theorem \ref{hansthm}, then 
$\fpb{s}(\Syz_C(X,Y,Z))$ is not semistable. Let $p^s=nl+r$ with $l\in\N$ and $0\leq r\leq n-1$. Then the Frobenius pull-back 
$\fpb{s}(\Syz_C(X,Y,Z))$ has a strong Harder-Narasimhan filtration given by
$$0\ra \Oc_C\left(-n\left(l+1+\frac{l}{2}\right)\right)\ra\fpb{s}(\Syz_C(X,Y,Z))\ra \Oc_C\left(n\left(l+1+\frac{l}{2}\right)-3p^s\right)\ra 0,$$
if $l$ is even and 
$$0\ra \Oc_C\left(-n\left(l+\left\lfloor\frac{l}{2}\right\rfloor\right)-3r\right)\ra\fpb{s}(\Syz_C(X,Y,Z))\ra \Oc_C\left(n\left(l+\left\lfloor\frac{l}{2}\right\rfloor\right)+3r-3p^s\right)\ra 0,$$
if $l$ is odd. Moreover, under the additional assumption $\gcd(p,n)=1$ this filtration is minimal for $p\geq d-3$ and $s\geq 1$ in the sense that $\fpb{s-1}(\Syz_C(X,Y,Z))$ is semistable.
\end{thm}

\begin{proof}
At first one needs to prove that the inclusions in the above sequences are non-trivial. This is done in \cite[Lemma 4.2.8]{almar}. In the second step 
one has to show that the kernels above are in fact the maximal destabilizing subbundles. This is done in \cite[Theorem 4.2.12]{almar} by computing the 
degree of the maximal destabilizing subbundles. The supplement is proven in \cite[Theorem 4.2.12]{almar} by using Corollary \ref{coralmar}.
\end{proof}

\begin{rem}\label{dundr}
In the situation of Theorem \ref{almarfilt}, one has
$$\left\lVert\left(\frac{p^s}{n},\frac{p^s}{n},\frac{p^s}{n}\right),L_\text{odd}\right\rVert_1<1$$
by assumption. If $l$ is even, this distance is achieved for 
$$\left(\left\lceil\frac{p^s}{n}\right\rceil,\left\lceil\frac{p^s}{n}\right\rceil,\left\lceil\frac{p^s}{n}\right\rceil\right)\in L_\text{odd},$$
showing $3r>2n$. Similarly, if $l$ is odd one obtains $3r<n$.
\end{rem}

From Theorem \ref{almarfilt} one obtains the Hilbert-Kunz function of $R$ for large $q$.

\begin{thm}[Kaid]\label{almarhkf}
In the situation of Theorem \ref{almarfilt} one has 
$$\HKF(R,p^e)=\HKM(R)p^{2e}\text{ for all } e\gg 0.$$
\end{thm}

\begin{proof}
By Theorem \ref{almarfilt} we have a short exact sequence
$$0\ra\Oc_C(\alpha)\ra\fpb{s}(\Syz_C(X,Y,Z))\ra\Oc_C(\beta)\ra 0.$$
Since this is the strong Harder-Narasimhan filtration of $\fpb{s}(\Syz_C(X,Y,Z))$, we obtain by Remark \ref{pbofstronghnfilt} the strong 
Harder-Narasimhan filtrations
$$0\ra\Oc_C(\alpha\cdot p^t)\ra\fpb{s+t}(\Syz_C(X,Y,Z))\ra\Oc_C(\beta\cdot p^t)\ra 0.$$
These sequences split, whenever 
$$\Ext^1\left(\Oc_C(\beta\cdot p^t),\Oc_C(\alpha\cdot p^t)\right)\cong \lK^1\left(C,\Oc_C((\alpha-\beta)\cdot p^t)\right)$$
vanishes. Using Serre duality and $\omega_C\cong\Oc_C(n-3)$ (cf. \cite[Example II.8.20.3]{hartshorne}), we have to show that
$$\begin{aligned}
& \Oc_C((2n-3r)\cdot p^t+n-3) && \text{if }l\text{ is even or}\\
& \Oc_C((3r-n)\cdot p^t+n-3) && \text{if }l\text{ is odd}
\end{aligned}$$
have no global sections. This is true if $(2n-3r)\cdot p^t+n-3$ resp. $(3r-n)\cdot p^t+n-3$ is negative. In view of Remark \ref{dundr} this happens for 
$t$ big enough. In particular, we must have $t\geq 1$. Note that $t=1$ is enough in the cases $p>n-3$.

Then one can compute the Hilbert-Kunz function for large $e$ with the help of Lemma \ref{hkffrei}. See \cite[Corollary 4.2.17]{almar} for the complete proof.
\end{proof}

\begin{exa}
Let $p=5$ and $n=6$. Then $s=1$ in Han's Theorem \ref{hansthm} and by Theorem \ref{almarfilt} the strong Harder-Narasimhan filtration of 
$\fpb{}(\Syz_C(X,Y,Z))$ is given by
$$0\ra\Oc_C(-6)\ra\Syz_C(X^5,Y^5,Z^5)\ra\Oc_C(-9)\ra 0.$$
Taking the $t$-th Frobenius pull-back of this sequence, the resulting sequence splits if 
$$-3\cdot 5^t+3<0\Leftrightarrow t\geq 1.$$
Explicitly, we have 
$$\fpb{e}(\Syz_C(X,Y,Z))\cong\Oc_C(-6\cdot 5^{e-1})\oplus\Oc_C(-9\cdot 5^{e-1})$$
for $e\geq 2$. By Theorem \ref{almarhkf} and Corollary \ref{hkmdiag} the Hilbert-Kunz function of $R$ for $e\geq 2$ is given by
$$\HKF(R,5^e)=\HKM(R)\cdot 5^e=\frac{126}{25}\cdot 5^e.$$
\end{exa}

\begin{exa}\label{p3n7}
Let $p=3$ and $n=7$. Then $\delta\left(\tfrac{1}{7},\tfrac{1}{7},\tfrac{1}{7}\right)=\tfrac{1}{63}$ with $s=2$. By Theorem \ref{almarfilt} we have the 
strong Harder-Narasimhan filtration
$$0\ra\Oc_C(-13)\ra\fpb{2}(\Syz_C(X,Y,Z))\ra\Oc_C(-14)\ra 0.$$
Since $-3^t+4<0\Leftrightarrow t\geq 2$, the $e$-th Frobenius pull-backs of $\Syz_C(X,Y,Z)$ split as 
$$\Oc_C\left(-13\cdot p^{e-2}\right)\oplus\Oc_C\left(-14\cdot p^{e-2}\right)$$
for $e\geq 4$. By Theorem \ref{almarhkf} and 
Corollary \ref{hkmdiag} we have
$$\HKF(R,3^e)=\HKM(R)\cdot 3^e=\frac{427}{81}\cdot 3^e$$
for $e\geq 4$. The other values are $\HKF(R,3^0)=1$, $\HKF(R,3^1)=27$, $\HKF(R,3^2)=419$ and $\HKF(R,3^3)=3843$ by an explicit computation with CoCoA. Moreover, 
another explicit computation shows
$$\Syz_R(X^9,Y^9,Z^9)\cong\Syz_R(X^5,Y^5,Z^5)(-6) \text{ and } \Syz_R(X^{27},Y^{27},Z^{27})\cong R(-39)\oplus R(-42).$$
\end{exa}

We now turn to the case, where $\Syz_C(X,Y,Z)$ is strongly semistable. 

\begin{lem}\label{sssandnontri}
If $p$ is odd, the bundle $\Syz_C(X,Y,Z)$ is not strongly semistable if and only if there is a $q=p^e$ such that 
$\Syz_C(X^q,Y^q,Z^q)\cong\Oc_C(l_1)\oplus\Oc_C(l_2)$ holds for some $l_1$, $l_2\in\Z$.
\end{lem}

\begin{proof}
Let $\Syz_C(X,Y,Z)$ be strongly semistable and assume that $\Syz_C(X^q,Y^q,Z^q)\cong\Oc_C(l_1)\oplus\Oc_C(l_2)$ for some $q$. Since $\Syz_C(X^q,Y^q,Z^q)$ is semistable, 
we must have $l_1=l_2$. Computing the degrees of the vector bundles $\Syz_C(X^q,Y^q,Z^q)$ and $\Oc_C(l)^2$, one obtains the contradiction $-3nq=2nl$.

If $\Syz_C(X,Y,Z)$ is not strongly semistable, then by the proof of Theorem \ref{almarhkf} the bundles $\Syz_C(X^q,Y^q,Z^q)$ split as 
$\Oc_C(l_1)\oplus\Oc_C(l_2)$ for all large q.
\end{proof}

\begin{cor}\label{deltafinprodim}
If $p$ is odd, the bundle $\Syz_C(X,Y,Z)$ is strongly semistable if and only if all quotients $R/(X^q,Y^q,Z^q)$ have infinite projective dimension.
\end{cor}

\begin{proof}
For a fixed $q$ the quotient $R/(X^q,Y^q,Z^q)$ has finite projective dimension if and only $\Syz_R(X^q,Y^q,Z^q)$ is free if and only if 
$\Syz_C(X^q,Y^q,Z^q)\cong\Oc_C(l_1)\oplus\Oc_C(l_2)$ for some $l_1$, $l_2\in\Z$.
\end{proof}

\begin{rem}\label{remchar2}
Note that Lemma \ref{sssandnontri} and Corollary \ref{deltafinprodim} hold in characteristic two under the further assumption that 
$\Syz_C(X,Y,Z)$ is not trivialized by the Frobenius.
\end{rem}

In \cite{vraciu} the authors study on Fermat rings $R$ of degree $n$ the minimal projective resolution of the quotients 
$$R/\left(X^N,Y^N,Z^N\right)$$ 
depending on $N$ and $n$. In particular they have shown the following.

\begin{thm}[Kustin, Rahmati, Vraciu]
Let $R=k[X,Y,Z]/(X^n+Y^n+Z^n)$, where $k$ denotes a field of characteristic $p\geq 0$ and $n\in\N_{\geq 1}$. For a given integer $N\geq 1$ the quotient 
$R/(X^N,Y^N,Z^N)$ has finite projective dimension as $R$-module if and only if one of the following conditions is satisfied

\begin{enumerate}
 \item $n|N$,
 \item $p=2$ and $n\leq N$,
 \item $p$ is odd and there exist positive integers $J$ and $e$ with $J$ odd such that
$$\left|Jp^e-\frac{N}{n}\right|<\begin{cases}3^{e-1} & \text{if }p=3,\\ \frac{p^e-1}{3} & \text{if }p^e\equiv 1\text{ mod }3,\\ \frac{p^e+1}{3} & \text{if }p^e\equiv 2\text{ mod }3\end{cases}$$
\end{enumerate}
\end{thm}

\begin{proof}
See \cite[Theorem 6.3]{vraciu}.
\end{proof}

Moreover, Kustin, Rahmati and Vraciu computed the minimal graded free resolution of $R/(X^N,Y^N,Z^N)$ if this quotient has infinite projective dimension. 

\begin{defi}
Let $r$ and $s$ be positive integers with $r+s=n$. Then we define 
$$\phi_{r,s}:=\begin{pmatrix}0 & Z^r & -Y^r & X^s\\ -Z^r & 0 & X^r & Y^s\\ Y^r & -X^r & 0 & Z^s\\ -X^s & -Y^s & -Z^s & 0\end{pmatrix}.$$
\end{defi}

Note that $\phi_{r,s}\phi_{s,r}=-(X^n+Y^n+Z^n)\Id_4$, hence $(\phi_{r,s},\phi_{s,r})$ are matrix factorizations for $-(X^n+Y^n+Z^n)$. They are 
equivalent to the matrix factorizations from Example \ref{diagrank2matfac} with $d_i=n$ and $a=b=c=r$. This gives the (ungraded) isomorphisms 
$$\coKer(\phi_{r,s})\cong\Syz_R(X^r,Y^r,Z^r).$$

\begin{thm}[Kustin, Rahmati, Vraciu]\label{inftyfree}
Let $N=\theta\cdot n+r$ with $\theta\in\N$ and $r\in\{1,\ldots,n-1\}$. Assume that 
$$Q:=R/\left(X^N,Y^N,Z^N\right)$$ 
has infinite projective dimension. If $\theta=2\cdot\eta-1$, then the homogenous minimal free resolution of $Q$ is given by
$$\xymatrix{
\ldots \ar[r]^{\phi_{r,n-r}} & F_3 \ar[r]^{\phi_{n-r,r}} & F_2 \ar[r]^{\phi_{r,n-r}} & F_1\ar[r] & R(-N)^3\ar[r] & R
}$$
with 
\begin{align*}
F_1 & :=R(-3\eta n+n-r)^3 \oplus R(-3\eta n+2n-3r),\\
F_2 & :=R(-3\eta n+n-2r)^3 \oplus R(-3\eta n),\\
F_3 & :=R(-3\eta n-r)^3 \oplus R(-3\eta n+n-3r).
\end{align*}
If $\theta=2\cdot\eta$, then the homogenous minimal free resolution of $Q$ is given by
$$\xymatrix{
\ldots \ar[r]^{\phi_{n-r,r}} & F_3 \ar[r]^{\phi_{r,n-r}} & F_2 \ar[r]^{\phi_{n-r,r}} & F_1 \ar[r] & R(-N)^3\ar[r] & R
}$$
with
\begin{align*}
F_1 & :=R(-3\eta n-2r)^3 \oplus R(-3\eta n-n),\\
F_2 & :=R(-3\eta n-n-r)^3 \oplus R(-3\eta n-3r),\\
F_3 & :=R(-3\eta n-n-2r)^3 \oplus R(-3\eta n-2n).
\end{align*}
\end{thm}

\begin{proof}
The statement follows from \cite[Theorem 3.5]{vraciu} combined with \cite[Theorems 5.14 and 6.1]{vraciu}.
\end{proof}

\begin{cor}\label{inftysyz}
Under the hypothesis of Theorem \ref{inftyfree}, we have
$$\begin{aligned}
\Syz_R\left(X^N,Y^N,Z^N\right)(m) & \cong \begin{cases}\coKer(\phi_{r,n-r}) & \text{if }\theta\text{ is even,}\\ \coKer(\phi_{n-r,r}) & \text{if }\theta\text{ is odd,}\end{cases}\\
    & \cong \begin{cases}\Syz_R(X^r,Y^r,Z^r) & \text{if }\theta\text{ is even,}\\ \Syz_R(X^{n-r},Y^{n-r},Z^{n-r}) & \text{if }\theta\text{ is odd.}\end{cases}
\end{aligned}$$
for some $m\in\Z$.
\end{cor}

In \cite{holgeralmar} the authors proved that $\Syz_C(X,Y,Z)$ admits a $(0,1)$-Frobenius periodicity if $p\equiv -1\text{ }(2n)$ (cf. Example \ref{fermatexa2}). 
The last corollary was the last ingredient necessary to generalize this result.

\begin{thm}\label{frobpernontri}
Let $C$ be the projective Fermat curve of degree $n$ over an algebraically closed field of characteristic $p>0$ with $\gcd(p,n)=1$. 
The bundle $\Syz_C(X,Y,Z)$ admits a Frobenius periodicity if and only if 
$$\delta\left(\frac{1}{n},\frac{1}{n},\frac{1}{n}\right)=0.$$
Moreover, the length of this periodicity is bounded above by the order of $p$ modulo $2n$.  
\end{thm}

\begin{proof}
If $\Syz_C(X,Y,Z)$ admits a Frobenius periodicity, it has to be strongly semistable. If it is strongly semistable we have to distinguish two cases. 
If $\Syz_C(X,Y,Z)$ is trivialized by the Frobenius (which might only happen if $p=2$ by Lemma \ref{sssandnontri}), we have of course a Frobenius periodicity.
 Otherwise, we obtain a Frobenius periodicity by Corollary \ref{inftysyz}, which can be applied by Corollary \ref{deltafinprodim} (resp. Remark \ref{remchar2}). 
The supplement follows since the isomorphism class of $\Syz_R(X^q,Y^q,Z^q)$ depends on $r$ and the parity of $\tfrac{q-r}{n}$.
\end{proof}

This result enables us to compute the Hilbert-Kunz function of $R$, if $\Syz_C(X,Y,Z)$ is strongly semistable. If $\Syz_C(X,Y,Z)$ is trivialized by 
the Frobenius, we can use Lemma \ref{hkffrei} to obtain $\HKF(R,2^e)=3n\cdot 4^{e-1}$ for $e\gg 0$.

\begin{thm}\label{hkfsss}
Let $R:=k[X,Y,Z]/(X^n+Y^n+Z^n)$ with an algebraically closed field $k$ of characteristic $p>0$ coprime to $n$. Assume that $\Syz_C(X,Y,Z)$ is strongly 
semistable and not trivialized by the Frobenius. Let $q=p^e=n\theta+r$ with $\theta\in\N$ and $0\leq r<n$. Then 
$$\HKF(R,q)= \left\{\begin{aligned} & \frac{3n}{4}\cdot q^2-\frac{3n}{4}r^2+r^3 && \text{if }\theta\text{ is even,}\\ & \frac{3n}{4}\cdot q^2-\frac{3n}{4}(n-r)^2+(n-r)^3 && \text{if }\theta\text{ is odd.}\end{aligned}\right.$$
\end{thm}

\begin{proof}
If $\Syz_C(X,Y,Z)$ is strongly semistable, the quotients $R/(X^q,Y^q,Z^q)$ have infinite projective dimension 
and their resolutions are given by Theorem \ref{inftyfree}. Now combine Corollary \ref{inftysyz} with Lemma \ref{hkfdreivar}. 
\end{proof}

\begin{exa}\label{fermatexa1}
Let $p\equiv 1\text{ }(2n)$. We have always $p^e=n\theta +1$ for some even $\theta$. An easy computation shows 
$\delta\left(\tfrac{1}{n},\tfrac{1}{n},\tfrac{1}{n}\right)=0$. By Corollary \ref{inftysyz}, we have 
$$\Syz_R(X^q,Y^q,Z^q)\cong\Syz(X,Y,Z)\left(-\frac{3}{2}\cdot(q-1)\right)$$
for all $q$, hence $\Syz_C(X,Y,Z)$ admits a $(0,1)$-Frobenius periodicity. With Theorem \ref{hkfsss} we obtain
$$\HKF(R,q)=\frac{3n}{4}\cdot q^2+1-\frac{3n}{4}.$$
\end{exa}

\begin{exa}\label{fermatexa2}
Let $p\equiv -1\text{ }(2n)$. We have $p^e=n\theta +n-1$ for some odd $\theta$ if $e$ is odd and $p^e=n\theta +1$ for some even $\theta$ if $e$ is even. 
Again, it is easy to show $\delta\left(\tfrac{1}{n},\tfrac{1}{n},\tfrac{1}{n}\right)=0$. By Corollary \ref{inftysyz}, we have 
$$\Syz_R(X^q,Y^q,Z^q)\cong\Syz(X,Y,Z)\left(-\frac{3}{2}\cdot(q-1)\right)$$
for all $q$, hence $\Syz_C(X,Y,Z)$ admits a $(0,1)$-Frobenius periodicity. With Theorem \ref{hkfsss} we obtain
$$\HKF(R,q)=\frac{3n}{4}\cdot q^2+1-\frac{3n}{4}.$$
This recovers Theorem 3.4 and Corollary 4.1 of \cite{holgeralmar}.
\end{exa}

\begin{exa}\label{fermatexa3}
Let $n=5$. If $p\equiv \pm 1\text{ }(10)$, the Hilbert-Kunz function of $R$ was computed in the two previous examples and is given by 
$$\HKF(R,q)=\frac{15}{4}\cdot q^2-\frac{11}{4}.$$ 
In this case we can also compute the Hilbert-Kunz function for $p\equiv \pm 3\text{ }(10)$. If $p\equiv 3\text{ }(10)$, we have $p^e=5\theta+r$, 
where $(\theta,r)$ is of the form
$$(\theta,r)=\begin{cases}(\text{even},3) & \text{if }e\equiv 1\text{ }(4),\\ (\text{odd},4) & \text{if }e\equiv 2\text{ }(4),\\ (\text{odd},2) & \text{if }e\equiv 3\text{ }(4),\\ (\text{even},1) & \text{if }e\equiv 0\text{ }(4).\end{cases}$$
If $p\equiv 7\text{ }(10)$, we have $p^e=5\theta+r$, where $(\theta,r)$ is of the form
$$(\theta,r)=\begin{cases}(\text{odd},2) & \text{if }e\equiv 1\text{ }(4),\\ (\text{odd},4) & \text{if }e\equiv 2\text{ }(4),\\ (\text{even},3) & \text{if }e\equiv 3\text{ }(4),\\ (\text{even},1) & \text{if }e\equiv 0\text{ }(4).\end{cases}$$
It easy to check that $\delta\left(\tfrac{1}{5},\tfrac{1}{5},\tfrac{1}{5}\right)=0$ in both cases. From Corollary \ref{inftysyz} we obtain 
$$\fpb{e}(\Syz_C(X,Y,Z))\cong \left\{\begin{aligned} & \Syz_R(X^3,Y^3,Z^3)\left(-\frac{3}{2}\cdot(q-3)\right) && \text{if }e\text{ is odd},\\ & \Syz_R(X,Y,Z)\left(-\frac{3}{2}\cdot(q-1)\right) && \text{if }e\text{ is even}.\end{aligned}\right.$$
Since the Hilbert-series
\begin{align*}
\lK_{\Syz_R(X,Y,Z)}(t) &= \frac{-3t^6+t^5-t^3+3t^2}{(1-t)^3},\\
\lK_{\Syz_R(X^3,Y^3,Z^3)}(t) &= \frac{-t^9-3t^8+3t^6+t^5}{(1-t)^3}
\end{align*}
do not differ by a factor $t^l$, the $R$-modules $\Syz_R(X,Y,Z)$ and $\Syz_R(X^3,Y^3,Z^3)$ are not isomorphic, hence $\Syz_C(X,Y,Z)$ admits a $(0,2)$-Frobenius periodicity.
The Hilbert-Kunz function of $R$ is given by
$$\HKF(R,q)=\left\{\begin{aligned} & \frac{15}{4}\cdot q^2-\frac{27}{4} && \text{if }e\text{ is odd,}\\ & \frac{15}{4}\cdot q^2-\frac{11}{4} && \text{if }e\text{ is even.}\end{aligned}\right.$$
\end{exa}

\begin{exa}\label{fermatexa4}
Let $p=37$ and $n=14$. Then $p^e=14\cdot \theta +r$ with 
$$(\theta,r)=\begin{cases}(\text{even},9) & \text{if }e\equiv 1\text{ }(3),\\ (\text{odd},11) & \text{if }e\equiv 2\text{ }(3),\\ (\text{even},1) & \text{if }e\equiv 0\text{ }(3).\end{cases}$$
This gives the Hilbert-Kunz function
$$\HKF(R,37^e)=\left\{\begin{aligned} & \frac{21}{2}\cdot 37^{2e}-\frac{243}{2} && \text{if }e\equiv 1\text{ }(3),\\ & \frac{21}{2}\cdot 37^{2e}-\frac{135}{2} && \text{if }e\equiv 2\text{ }(3),\\ & \frac{21}{2}\cdot 37^{2e}-\frac{19}{2} && \text{if }e\equiv 0\text{ }(3).\end{aligned}\right.$$
In this example, we see directly that we have a $(0,3)$-Frobenius periodicity 
$$\fpb{3}(\Syz_C(X,Y,Z))\cong\Syz_C(X,Y,Z)\left(-\tfrac{3}{2}\cdot (q-1)\right).$$
\end{exa}

The next examples deal with the question which Frobenius periodicities of the bundles $\Syz_C(X,Y,Z)$ can be achieved. A sufficient condition for 
having a Frobenius periodicity is (cf. Theorem \ref{frobpernontri})
$$\delta\left(\frac{1}{n},\frac{1}{n},\frac{1}{n}\right)=0,$$
which is equivalent to the condition that the distances of all triples $v_e:=\left(\tfrac{p^e}{n},\tfrac{p^e}{n},\tfrac{p^e}{n}\right)$ to 
$L_{\text{odd}}$ are at least one. Let $p^e=\theta_e\cdot n+r_e$ with $\theta_e,r_e\in\N$ and $0\leq r_e<n$. The nearest element to $v_e$ in 
$L_{\text{odd}}$ is given by the component-wise round-up of $v_e$ if $\theta_e$ is even and by the component-wise round-down of $v_e$ if 
$\theta_e$ is odd. This leads to the conditions 
$$\begin{aligned}
&\begin{aligned}
& 3\cdot \left(1-\frac{r_e}{n}\right)\geq 1 && \text{if }\theta_e\text{ is even,}\\
& 3\cdot \frac{r_e}{n}\geq 1 && \text{if }\theta_e\text{ is odd.}
\end{aligned}
&\Leftrightarrow &
\begin{aligned}
& 2\cdot n\geq 3\cdot r_e && \text{if }\theta_e\text{ is even,}\\
& 3\cdot r_e\geq n && \text{if }\theta_e\text{ is odd.}
\end{aligned}
\end{aligned}$$

\begin{exa}\label{fermatexa5}
Let $p$ be odd, $l\in\N$, $l\geq 0$ and $n=\tfrac{p^{l+1}+1}{2}$. For $0\leq e\leq 2l+2$ we have 
\begin{equation}\label{respe}
p^e=\left\{
\begin{aligned}
& 0\cdot n+p^e && \text{if }0\leq e\leq l,\\
& \left(2\cdot p^{e-l-1}-1\right)\cdot n+n-p^{e-l-1} && \text{if }l+1\leq e\leq 2l+1,\\
& \left(2\cdot p^{l+1}-2\right)\cdot n+1 && \text{if }e= 2l+2.
\end{aligned}\right.
\end{equation}
This shows that $p^e$ is of the form $\text{even}\cdot n+p^{e'}$ or $\text{odd}\cdot n+n-p^{e'}$ for some $0\leq e'\leq l$.
Since $2n=p^{l+1}+1\geq 3p^{e'}$ for all $0\leq e'\leq l$, we see that $\delta\left(\tfrac{1}{n},\tfrac{1}{n},\tfrac{1}{n}\right)=0$.
By Corollary \ref{inftysyz} we find the isomorphism 
$$\fpb{e}(\Syz_C(X,Y,Z))\cong\Syz_C\left(X^{p^{e'}},Y^{p^{e'}},Z^{p^{e'}}\right)\left(-\frac{3}{2}\cdot (p^e-p^{e'})\right)$$
with $e\equiv e'\text{ mod }l+1$ and $e'\in\{0,\ldots,l\}$. Since the Hilbert-series
$$\lK_{\Syz_R\left(X^{p^e},Y^{p^e},Z^{p^e}\right)}(t)=\frac{3\cdot t^{2p^e}-3\cdot t^{n+p^e}+t^n-t^{3p^e}}{(1-t)^3}$$
for $e\in\{0,\ldots,l\}$ do not differ by a factor $t^a$, $a\in\Z$, the $R$-modules are non-isomorphic. All in all, we found a $(0,l+1)$-Frobenius 
periodicity. The Hilbert-Kunz function is given by
$$\HKF(R,p^e)=\frac{3n}{4}\cdot\left(p^{2e}-p^{2e'}\right)+p^{3e'}$$
with $e\equiv e'\text{ mod }l+1$ and $e'\in\{0,\ldots,l\}$.
\end{exa}

\begin{exa}\label{fermatexa6}
Let $p=2$ and $n=3$. Then $2^e$ is of the form $\text{even}\cdot n+2$ or $\text{odd}\cdot n+1$. Since $2\cdot n=3\cdot 2\geq 3\cdot 2^e$ 
for $0\leq e\leq 1$ and $3\cdot 1=n$, we see that $\delta\left(\tfrac{1}{n},\tfrac{1}{n},\tfrac{1}{n}\right)=0$.
Applying Corollary \ref{inftysyz}, we find the isomorphism
$$\fpb{e}(\Syz_C(X,Y,Z)) \cong
\left\{\begin{aligned}
& \Syz_C\left(X,Y,Z\right) &&\text{if }e=0,\\
& \Syz_C\left(X^2,Y^2,Z^2\right)\left(-3(p^{e-1}-1)\right) &&\text{if }e\geq 1.
\end{aligned}\right.$$
Since the Hilbert-series of $\Syz_R(X,Y,Z)$ and $\Syz_R(X^2,Y^2,Z^2)$ are given by
$$\frac{3t^2-3t^4+t^3-t^3}{(1-t)^3}
\quad \text{and} \quad
\frac{3t^{4}-3t^{5}+t^3-t^{6}}{(1-t)^3},
$$
we see that the two modules are non-isomorphic. We obtain a $(1,2)$-Frobenius periodicity. By Theorem \ref{hkfsss} the Hilbert-Kunz function is given by
$$\HKF(R,2^e)=\begin{cases}
1 & \text{if }e=0,\\
9\cdot 2^{2e-2}-1 & \text{if }e\geq 1. 
\end{cases}$$
\end{exa}

\begin{exa}\label{fermatexa7}
Let $p=2$ and $2^l\leq n<2^{l+1}$, hence $n=2^l+x$ with $1\leq x<2^l$. We want to show that $\Syz_C(X,Y,Z)$ admits a Frobenius periodicity on the smooth curve $C$ 
if and only if $n=3$. Since we need $\delta\left(\tfrac{1}{n},\tfrac{1}{n},\tfrac{1}{n}\right)=0$, we must have
$$2n=2^{l+1}+2x\geq 3\cdot 2^l\geq 3\cdot 2^e$$ 
for all $0\leq e\leq l$. This is equivalent to $x\geq 2^{l-1}$. Since $2^{l+1}=n+2^l-x$ with $0<2^l-x<2^l<n$, we need that the inequality
$3\cdot (2^l-x)\geq n$ holds. This is equivalent to $x\leq 2^{l-1}$, hence $n=2^l+2^{l-1}=3\cdot 2^{l-1}$. Because of the assumption $\gcd(2,n)=1$, only the case $n=3$ is left.
\end{exa}

From the Examples \ref{fermatexa6} - \ref{fermatexa7} we obtain the following corollary.

\begin{cor}
Let $p=2$. The bundle $\Syz_C(X,Y,Z)$ admits a Frobenius periodicity on the smooth curve $C$ if and only if $n=3$.
\end{cor}

\section{Open questions}
At the end, we gather some open questions related to topics of this thesis.

\subsection{Questions concerning rings of dimension two}



In Chapter \ref{chapmatfac} we used some sheaf-theoretic methods to compute from a given matrix factorization $(\phi,\psi)$ a first syzygy module of an ideal such that this module is 
isomorphic to the cokernel of the matrix factorization. At some point in Construction \ref{supposethat} we used the phrase ``Suppose that we can choose a 
$J\subseteq\{1,\ldots,n\}$...'', where $J$ was the set of indices of the columns of $\psi$ we had to keep. In general, it is not understood under which circumstances this $J$ has to exist.

\begin{que}
Are there necessary (and sufficient) conditions (on the hypersurface $R$) ensuring that the set $J$ exists?
\end{que} 

Another question is what happens to morphisms between matrix factorizations under this construction.

\begin{que}
Is it possible to make the construction in Chapter \ref{chapmatfac} categorical?
\end{que}

Another related question is whether sheaf-theoretic tools are necessary or not.

\begin{que}
Is there a completely algebraic way to obtain first syzygy modules of ideals from a given matrix factorization?
\end{que}

In Chapter \ref{pbmaxcm} we computed the pull-backs of the indecomposable, maximal Cohen-Macau\-lay modules over surface rings of type ADE along the 
inclusions induced by normal subgroups. From these computations we obtained that the Frobenius pull-backs of $\Syz_R(\mm)$, where $R$ is of type D or E, 
are indecomposable. This approach is quite complicated and works only in this very special situation. This leads to the following question.

\begin{que}
Is there a (numerical) invariant of vector bundles, which measures the splitting behaviour?
\end{que}

We turn back to the Fermat rings $R=k[X,Y,Z]/(X^n+Y^n+Z^n)$ with an algebraically closed field $k$ of characteristic $p>0$. In the case where 
$\Sc:=\Syz_{\Proj(R)}(X,Y,Z)$ is not strongly semistable, we saw in Theorem \ref{almarfilt} that the strong Harder-Narasimhan filtration of $\fpb{s}(\Sc)$, 
where $s\geq 0$ comes from Han's Theorem, is of the form
$$0\ra\Oc_{\Proj(R)}(l_1)\ra\fpb{s}(\Sc)\ra\Oc_{\Proj(R)}(l_2)\ra 0$$
for some $l_1$, $l_2\in\Z$. In the proof of Theorem \ref{almarhkf} we saw that there is a $t\in\N$ such that the $t$-th Frobenius pull-back of the above sequence splits.
In this situation we can compute $\HKF(R,p^e)$, when $p^e<n$ or $e\geq s+t$.

\begin{que}
Let $C$ be a projective Fermat curve of degree $n$, let $s$ be the integer from Han's Theorem and $t$ be minimal such that 
$$\fpb{s+t}(\Syz_C(\mm))\cong\Oc_C(m_1)\oplus\Oc_C(m_2)$$
for some $m_1$, $m_2\in\Z$. Find a way to compute $\HKF(R,p^e)$ in the cases $n<p^e\leq p^{s+t-1}$.
\end{que}

\begin{que}
The situation is the same as in the previous question. Is it possible that $\Syz_R(\mm\qpot)$ splits as a direct sum of two non-free $R$-modules of rank one?
\end{que}

The next questions deal with general diagonal hypersurfaces of dimension two. We denote by $R$ the standard-graded Fermat ring of degree $n$ and by $C$ its projective spectrum. 
We will denote by $S$ a general diagonal hypersurface with exponent-triple $(d_1,d_2,d_3)\in\N^3_{\geq 2}$ and by $D$ its projective spectrum.

\begin{exa}
Han computed $\HKF(S,3^e)$ in her thesis \cite{handiss}, where $S$ is the diagonal hypersurface with exponent-triple $(5,8,8)$. The Hilbert-Kunz multiplicity is 
$\tfrac{24}{5}$ and the tail of $\HKF(S,3^e)$ is $-\tfrac{19}{5}$ if $e$ is even and $-\tfrac{81}{5}$ if $e$ is odd. Since 
$\delta\left(\tfrac{1}{5},\tfrac{1}{8},\tfrac{1}{8}\right)=0$, we know by Corollary \ref{deltastrongsemistable} that 
$\Syz_C(X^8,Y^5,Z^5)$ is strongly semistable on the projective Fermat curve of degree 40. A direct computation using CoCoA shows 
$$\fpb{e}(\Syz_D(X,Y,Z))\cong\left\{\begin{aligned}
& \Syz_D(X^3,Y^3,Z^3)(-9q+27) && \text{if }e\cong 1\text{ }(4),\\
& \Syz_D(X,Y^7,Z^7)(-9q+39) && \text{if }e\cong 2\text{ }(4),\\
& \Syz_D(X^3,Y^5,Z^5)(-9q+37) && \text{if }e\cong 3\text{ }(4),\\
& \Syz_D(X,Y,Z)(-9q+9) && \text{if }e\cong 0\text{ }(4).
\end{aligned}\right.$$
This shows that there is a $(0,e_0)$-Frobenius periodicity 
$$\fpb{4}(\Syz_D(X,Y,Z))\cong\Syz_D(X,Y,Z)(m)\text{ for some }m\in\Z$$
of length $e_0\leq 4$. Computing the Hilbert-series of the corresponding $S$-modules, one sees that they do not differ by a factor $t^m$, hence they cannot be 
isomorphic and the periodicity has exactly length four.

Using Formula \ref{formnontr} on page \pageref{formnontr}, we recover Han's result. With $3^e=5\theta +r$, $\theta\in\N$, $r\in\{0,\ldots,4\}$, the tuple $(\theta,r)$ is cyclic of the form 
$(\text{even},1)$, $(\text{even},3)$, $(\text{odd},4)$, $(\text{even},2)$ and with $3^e=8\theta'+r'$, $\theta'\in\N$, $r'\in\{0,\ldots,7\}$, the tuple $(\theta',r')$ is cyclic of the form 
$(\text{even},1)$, $(\text{even},3)$, $(\text{odd},1)$, $(\text{odd},3)$ for $e\geq 0$. Hence, we obtain in this case (for some $m\in\Z$)
$$\fpb{e}(\Syz_D(X,Y,Z))(m)\cong\left\{\begin{aligned}
& \Syz_D(X^r,Y^{r'},Z^{r'}) && \text{if }\theta\text{ and }\theta'\text{ are even},\\
& \Syz_D(X^{5-r},Y^{8-r'},Z^{8-r'}) && \text{if }\theta\text{ and }\theta'\text{ are odd},
\end{aligned}\right.$$
which has a similar shape as Corollary \ref{inftyfree}.
\end{exa}

The surface rings of type $A_n$ (in odd characteristics), $E_6$ and $E_8$ are other diagonal hypersurfaces and we saw that the corresponding $\delta$ vanishes in all characteristics coprime 
to the degree of the defining polynomial. In these cases we saw that the modules $\Syz(X^q,Y^q,Z^q)$ satisfy isomorphisms similar to Corollary \ref{inftyfree}.

For $1\leq a_i\leq d_i-1$ we define the matrix factorizations $(\phi_{a_1,a_2,a_3},\psi_{a_1,a_2,a_3})$ with
\begin{align*}
\phi_{a_1,a_2,a_3}:= &
\begin{pmatrix}
X^{a_1} & Y^{a_2} & Z^{a_3} & 0\\
-Y^{d_2-a_2} & X^{d_1-a_1} & 0 & Z^{a_3}\\
-Z^{d_3-a_3} & 0 & X^{d_1-a_1} & -Y^{a_2}\\
0 & -Z^{d_3-a_3} & Y^{d_2-a_2} & X^{a_1}
\end{pmatrix}\\
\psi_{a_1,a_2,a_3}:= &
\begin{pmatrix}
X^{d_1-a_1} & -Y^{a_2} & -Z^{a_3} & 0\\
Y^{d_2-a_2} & X^{a_1} & 0 & -Z^{a_3}\\
Z^{d_3-a_3} & 0 & X^{a_1} & Y^{a_2}\\
0 & Z^{d_3-a_3} & -Y^{d_2-a_2} & X^{d_1-a_1}
\end{pmatrix}.
\end{align*}
Assume that $Q:=k[X,Y,Z]/(X^N,Y^N,Z^N,X^{d_1}+Y^{d_2}+Z^{d_3})$ has infinite projective dimension as $S:=k[X,Y,Z]/(X^{d_1}+Y^{d_2}+Z^{d_3})$-module. 
Let $N:=d_i\theta_i+r_i$ with $\theta_i\in\N$ and $0\leq r_i\leq d_i-1$. Let $A_i:=r_i$ if $\theta_i$ is even and $A_i:=d_i-r_i$ if $\theta_i$ is odd.

\begin{que}
Is it true that the minimal $S$-free resolution of $Q$ is given by 
$$\ldots \stackrel{\phi_{A_1,A_2,A_3}}{\lra} S^4 \stackrel{\psi_{A_1,A_2,A_3}}{\lra} S^4\stackrel{\phi_{A_1,A_2,A_3}}{\lra} S^4 \lra S^3\lra S$$
up to the degree shifts necessary to make the resolution homogeneous?
\end{que}

\begin{exa}
Consider $R:=k[X,Y,Z]/(X^{12}+Y^{12}+Z^{12})$ and $I:=(X^4,Y^4,Z^3)$ in characteristic five. Then 
$$\delta\left(\frac{4}{12},\frac{4}{12},\frac{3}{12}\right)=\frac{1}{12}\cdot \frac{1}{5},$$
hence $s=1$ in Han's Theorem and $\Syz_C(I)$ is not strongly semistable. A direct computation shows that the minimal degree of a generator of
$\Syz_R(X^{20},Y^{20},Z^{15})$ is 27 and that we have a splitting 
$$\Syz_R(X^{100},Y^{100},Z^{75})\cong R(-135)\oplus R(-140).$$
This gives for $e\geq 2$ 
\begin{align*}
\HKF(I,R,5^e)&=\frac{2928}{25}\cdot 25^e\text{ and}\\
\HKF(k[X,Y,Z]/(X^3+Y^3+Z^4),5^e)&=\frac{61}{25}\cdot 25^e
\end{align*}
by Lemma \ref{hkffrei}. In the cases $E_6$ and $E_8$ we had in the characteristics two, three (and five) a splitting for $e\geq 1$. 
In these cases one has $\delta\neq 0$ and $s=1$.
\end{exa}

\begin{que}
Assume $\delta\left(\tfrac{a}{n},\tfrac{b}{n},\tfrac{c}{n}\right)\neq 0$. Do the modules $\Syz_R(X^{aq},Y^{bq},Z^{cq})$ become free for large $e$?
\end{que}

Equivalently, one could ask the following question.

\begin{que}
Assume $\delta\left(\tfrac{1}{d_1},\tfrac{1}{d_2},\tfrac{1}{d_3}\right)\neq 0$. Do the modules $\Syz_S(\mm\qpot)$ become free for large $e$?
\end{que}

\subsection{Questions concerning rings of dimension at least three}

As soon as the dimension of $R$ grows, there are very few explicit examples of Hilbert-Kunz multiplicities and Hilbert-Kunz functions. 
One might hope to use Kn\"orrer's periodicity (cf. \cite[Theorem 12.10]{yobook}) to obtain the Hilbert-Kunz function of 
$$T:=k[U_1,\ldots,U_n,X,Y,Z]/\left(\sum_{i=1}^nU_i^2+F(X,Y,Z)\right),$$
where $F$ is of type ADE. Recall that Kn\"orrer's periodicity says that the categories $\underline{\MF}(R)$ and $\underline{\MF}(S)$ with 
$R=k[X_1,\ldots,X_n]/(G)$, $S=k[U,V,X_1,\ldots,X_n]/(U^2+V^2+G)$ and $\chara(k)\neq 2$ are equivalent via the functor $\_\hat{\otimes}(U+iV,U-iV)$. 
In general first syzygy modules of ideals are not (maximal) Cohen-Macaulay, but maybe they appear in the two periodic free resolutions coming from 
the matrix factorizations in which case they would be (even maximal) Cohen-Macaulay. This does not happen. We prove this for $\Syz_T(\mm)$, where $F$ is 
of type $E_6$. We have 

\begin{align*}
\lK_{\Syz_T(\mm)}(t) & = \frac{((n+1)t^6+t^4+t^3-1)(1-t^{12})}{(1-t^6)^{n+1}(1-t^4)(1-t^3)}+1\\
& = \frac{((n+1)t^6+t^4+t^3-1)(1+t^6)+(1-t^6)^n(1-t^4)(1-t^3)}{(1-t^6)^n(1-t^4)(1-t^3)}.
\end{align*}

If $\Syz_T(\mm)$ would be Cohen-Macaulay, all coefficients in the numerator of its Hilbert-series have to be positive (cf. Lemma \ref{hilbsercm}). 
The numerator's leading term is $(-1)^nt^{6n+7}$, which is negative for $n$ odd. If $n$ is even, the second highest term of 
$(1-t^6)^n(1-t^4)(1-t^3)$ is $-t^{6n+4}$, which does not cancel with a term from $((n+1)t^6+t^4+t^3-1)(1+t^6)$. This shows that one needs to work 
with syzygy modules of higher order.

\begin{que}
Is it possible to compute from a given matrix factorization a higher order syzygy module of an ideal which is isomorphic to the 
cokernel of the matrix factorization?
\end{que}

\begin{appendix}
\chapter{An implementation of Corollary 2.31 for CoCoA}\label{app.myhkmthm}
We give an implementation of Corollary \ref{myhkmcor} in CoCoA.

\begin{func}An implementation of the floor function for \verb Q $\in\Q$.

\begin{verbatim}
Define Floor(Q)
M:=Num(Q);
N:=Den(Q);
If Mod(M,N)=0 
 Then Return M/N;
 Else I:=0; While I<=M/N Do I:=I+1; EndWhile; Return I-1;
EndIf;
EndDefine;
\end{verbatim}
\end{func}

\begin{func}An implementation of the ceil function for \verb Q $\in\Q$.

\begin{verbatim}
Define Ceil(Q)
M:=Num(Q);
N:=Den(Q);
If Mod(M,N)=0 
 Then Return M/N;
 Else I:=M; While I>=M/N Do I:=I-1; EndWhile; Return I+1;
EndIf;
EndDefine;
\end{verbatim}
\end{func}

\begin{func}This function returns for a list \verb L $\in ([0,\ldots,\infty)\cap\Q)^3$ the minimal distance to the lattice $L_{\text{odd}}$.

\footnotesize
\begin{verbatim}
Define DistLodd(L)
Lodd:=[[Floor(L[1]),Floor(L[2]),Floor(L[3])],
[Floor(L[1]),Floor(L[2]),Ceil(L[3])],[Floor(L[1]),Ceil(L[2]),Floor(L[3])],
[Floor(L[1]),Ceil(L[2]),Ceil(L[3])],[Ceil(L[1]),Floor(L[2]),Floor(L[3])],
[Ceil(L[1]),Floor(L[2]),Ceil(L[3])],[Ceil(L[1]),Ceil(L[2]),Floor(L[3])],
[Ceil(L[1]),Ceil(L[2]),Ceil(L[3])]];
Dist:=[];
For I:=8 To 1 Step -1 Do
 If IsEven(Sum(Lodd[I])) 
  Then Remove(Lodd,I);
  Else 
   Append(Dist,Abs(L[1]-Lodd[I][1])+Abs(L[2]-Lodd[I][2])+Abs(L[3]-Lodd[I][3]));
 EndIf;
EndFor;
Return Min(Dist);
EndDefine;
\end{verbatim}
\normalsize
\end{func}

\begin{func}This function computes a lower bound for the $s$ in Han's Theorem. The argument \verb P  is a prime number and the argument \verb T  is a list 
of three non-negative rational numbers.

\begin{verbatim}
Define Lower(P,T)
M:=Max(T);
L:=0;
If 3/2*M<=1 
 Then
  While P^L>=3/2*M Do L:=L-1; EndWhile; L:=L+1;
 Else
  While P^L<3/2*M Do L:=L+1; EndWhile;
EndIf;
Return L;
EndDefine;
\end{verbatim}
\end{func}

\begin{func}This function computes an upper bound for the $s$ in Han's Theorem. The argument \verb P  is a prime number and the argument \verb T  is a list 
of three non-negative rational numbers.

\small
\begin{verbatim}
Define Upper(P,T)
U:=1;
Residues:=[];
M:=2*LCM(Den(T[1]),Den(T[2]),Den(T[3]));
A:=[Mod(P*Num(T[1]),M),Mod(P*Num(T[2]),M),Mod(P*Num(T[3]),M)];
While Not (A IsIn Residues) Do
 Append(Residues,A);
 U:=U+1;
 A:=[Mod(P^U*Num(T[1]),M),Mod(P^U*Num(T[2]),M),Mod(P^U*Num(T[3]),M)];
EndWhile;
Return U-1;
EndDefine;
\end{verbatim}
\normalsize
\end{func}

\begin{func}This function computes the characteristic \verb P  value of Han's $\delta$ function of the list \verb T  of three non-negative rational numbers.

\begin{verbatim}
Define Delta(P,T)
If 2*Max(T)>=Sum(T) Then Return 2*Max(T)-Sum(T); EndIf;
L:=Lower(P,T);
U:=Upper(P,T);
For S:=L To U Do
 D:=DistLodd(P^S*T);
 If D<1 Then Return (1-D)/P^S; EndIf;
EndFor;
Return 0;
EndDefine;
\end{verbatim}
\end{func}

\begin{func}This function is the final implementation of Corollary \ref{myhkmcor}. Within the notation of this Corollary, the argument \verb Exp  is a 
$3\times 3$ matrix whose $i$-th row consists of the exponents in the $i$-th summand of the trinomial $F$. The argument \verb P  is a prime number, the 
argument \verb T  is the list $[a,b,c]$ and \verb Type  denotes the type of $F$, hence 1 or 2.

\footnotesize
\begin{verbatim}
Define HKMTri(Exp,P,T,Type)
D:=Sum(Exp[1]);
Alphap:=T[1]*(D-Exp[2][3]-Exp[3][2])-T[2]*Exp[3][1]-T[3]*Exp[2][1];
Betap:=T[2]*(D-Exp[3][1]-Exp[1][3])-T[1]*Exp[3][2]-T[3]*Exp[1][2];
Gammap:=T[3]*(D-Exp[1][2]-Exp[2][1])-T[1]*Exp[2][3]-T[2]*Exp[1][3];
If Alphap>=0 Then Alphau:=Alphap; Alphad:=0; 
 Else Alphau:=0; Alphad:=-Alphap; EndIf;
If Betap>=0 Then Betau:=Betap; Betad:=0; 
 Else Betau:=0; Betad:=-Betap; EndIf;
If Gammap>=0 Then Gammau:=Gammap; Gammad:=0; 
 Else Gammau:=0; Gammad:=-Gammap; EndIf;
If Type=1 Then Alpha:=Alphau+Gammad; 
 Else Alpha:=Alphau+Betad+Gammad; EndIf;
If Type=1 Then Beta:=Betau+Alphad; 
 Else Beta:=Betau; EndIf;
If Type=1 Then Gamma:=Gammau+Betad; 
 Else Gamma:=Gammau+Alphad; EndIf;
Lam:=Det(Mat(Exp))/D;
Q:=2*(T[1]*T[2]+T[1]*T[3]+T[2]*T[3])-T[1]^2-T[2]^2-T[3]^2;
Return D*Q/4+(Lam)^2/(4*D)*(Delta(P,[Alpha/Lam,Beta/Lam,Gamma/Lam]))^2;
EndDefine;
\end{verbatim}
\normalsize
\end{func}

\chapter{Computer computations to Chapter 5}
We give the computer computations needed for the argumentation in Chapter \ref{pbmaxcm}. We print only those output lines that are necessary 
for the arguments.

\section{Normal subgroups with GAP and rings of invariants with Singular}\label{app.group}
In this section we compute explicit generators for the normal subgroups of $\T$, $\Obb$ and $\I$ (with their embeddings into $\SL_2(k)$ as given 
in Chapter \ref{pbmaxcm}, whenever the generators are defined in $k^{2\times 2}$). 
Moreover, we use Singular to compute the invariants of $k[x,y]$ under these normal subgroups, where we work over the smallest finite extensions 
of $\Q$ we need to define the generators over $k$. For these computations one needs the library \verb finvar .

We compute the normal subgroups of $\T$ with GAP.

\scriptsize
\begin{verbatim}
gap> o:=[[E(4),0],[0,-E(4)]];
[ [ E(4), 0 ], [ 0, -E(4) ] ]

gap> t:=[[0,1],[-1,0]];
[ [ 0, 1 ], [ -1, 0 ] ]

gap> w:=Sqrt(2);
E(8)-E(8)^3

gap> m:=[[E(8)^7/w,E(8)^7/w],[E(8)^5/w,E(8)/w]];
[ [ 1/2-1/2*E(4), 1/2-1/2*E(4) ], [ -1/2-1/2*E(4), 1/2+1/2*E(4) ] ]

gap> NormalSubgroups(Group(o,t,m));             
[ Group([ [ [ E(4), 0 ], [ 0, -E(4) ] ], [ [ 0, 1 ], [ -1, 0 ] ], 
      [ [ 1/2-1/2*E(4), 1/2-1/2*E(4) ], [ -1/2-1/2*E(4), 1/2+1/2*E(4) ] ] ]), 
  Group([ [ [ 0, E(4) ], [ E(4), 0 ] ], [ [ E(4), 0 ], [ 0, -E(4) ] ], 
      [ [ -1, 0 ], [ 0, -1 ] ] ]), 
  Group([ [ [ -1, 0 ], [ 0, -1 ] ] ]), Group([  ]) ]
\end{verbatim}
\normalsize

Then Singular gives the corresponding rings of invariants. The three output lines give the primary invariants (prim), the secondary invariants (sec) and 
the indecomposable, secondary invariants (irrsec). Then $k[x,y]^G$ is a finitely generated module over $k[\text{prim}]$ with module generators sec. As a 
$k$-algebra $k[x,y]^G$ is generated by prim and irrsec.

\scriptsize
\begin{verbatim}
> ring R=(0,Q),(x,y),lp;minpoly=rootofUnity(8);
> number w=Q-Q^3;
> w;
(-Q3+Q)
> matrix o[2][2]=Q2,0,0,-Q2;
> matrix t[2][2]=0,1,-1,0;
> matrix m[2][2]=Q7/w,Q7/w,Q5/w,Q/w;

> print(invariant_ring(o,t,m));
x5y-xy5,x8+14*x4y4+y8
1,x12-33*x8y4-33*x4y8+y12
x12-33*x8y4-33*x4y8+y12

> print(invariant_ring(o,o*t));
x2y2,x4+y4
1,x5y-xy5
x5y-xy5
\end{verbatim}
\normalsize

The normal subgroups of $\Obb$ are given by

\scriptsize
\begin{verbatim}
gap> k:=[[E(8),0],[0,E(8)^7]];
[ [ E(8), 0 ], [ 0, -E(8)^3 ] ]

gap> NormalSubgroups(Group(o,t,m,k));
[ Group([ [ [ E(4), 0 ], [ 0, -E(4) ] ], [ [ 0, 1 ], [ -1, 0 ] ], 
      [ [ 1/2-1/2*E(4), 1/2-1/2*E(4) ], [ -1/2-1/2*E(4), 1/2+1/2*E(4) ] ], 
      [ [ E(8), 0 ], [ 0, -E(8)^3 ] ] ]), 
  Group([ [ [ -1/2-1/2*E(4), 1/2-1/2*E(4) ], 
	    [ -1/2-1/2*E(4), -1/2+1/2*E(4) ] ], 
      [ [ 0, E(4) ], [ E(4), 0 ] ], 
      [ [ E(4), 0 ], [ 0, -E(4) ] ], [ [ -1, 0 ], [ 0, -1 ] ] ]), 
  Group([ [ [ 0, E(4) ], [ E(4), 0 ] ], [ [ E(4), 0 ], [ 0, -E(4) ] ], 
      [ [ -1, 0 ], [ 0, -1 ] ] ]), 
  Group([ [ [ -1, 0 ], [ 0, -1 ] ] ]), Group([  ]) ]
\end{verbatim}
\normalsize

The rings of invariants are

\scriptsize
\begin{verbatim}
> matrix k[2][2]=Q,0,0,Q7;

> print(invariant_ring(o,t,m,k));
x8+14*x4y4+y8,x10y2-2*x6y6+x2y10
1,x17y-34*x13y5+34*x5y13-xy17
x17y-34*x13y5+34*x5y13-xy17

> print(invariant_ring(o,o*t,k*k*k*k*k*k*k*m*k*o*o*o));
x5y-xy5,x8+14*x4y4+y8
1,x12-33*x8y4-33*x4y8+y12
x12-33*x8y4-33*x4y8+y12
\end{verbatim}
\normalsize

We compute the normal subgroups of $\I$.

\scriptsize
\begin{verbatim}
gap> w:=Sqrt(5);
E(5)-E(5)^2-E(5)^3+E(5)^4

gap> o:=[[-E(5)^3,0],[0,-E(5)^2]];
[ [ -E(5)^3, 0 ], [ 0, -E(5)^2 ] ]

gap> t:=[[(-E(5)+E(5)^4)/w,(E(5)^2-E(5)^3)/w],
	  [(E(5)^2-E(5)^3)/w,(E(5)-E(5)^4)/w]];
[ [ -1/5*E(5)-2/5*E(5)^2+2/5*E(5)^3+1/5*E(5)^4, 
    2/5*E(5)-1/5*E(5)^2+1/5*E(5)^3-2/5*E(5)^4 ], 
  [ 2/5*E(5)-1/5*E(5)^2+1/5*E(5)^3-2/5*E(5)^4, 
    1/5*E(5)+2/5*E(5)^2-2/5*E(5)^3-1/5*E(5)^4 ] ]

gap> NormalSubgroups(Group(o,t)); 
[ Group([  ]), <group of 2x2 matrices of size 2 in characteristic 0>, 
  <group of 2x2 matrices of size 120 in characteristic 0> ] 
\end{verbatim}
\normalsize

The ring of invariants is given by

\scriptsize
\begin{verbatim}
> ring R=(0,Q),(x,y),lp;minpoly=rootofUnity(5);
> matrix o[2][2]=-Q3,0,0,-Q2;
> number w=Q-Q2-Q3+Q4;
> matrix t[2][2]=(-Q+Q4)/w,(Q2-Q3)/w,(Q2-Q3)/w,(Q-Q4)/w;
> print(invariant_ring(o,t));
x11y+11*x6y6-xy11,x20-228*x15y5+494*x10y10+228*x5y15+y20
1,x30+522*x25y5-10005*x20y10-10005*x10y20-522*x5y25+y30
x30+522*x25y5-10005*x20y10-10005*x10y20-522*x5y25+y30
\end{verbatim}
\normalsize

\section{Macaulay2 computations in the $E_6$-case}\label{m2.mcmpb.e6}
\scriptsize
\begin{verbatim}
+ M2.exe --no-readline --print-width 79
Macaulay2, version 1.4
with packages: ConwayPolynomials, Elimination, IntegralClosure, LLLBases,
               PrimaryDecomposition, ReesAlgebra, TangentCone

i1 : A=ZZ[U,V,W]/(U^2+4*V^3-V*W^2)

i2 : X=W^3-36*V^2*W

i3 : Y=W^2+12*V^2

i4 : Z=U

i5 : S1=mingens image syz gens ideal(X,Y,Z)

o5 = {3} | 0       -V   W       -U   |
     {2} | -U      -2VW 12V2-W2 -3UW |
     {1} | 12V2+W2 3UW  36UV    4W3  |

             3       4
o5 : Matrix A  <--- A

i6 : S2=mingens image syz gens ideal(X,Y^2,Y*Z,Z^2)

o6 = {3} | 0       0       2VW  12V2-W2 UV   UW   -U2  |
     {4} | 0       -U      V    W       0    U    0    |
     {3} | -U      12V2+W2 3U   0       2VW  -2W2 -3UW |
     {2} | 12V2+W2 0       -6W2 -72VW   -3UW 36UV 4W3  |

             4       7
o6 : Matrix A  <--- A

i7 : N=mingens image syz gens ideal(X*Y,X*Z,Y^2,Y*Z^2,Z^3)

o7 = {5} | 0       -U      V   W       0        0    U     0    0       0   
     {4} | 0       12V2+W2 0   0       0        2VW  -2W2  UV   UW      -U2 
     {4} | 0       0       2VW 12V2-W2 U2       UV   UW    0    0       0   
     {4} | -U      0       -3W 36V     -12V2-W2 3U   0     2VW  12V2-W2 -3UW
     {3} | 12V2+W2 0       0   0       0        -6W2 -72VW -3UW 36UV    4W3 
     --------------------------------------------------------------------------
     -U2W       |
     3UW3       |
     0          |
     -2W4+27U2V |
     81U3       |

             5       11
o7 : Matrix A  <--- A

i8 : L=mingens image syz gens ideal(X,Y,Z^2)

o8 = {3} | V   W       0        U2       |
     {2} | 2VW 12V2-W2 U2       0        |
     {2} | -3W 36V     -12V2-W2 36V2W-W3 |

             3       4
o8 : Matrix A  <--- A
\end{verbatim}
\normalsize

\section{Macaulay2 computations in the $E_8$-case}\label{m2.mcmpb.e8}
\scriptsize
\begin{verbatim}
+ M2.exe --no-readline --print-width 79
Macaulay2, version 1.4
with packages: ConwayPolynomials, Elimination, IntegralClosure, LLLBases,
               PrimaryDecomposition, ReesAlgebra, TangentCone

i1 : A=ZZ[R,S,T]/(T^2-R*S)

i2 : X=R^15+522*R^10*T^5-10005*R^5*T^10-10005*S^5*T^10-522*S^10*T^5+S^15

i3 : Y=R^10-228*R^5*T^5+494*T^10+228*S^5*T^5+S^10

i4 : Z=R^5*T-S^5*T+11*T^6

i5 : S1=mingens image syz gens ideal(X,Y,Z)

i6 : S2=mingens image syz gens ideal(X,Y^2,Y*Z,Z^2)

i7 : S3=mingens image syz gens ideal(X*Y,X*Z,Y^2,Y*Z^2,Z^3)

i8 : S4=mingens image syz gens ideal(X*Y,X*Z^2,Y^3,Y^2*Z,Y*Z^3,Z^4)

i9 : S5=mingens image syz gens ideal(X*Y^2,X*Y*Z^2,X*Z^4,Y^4,Y^3*Z,Y^2*Z^3,Z^5)

i10 : S6=mingens image syz gens ideal(X,Y^2,Y*Z,Z^3)

i11 : S7=mingens image syz gens ideal(X*Y,X*Z,Y^2,Y*Z^2,Z^4)

i12 : S8=mingens image syz gens ideal(X,Y,Z^2)

i13 : B=QQ[R,S,T]/(T^2-R*S)

i14 : T1=map(B,A)**S1

i15 : T2=map(B,A)**S2

i16 : T3=map(B,A)**S3

i17 : T4=map(B,A)**S4

i18 : T5=map(B,A)**S5

i19 : T6=map(B,A)**S6

i20 : T7=map(B,A)**S7

i21 : T8=map(B,A)**S8

i22 : M1=mingens image T1

o22 = {15} | R                              -S                            
      {10} | -R6+11S4T2-66RT5               S6-11R4T2-66ST5               
      {6}  | 12S9T-684R6T4+2052S4T6+2964RT9 12R9T+684S6T4-2052R4T6+2964ST9
      -------------------------------------------------------------------------
      T                              -T                             |
      -R5T+11S5T-66T6                -11R5T+S5T-66T6                |
      12S10-684R5T5+2052S5T5+2964T10 12R10-2052R5T5+684S5T5+2964T10 |

              3       4
o22 : Matrix B  <--- B

i23 : M2=mingens image T2

o23 = {15} | -5R5T-5S5T               R6-11S4T2+66RT5              
      {20} | -T                       -R                           
      {16} | 6R5-6S5+216T5            60S4T-180RT4                 
      {12} | 3240R5T4-3240S5T4-9360T9 72S9-864R6T3+9072S4T5+8424RT8
      -------------------------------------------------------------------------
      S6-11R4T2-66ST5               |
      -S                            |
      -60R4T-180ST4                 |
      72R9+864S6T3-9072R4T5+8424ST8 |

              4       3
o23 : Matrix B  <--- B

i24 : M3=mingens image T3

o24 = {25} | -9T2                 -9T2                
      {21} | -4R5T+104S5T+756T6   -104R5T+4S5T+756T6  
      {20} | -35R5T2-55S5T2+110T7 -55R5T2-35S5T2-110T7
      {22} | 48R5+168S5-2160T5    168R5+48S5+2160T5   
      {18} | 64800S5T4+93600T9    64800R5T4-93600T9   
      -------------------------------------------------------------------------
      ST                          RT                         
      4S6+16R4T2-84ST5            -4R6-16S4T2-84RT5          
      3S6T+7R4T3+22ST6            3R6T+7S4T3-22RT6           
      -24R4T-432ST4               -24S4T+432RT4              
      2880S6T3-10080R4T5+18720ST8 2880R6T3-10080S4T5-18720RT8
      -------------------------------------------------------------------------
      R2                                 -S2                               
      -60S3T3+180R2T4                    -60R3T3-180S2T4                   
      -R7+11S3T4-66R2T5                  S7-11R3T4-66S2T5                  
      216S3T2-288R2T3                    -216R3T2-288S2T3                  
      288S8T-576R7T2+26208S3T6+14976R2T7 288R8T+576S7T2-26208R3T6+14976S2T7
      -------------------------------------------------------------------------
      3RT                      -3ST                     |
      -2R6-158S4T2+408RT5      -2S6-158R4T2-408ST5      |
      -R6T+31S4T3-176RT6       S6T-31R4T3-176ST6        |
      528S4T-504RT4            -528R4T-504ST4           |
      720S9+60480S4T5+28080RT8 720R9-60480R4T5+28080ST8 |

              5       8
o24 : Matrix B  <--- B

i25 : M4=mingens image T4

o25 = {25} | 5R5T2+5S5T2               2S6T+8R4T3+33ST6          
      {27} | -12R5-12S5                36R4T-252ST4              
      {30} | T2                        ST                        
      {26} | 6R5T-6S5T-84T6            3S6+27R4T2-18ST5          
      {28} | 720T4                     -72R4-576ST3              
      {24} | 7776R5T3-7776S5T3-22464T8 864S6T2-22464R4T4+33696ST7
      -------------------------------------------------------------------------
      -2R6T-8S4T3+33RT6          R7-11S3T4+66R2T5                
      -36S4T-252RT4              -216S3T2+288R2T3                
      -RT                        -R2                             
      3R6+27S4T2+18RT5           60S3T3-180R2T4                  
      -72S4+576RT3               576S3T-288R2T2                  
      864R6T2-22464S4T4-33696RT7 864S8-864R7T+56160S3T5+11232R2T6
      -------------------------------------------------------------------------
      S7-11R3T4-66S2T5                 |
      -216R3T2-288S2T3                 |
      -S2                              |
      -60R3T3-180S2T4                  |
      -576R3T-288S2T2                  |
      864R8+864S7T-56160R3T5+11232S2T6 |

              6       5
o25 : Matrix B  <--- B

i26 : M5=mingens image T5

o26 = {35} | S8+3R7T-39S3T5-143R2T6 R8-3S7T+39R3T5-143S2T6 25R5T3+25S5T3      
      {37} | 0                      0                      0                  
      {39} | 0                      0                      0                  
      {40} | -S3+3R2T               -R3-3S2T               5T3                
      {36} | -6R7-294S3T4+84R2T5    -6S7-294R3T4-84S2T5    42R5T2-42S5T2-288T7
      {38} | -720S3T2-1440R2T3      -720R3T2+1440S2T3      -72R5+72S5+1008T5  
      {30} | 0                      0                      0                  
      -------------------------------------------------------------------------
      9R6T2+41S4T4-176RT7                          
      -6R6+306S4T2+684RT5                          
      -720S4+1440RT3                               
      5RT2                                         
      -8R6T-152S4T3+42RT6                          
      -216S4T-1512RT4                              
      509760S9T4-976320R6T7+5088960S4T9+2790720RT12
      -------------------------------------------------------------------------
      9S6T2+41R4T4+176ST7                          
      -6S6+306R4T2-684ST5                          
      -720R4-1440ST3                               
      5ST2                                         
      8S6T+152R4T3+42ST6                           
      216R4T-1512ST4                               
      509760R9T4+976320S6T7-5088960R4T9+2790720ST12
      -------------------------------------------------------------------------
      50S5T3-275T8                                     
      -150R5T+150S5T+2100T6                            
      -18000T4                                         
      5T3                                              
      -83R5T2-167S5T2-288T7                            
      228R5+372S5+1008T5                               
      108000R10T3-108000S10T3-12312000R5T8-12312000S5T8
      -------------------------------------------------------------------------
      574S6T2+976R4T4+2211ST7                          
      -1071S6+1521R4T2-15894ST5                        
      -1080R4+125280ST3                                
      155ST2                                           
      -342S6T+2352R4T3-8433ST6                         
      -3924R4T+27468ST4                                
      764640S11T2+88633440S6T7+79535520R4T9+4186080ST12
      -------------------------------------------------------------------------
      574R6T2+976S4T4-2211RT7                          
      -1071R6+1521S4T2+15894RT5                        
      -1080S4-125280RT3                                
      155RT2                                           
      342R6T-2352S4T3-8433RT6                          
      3924S4T+27468RT4                                 
      764640R11T2-88633440R6T7-79535520S4T9+4186080RT12
      -------------------------------------------------------------------------
      -7S7T-43R3T5-198S2T6                                    
      -360R3T3+720S2T4                                        
      864R3T+2592S2T2                                         
      -5S2T                                                   
      -12S7-168R3T4-108S2T5                                   
      -288R3T2+2016S2T3                                       
      5184S12T-616896R8T5-580608S7T6+6697728R3T10-3348864S2T11
      -------------------------------------------------------------------------
      -7R7T-43S3T5+198R2T6                                    
      -360S3T3-720R2T4                                        
      864S3T-2592R2T2                                         
      -5R2T                                                   
      12R7+168S3T4-108R2T5                                    
      288S3T2+2016R2T3                                        
      5184R12T-616896S8T5+580608R7T6-6697728S3T10-3348864R2T11
      -------------------------------------------------------------------------
      R7T+49S3T5-264R2T6                                     
      720S3T3-360R2T4                                        
      2592S3T+864R2T2                                        
      5R2T                                                   
      -6R7-234S3T4+504R2T5                                   
      -1584S3T2-288R2T3                                      
      5184S13-580608S8T5-616896R7T6-3348864S3T10+6697728R2T11
      -------------------------------------------------------------------------
      -S7T-49R3T5-264S2T6                                     |
      -720R3T3-360S2T4                                        |
      -2592R3T+864S2T2                                        |
      -5S2T                                                   |
      -6S7-234R3T4-504S2T5                                    |
      -1584R3T2+288S2T3                                       |
      5184R13+580608R8T5-616896S7T6-3348864R3T10-6697728S2T11 |

              7       12
o26 : Matrix B  <--- B

i27 : M6=mingens image T6

o27 = {15} | S8+3R7T-39S3T5-143R2T6 R8-3S7T+39R3T5-143S2T6 25R5T3+25S5T3      
      {20} | -S3+3R2T               -R3-3S2T               5T3                
      {16} | -6R7-294S3T4+84R2T5    -6S7-294R3T4-84S2T5    42R5T2-42S5T2-288T7
      {18} | -720S3T2-1440R2T3      -720R3T2+1440S2T3      -72R5+72S5+1008T5  
      -------------------------------------------------------------------------
      |
      |
      |
      |

              4       3
o27 : Matrix B  <--- B

i28 : M7=mingens image T7

o28 = {25} | -S3-2R2T              -R3+2S2T              5RT2                
      {21} | 4R7-104S3T4-156R2T5   4S7-104R3T4+156S2T5   4R6T-104S4T3-156RT6 
      {20} | S8-2R7T-84S3T5+77R2T6 R8+2S7T+84R3T5+77S2T6 15R6T2+35S4T4-110RT7
      {22} | -480S3T3-360R2T4      480R3T3-360S2T4       -24R6+144S4T2+576RT5
      {24} | -288S3T-2016R2T2      288R3T-2016S2T2       -288S4-2016RT3      
      -------------------------------------------------------------------------
      5ST2                 |
      -4S6T+104R4T3-156ST6 |
      15S6T2+35R4T4+110ST7 |
      -24S6+144R4T2-576ST5 |
      -288R4+2016ST3       |

              5       4
o28 : Matrix B  <--- B

i29 : M8=mingens image T8

o29 = {15} | -R4+7ST3                 -S4-7RT3                
      {10} | R9+17S6T3+119R4T5+187ST8 S9-17R6T3-119S4T5+187RT8
      {12} | -24S6T+624R4T3-936ST6    24R6T-624S4T3-936RT6    
      -------------------------------------------------------------------------
      -R3T+7S2T2                 -S3T-7R2T2                 |
      R8T+17S7T2+119R3T6+187S2T7 S8T-17R7T2-119S3T6+187R2T7 |
      -24S7+624R3T4-936S2T5      24R7-624S3T4-936R2T5       |

              3       4
o29 : Matrix B  <--- B

i30 : inducedMap(image M1,image T1)

o30 = {16} | 1 0 0 0    |
      {16} | 0 1 0 0    |
      {16} | 0 0 1 1/12 |
      {16} | 0 0 0 1/12 |

o30 : Matrix

i31 : inducedMap(image M2,image T2)

o31 = {21} | 1 0 0 -1/6R -1/6S -1/6T -1/36T |
      {21} | 0 1 0 1/6T  0     1/6S  1/72S  |
      {21} | 0 0 1 0     1/6T  0     1/72R  |

o31 : Matrix

i32 : inducedMap(image M3,image T3)

o32 = {27} | 1 19/180 0 0 0 0 0     0    17/48R -11/168S 409/1800T 97/900T |
      {27} | 0 1/180  0 0 0 0 0     0    0      0        1/1800T   1/300T  |
      {27} | 0 0      1 0 0 0 0     3/10 0      409/280T 0         0       |
      {27} | 0 0      0 1 0 0 -3/10 0    11/80T 0        0         0       |
      {27} | 0 0      0 0 1 0 0     0    0      0        0         0       |
      {27} | 0 0      0 0 0 1 0     0    0      0        0         0       |
      {27} | 0 0      0 0 0 0 1/10  0    1/60T  0        1/60S     1/720S  |
      {27} | 0 0      0 0 0 0 0     1/10 0      1/60T    0         1/720R  |

o32 : Matrix

i33 : inducedMap(image M4,image T4)

o33 = {32} | 1 0 0 0 0 -5/3T 5/6T    -1/3S 1/3R  0    0     -1/12R -1/12S
      {32} | 0 1 0 0 0 -2/3R 1/4R    -8/3T 0     0    -2/3S 0      -5/6T 
      {32} | 0 0 1 0 0 -7/3S -11/12S 0     -8/3T 2/3R 0     5/6T   0     
      {32} | 0 0 0 1 0 0     0       0     0     1/3T 0     1/12S  0     
      {32} | 0 0 0 0 1 0     0       0     0     0    1/3T  0      1/12R 
      -------------------------------------------------------------------------
      23/24RT -29/24ST -3/8T2   167/144T2 RT2          9/2R3T-13/2S2T2  |
      25/72R2 -3/8T2   -11/72RT 91/216RT  5/3S3+1/3R2T 3/2R4+227/6ST3   |
      7/24T2  29/72S2  -13/24ST 125/216ST 14/3T3       13/6S3T+61/6R2T2 |
      1/72ST  0        1/72S2   1/864S2   5/3ST2       -85/6RT3         |
      0       1/72RT   0        1/864R2   -10/3S2T     -155/6T4         |

o33 : Matrix

i34 : inducedMap(image M5,image T5)

o34 = {43} | 1 0 0 0 0 0     0       0       0 0 0   0    3/2T 1/2R  0       
      {43} | 0 1 0 0 0 0     0       0       0 0 0   0    1/2S -3/2T 0       
      {43} | 0 0 1 0 0 -1/25 0       0       0 0 0   0    0    0     59/375T 
      {43} | 0 0 0 1 0 0     0       -31/295 0 0 0   0    0    0     -1/30S  
      {43} | 0 0 0 0 1 0     -31/295 0       0 0 0   0    0    0     1/30R   
      {43} | 0 0 0 0 0 1/25  0       0       0 0 0   0    0    0     -59/375T
      {43} | 0 0 0 0 0 0     1/295   0       0 0 0   0    0    0     0       
      {43} | 0 0 0 0 0 0     0       1/295   0 0 0   0    0    0     0       
      {43} | 0 0 0 0 0 0     0       0       1 0 0   -1/6 0    0     0       
      {43} | 0 0 0 0 0 0     0       0       0 1 1/6 0    0    0     0       
      {43} | 0 0 0 0 0 0     0       0       0 0 1/6 0    0    0     0       
      {43} | 0 0 0 0 0 0     0       0       0 0 0   1/6  0    0     0       
      -------------------------------------------------------------------------
      0     0        0          0          -T    T     -1/3R  0       0     
      0     0        0          0          -1/3S 1/3S  T      0       0     
      -1/2S -41/150T 1/18S      1/18R      0     0     0      -13/72R 13/72S
      0     29/120S  0          -119/1770T 1/2R  -1/3R 0      1/6T    0     
      1/2T  -1/120R  -119/1770T 0          0     0     1/3S   0       -1/6T 
      0     1/25T    0          0          0     0     0      0       0     
      0     0        1/2655T    0          0     0     0      0       0     
      0     0        0          1/2655T    0     0     0      0       0     
      0     0        0          0          0     0     -1/18T 0       0     
      0     0        0          0          1/6T  1/18T 0      0       0     
      0     0        0          0          0     1/18T 0      1/72S   0     
      0     0        0          0          0     0     1/18T  0       1/72R 
      -------------------------------------------------------------------------
      ST             0           0             0           -20T3        
      1/3S2          0           0             0           35/6ST2      
      -47/720RT      19/144ST    -7841/54000T2 53/540T2    3/2S3+13/6R2T
      -10139/42480T2 1/144S2     151/720ST     -181/1440ST 0            
      1/720R2        233/42480T2 13/720RT      -11/1440RT  0            
      0              0           -191/2250T2   19/540T2    0            
      0              -73/15930T2 0             0           0            
      -17/15930T2    0           0             0           0            
      0              0           0             0           0            
      0              0           0             0           0            
      1/432ST        0           1/432S2       1/5184S2    0            
      0              1/432RT     0             1/5184R2    0            
      -------------------------------------------------------------------------
      -2R3T-7/2S2T2    893/72T4                -823/48RT3              |
      10/3S3T-15/2R2T2 -823/48ST3              -893/72T4               |
      109/15T4         -233/720S3T-773/720R2T2 233/720R3T-773/720S2T2  |
      3/2R4-3ST3       11/80S4+1171/80RT3      -37/240S3T+1637/240R2T2 |
      3/2S4+3RT3       -37/240R3T-1637/240S2T2 -11/80R4+1171/80ST3     |
      -3/5T4           0                       0                       |
      0                0                       0                       |
      0                0                       0                       |
      0                0                       2183/144T4              |
      0                2183/144T4              0                       |
      0                -649/48T4               0                       |
      0                0                       649/48T4                |

o34 : Matrix

i35 : inducedMap(image M6,image T6)

o35 = {23} | 1 0 0 3/2T 1/2R  -5/6R2 -5/2ST -5/6RT -5/2T2 -20T3        
      {23} | 0 1 0 1/2S -3/2T 5/2RT  -5/6S2 5/2T2  -5/6ST 35/6ST2      
      {23} | 0 0 1 0    0     1/3ST  1/3RT  1/3S2  1/3R2  3/2S3+13/6R2T
      -------------------------------------------------------------------------
      -35/6RT2      7/72R3T2+23/36S2T3    |
      -20T3         7/72S3T2-23/36R2T3    |
      3/2R3-13/6S2T 1/72R5-1/72S5+53/36T5 |

o35 : Matrix

i36 : inducedMap(image M7,image T7)

o36 = {28} | 1 0 0 0 -3/2T 3/4R  0     27/8T2 -39/8RT 1/4R2  5/24R2T       
      {28} | 0 1 0 0 3/4S  3/2T  0     39/8ST 27/8T2  1/2RT  -5/24RT2      
      {28} | 0 0 1 0 -1/4R 0     7/4S  1/8RT  -7/8R2  1/2S2  1/24R3-1/12S2T
      {28} | 0 0 0 1 0     -1/4S -3/4R -7/8S2 -1/8ST  -5/4T2 5/24T3        
      -------------------------------------------------------------------------
      5/24ST2         5/24RT2         5/24T3          -5/8T4          
      5/24S2T         -5/24T3         5/24ST2         -5/6ST3         
      5/24T3          -1/12S3+1/24R2T 5/24RT2         -1/24S4-17/24RT3
      -1/24S3-1/12R2T 5/24ST2         -1/12R3-1/24S2T 7/24R3T+1/4S2T2 
      -------------------------------------------------------------------------
      5/6RT3           -5/72R3T3+5/96S2T4       |
      -5/8T4           5/72S3T3+5/96R2T4        |
      -7/24S3T+1/4R2T2 1/288S6-1/48R4T2+1/12ST5 |
      -1/24R4+17/24ST3 1/288R6-1/48S4T2-1/12RT5 |

o36 : Matrix

i37 : inducedMap(image M8,image T8)

o37 = {19} | 1 0 0 0 0      0     -493/24RT2  |
      {19} | 0 1 0 0 0      0     -3757/24ST2 |
      {19} | 0 0 1 0 1/2S   17/2T 221/24S3    |
      {19} | 0 0 0 1 -17/2T 1/2R  -29/24R3    |

o37 : Matrix

i38 : C=ZZ/(7)[R,S,T]/(T^2-R*S)

i39 : T3=map(C,A)**S3

i40 : M3=mingens image T3

o40 = {25} | -2T2     -2T2     -2ST      -2RT      R2             
      {21} | 3R5T-S5T R5T-3S5T -S6+3R4T2 R6-3S4T2  3S3T3-2R2T4    
      {20} | S5T2-2T7 R5T2+2T7 S6T-2ST6  R6T+2RT6  -R7-3S3T4-3R2T5
      {22} | -R5+3T5  -S5-3T5  -R4T+3ST4 -S4T-3RT4 -S3T2-R2T3     
      {18} | S5T4+3T9 R5T4-3T9 S6T3+3ST8 R6T3-3RT8 S8T-2R7T2+3R2T7
      -------------------------------------------------------------------------
      -S2             -3RT           3ST            |
      3R3T3+2S2T4     2R6-3S4T2-2RT5 2S6-3R4T2+2ST5 |
      S7+3R3T4-3S2T5  R6T-3S4T3+RT6  -S6T+3R4T3+ST6 |
      R3T2-S2T3       -3S4T          3R4T           |
      R8T+2S7T2+3S2T7 S9-3RT8        R9-3ST8        |

              5       8
o40 : Matrix C  <--- C

i41 : D=ZZ/(59)[R,S,T]/(T^2-R*S)

i42 : T5=map(D,A)**S5

i43 : M5=mingens image T5

o43 = {35} | S8+3R7T+20S3T5-25R2T6 R8-3S7T-20R3T5-25S2T6 R5T3+S5T3         
      {37} | 0                     0                     0                 
      {39} | 0                     0                     0                 
      {40} | -S3+3R2T              -R3-3S2T              12T3              
      {36} | -6R7+S3T4+25R2T5      -6S7+R3T4-25S2T5      -29R5T2+29S5T2+5T7
      {38} | -12S3T2-24R2T3        -12R3T2+24S2T3        16R5-16S5+12T5    
      {30} | 0                     0                     0                 
      -------------------------------------------------------------------------
      14S6T2-28R4T4+5ST7 -14R6T2+28S4T4+5RT7 -18S5T3-19T8           
      -29S6+4R4T2-2ST5   29R6-4S4T2-2RT5     -5R5T+5S5T+11T6        
      R4+2ST3            -S4+2RT3            -10T4                  
      -25ST2             25RT2               10T3                   
      19S6T+7R4T3+26ST6  19R6T+7S4T3-26RT6   11R5T2+20S5T2+14T7     
      -18R4T+8ST4        -18S4T-8RT4         -16R5-23S5+10T5        
      S6T7-12R4T9+18ST12 R6T7-12S4T9-18RT12  R10T3-S10T3+4R5T8+4S5T8
      -------------------------------------------------------------------------
      -2S6T2-11R4T4-20ST7     -2R6T2-11S4T4+20RT7    
      -3S6+21R4T2-19ST5       -3R6+21S4T2+19RT5      
      4R4+18ST3               4S4-18RT3              
      -19ST2                  -19RT2                 
      -24S6T-21R4T3+5ST6      24R6T+21S4T3+5RT6      
      26R4T-5ST4              -26S4T-5RT4            
      S11T2-R9T4+7R4T9+13ST12 R11T2-S9T4-7S4T9+13RT12
      -------------------------------------------------------------------------
      23S7T-2R3T5+10S2T6            23R7T-2S3T5-10R2T6           
      -14R3T3+28S2T4                -14S3T3-28R2T4               
      10R3T-29S2T2                  10S3T+29R2T2                 
      8S2T                          8R2T                         
      -28S7+21R3T4-16S2T5           28R7-21S3T4-16R2T5           
      -23R3T2-16S2T3                23S3T2-16R2T3                
      S12T-R8T5+6S7T6-6R3T10+3S2T11 R12T-S8T5-6R7T6+6S3T10+3R2T11
      -------------------------------------------------------------------------
      22R7T+16S3T5-26R2T6          -22S7T-16R3T5-26S2T6         |
      28S3T3-14R2T4                -28R3T3-14S2T4               |
      -29S3T+10R2T2                29R3T+10S2T2                 |
      -8R2T                        8S2T                         |
      -14R7-15S3T4-4R2T5           -14S7-15R3T4+4S2T5           |
      21S3T2-23R2T3                21R3T2+23S2T3                |
      S13+6S8T5-R7T6+3S3T10-6R2T11 R13-6R8T5-S7T6+3R3T10+6S2T11 |

              7       12
o43 : Matrix D  <--- D
\end{verbatim}
\normalsize
\end{appendix}

\nocite{CocoaSystem}

\begin{flushleft}
\bibliographystyle{alpha}
\bibliography{bibfile}
\end{flushleft}
\thispagestyle{empty}

\end{document}